\documentclass[11pt,american,english]{article}

\usepackage[english]{babel}
\usepackage{constants}
\usepackage[utf8x]{inputenc}
\usepackage[T1]{fontenc}
\usepackage{blindtext}
\usepackage{titlesec}
\usepackage{float}
\usepackage{xcolor}
\usepackage{enumerate}
\restylefloat{table}
\usepackage{authblk}
\usepackage[a4paper,top=3cm,bottom=3cm,left=2cm,right=1.5cm,marginparwidth=1.75cm]{geometry}
\usepackage{amsmath}
\usepackage{bm}
\usepackage{graphicx}
\usepackage[colorinlistoftodos]{todonotes}
\usepackage[colorlinks=true, allcolors=blue]{hyperref}
\usepackage{tikz}
\usepackage{caption}
\usetikzlibrary{decorations.pathreplacing}
\numberwithin{equation}{section}
\newtheorem{theorem}{Theorem}[section]
\newtheorem{corollary}[theorem]{Corollary}
\newtheorem{prpstn}[theorem]{Proposition}
\newtheorem{lemma}[theorem]{Lemma}
\newtheorem{definition}[theorem]{Definition}

\newtheorem{remark}[theorem]{Remark}

\usepackage{amssymb}
\usepackage{mathrsfs}
\usepackage{color}
\newcommand{\qed}{\nobreak \ifvmode \relax \else
      \ifdim\lastskip<1.5em \hskip-\lastskip
      \hskip1.5em plus0em minus0.5em \fi \nobreak
      \vrule height0.75em width0.5em depth0.25em\fi}
\newenvironment{proof}[1][Proof]{\begin{trivlist}
\item[\hskip \labelsep {\bfseries #1}]}{\hfill\qed\end{trivlist}}
\usepackage{dsfont}
\def\no{\nonumber}
\usepackage{mathtools}
\usepackage{stmaryrd}

\newcommand{\R}{\mathbb{R}}

\newcommand{\Z}{\mathbb{Z}}
\newcommand{\ot}{\#}

\newcommand{\p}{\partial}
\newcommand{\mc}{\mathcal}

\renewcommand{\/}[2]{\frac{#1}{#2}}
\newcommand{\fr}{\frac}

\def\eps{\epsilon}
\def\wt{\widetilde}
\def\wh{\widehat}
\def\wb{\overline}
\DeclareFontFamily{U}{mathx}{}
\DeclareFontShape{U}{mathx}{m}{n}{<-> mathx10}{}
\DeclareSymbolFont{mathx}{U}{mathx}{m}{n}
\DeclareMathAccent{\widecheck}{0}{mathx}{"71}
\newcommand{\vertiii}[1]{{\left\vert\kern-0.25ex\left\vert\kern-0.25ex\left\vert #1 
    \right\vert\kern-0.25ex\right\vert\kern-0.25ex\right\vert}}
\newcommand{\Lip}{\mathrm{Lip}}

\allowdisplaybreaks[1]

\author[,1]{\normalsize Alain Bensoussan\footnote{E-mail: axb046100@utdallas.edu.}}
\author[,2]{\normalsize Tak Kwong Wong\footnote{E-mail: takkwong@maths.hku.hk.}}
\author[,3]{\normalsize Sheung Chi Phillip Yam\footnote{E-mail: scpyam@sta.cuhk.edu.hk.}}
\author[,3,4]{\normalsize Hongwei Yuan\footnote{E-mail: hwyuan@link.cuhk.edu.hk; hwyuan@um.edu.mo.}}
\affil[1]{\small\it Naveen Jindal School of Management, University of Texas at Dallas}
\affil[2]{\small\it Department of Mathematics, The University of Hong Kong}
\affil[3]{\small\it Department of Statistics, The Chinese University of Hong Kong}
\affil[4]{\small\it Department of Mathematics, University of Macau}

\title{A Theory of First Order Mean Field Type Control Problems and their Equations}
\begin{document}
\maketitle

\begin{abstract}
In this article, by using several new crucial {\it a priori} estimates which are still absent in the literature, we provide a comprehensive resolution of the first order generic mean field type control problems and also establish the global-in-time existence and uniqueness of classical solutions of their Bellman and master equations. 
Rather than developing the analytical approach via tackling the Bellman and master equation directly, we apply the maximum principle approach by considering the induced forward-backward ordinary differential equation (FBODE) system;
indeed, we first show the local-in-time unique existence of the solution of the FBODE system for a variety of terminal data by Banach fixed point argument, and then provide crucial {\it a priori} estimates of bounding the sensitivity of the terminal data for the backward equation by utilizing a monotonicity condition that can be deduced from the positive definiteness of the Schur complement of the Hessian matrix of the Lagrangian in the lifted version and manipulating first order condition appropriately; this uniform bound over the whole planning horizon $[0,T]$ allows us to partition $[0,T]$ into a finite number of sub-intervals with a common small length and then glue the consecutive local-in-time solutions together to form the unique global-in-time solution of the FBODE system. The regularity of the global-in-time solution follows from that of the local ones due to the regularity assumptions on the coefficient functions. Moreover, the regularity of the value function will also be shown with the aid of the regularity of the solution couple of the FBODE system and the regularity assumptions on the coefficient functions, with which we can further deduce that this value function and its linear functional derivative satisfy the Bellman and master equations, respectively; further analysis of the unique nature of the FBODE solution implies the uniqueness of the classical solutions of these respective equations. Finally, to illustrate the effectiveness of our proposed general theory, we also provide the resolution of a non-trivial non-linear-quadratic example with non-separable Hamiltonian which cannot yet be explained in the contemporary literature. 	
\end{abstract}
\begin{center}
\begin{minipage}{5.5in}
{\bf Mathematics Subject Classifications (2020)}: Primary 93E20; Secondary 49J55, 49N80, 60H30\\
{\bf Key words}: Mean Field Type Control; Wasserstein Metric Space; Forward-Backward Differential Systems; Decoupling Field; Jacobian Flows; General Monotonicity Conditions; Bellman Equations; Master Equations; Non-linear-quadratic Examples.
\end{minipage}
\end{center}

\tableofcontents

\section{Introduction}
Mean field type control problems and games are a branch of mathematical control theory and game theory that deals with systems involving a large number of interacting agents. It focuses on studying the behavior of the system as the number of agents tends to infinity, leading to a mean field approximation. This approach allows for the analysis and understanding of complex systems with a large population of agents. In recent years, there has been significant progress in the theory of mean field type control and games, driven by both theoretical developments and practical applications. Particularly, the recent monographs of  
Cardaliaguet, Delarue, Lasry and Lions~\cite{cardaliaguet2019master}, Carmona and Delarue~\cite{carmona2018probabilistic}, Gomes, Pimentel and Voskanyan~\cite{gomes2016regularity} and Bensoussan, Frehse and Yam \cite{bensoussan2013mean} have made substantial contributions in understanding the dynamics, equilibrium, and optimization problems in such systems.
For a comprehensive understanding of the stochastic (Pontryagin) maximum principle, McKean-Vlasov dynamics, forward-backward stochastic differential equations (FBSDEs), and their connection to mean field type control and games, the following references can be consulted: \cite{andersson2011maximum} \cite{bensoussan2020control} \cite{bensoussan2016linear} \cite{homan2023game} \cite{homan2023control} \cite{bensoussan2015well} \cite{buckdahn2011general} \cite{buckdahn2009meanLim} \cite{buckdahn2017meanSDC} \cite{buckdahn2009mean} \cite{buckdahn2017mean} \cite{carmona2015forward} \cite{chassagneux2022probabilistic} \cite{cosso2023optimal} \cite{cosso2019zero} \cite{djete2022mckean} \cite{huang2006large} \cite{pardoux2005backward} \cite{pham2016linear} \cite{PHW} \cite{pham2018bellman}. For a deeper understanding of the dynamic programming principle, Bellman equations, master equations, and their connection to mean field type control and games, the following references can be helpful: \cite{bensoussan2015master} \cite{bensoussan2019control} \cite{cardaliaguet2022splitting} \cite{cosso2021master} \cite{gangbo2020global} \cite{gangbo2022mean} \cite{gangbo2015existence} \cite{lauriere2014dynamic} \cite{li2012stochastic}.
All these results contribute to the rigorous treatment of mean field games, mean field type control problems and their associated equations. 

In this paper, we consider a generic dynamics where the drift function $f(x,\mu,\alpha)$ depends on not only controls $\alpha$ but also the state variables $x$ and mean-field measure term $\mu$,
while the existing literature usually considers simpler cases, such as $f(x,\mu,\alpha)\equiv\alpha$. In addition, the running cost function $g(x,\mu,\alpha)$ that considered here 
can be both non-separable and of quadratic-growth in the whole space of $\R^{d_x}\times \mc{P}_2(\R^{d_x})$ in $(x,\mu)$, and thus altogether the Hamiltonian $\displaystyle h(x,\mu,z)= \min_{\alpha} \left(f(x,\mu,\alpha)\cdot z  + g(x,\mu,\alpha)\right)$ can also be non-separable and of quadratic-growth in the whole space $\R^{d_x}\times \mc{P}_2(\R^{d_x})\times \R^{d_x}$ in $(x,\mu,z)$, while the existing literature mostly considers compact torus space $x\in\mathbb{T}$ (see \cite{cardaliaguet2019master} for example) or only allows the Hamiltonian of linear-growth but not of quadratic-growth at the far-end filed of the space of $\R^{d_x}\times \mc{P}_2(\R^{d_x})$ (see \cite{gangbo2022mean} for example). Roughly speaking, we consider that the drift function $f(x,\mu,\alpha)$ is Lipschitz continuous and grows at most linearly, the running cost function $g(x,\mu,\alpha)$ is convex in the control $\alpha$ and satisfies a monotonicity condition that can be deduced from the positive definiteness of the Schur complement of the Hessian matrix of the Lagrangian in the lifted version (see the discussion in Section \ref{sec:motivation} for details). Our newly proposed conditions \eqref{positive_g_mu} and \eqref{positive_k_mu} are closely related to Lasry-Lions monotonicity condition and displacement monotonicity condition, which are respectively first proposed in the interesting and important works of \cite{cardaliaguet2019master,carmona2018probabilistic} and \cite{gangbo2022mean,gangbo2020global}; see Section \ref{assumptions} for our precise assumptions on the coefficient functions and Section \ref{subsec:dis} for the relationship and comparison between our convexity assumptions and others. Under these assumptions, we have the following main results: the global-in-time solvability of the generic mean field type control problems, as well as the global-in-time existence, uniqueness and classical regularity of the corresponding forward-backward ordinary differential equation (FBODE) system (see Theorem \ref{GlobalSol} for existence and Theorem \ref{Thm6_4} for uniqueness), and that of the solutions of Bellman (see Theorem \ref{GlobalSol_Bellman}) and master equations (see Theorem \ref{thm_master}). 
It follows from a verification theorem (see Theorem \ref{Verification}) that the value function defined by \eqref{definev} solves the generic mean field type control problems (see \eqref{SDE1}-\eqref{valfun}). 
Finally, to illustrate the effectiveness of our proposed general theory, we also provide the resolution of a non-trivial non-linear-quadratic example (see Section \ref{sec:nonLQ}) with non-separable Hamiltonian which can unlikely be covered in the contemporary literature.

Our strategy of solving the generic mean field type control problems (see \eqref{SDE1}-\eqref{valfun}) is as follows. Firstly, in contrast to the simpler cases such as $f(x,\mu,\alpha)\equiv\alpha$, in our generic setting of the drift and running cost functions, the optimal control cannot be solved explicitly; so we introduce a concept of {\it cone condition} (see \eqref{c_k_0}), which is a property but not an assumption that our constructed solution can be shown to fulfill (see \eqref{eq_7_44}), and it can be incorporated with the assumption on the asymptotic behavior of second order derivatives of drift function (see \eqref{bdd_d2_f}) to guarantee the solvability of the optimal control from the first order condition (see \eqref{first_order_condition}) by using a generalized version of implicit function theorem (see Theorem \ref{GIFT} in Appendix) that includes measure arguments. Second, we show a verification theorem and a maximum principle (see Theorem \ref{MPBE}) to deduce solving the generic mean field type control problems (see \eqref{SDE1}-\eqref{valfun}) to solving a FBODE system (see \eqref{MPIT}). Third, we use fixed point arguments to prove the local-in-time existence, uniqueness and regularity of solutions to a more general FBODE system \eqref{fbodesystem}, which is basically the same as the FBODE system \eqref{MPIT} except imposing a more general terminal data $p(x,\mu)$. The length of the existing time of local solution depends mainly on the bound of the derivatives of the terminal data $p(x,\mu)$, as well as some other uniform-in-time constants related to the given drift and running cost functions; see Theorem \ref{Thm6_1} and \ref{Thm6_2}. Next, to paste the local-in-time solution together to obtain the global-in-time classical solution of the FBODE system \eqref{MPIT}, the key point is to obtain {\it a priori} uniform-in-time estimates on the derivatives of the backward dynamics $Z^{t,m}_s(x)$ in initial data $x\in\R^{d_x}$ and $m\in\mc{P}_2(\R^{d_x})$, since $Z^{t,m}_{t_i}(x)$, $i=1,2,...,N+1$ are the terminal data $p$ on the sub-interval $[t_{i-1},t_i]$. Under the guidance of the calculus in the lifted Hilbert space which is much simplified and provides us useful insights of the algebraic manipulations and estimations (see Section \ref{sec:motivation}), we indeed can obtain our new crucial {\it a priori} estimates (see Theorem \ref{Crucial_Estimate}), under the assumptions that the Schur complement is positive-definite (also see \eqref{positive_g_x} and \eqref{positive_g_mu}) and the terminal cost function is convex (also see \eqref{positive_k} and \eqref{positive_k_mu}); these two assumptions are analogous to Lasry-Lions monotonicity conditions and displacement monotonicity conditions. Therefore, using both the local-in-time results and the new crucial {\it a priori} estimates, we can prove our global-in-time results. The ideas of proving local-in-time solutions of FBODE system and then pasting them together to obtain the global-in-time solution if one can uniformly bound the Lipschitz constant or first order derivatives of backward dynamics in initial data, can also be found in the interesting work of \cite{chassagneux2022probabilistic} where the authors only considered linear dynamics. In that work, the authors also give a bound on that Lipschitz constant under some relatively strong assumptions which seem not to be easily relaxed to deal with generic cases with more general dynamics. In our paper, we can deal with generic nonlinear dynamics since we can provide {\it a priori} uniform-in-time estimates on the derivatives of the backward dynamics $Z^{t,m}_s(x)$ in initial data $x\in\R^{d_x}$ and $m\in\mc{P}_2(\R^{d_x})$, which is a crucial aspect of our paper.  

Our approach is novel at least in the following four aspects. (i) We find an optimal flow of state variables $X_s$ instead of finding an optimal flow of measures $\mu_s$ (the classical approach); the state variables $X_s$ push the initial measure $m$ forward to be the optimal flow of measures $\mu_s=X_s\ot m$ by a push-forward operator $\ot$; this is inspired by the interesting idea of ``lifting'', first introduced by P.-L.~Lions \cite{LionsLecture2}, but our present approach only lifts a much smaller subset of $L^2$-random variables but not all; this approach is further developed and enhanced in \cite{bensoussan2020control}, \cite{homan2023game} and \cite{homan2023control}. (ii) We first propose the concept of {\it cone condition}s to deal with the unique solvability of optimal controls; in contrast to simpler cases such as $f(x,\mu,\alpha):=\alpha$ which the existing literature usually considers, the optimal controls can be explicitly solved from the first order condition, and hence, the {\it cone condition}s are not necessary for the simpler cases. (iii) We provide a way to manipulate the first order condition efficiently in the estimations to compensate the lack of explicit formula of the optimal control $\alpha$; as a direct consequence, the Schur complement of the Hessian matrix of the Lagrangian, which has not be found in the existing literature, arises naturally. (iv) Our new crucial {\it a priori} estimates on the derivatives of the backward dynamics $Z^{t,m}_s(x)$ are first proven in this work, and are the key estimates that guarantee the {\it cone conditions} and our the global-in-time results.

The direct consequences and advantages of our novelties are as follows. (i) By using the push-forward operator $\ot$, we can transform the study of $\mu_s$ in Wasserstein metric space to $X_s$ in the Hilbert space $L^2_m$, and avoid the difficulties caused by the lack of linear structure property in Wasserstein metric space. (ii) In principle, the optimal control $\wh{\alpha}(x,\mu,z)$ is solved from the first order condition $\p_\alpha f(x,\mu,\wh{\alpha})\cdot z+\p_\alpha g(x,\mu,\wh{\alpha})=0$ where $f$ is the draft function and $g$ is the running cost; since the drift function $f$ can be non-linear in control (namely $\p_\alpha\p_\alpha f(x,\mu,\alpha)\neq 0$) in our generic setting, when the absolute value of variable $z$ is large, we are not able to apply the implicit function theorem to guarantee the unique solvability of the optimal control from the first order condition, even if we assume $g(x,\mu,\alpha)$ is convex in $\alpha$; however, the proposed {\it cone condition}, which is shown to be fulfilled by our constructed solution of the FBODE system, basically suggests that for any given $(x,\mu)$, only those $z$'s in a bounded subset are related to the solutions of FBODE,  so we can still apply the implicit function theorem in the related region while solving the FBODE. (iii) While deriving our new crucial {\it a priori} estimations, one of the technical difficulty is the lack of explicit formula for the optimal control $\wh{\alpha}$, and hence, the derivatives of $\wh{\alpha}$ have to be obtained via differentiating the first order condition implicitly. Using these identities coming from the first order condition, we can carry on the estimations without the explicit expression of $\wh{\alpha}$. In the course of the estimations, even though the expressions are messy, from the perspective inspired by the calculus in the Hilbert space approach, the crucial term is just the Schur complement. As a result, we can use an appropriate positivity condition on the Schur complement to establish the new crucial estimates. (iv) The solution of FBODE system is proven to fulfill the {\it cone condition}s on the whole time interval $[0,T]$ and satisfy the crucial {\it a priori} estimates at the same time. By using our new crucial {\it a priori} estimates on the derivatives of the backward dynamics $Z^{t,m}_s(x)$, we can obtain an uniform lower bound on the length of the existing time of local solutions in different time sub-intervals so that we can glue these local solutions to construct a global solution in the whole time interval $[0,T]$ for any fixed large time $T$. To the best of our knowledge, showing the global-in-time existence of FBODE system in this way is also new in the literature.

Before ending this introduction, let us also emphasize that our approach is quite different from the other approaches in the literature. The Hilbert space approach used in \cite{bensoussan2020control} for instance has the advantages that the mathematical technicalities can be simplified considerably and the problem becomes more transparent, while the advantage of the classical approach in the Wasserstein
metric space is that the assumptions on the second order derivatives of coefficient functions are milder; also see Remark \ref{rem_DD} and the last paragraph in Section \ref{sec:motivation} for more discussion on the second order derivatives. In this paper, to make use of the advantages of both approaches, we first heuristically solve the mean field type control problem via the Hilbert space approach (see Section \ref{sec:motivation} for more details), and then derive the corresponding estimates and prove our results rigorously in the original Wasserstein metric space under the guidance of the calculus after that partial lifting procedure. In this new hybrid approach, we obtain useful insights on algebraic manipulations and deriving estimates from the heuristic arguments (in the Hilbert space setting), but are able to show our mian results under milder assumptions since our actual estimates are derived in the original Wasserstein metric space setting. According to our results in \cite{homan2023game}, \cite{homan2023control} and this paper, our approach can also be extended to study generic mean field game problems, especially for those with relatively small mean field effect.
Our present proposed theory can be readily extended to furnish one to tackle generic second order master equations with the present of Brownian driving noise.




The rest of this article is organized as follows.
In Section \ref{sec:ProblemSetting}, we introduce the Wasserstein space of measures, the push-forward operator and the setting of first order mean field type control problem. In Section \ref{sec:Derivatives}, we introduce various derivatives of functionals defined on the Wasserstein space.
In Section \ref{sec:motivation}, we give some motivation of our approach that was inspired by the calculus in the Hilbert space, and the reason why we must obtain our generic results by working in the Wasserstein space.
In Section \ref{sec:assumptions} we introduce the assumptions used in our results and some direct consequences of and useful properties implied by these assumptions, and we also compare these assumptions with other assumptions used in the literature.
In Section \ref{sec:Bellman}, we introduce the Bellman equation, prove the verification theorem and Pontryagin maximum principle; in particular, this explains why in order to solve the original first order mean field type control problem, it suffices to solve a forward-backward ordinary differential equations (FBODE). In Section \ref{sec:local}, we prove the existence and regularity of local-in-time solutions of the FBODE, and then of global-in-time solutions in Section \ref{sec:global} by using a new crucial {\it a priori} estimate. We also establish the uniqueness in Section \ref{sec:unique}.
In Section \ref{sec:master}, we introduce the master equation and show its global-in-time existence and uniqueness.
Finally, in Section \ref{sec:nonLQ}, we provide a non-trivial example of non-linear-quadratic mean field type control problem with non-separable Hamiltonian that satisfies our assumptions in Theorem \ref{GlobalSol}.

\section{Generic First Order Mean Field Type Control Problem}\label{sec:ProblemSetting}
Consider a probability space $\left(\Omega, \mathcal{F},\mathbb{P}\right)$\footnote{There is no need to incorporate a filtration here as this is just a first order mean field type control problem, the only randomness solely sets off at the initial state.} and all its square-integrable random $\R^d$-vectors, namely, 
\begin{align*}
\mathcal{H}^d:=L^{2}(\Omega,\mathcal{F},\mathbb{P};\mathbb{R}^{d}),
\end{align*}
which is equipped with an inner product:
\begin{align}
    \langle X,Y\rangle_{\mathcal{H}^d}:= \int_{\Omega} X(\omega)\cdot Y(\omega)d\mathbb{P}(\omega),\ \forall\ X,Y\in\mathcal{H}^d,
\end{align} 
which induces the corresponding norm denoted by $\| \cdot \|_{\mathcal{H}^d}$; for the sake of convenience, in the rest of this article we skip the subscripts of both inner product and norm if there is no cause of ambiguity. Here $X\cdot Y=X^\top Y$ stands for the inner product in $\R^d$, and we just denote the usual norm in $\R^d$ by $|\cdot|$ without any indication to the dependence on the dimension $d$. Let $\mathcal{P}_{p}(\mathbb{R}^{d})$, $p=1,2,...$, be the space of all probability measures on $\mathbb{R}^{d}$, each of which has a finite $p$-th order moment, equipped with the Wasserstein metric $W_{p}(\mu,\nu)$ defined by: 
\begin{equation} \label{eq_2_2}
W_{p}(\mu,\nu):=\bigg(\inf_{\pi\in\Pi(\mu,\nu)}\int_{\mathbb{R}^{d}\times \mathbb{R}^{d}}|\xi-\eta|^{p}d\pi(\xi,\eta)\bigg)^{\frac{1}{p}} , 
\end{equation}
where $\Pi(\mu,\nu)$ denotes the set of all joint probability measures
on $\mathbb{R}^{d}\times \mathbb{R}^{d}$ such that the marginals are the probability measures $\mu$ and $\nu$, respectively. Also define $\|\mu\|_p=\left(\int_{\R^{d}}|x|^p\,d\mu(x)\right)^{1/p}$. Clearly, $\mathcal{P}_{p'}(\mathbb{R}^d)\subset \mathcal{P}_{p}(\mathbb{R}^d)$ for any $1\leq p\leq p'<\infty$.

Fix a $m \in \mathcal{P}_p(\mathbb{R}^d)$, we denote $L^{p,d_1,d_2}_m:= L^p_m(\mathbb{R}^{d_1};\mathbb{R}^{d_2})$, the set of all measurable maps $\Phi:\R^{d_1}\rightarrow \R^{d_2}$ such that $\int_{\mathbb{R}^{d_1}}|\Phi(x)|^pd m(x)<\infty$. We equip $L^{p,d_1,d_2}_m$ with a norm $\| \cdot \|_{L^{p,d_1,d_2}_m}$ or simply $\| \cdot \|_{L^{p,d}_m}$ when $d_1=d_2=d$:
\begin{align}\label{eq_2_3}
    \|X\|_{L^{p,d_1,d_2}_m}:= \bigg(\int_{\mathbb{R}^{d_1}}|X(x)|^pd m(x)\bigg)^{\frac{1}{p}},\ \forall\ X\in L^{p,d_1,d_2}_m;
\end{align} 
clearly, $L^{p,d_1,d_2}_m$ is a Banach space and $L^{p',d_1,d_2}_m\subset L^{p,d_1,d_2}_m$ for any $1\leq p\leq p'<\infty$.
For $p=2$, we further equip $L^{2,d_1,d_2}_m$ with an inner product:
\begin{align}
    \langle X,Y\rangle_{L^{2,d_1,d_2}_m}:= \int_{\mathbb{R}^{d_1}} X(x)\cdot Y(x)d m(x),\ \forall\ X,Y\in L^{2,d_1,d_2}_m;
\end{align}
clearly, $L^{2,d_1,d_2}_m$ is a Hilbert space. In addition, we can regard the space $L^{2,d_1,d_2}_m$ as a tangent space attached to a point $m$ in the Wasserstein metric space $\mathcal{P}_2(\mathbb{R}^d)$, and we use subscript $m$ to indicate this nature.
\begin{definition}\label{defot}
Given a probability measure $m\in\mathcal{P}_2(\mathbb{R}^d)$, $X\in L^{2,d}_m$, define another one $ X \ot m  \in \mathcal{P}_2(\mathbb{R}^d)$ so that for every measurable function $\phi:\mathbb{R}^d\to\mathbb{R}$ such that ${\displaystyle \sup_{x\in\R^d}} \dfrac{|\phi(x)|}{1+|x|^2}<\infty$, 
\begin{align}
    \int_{\R^d} \phi(x)d( X \ot m )(x)=\int_{\R^d} \phi(X(x))dm(x).
\end{align}
\end{definition}
This new probability measure $ X \ot m$ is actually the push-forward map of $m$ under the map $X(\cdot)$.\footnote{More formally, a push-forward, $X\ot$, of a Borel map $X: \mathbb{R}^d \rightarrow \mathbb{R}^d$ is such that for any $\mu \in \mathcal{P}_{2}(\mathbb{R}^{d})$, the measure $X\ot \mu \in \mathcal{P}_{2}(\mathbb{R}^{d})$ is defined so that $X\ot \mu (A) := \mu(X^{-1}(A)), \text{ for any Borel set } A \subseteq \mathbb{R}^d$.} 

In this article, we denote the state by $x\in\mathbb{R}^{d_x}$ and the control by $\alpha\in\mathbb{R}^{d_\alpha}$,
also consider continuous coefficient and cost functions:
\begin{align*}
&f(x,\mu,\alpha)=(f_1,...,f_{d_x})(x,\mu,\alpha): \mathbb{R}^{d_x}\times\mathcal{P}_2(\mathbb{R}^{d_x})\times\mathbb{R}^{d_\alpha}\to\mathbb{R}^{d_x};\\
&g(x,\mu,\alpha): \mathbb{R}^{d_x}\times\mathcal{P}_2(\mathbb{R}^{d_x})\times\mathbb{R}^{d_\alpha} \to\mathbb{R};\\ &k(x,\mu) : \mathbb{R}^{d_x}\times\mathcal{P}_2(\mathbb{R}^{d_x})\to\mathbb{R},
\end{align*}
where these differentiable, in both $x$ and $\alpha$, functions can be allowed to depend on time, but we here omit this dependence to avoid unnecessary technicalities; yet all the following discussions remain valid with time dependent coefficients and cost functions. Besides, $f$ plays the role as a drift function, $g$ is the running cost, and $k$ stands for the terminal cost functional. 
The argument $\mu$ is the place for the mean field term of the state. 
The generic first order mean field type control problem is formulated as follows: for $t\in[0,T]$, $m\in\mathcal{P}_2(\R^{d_x})$ and any admissible (see definition below) control process $\alpha=(\alpha_t)_{t\in[0,T]}$, 
consider two type of controlled state equations given by:
\begin{align}\label{SDE1}
x^{t,\xi,\alpha}_s =&\ \ \xi+\int_t^s f\left(x^{t,\xi,\alpha}_\tau,\mathbb{L}_{x^{t,\xi,\alpha}_\tau},\alpha_\tau\right)d\tau,\ s\in[t,T];\\\label{SDE2}
x^{t,x,\xi,\alpha}_s =&\ \ x+\int_t^s f\left(x^{t,x,\xi,\alpha}_\tau,\mathbb{L}_{x^{t,\xi,\alpha}_\tau},\alpha_\tau\right)d\tau,\ s\in[t,T],
\end{align}
where $\xi\in\mathcal{H}^{d_x}$ such that $m=\mathbb{L}_\xi$ is the initial distribution at time $t$, here $\mathbb{L}_\xi$ denotes the law of $\xi$, and $\mathbb{L}_{x^{t,\xi,\alpha}_\tau}$ is the  probability distribution of
$x^{t,\xi,\alpha}_\tau$, which is the pushed-forward measure by the dynamics \eqref{SDE1} of $m$, now at time $\tau$. The set of admissible control processes is defined as $\mathbb{A}:=L^2([0,T];\mathcal{H}^{d_\alpha})$. The existence and uniqueness of solutions to \eqref{SDE1}-\eqref{SDE2} can be proven under standard Lipschitz and linear-growth conditions; moreover this solution satisfies $\sup_{s\in[t,T]}\mathbb{E}[|x^{t,\xi,\alpha}_s|^2]\leq C(1+\mathbb{E}|\xi|^2)<\infty$, see \cite{kurtz1999particle} and \cite{sznitman1991topics} for instance, where the constant $C$ is independent of $\xi$ or $\mathbb{L}_\xi=m$, but depends solely on $T$. For the class of the admissible controls, we only consider
all having a feedback form $\alpha(t,x,\mu)$, a continuous map from $[0,T]\times\mathbb{R}^{d_x}\times\mathcal{P}_2(\mathbb{R}^{d_x})$ to $\mathbb{R}^{d_\alpha}$; also use $\mathbb{A}$ to denote this collection. To save notation, we write $\alpha(t,x,\mu)$ or $\alpha(x,\mu)$ for simplicity in case of no cause of ambiguity. Define $x^{t,x,m,\alpha}_s:=x^{t,x,\xi,\alpha}_s$ since, under the standard mean-field type control problem setting, $x^{t,x,\xi,\alpha}_s$ depends on $\xi$ only through its law $m$; then $x^{t,\xi,\alpha}_s=x^{t,x,m,\alpha}_s\big|_{x=\xi}$ and the pair $(x^{t,x,m,\alpha}_s,x^{t,\xi,\alpha}_s)_{s\in[t,T]}$ has the flow property, which further implies that $\mathbb{L}_{x^{t,\xi,\alpha}_s}=x^{t,\cdot,m,\alpha}_s\ot m$; also see \cite{buckdahn2017mean} for a proof.

Now, we define the objective function in terms of $\alpha$, 
\begin{align}\label{cost}
j(t,m,\alpha) :=&\  \int^T_t \mathbb{E}\left[g(x^{t,\xi,\alpha}_s,\mathbb{L}_{x^{t,\xi,\alpha}_s},\alpha(s,x^{t,\xi,\alpha}_s,\mathbb{L}_{x^{t,\xi,\alpha}_s}))\right]ds +\mathbb{E}\left[k(x^{t,\xi,\alpha}_T,\mathbb{L}_{x^{t,\xi,\alpha}_T})\right],\text{ for }t\in[0,T].
\end{align}
Our main purpose in this work is to find the optimal control $\widehat{\alpha}$ such that the value function,
\begin{align}\label{valfun}
v(t,m):=j(t,m,\widehat{\alpha})=\min_{\alpha\in\mathbb{A}} j(t,m,\alpha),
\end{align}
for any $(t,m)\in[0,T]\times \mathcal{P}_2(\R^{d_x})$, that means the global unique solvability of the generic first order mean field type control problem \eqref{valfun}; in our knowledge, our solution will provide a settlement of a standing problem in this one of the most popular research areas of mean field games and mean field type control problems, and will also contribute to the mathematical understanding on the supreme efficiency of Deep Neural
Network after applying a batch normalization; indeed, we shall study the well-posedness of the generic first order mean field type control problem \eqref{SDE1}-\eqref{valfun} under Assumptions $\bf{(a1)}$-$\bf{(a3)}$,
to be specified in Section \ref{assumptions}.

\section{Derivatives in Wasserstein Metric Space}\label{sec:Derivatives}
We first list a few useful properties related to the push forward map $\ot$.
In the following, $L^2([t,T];\mathcal{B})$ stands for the set of all processes $s\in[t,T]\mapsto X(s)\in\mathcal{B}$, with a range being a Banach space $\mathcal{B}$, such that $\int_t^T \| X(s)\|_{\mathcal{B}}^2 ds <\infty$.
\begin{prpstn}
\label{property_push_forward}
The following properties hold for every $m\in\mathcal{P}_2(\R^d)$.
\begin{enumerate}
\item Let $X$, $Y\in L^{2,d}_m$, the composition, $X\circ Y\in L^{2,d}_m$, satisfies transitive property of $(X\circ Y)\ot  m = X\ot (Y\ot  m)$;
\item If $X(x) = x$ is the identity map in $L^{2,d}_m$, then $ X \ot m  =m$;
\item $\|X\ot m\|_2=\|X\|_{L^2_m}$ for any $X\in L^2_m$, that is the isometric property.
\end{enumerate}
\end{prpstn}
\begin{proof}
Statements 2 and 3 are clear by Definition \ref{defot}, so we omit their proofs. We only prove the first statement here. For any test function $\phi:\R^d\to \R$ with the appropriate boundedness as specified in Definition \ref{defot},
\small\begin{align*}
    \int_{\R^d} \phi(x)d( (X\circ Y) \ot m )(x)=\int_{\R^d} \phi(X(Y(x)))dm(x)=\int_{\R^d} \phi(X(x))d (Y\ot m)(x)=\int_{\R^d} \phi(x)d (X\ot (Y\ot m))(x).
\end{align*}\normalsize
\end{proof}

For the sake of convenience, we now provide a brief review on the definitions of different derivatives in the respective generic Hilbert space and Wasserstein metric space $\mathcal{P}_2(\mathbb{R}^d)$.
Given a probability measure $\mu\in\mathcal{P}_{2}(\mathbb{R}^{d})$, 
the tangent space $\mathcal{T}(\mu)$ of $\mathcal{P}_{2}(\mathbb{R}^{d})$ at $\mu$ is defined as 
\begin{equation}
\mathcal{T}(\mu):=\text{the closure of }\left\{\nabla \Phi_\mu^\nu \bigg| \nabla \Phi_\mu^\nu:=\mathop{\arg\min}_{\nabla \phi \in \mathcal{S}_\mu^\nu} \int_{\mathbb{R}^{d}}|x-\nabla \phi(x)|^{2}d\mu(x)  \text{ for some } \nu\in \mathcal{P}_{2}(\mathbb{R}^{d})\right\}\text{ in }L_\mu^{2,d},
\end{equation}
where $\mathcal{S}_\mu^\nu$ is the set of all Borel maps $T:\mathbb{R}^d \rightarrow \mathbb{R}^d$ such that $T\ot \mu =\nu$, each $\nabla \Phi_\mu^\nu$ is a Brenier's map pushing $\mu$ to $\nu$.

\begin{definition}
(Wasserstein differentiability \cite{ambrosio2005gradient}) A function $f: \mu\in \mathcal{P}_2(\mathbb{R}^{d})\mapsto f(\mu)\in \mathbb{R}$ is said to be {\it Wasserstein differentiable}
at a $\mu_0\in \mathcal{P}_2(\mathbb{R}^{d})$, in the Wasserstein metric space, if there exists a function $\nabla_\mu f(\mu_0)(\cdot)$, called the {\it Wasserstein gradient}, belonging to the
tangent space $\mathcal{T}(\mu_0)$ with the property that, for any sequence of probabilities $\mu_n\in\mathcal{P}_{2}(\mathbb{R}^{d})$ such that $W_{2}(\mu_n,\mu_0)\rightarrow0$ as $n \rightarrow \infty$, 
\begin{equation} \label{eq:4-9}
\frac{f(\mu_n)-f(\mu_0)-\displaystyle\int_{\mathbb{R}^{d}}\nabla_\mu f(\mu_0)(x)\cdot (\nabla \Phi_{\mu_0}^{\mu_n}(x)-x)d \mu_0(x)}{W_{2}(\mu_n,\mu_0)}\rightarrow0 \,\text{ as }n \rightarrow \infty. 
\end{equation}
\end{definition}

\begin{definition}\label{first_LFD}
(Linear Functional Derivative \cite{carmona2018probabilistic}). A function $f: \mu\in \mathcal{P}_2(\mathbb{R}^{d})\mapsto f(\mu)\in \mathbb{R}$ is said to have a {\it linear functional derivative} if there exists a function
\begin{align*}
    \/{\delta f}{\delta \mu}:\mathcal{P}_2(\mathbb{R}^{d})\times \mathbb{R}^{d} \to  \mathbb{R}
\end{align*}
which is continuous with respect to the product topology of $\mathcal{P}_2(\R^d)$ and $\R^d$ so that the followings hold:
\begin{enumerate}
    \item [(i)]
there is another function $c: \mathcal{P}_2(\mathbb{R}^{d})\rightarrow [0,\infty)$ which is bounded on any bounded subsets of $\mathcal{P}_2(\R^d)$, such that
\begin{align}
\left| \/{\delta f}{\delta \mu}(\mu)(x)\right|\leq c(\mu)(1+|x|^2),\text{ for any }\mu\in\mathcal{P}_2(\R^d)\text{ and }x\in \R^d;
\end{align}
\item [(ii)] for all $\nu_1$, $\nu_2\in \mathcal{P}_2(\mathbb{R}^{d})$, it also holds that
\begin{align}\label{delta_mu_f}
    f(\nu_1)-f(\nu_2) = \int_0^1 \int_{\mathbb{R}^{d}}\/{\delta f}{\delta \mu}(\theta \nu_1+(1-\theta)\nu_2)(x)d(\nu_1-\nu_2)(x)d\theta.
\end{align}
\end{enumerate}
Note that $\/{\delta f}{\delta \mu}$ defined here is unique only up to  a constant, and the following normalization condition is taken in the rest of this article:
\begin{align}\label{eq_3_5_n}
\int_{\mathbb{R}^{d}} \/{\delta f}{\delta \mu}(\mu)(x)d\mu(x)=0,
\end{align}
which in turn ensures the functional derivative of a constant function is zero. 

\end{definition}

\begin{remark}
Clearly, \eqref{delta_mu_f} implies that, for any $m$, $m'\in \mathcal{P}_2(\mathbb{R}^{d})$, 
\begin{align}\label{delta_mu_f_1}
   \lim_{s\to 0^+} \frac{f((1-s)m+sm')-f(m)}{s} = \int_{\mathbb{R}^{d}}\/{\delta f}{\delta \mu}(m)(x)d(m'-m)(x).
\end{align}
Besides, if $\mu$ has a density $\/{d\mu}{dx}$ with respect to Lebesgue measure, the functional derivative is precisely the usual functional derivative of $f$ (if we regard $f$ as a function of $\frac{d\mu}{dx}$ instead) with respect to the density $\/{d\mu}{dx}$ over the usual $L^2(\R^d)$ space, and so the notation presented in Definition \ref{first_LFD} can be regarded as a completion of the space of these usual functional derivatives.
\end{remark}

\begin{definition}
(Second Order Linear Functional Derivative \cite{bensoussan2020control}). A function $f: \mu\in \mathcal{P}_2(\mathbb{R}^{d})\mapsto f(\mu)\in \mathbb{R}$ is said to have a second order linear functional derivative if $f$ has the first order linear functional derivative as defined in Definition \ref{first_LFD}, and there exists another function $\/{\delta^2 f}{\delta \mu^2}:\mathcal{P}_2(\mathbb{R}^{d})\times\mathbb{R}^{d}\times \mathbb{R}^{d} \to \mathbb{R}$ such that it is continuous with respect to the corresponding product topology of $\mathcal{P}_2(\mathbb{R}^{d})\times\mathbb{R}^{d}\times \mathbb{R}^{d}$, and 
\begin{enumerate}
    \item [(i)] there is a functional $c: \mathcal{P}_2(\mathbb{R}^{d})\rightarrow [0,\infty)$ which is bounded on any bounded subsets, so that 
\begin{align}
\left| \/{\delta^2 f}{\delta \mu^2}(\mu)(x,\widetilde{x})\right|\leq c(\mu)(1+|x|^2+|\widetilde{x}|^2),
\end{align} 
and, together with its first order linear functional derivative, it also satisfies: 
\item[(ii)] for any $\nu_1,\nu_2 \in \mathcal{P}_2(\R^d)$,
\begin{align*}
    &f(\nu_1) - f(\nu_2) \\
    =&\ \int_{\mathbb{R}^{d}}\/{\delta f}{\delta \mu}(\nu_2)(x) d(\nu_1-\nu_2)(x)\\
& + \int_0^1\int_0^1 \int_{\mathbb{R}^{d}}\int_{\mathbb{R}^{d}}\theta \/{\delta^2 f}{\delta \mu^2}(\nu_2+\lambda \theta(\nu_1-\nu_2))(x,\widetilde{x})d(\nu_1-\nu_2)(x)d(\nu_1-\nu_2)(\widetilde{x})d\lambda d\theta.
\end{align*}
\end{enumerate}
Again, we also require the normalization conditions in addition to \eqref{eq_3_5_n} for the first order linear functional derivative:
$$\int_{\mathbb{R}^{d}} \/{\delta^2 f}{\delta \mu^2}(\mu)(x,\widetilde{x})d\mu(\widetilde{x}) = 0,\  \forall x;$$
$$\int_{\mathbb{R}^{d}} \/{\delta^2 f}{\delta \mu^2}(\mu)(x,\widetilde{x})d\mu(x) = 0,\  \forall \widetilde{x}.$$
\end{definition}
To add a point, we can also show the following symmetric property of the second-order linear functional derivatives:
\begin{align}
    \/{\delta^2 f}{\delta \mu^2}(\mu)(x,\widetilde{x})=\/{\delta^2 f}{\delta \mu^2}(\mu)(\widetilde{x},x),\text{ for any }x,\widetilde{x}\in \R^d.
\end{align}

Let $\mathcal{H}$ and $\mathcal{H}'$ be two Hilbert spaces and $\mathcal{L}(\mathcal{H};\mathcal{H}')$ be the set of all bounded linear operators from $\mathcal{H}$ to $\mathcal{H}'$ equipped with the usual operator norm: 
\begin{align}
	\|F\|_{\mathcal{L}(\mathcal{H};\mathcal{H}')} := \sup_{X\in\mathcal{H}} \/{\|F(X)\|_\mathcal{H'}}{\|X\|_\mathcal{H}}.
\end{align}
\begin{definition}
(\textit{Fr\'echet differentiability}). A function $F: \mathcal{H}\to\mathcal{H'}$ is said to be \textit{Fr\'echet differentiable} at $X$ if there is a bounded linear operator $D_XF (X) \in \mathcal{L}(\mathcal{H};\mathcal{H}')$, such that for all $Y \in \mathcal{H}$,
\begin{align*}
	\lim_{\|Y\|_\mathcal{H}\to 0}\/{\|F (X+Y)-F (X)-D_XF (X)(Y)\|_\mathcal{H'}}{\|Y\|_\mathcal{H}} = 0.
\end{align*}
\end{definition}
\begin{definition}\label{def_L_diff}
({\it L-differentiability} at a $\mu_0\in \mathcal{P}_2(\mathbb{R}^{d}$) \cite{bensoussan2017interpretation,carmona2018probabilistic}). A function $f: \mu\in \mathcal{P}_2(\mathbb{R}^{d})\mapsto f(\mu)\in \mathbb{R}$ is said to be {\it $L$-differentiable} at a $\mu_0 \in \mathcal{P}_2(\mathbb{R}^{d})$ if there exists a $X_0\in\mathcal{H}^{d}$ with the law $\mu_0$, denoted by $\mathbb{L}_{X_0} = \mu_0$, such that its lifted version $F (X):= f(\mathbb{L}_X)$ is Fr\'echet differentiable at $X_0$; and we also call this Fr\'echet derivative $D_X F(X_0)(\cdot)\in\mathcal{L}(\mathcal{H}^d;\R)$ as the {\it $L$-derivative} of $f(\mu)$ at $\mu_0$, denote it by
$\partial_\mu f(\mu_0)$.
\end{definition}
Note that this {\it L-derivative} $\partial_\mu f(\mu_0)$ is uniquely defined, except $\mathbb{P}$-null set of points, in the sense that its definition is independent of the choice of $X_0$ corresponding to the same $\mu_0$; also see Proposition 5.24 in \cite{carmona2018probabilistic} and \cite{bensoussan2017interpretation} for details. 
Generally, the function $f$ is {\it $L$-differentiable} if it is {\it $L$-differentiable} at every $\mu\in \mathcal{P}_2(\mathbb{R}^{d})$, so that the {\it $L$-derivative} as a function denoted by $\partial_\mu f(\cdot)(\cdot):(\mu,x)\in \mathcal{P}_2(\mathbb{R}^{d})\times \R^d \mapsto \partial_\mu f(\mu)(x)\in\mathbb{R}^{d}$ is jointly measurable. Certainly, we also have
\begin{equation}
\partial_{\mu}f(\mathbb{L}_X)(X)=D_{X}F(X) \:\text{a.s.}, \quad \forall X\,\in\mathcal{H}^{d}. \label{eq:4-110}
\end{equation}

In Theorem 5.64 of \cite{carmona2018probabilistic}, Carmona and Delarue have shown that the $L$-derivative coincides with the Wasserstein gradient when they are both defined; that is, $\partial_\mu f(\mu)=\nabla_\mu f(\mu)$; see also \cite{bensoussan2019control} for an alternative approach by directly linking $D_{X}F(X)$ and $\nabla_\mu f(\mu)$. 
Due to the pure metric space nature, but not a vector space, of Wasserstein metric space, using Wasserstein gradient to develop control theory is more complicated than using $L$-derivative, and so the later will be used in describing the regularity of our solution of FBODE system, see Theorem \ref{Thm6_2} for example. 
Besides, we have the following properties for linear functional derivatives and $L$-derivatives. 

\begin{prpstn}\label{prop_3_7}(Proposition 5.51 in \cite{carmona2018probabilistic}) 
Suppose that $f: \mu\in \mathcal{P}_2(\mathbb{R}^{d})\mapsto f(\mu)\in \mathbb{R}$ is {\it $L$-differentiable} on $\mathcal{P}_2(\R^d)$ and its Fr\'echet derivative $D_{X}F(\cdot)$ is uniformly Lipschitz in its argument. Also assume that its $L$-derivative $\partial_\mu f(\cdot)(\cdot):(\mu,x)\in \mathcal{P}_2(\R^d)\times \R^d\mapsto \partial_\mu f(\mu)(x)\in \R$ is jointly continuous with respect to the natural product topology. Then $f$ has a linear functional derivative $\/{\delta f}{\delta \mu}(\cdot)(\cdot)$ such that 
\begin{equation}
\partial_\mu f(\mu)(x)=\partial_x \/{\delta f}{\delta \mu}(\mu)(x),\text{ for any }   \mu\in \mathcal{P}_2(\mathbb{R}^{d}). \label{eq:4-111}
\end{equation}
\end{prpstn}
Also see \cite{bensoussan2017interpretation} for \eqref{eq:4-111} for an alternative derivation.
\begin{definition}
A function $f:\mu\in\mathcal{P}_2(\mathbb{R}^{d}) \mapsto f(\mu)\in \mathbb{R}$ is called regularly linear-functionally differentiable 
in $\mu\in\mathcal{P}_2(\mathbb{R}^{d})$ if\\
(i) $f$ has a linear functional derivative $\/{\delta f}{\delta \mu}(\mu)(x)$ which, for each $\mu \in \mathcal{P}_{2}(\mathbb{R}^{d})$, is differentiable in $x$;\\
(ii) $\partial_x\/{\delta f}{\delta \mu}(\cdot)(\cdot):\mathcal{P}_2(\R^d)\times \R^d\rightarrow \R^d$ is jointly continuous with respect to the natural product topology such that
\begin{equation}
\left|\partial_x\/{\delta f}{\delta \mu}(\mu)(x)\right|\leq c(\mu)(1+|x|),\label{eq:4-113}
\end{equation}
where $c(\mu)$ is another function in $\mu$ being bounded on any bounded subsets of $\mathcal{P}_{2}(\mathbb{R}^{d})$.
\end{definition}
Without the cause of ambiguity, we simply call $f$ to be regularly differentiable in $\mu\in\mathcal{P}_2(\mathbb{R}^{d})$. The derivative $\p_x\frac{\delta f}{\delta \mu}:\mathcal{P}_2(\R^d)\times \R^d\rightarrow \R^d$ is sometimes called the intrinsic derivative, see \cite{cardaliaguet2019master} for example.

\begin{prpstn}\label{Ldif} (Propositions 5.44 and 5.48 in \cite{carmona2018probabilistic})
Suppose that  $f:\mathcal{P}_2(\mathbb{R}^{d}) \to \mathbb{R}$ is regularly differentiable in $\mu\in\mathcal{P}_2(\mathbb{R}^{d})$. Then, $f$ is $L$-differentiable such that \eqref{eq:4-111} holds. Furthermore, for any $\mu\in\mathcal{P}_2(\R^d)$, and any sequence $\mu_n$ with $W_{2}(\mu_n,\mu)\rightarrow0$ as $n \rightarrow \infty$, we also have  
\begin{equation}\label{prop33}
\frac{f(\mu_n)-f(\mu)-\int_{\mathbb{R}^{d}}\/{\delta f}{\delta \mu}(\mu)(x) d(\mu_n-\mu)(x)}{W_{2}(\mu_n,\mu)}\rightarrow0.
\end{equation}
\end{prpstn}
\begin{definition}\label{def_3_10}
A function $f:\mu\in\mathcal{P}_2(\mathbb{R}^{d}) \mapsto f(\mu)\in \mathbb{R}$ is called second-order regularly linear-functionally differentiable in $\mu\in\mathcal{P}_2(\mathbb{R}^{d})$ if\\
(i) $f$ has a linear functional derivative $\frac{\delta f}{\delta \mu}(\mu)(x)$, and also has a second-order linear functional derivative $\/{\delta^2 f}{\delta \mu^2}(\mu)(x,\wt{x})$ which, for each $\mu \in \mathcal{P}_{2}(\mathbb{R}^{d})$, has derivatives $\p_{x}\p_x\frac{\delta f}{\delta \mu}(\mu)(x)$ and $\p_x\p_{\wt{x}}\/{\delta^2 f}{\delta \mu^2}(\mu)(x,\wt{x})$;\\
(ii) these $\partial_{x}\p_x\/{\delta f}{\delta \mu}:\mathcal{P}_2(\R^d)\times \R^d\rightarrow \R^d$ and $\p_x\p_{\wt{x}}\/{\delta^2 f}{\delta \mu^2}:\mathcal{P}_2(\R^d)\times \R^d\times \R^d\rightarrow \R^d$ are jointly continuous with respect to the corresponding product topologies such that
\begin{equation}
\left|\p_{x}\p_x\/{\delta f}{\delta \mu}(\mu)(x)\right|+\left|\p_x\p_{\wt{x}}\/{\delta^2 f}{\delta \mu^2}(\mu)(x,\wt{x})\right|\leq c(\mu),
\end{equation}
where $c(\mu)$ is another function being bounded on any bounded subsets of $\mathcal{P}_{2}(\mathbb{R}^{d})$. 
\end{definition}
Without the cause of ambiguity, we simply call $f$ to be second-order regularly differentiable in $\mu\in\mathcal{P}_2(\mathbb{R}^{d})$.

Finally, let us end this section by providing the following important remark about the regularity requirement of coefficient functions.
\begin{remark}\label{rem_DD}
To get rid of possible restrictions on the assumptions on the second-order Fr\'{e}chet differentiability of coefficient functions as in \cite{bensoussan2020control} and \cite{bensoussan2019control}, we use a different approach to deal with the second-order derivatives on the Wasserstein space of measures $\mathcal{P}_2(\R^n)$. Consider a function $f:\mu\in\mathcal{P}_2(\R^n)\mapsto f(\mu)\in\R$ defined on $\mathcal{P}_2(\R^n)$, then the Fr\'{e}chet derivative $D_X F(X)$ of its lifted version $F: X\in L^{2}_m(\R^n)\to F(X):=f(X\ot m)\in\R$ in the Hilbert space $L^{2}_m(\R^n)$ projects down to a derivative $\p_x\frac{\delta}{\delta\mu} f(\mu)(x)$ of $f$ on $\mathcal{P}_2(\R^n)$, where $\frac{\delta f}{\delta \mu}:(\mu,x)\in\mathcal{P}_2(\mathbb{R}^n)\times \mathbb{R}^n \mapsto  \frac{\delta}{\delta\mu} f(\mu)(x)\in\mathbb{R}$ is a linear functional derivative of $f$; see Proposition \ref{prop_3_7} and \ref{Ldif}. 
However, when it comes to the second-order Fr\'{e}chet derivative $D_XD_XF(X)$, a restriction on the original function $f(\mu)$ intervenes; since $D_XF(X)=\p_x\frac{\delta}{\delta\mu} f(X\ot m)(X(\cdot))$, $D_XF(X)\in L^{2}_m(\R^n)$ is regarded as a lifting version of $\p_x\frac{\delta}{\delta\mu} f(\mu)(x)\in\R^n$; the argument $*$ of $\p_x\frac{\delta}{\delta\mu} f(\mu)(*)$ is changed from $x\in \R^n$ to $X\in L^{2}_m(\R^n)$ with its corresponding norms also being changed in this lifting procedure, which causes a problem when one links the second-order Fr\'{e}chet derivative $D_XD_XF(X)$ with the second-order derivatives $\p_x\p_x\frac{\delta}{\delta\mu}f(\mu)(x)$ and $\p_x\p_{\wt{x}}\frac{\delta^2}{\delta\mu^2}f(\mu)(x,\wt{x})$ on $\mathcal{P}_2(\R^n)$; roughly speaking, the existence of $\p_x\p_x\frac{\delta}{\delta\mu}f(\mu)(x)$ and $\p_x\p_{\wt{x}}\frac{\delta^2}{\delta\mu^2}f(\mu)(x,\wt{x})$ cannot imply the existence of the second-order Fr\'{e}chet derivative $D_XD_XF(X)$; also see Example 2.3 in \cite{buckdahn2017mean} and the last paragraph in Section \ref{sec:motivation}. Therefore, we use $\p_x\p_x\frac{\delta}{\delta\mu}f(\mu)(x)$ and $\p_x\p_{\wt{x}}\frac{\delta^2}{\delta\mu^2}f(\mu)(x,\wt{x})$ instead of $D_XD_XF(X)$ when it comes to the second derivative on measures; see Definition \ref{def_3_10}.
\end{remark}

In summary, we have just reviewed the concepts of Wasserstein differentiability, linear functional derivatives, $L$-derivatives, regular differentiability over the metric space $\mathcal{P}_2(\R^d)$, and their relationships. We are now ready to consider the first order mean field type control problem \eqref{SDE1}-\eqref{valfun}. For our first order problem setting, the push-forward mapping $X(\cdot)\in L^{2,d_x}_{m}$ governed by \eqref{MPIT} is deterministic
since there is no new source of randomness other than those inherited from the initial data; while in our future work on the second order case, $X$ is a random variable that obtains additional randomness, other than that of the initial data, from the Brownian motion that drives the dynamics of $X$, and so by that time we have to deal with the randomness from the Brownian motion and that from the initial data separately by only lifting to the Hilbert space that corresponds to the former only.
To this end, this is the major difference behind the lifting procedure first proposed in \cite{bensoussan2020control} from that in P.-L. Lions \cite{LionsLecture1,LionsLecture2} (or the edited version written by Cardaliaguet \cite{cardaliaguet2010notes}).

\section{Heuristics Behind Our Approach Guided by Lifting}\label{sec:motivation}
We here sketch out the main idea behind our proposed approach that involves building crucial {\it a priori} estimates in connection with the Schur complement of the matrices of coefficient functions corresponding to the lifted version of the FBODE obtained by the maximum principle, and then we further explain the difficulty encountered in face at the second-order Fr\'{e}chet differentiability of various coefficient functions, by then we have to modify our argument. More precisely, by adopting the ``lifting'' idea introduced by P.-L. Lions \cite{LionsLecture2}, we define, for each $m_0\in \mathcal{P}_2(\R^{d_x})$, the lifted coefficient functions and objective functional,
\begin{align}\label{lifted:CapitalF}
&F:(X,A)\in L^{2,d_x}_{m_0}\times L^{2,d_x,d_\alpha}_{m_0}\mapsto F(X,A)(\cdot)=f(X(\cdot),X\ot m_0,A(\cdot))\in  L^{2,d_x}_{m_0};\\\label{lifted:GBar}
&\wb{G}:(X,A)\in L^{2,d_x}_{m_0}\times L^{2,d_x,d_\alpha}_{m_0}\mapsto \wb{G}(X,A)=\int_{\R^{d_x}} g(X(x),X\ot m_0,A(x))dm_0(x)\in \R;\\\label{lifted:KBar}
&\wb{K}:X\in L^{2,d_x}_{m_0} \mapsto \wb{K}(X)=\int_{\R^{d_x}} k(X(x),X\ot m_0)dm_0(x)\in \R;
\\\label{lifted:ObjFunJ}
&J:(t,Y,A)\in [0 ,T]\times L^{2,d_x}_{m_0}\times \mathbb{A}_{m_0}\mapsto J(t,Y,A)=\int_t^T \wb{G}(X^{t,Y,A}_s,A_s) ds+\wb{K}(X^{t,Y,A}_T)\in \R.
\end{align}
We can regard $F$, $\wb{G}$ as operators acting on the ``tangent space'' (also see \cite{bensoussan2020control}) $L^{2,d_x}_{m_0}\times L^{2,d_x,d_\alpha}_{m_0}$ attached to the measure $m_0$.
For every $t\in [0 ,T]$,  and $Y\in L^{2,d_x}_{m_0}$, $A_\cdot \in \mathbb{A}_{m_0}:=L^2([0 ,T];L^{2,d_x,d_\alpha}_{m_0})$, consider the dynamics
\begin{align}\label{lifted:DynX}
X^{t,Y,A}_s(\cdot)=Y(\cdot)+\int_t^s F(X^{t,Y,A}_\tau,A_\tau)(\cdot)d\tau,\ s\in[t,T].
\end{align}
Note that, for any $m\in\mathcal{P}_2(\R^{d_x})$, there exists at least one $Y_{m_0}^m\in L^{2,d_x}_{m_0}$, for instance a Brenier's  map, such that $m=Y_{m_0}^m\ot m_0$ pushing $m_0$ forward to $m$; also see \cite{ambrosio2005gradient,BHYbook,bensoussan2017interpretation,bensoussan2015master,bensoussan2019control,carmona2018probabilistic,villani2003topics,villani2009optimal}.  Then, recalling the notations in Section \ref{sec:ProblemSetting}, we have $x^{t,\xi,\alpha}_s=x^{t,Y_{m_0}^m(x),m,\alpha}_s\big|_{x=\xi_0}=X^{t,Y_{m_0}^m,A}_s(x)\big|_{x=\xi_0}$ and $j(t,m,\alpha)=J^{m_0}(t,Y_{m_0}^m,A)$ where $m=Y_{m_0}^m\ot m_0$, $\xi=Y^m_{m_0}(\xi_0)$, $\alpha_s=A_s(\xi_0)$, 
and $\xi_0\in \mathcal{H}^{d_x}$ such that $\mathbb{L}_{\xi_0}=m_0$. After the lifting procedure, we aim to find the optimal control $\widehat{A}$ such that
\begin{align}\label{lifted:ValFun}
V(t,Y):=J(t,Y,\widehat{A})=\min_{A\in\mathbb{A}_{m_0}} J(t,Y,A).
\end{align} 
In accordance with the usual verification theorem and Pontryagin maximum principle but now in the lifted version (see Theorem \ref{Verification} and \ref{MPBE} for the unlifted version), we can convert the solving for the abstract control problem \eqref{lifted:DynX}-\eqref{lifted:ValFun} to the resolution of the lifted FBODE system in the Hilbert space of $L^{2,d_x}_{m_0}$ (see \eqref{MPIT} for the unlifted version)
\begin{align}
\label{lifted:MPIT}
\begin{cases}       
\dfrac{d}{ds}X^{t,Y}_s = F\left(X^{t,Y}_s,A\left(X^{t,Y}_s,Z^{t,Y}_s\right)\right),\text{ for any } s\in[t,T],\\
\ \ \ \,X^{t,Y}_t= Y;\\
\dfrac{d}{ds}Z^{t,Y}_s  = -D_X H\left(X^{t,Y}_s,Z^{t,Y}_s\right),\\
\ \ \ \,Z^{t,Y}_T=D_X \wb{K}(X^{t,Y}_T),
\end{cases}
\end{align}\normalsize
where $D_X$ is Fr\'{e}chet derivative over $L^{2,d_x}_{m_0}$, and $A(X,Z)$ solves for the lifted first order condition (see \eqref{first_order_condition} for the unlifted version):
\begin{align}\label{lifted:First_Order_Con}
D_A L(X,A;Z)=\left\langle D_A F(X,A), Z\right\rangle_{L^{2,d_x}_{m_0}} + D_A \wb{G}(X,A)=0,
\end{align} 
$L(X,A;Z)=\left\langle F(X,A), Z\right\rangle_{L^{2,d_x}_{m_0}} + \wb{G}(X,A)$ is the lifted Lagrangian and $$H(X,Z)= \left\langle F(X,A(X,Z)), Z\right\rangle_{L^{2,d_x}_{m_0}}  + \wb{G}(X,A(X,Z))$$ is the lifted Hamiltonian. Define
\begin{align}\label{lifted:Gamma_alpha}
\Gamma (s,Y):=D_X \wb{K}(X^{t,Y}_T)+ \int_s^T &D_X H\left(X^{s,Y}_\tau,\Gamma (\tau,X^{s,Y}_\tau)\right)d\tau,
\end{align}
and therefore $Z^{t,Y}_s=\  \Gamma(s,X^{t,Y,}_s)$. Now, we intend to derive an {\it a priori} estimate of $D_Y Z^{t,Y}_t=D_Y \Gamma(t,X^{t,Y}_t)\\=D_Y \Gamma(t,Y)$, which is crucial in getting the sensitivity and uniqueness of the FBODE system \eqref{lifted:MPIT}, based on which we can paste the local-in-time solution of the FBODE system \eqref{lifted:MPIT} together to obtain a global-in-time one. 
By taking the Fr\'{e}chet derivative with respect to $Y$ in \eqref{lifted:MPIT}, we have
the following linear FBODE system; also see \eqref{p_mX} and \eqref{p_xX} for the unlifted version, there we take the derivatives with respect to the initial state $x\in\R^{d_x}$ and the initial distribution $m\in\mathcal{P}_2(\R^{d_x})$, as this can result in better estimates as we are about to explain here:
\small\begin{align}\label{lifted:D_YXexistence}
\left\{ \begin{aligned}
\dfrac{d}{ds}D_Y X^{t,Y}_s =&\  D_X F\left(X^{t,Y}_s,A\left(X^{t,Y}_s,Z^{t,Y}_s\right)\right)(D_Y X^{t,Y}_s)\\
&+D_A F\left(X^{t,Y}_s,A\left(X^{t,Y}_s,Z^{t,Y}_s\right)\right)\left(D_X A\left(X^{t,Y}_s,Z^{t,Y}_s\right)(D_Y X^{t,Y}_s)+D_Z A\left(X^{t,Y}_s,Z^{t,Y}_s\right)(D_Y Z^{t,Y}_s)\right)\\
\ \ \ \,D_Y X^{t,Y}_t=&\  Id;\\
\dfrac{d}{ds}D_Y Z^{t,Y}_s  =&\  -D_XD_X H\left(X^{t,Y}_s,Z^{t,Y}_s\right)(D_Y X^{t,Y}_s)-D_ZD_X H\left(X^{t,Y}_s,Z^{t,Y}_s\right)(D_Y Z^{t,Y}_s)\\
\ \ \ \,D_Y Z^{t,Y}_T=&\ D_XD_X \wb{K} (X^{t,Y}_T)(D_Y X^{t,Y}_T),
\end{aligned} \right.
\end{align}\normalsize
where $Id:x\in\R^{d_x}\mapsto x\in\R^{d_x}$ is the identity mapping on $\R^{d_x}$.

Suppose for the moment that the Fr\'{e}chet derivatives up to  the second order in the following calculations exist, then, for any $\wt{Y}\in L^{2,d_x}_{m_0}$, by using \eqref{lifted:First_Order_Con} and \eqref{lifted:D_YXexistence}, we have, under the assumptions that the running cost function $\wb{G}(X,A)$ is convex in control $A$, and the second order derivative of the drift coefficient $F(X,A)$ with respect to control $A$ is bounded and decays at the proximity of infinity,
\small\begin{align}\no
&\bigg\langle D_Y Z^{t,Y}_T(\wt{Y}),D_Y X^{t,Y}_T(\wt{Y})\bigg\rangle_{ L^{2,d_x}_{m_0}}-\bigg\langle D_Y Z^{t,Y}_t(\wt{Y}),D_Y X^{t,Y}_t(\wt{Y})\bigg\rangle_{ L^{2,d_x}_{m_0}}\\\no
=&\ -\int_t^T \left(D_AD_A L\right)^{-1}\bigg(\left\langle  D_{A} F, D_Y Z^{t,Y}_s(\wt{Y})\right\rangle_{ L^{2,d_x,d_\alpha}_{m_0}},\left\langle  D_{A} F, D_Y Z^{t,Y}_s(\wt{Y})\right\rangle_{ L^{2,d_x,d_\alpha}_{m_0}}\bigg)ds\\\no
&-\int_t^T\bigg(D_XD_X L\big(D_Y X^{t,Y}_s(\wt{Y}),D_Y X^{t,Y}_s(\wt{Y})\big)\\\label{lifted:estimate_1}
&\ \ \ \ \ \ \ \ \ \ \ \ -\left(D_AD_A L\right)^{-1}\bigg(\left\langle D_AD_X L,D_Y X^{t,Y}_s(\wt{Y}) \right\rangle_{L^{2,d_x,d_\alpha}_{m_0}},\left\langle D_XD_A L,D_Y X^{t,Y}_s(\wt{Y}) \right\rangle_{L^{2,d_x,d_\alpha}_{m_0}}\bigg)ds
\end{align}\normalsize
where the arguments $\left(X^{t,Y}_s,A\left(X^{t,Y}_s,Z^{t,Y}_s\right);Z^{t,Y}_s\right)$ of the Lagrangian $L(X,A;Z)$ and the arguments $\left(X^{t,Y}_s,A\left(X^{t,Y}_s,Z^{t,Y}_s\right)\right)$ of the drift function $F(X,A)$ are omitted for simplicity if there is no cause of ambiguity; see Appendix \ref{lifted:estimate_1_detail} for the justification of the definition of $\left(D_AD_A L\right)^{-1}$ and a derivation of \eqref{lifted:estimate_1}. 
Note that the term $D_XD_X L-\left(D_AD_A L\right)^{-1}\left(D_AD_X L, D_XD_A L\right)$ is the Schur complement of the Hessian matrix of the Lagrangian $L(X,A;Z)$ in $(X,A)$, and so its convexity will imply the positive definiteness of the Schur complement. 
We point out that the Schur complement has a close relationship with the Lasry-Lions monotonicity condition; consider a special case that the drift function is simply $f(x,\mu,\alpha)=\alpha$ and the running cost function is $g(x,\mu,\alpha)=\frac{1}{2}\alpha^2+g_1(x,\mu)$, then $L(X,A;Z)=\left\langle A, Z\right\rangle_{L^{2,d_x}_{m_0}} + \frac{1}{2}\int_{\R^{d_x}} |A(x)|^2 dm_0(x)+\int_{\R^{d_x}} g_1(X(x),X\ot m_0)dm_0(x)$ and $D_AD_X L(X,A;Z)=D_XD_A L(X,A;Z)=0$; thus the Schur complement becomes 
\begin{align}\no
&D_XD_X \bigg(\int_{\R^{d_x}} g_1(X(x),X\ot m_0)dm_0(x)\bigg)\big(\wt{X},\wt{X}\big)\\\no
=&\ \int_{\R^{d_x}}\big(\wt{X}(\wt{x})\cdot\p_x\big)\big(\wt{X}(\wt{x})\cdot\p_x\big)G_1(x,\mu)\bigg|_{x=X(\wt{x}),\mu=X\ot m_0}dm_0(\wt{x})\\\label{lift:eq_4_12}
&+\int_{\R^{d_x}}\int_{\R^{d_x}}(\wt{X}(\wt{x})\cdot \p_{x})(\wt{X}(\wh{x})\cdot \p_\mu)G_1\big(x,\mu\big)\bigg|_{x=X(\wt{x}),\mu=X\ot m_0}\big(X(\wh{x})\big)dm_0(\wh{x})dm_0(\wt{x})
\end{align}
where $G_1(x,\mu):=\frac{\delta}{\delta \mu}\bigg(\int_{\R^{d_x}} g_1(\wt{x},\mu)d\mu(\wt{x})\bigg)(x)=g_1(x,\mu)+\int_{\R^{d_x}}\frac{\delta}{\delta\mu}g_1(\wt{x},\mu)(x)d\mu(\wt{x})$; the non-negativeness of the second term on the right hand side of \eqref{lift:eq_4_12} is corresponding to the Lasry-Lions monotonicity condition (2.5) in Cardaliaguet-Delarue-Lasry-Lions' monograph \cite{cardaliaguet2019master}; also see \eqref{eq_141} and \eqref{positive_g_mu_1} in Section \ref{subsec:dis} for more details. %

Under the assumptions that the Schur complement is positive-definite (also see \eqref{positive_g_x} and \eqref{positive_g_mu}), the terminal cost function $\wb{K}(X)$ is convex in $X$ (also see \eqref{positive_k} and \eqref{positive_k_mu}) and $D_{A} F(X,A)^\top D_{A} F(X,A)$ is positive-definite (also see \eqref{positive_f}), we can derive from \eqref{lifted:estimate_1} that there exists a positive constant $c$ independent of $T$ such that 
\begin{align}\no
&c\left\|D_Y X^{t,Y}_T \right\|_{\mathcal{L}(L^{2,d_x}_{m_0};L^{2,d_x}_{m_0})}^2+c\int_t^T \bigg(\left\|D_Y Z^{t,Y}_s \right\|_{\mathcal{L}(L^{2,d_x}_{m_0};L^{2,d_x}_{m_0})}^2+\left\|D_Y X^{t,Y}_s \right\|_{\mathcal{L}(L^{2,d_x}_{m_0};L^{2,d_x}_{m_0})}^2\bigg) ds\\\label{lifted:estimate_2}
\leq&\  \bigg\|D_Y Z^{t,Y}_t \bigg\|_{\mathcal{L}(L^{2,d_x}_{m_0};L^{2,d_x}_{m_0})};
\end{align}
see Appendix \ref{lifted:estimate_2_detail} for a derivation of \eqref{lifted:estimate_2} and also see \eqref{L_star_1_1} and \eqref{cru_est_2} for the corresponding unlifted version of \eqref{lifted:estimate_2}. Also note that, by using the boundness assumptions on the derivatives on the right hand side of \eqref{lifted:D_YXexistence} (also see \eqref{bdd_d1_f}, \eqref{bdd_d2_f}, \eqref{bdd_d2_g_1},\eqref{bdd_d2_g_2} and \eqref{bdd_d2_k_1}), there exists a positive constant $C$ independent of $T$ such that
\begin{align}\no
&\left\|D_Y Z^{t,Y}_t\right\|_{\mathcal{L}(L^{2,d_x}_{m_0};L^{2,d_x}_{m_0})}^2=\left\|D_Y Z^{t,Y}_T\right\|_{\mathcal{L}(L^{2,d_x}_{m_0};L^{2,d_x}_{m_0})}^2-\int_t^T \frac{d}{ds}\left\|D_Y Z^{t,Y}_s\right\|_{\mathcal{L}(L^{2,d_x}_{m_0};L^{2,d_x}_{m_0})}^2ds\\\label{lifted:estimate_3}
\leq &\ C\left\|D_Y X^{t,Y}_T\right\|_{\mathcal{L}(L^{2,d_x}_{m_0};L^{2,d_x}_{m_0})}^2+C\int_t^T \bigg(\left\|D_Y X^{t,Y}_s\right\|_{\mathcal{L}(L^{2,d_x}_{m_0};L^{2,d_x}_{m_0})}^2+\left\|D_Y Z^{t,Y}_s\right\|_{\mathcal{L}(L^{2,d_x}_{m_0};L^{2,d_x}_{m_0})}^2\bigg)ds;
\end{align}\normalsize
also see \eqref{p_x_gamma} and \eqref{p_m_gamma_int} for the corresponding unlifted version of \eqref{lifted:estimate_3}. Therefore, by using \eqref{lifted:estimate_2} and \eqref{lifted:estimate_3}, we can obtain the desired {\it a priori} estimate on $D_Y Z^{t,Y}_t$, that is, $\left\|D_Y Z^{t,Y}_t\right\|_{\mathcal{L}(L^{2,d_x}_{m_0};L^{2,d_x}_{m_0})}\leq C/c$ uniformly in $T$.

However, the above formal argument for obtaining the crucial {\it a priori} estimate in lifted version are based on the existence of the second-order Fr\'{e}chet derivatives of the lifted function $F(X,A)$, $\wb{G}(X,A)$ and $\wb{K}(X)$, and we would like to explain why the second-order Fr\'{e}chet differentiability is too restrictive, especially while considering a generic drift function. Consider a smooth and linear-growth function $f:x\in\R\mapsto f(x)\in\R$ and its lifted version $F:X\in L^2(\R)\to F(X)(\cdot):=f(X(\cdot))\in L^2(\R)$; suppose that $F$ is second-order Fr\'{e}chet differentiable, then there exists a positive constant $C$ such that $\big\|D_XD_XF(X)(\wt{X},\wh{X})\big\|_{L^2}\leq C\big\|\wt{X}\big\|_{L^2}\cdot\big\|\wh{X}\big\|_{L^2}$ for any $\wt{X}$, $\wh{X}\in L^2(\R)$; this inequality implies that $\p_x\p_x f(x)=0$ and thus $f(x)$ must be linear, since $\big\|D_XD_XF(X)(\wt{X},\wh{X})\big\|_{L^2}=\big(\int_{\R}\big|\p_x\p_xf(X(y))\wt{X}(y)\wh{X}(y)\big|^2\big)^{1/2}$; one can also see Example 2.3 in \cite{buckdahn2017mean}.
Therefore, to get rid of this restriction (so that we can also tackle the mean field type control problems with generic drift functions), we solve the generic mean field type control problem directly on the Wasserstein space by using the manipulations and estimations analogous to that in the lifted version, which is proven to be a successful strategy in the rest of this paper. For example, we shall separately estimate both the derivatives $\p_x Z^{t,m}_s(x)$ and $\p_m Z^{t,m}_s(x)$ of the pull-down version $Z^{t,m}_s(x)$ of $Z^{t,Y}_s$ instead of $D_Y Z^{t,Y}_s$, see \eqref{L_star_1} and \eqref{L_star_6_1}.

\section{Assumptions and Preliminary Results}\label{sec:assumptions}
\subsection{Assumptions}\label{assumptions}
We aim at solving the first order mean field type control problem \eqref{SDE1}-\eqref{valfun} under the following assumptions.\\
$\bf{(a1)}$ The drift function $f:(x,\mu,\alpha)\in \R^{d_x}\times \mathcal{P}_{2}(\mathbb{R}^{d_x})\times \R^{d_\alpha}\mapsto \R^{d_x}$ is second-order differentiable in $x\in \R^{d_x}$ and $\alpha\in \R^{d_\alpha}$ and is also second-order regularly differentiable in $\mu\in\mathcal{P}_2(\R^{d_x})$, so that all of the following derivatives exist and are jointly continuous in their corresponding arguments, and they also satisfy the following estimates. \\
$(i)$ For $f_i$ as the $i$-th component function of $f$,
\begin{align}\label{positive_f}
&\sum_{i=1}^{d_x}\sum_{j=1}^{d_x}\bigg(\xi_i \p_\alpha f_i(x,\mu,\alpha)\bigg)\cdot \bigg(\xi_j \p_\alpha f_j(x,\mu,\alpha)\bigg) \geq \lambda_f|\xi|^2,\ \forall \xi=(\xi_1,...,\xi_{d_x})\in\R^{d_x}.
\end{align}
$(ii)$ We have the following bounds on the first order derivatives of $f$:
\begin{align}\no&\sup_{(x,\mu,\alpha,\wt{x})\in \R^{d_x}\times \mathcal{P}_{2}(\mathbb{R}^{d_x})\times \R^{d_\alpha}\times\R^{d_x}}
\Big\{\|\p_\alpha f(x,\mu,\alpha)\|_{\mathcal{L}(\R^{d_\alpha};\R^{d_x})}\vee\|\p_\mu f(x,\mu,\alpha)(\wt{x})\|_{\mathcal{L}(\R^{d_x};\R^{d_x})}\\
\label{bdd_d1_f}&\ \ \ \ \ \ \ \ \ \ \ \ \ \ \ \ \ \ \ \ \ \ \ \ \ \ \ \ \ \ \ \ \ \ \ \ \ \ \ \ \ \vee\|\p_x f(x,\mu,\alpha)\|_{\mathcal{L}(\R^{d_x};\R^{d_x})}\Big\}
\leq \Lambda_f.
\end{align}
$(iii)$ We have the following bounds on the second order derivatives of $f$:
\begin{align}
\no\ &\sup_{(x,\mu,\alpha,\wt{x},\wh{x})\in \R^{d_x}\times \mathcal{P}_{2}(\mathbb{R}^{d_x})\times \R^{d_\alpha}\times\R^{d_x}\times\R^{d_x}}
\Big\{\|\p_\mu\p_\mu f(x,\mu,\alpha)(\wt{x},\wh{x})\|_{\mathcal{L}(\R^{d_x}\times \R^{d_x};\R^{d_x})}\vee\|\p_\alpha\p_\alpha f(x,\mu,\alpha)\|_{\mathcal{L}(\R^{d_\alpha}\times \R^{d_\alpha};\R^{d_x})}\\
\no&\vee\|\p_\mu\p_\alpha f(x,\mu,\alpha)(\wt{x})\|_{\mathcal{L}(\R^{d_\alpha}\times \R^{d_x};\R^{d_x})}\vee\|\p_\alpha\p_x f(x,\mu,\alpha)\|_{\mathcal{L}(\R^{d_x}\times \R^{d_\alpha};\R^{d_x})}\vee\|\p_x\p_x f(x,\mu,\alpha)\|_{\mathcal{L}(\R^{d_x}\times \R^{d_x};\R^{d_x})}
\\
\label{bdd_d2_f}&\vee\|\p_{\wt{x}}\p_\mu f(x,\mu,\alpha)(\wt{x})\|_{\mathcal{L}(\R^{d_x}\times \R^{d_x};\R^{d_x})}\vee\|\p_x\p_\mu f(x,\mu,\alpha)(\wt{x})\|_{\mathcal{L}(\R^{d_x}\times \R^{d_x};\R^{d_x})}
\Big\}\cdot (1+|x|+\|\mu\|_1)\leq \wb{l}_f.
\end{align}
$(iv)$ The second order derivatives are jointly Lipschitz continuous in their corresponding arguments with Lipschitz constant decaying at the proximity of infinity: for any $x,\,x',\,\wt{x},\,\wt x',\,\wh x,\,\wh x'\in\R^{d_x}$, $\alpha,\,\alpha'\in\R^{d_\alpha}$ and $\mu,\,\mu'\in\mc{P}_2(\R^{d_x})$,
\begin{align}\no
a)\ \ &\big\|\p_\mu\p_\mu f(x,\mu,\alpha)(\wt{x},\wh{x})-\p_\mu\p_\mu f(x',\mu',\alpha')(\wt{x}',\wh{x}')\big\|_{\mathcal{L}(\R^{d_x}\times \R^{d_x};\R^{d_x})}\\
\leq &\ \frac{\Lip_f}{1+\max\{|x|,|x'|\}+\max\{\|\mu\|_1,\|\mu'\|_1\}}\Big(|x-x'|+W_2(\mu,\mu')+|\alpha-\alpha'|+|\wt{x}-\wt{x}'|+|\wh{x}-\wh{x}'|\Big),\\\no 
b)\ \ &\big\|\p_\alpha\p_\alpha f(x,\mu,\alpha)-\p_\alpha\p_\alpha f(x',\mu',\alpha')\big\|_{\mathcal{L}(\R^{d_\alpha}\times \R^{d_\alpha};\R^{d_x})}\\
\leq &\ \frac{\Lip_f}{1+\max\{|x|,|x'|\}+\max\{\|\mu\|_1,\|\mu'\|_1\}}\Big(|x-x'|+W_2(\mu,\mu')+|\alpha-\alpha'|\Big),
\end{align}
and so are for $\p_\mu\p_\alpha f(x,\mu,\alpha)(\wt{x})$, $\p_\alpha\p_x f(x,\mu,\alpha)$, $\p_x\p_x f(x,\mu,\alpha)$, $\p_{\wt{x}}\p_\mu f(x,\mu,\alpha)(\wt{x})$ and $\p_x\p_\mu f(x,\mu,\alpha)(\wt{x})$;
here $\lambda_f$, $\Lambda_f$, $\wb{l}_f$ and $\Lip_f$ are some positive finite constants, and $\|\cdot\|_{\mathcal{L}(\mathcal{A};\mathcal{B})}$ is the operator norm of a linear mapping from $\mathcal{A}$ to $\mathcal{B}$.\\
$\bf{(a2)}$ The running cost function $g:(x,\mu,\alpha)\in \R^{d_x}\times \mathcal{P}_{2}(\mathbb{R}^{d_x})\times \R^{d_\alpha}\mapsto \R$ is second-order differentiable in $x\in \R^{d_x}$ and $\alpha\in \R^{d_\alpha}$ and  second-order regularly differentiable in $\mu\in\mathcal{P}_2(\R^{d_x})$ such that 
these second order derivatives are jointly Lipschitz continuous in their respective arguments, and they also satisfy the following estimates.\\ 
$(i)$ For all $\xi\in\R^{d_\alpha}$, 
\begin{align}
\label{positive_g_alpha}
\xi^\top\p_\alpha\p_\alpha g(x,\mu,\alpha)\xi\geq \lambda_g|\xi|^2.
\end{align}
$(ii)$ Define\footnote{Analogous to Lasry-Lions monotonicity condition.}
\begin{align}\label{def_G}
&G(x,\mu,\alpha):=\frac{\delta}{\delta\mu}\left(\int_{\R^{d_x}}g(\wt{x},\mu,\alpha)d\mu(\wt{x})\right)(x)=g(x,\mu,\alpha)+\int_{\R^{d_x}}\frac{\delta}{\delta\mu}g(\wt{x},\mu,\alpha)(x)d\mu(\wt{x}),
\end{align}
then, for all $\xi\in\R^{d_x}$ and any $\mu,\,\mu'\in\mathcal{P}_2(\R^{d_x})$,
\begin{align}\label{positive_g_x}
& a)\ \xi^\top\p_x\p_x G(x,\mu,\alpha)\xi\geq \lambda_{g}|\xi|^2>0,\\\label{positive_g_mu}
& b)\ \int_{\R^{d_x}} \Big(G(x,\mu',\alpha)-G(x,\mu,\alpha)\Big)d\big(\mu'-\mu\big)(x)\geq -l_g \left(\int_{\R} yd(\mu-\mu')(y)\right)^2 \text{ with }l_g<\lambda_g;
\end{align}
to simplify the computation of some constants\footnote{Indeed, our results will also be held as long as $l_g<\lambda_g$, but various constants in this paper depend on the ratio of $l_g$ to $\lambda_g$. Thus, to keep our presentation concise, we assume $l_g\leq \frac12 \lambda_g$ for simplicity.} in the proof of Theorem \ref{Crucial_Estimate}, we further assume that $l_g\leq \frac{1}{2}\lambda_g$.
\\$(iii)$ We have the following bounds on the second order derivatives of $g$:
\begin{align} \no
\ a)\ &\sup_{(x,\mu,\alpha,\wt{x},\wh{x})\in \R^{d_x}\times \mathcal{P}_{2}(\mathbb{R}^{d_x})\times \R^{d_\alpha}\times\R^{d_x}\times\R^{d_x}}
\Big\{\|\p_\mu\p_\mu g(x,\mu,\alpha)(\wt{x},\wh{x})\|_{\mathcal{L}(\R^{d_x}\times \R^{d_x};\R)}\vee\|\p_\alpha\p_\alpha g(x,\mu,\alpha)\|_{\mathcal{L}(\R^{d_\alpha}\times \R^{d_\alpha};\R)}\\
\label{bdd_d2_g_1}&\vee\|\p_x\p_\mu g(x,\mu,\alpha)(\wt{x})\|_{\mathcal{L}(\R^{d_x}\times \R^{d_x};\R)}\vee\|\p_x\p_x g(x,\mu,\alpha)\|_{\mathcal{L}(\R^{d_x}\times \R^{d_x};\R)}\vee\|\p_{\wt{x}}\p_\mu g(x,\mu,\alpha)(\wt{x})\|_{\mathcal{L}(\R^{d_x}\times \R^{d_x};\R)}\Big\}\leq \Lambda_g,\\\label{bdd_d2_g_2}
\ b)\ &\sup_{(x,\mu,\alpha,\wt{x})\in \R^{d_x}\times \mathcal{P}_{2}(\mathbb{R}^{d_x})\times \R^{d_\alpha}\times\R^{d_x}}
\Big\{\|\p_x\p_\alpha g(x,\mu,\alpha)\|_{\mathcal{L}(\R^{d_\alpha}\times \R^{d_x};\R)}\vee\|\p_\alpha\p_\mu g(x,\mu,\alpha)(\wt{x})\|_{\mathcal{L}(\R^{d_x}\times \R^{d_\alpha};\R)}\Big\}\leq \wb{l}_g,
\end{align}
here $\lambda_g$, $\Lambda_g$ and $\wb{l}_g$ are some positive finite constants, and $l_g$ is a finite non-negative constant.\\
$\bf{(a3)}$ The terminal cost function $k:(x,\mu)\in \R^{d_x}\times \mathcal{P}_{2}(\mathbb{R}^{d_x})\mapsto \R$ is second-order differentiable in $x\in \R^{d_x}$ and second-order regularly differentiable in $\mu\in\mathcal{P}_2(\R^{d_x})$ so that 
the second order derivatives are jointly Lipschitz continuous in their respective arguments; also these derivatives satisfy the following estimates: 
\\$(i)$ Define\footnote{Analogous to Lasry-Lions monotonicity condition.}
\begin{align}\label{def_K}
&K(x,\mu):=\frac{\delta}{\delta\mu}\left(\int_{\R^{d_x}}k(\wt{x},\mu)d\mu(\wt{x})\right)(x)=k(x,\mu)+\int_{\R^{d_x}}\frac{\delta}{\delta\mu}k(\wt{x},\mu)(x)d\mu(\wt{x}),
\end{align}
then, for all $\xi\in\R^{d_x}$ and any $\mu,\,\mu'\in\mathcal{P}_2(\R^{d_x})$,
\begin{align}\label{positive_k}
& a)\ \xi^\top\p_x\p_x K(x,\mu)\xi\geq \lambda_{k}|\xi|^2>0,\\\label{positive_k_mu}
& b)\ \int_{\R^{d_x}} \Big(K(x,\mu')-K(x,\mu)\Big)d\big(\mu'-\mu\big)(x)\geq -l_k \left(\int_{\R} yd(\mu-\mu')(y)\right)^2 \text{ with }l_k<\lambda_k;
\end{align}
to simplify the computation of some constants\footnote{Indeed, our results will also be held as long as $l_k<\lambda_k$, but various constants in this paper depend on the ratio of $l_k$ to $\lambda_k$. Thus, to keep our presentation concise, we assume $l_k\leq \frac12 \lambda_k$ for simplicity.} in the proof of Theorem \ref{Crucial_Estimate}, we further assume that $l_k\leq \frac{1}{2}\lambda_k$.
\\$(ii)$ We have the following bounds on the second order derivatives of $k$:
\begin{align}
\no\ &\sup_{(x,\mu,\wt{x},\wh{x})\in \R^{d_x}\times \mathcal{P}_{2}(\mathbb{R}^{d_x})\times\R^{d_x}\times\R^{d_x}}
\Big\{\|\p_\mu\p_\mu k(x,\mu)(\wt{x},\wh{x})\|_{\mathcal{L}(\R^{d_x}\times \R^{d_x};\R)}\vee\|\p_x\p_\mu k(x,\mu)(\wt{x})\|_{\mathcal{L}(\R^{d_x}\times \R^{d_x};\R)}\\
\label{bdd_d2_k_1}&\ \ \ \ \ \ \ \ \ \ \ \ \ \ \ \ \ \ \ \ \ \ \ \ \ \ \ \ \ \ \ \ \ \ \ \ \ \ \ \ \vee\|\p_x\p_x k(x,\mu)\|_{\mathcal{L}(\R^{d_x}\times \R^{d_x};\R)}\vee\|\p_{\wt{x}}\p_\mu k(x,\mu)(\wt{x})\|_{\mathcal{L}(\R^{d_x}\times \R^{d_x};\R)}\Big\}\leq \Lambda_k;
\end{align}
here $\lambda_k$ and $\Lambda_k$ are some positive finite constants, and $l_k$ is a finite non-negative constant.

\subsection{Derived Properties of Coefficient Functions and Controls}\label{properties}
Under Assumptions $\bf{(a1)}$-$\bf{(a3)}$, we have the following useful properties which play a vital role in developing the general theory. All the proofs in this section are put in Appendix \ref{App:prop}.

\begin{prpstn}\label{A1} 
The drift function $f(x,\mu,\alpha)$ is of linear-growth and Lipschitz continuous, that is, there exists a positive constant $L_f$, which can be taken as $L_f:=\max\{\Lambda_f,|f(0,\delta_0,0)|\}$, such that 
\begin{align}\label{ligf}
&|f(x,\mu,\alpha)|\leq L_f \left(1+|x|+ \|\mu\|_1+|\alpha| \right);\\\label{Lipf}
&\left|f(x,\mu,\alpha) - f(x',\mu',\alpha')\right|\leq L_f \big(|x-x'|+W_1(\mu,\mu')+|\alpha-\alpha'|\big).
\end{align} 
\end{prpstn}

\begin{prpstn}\label{A2} 
(Bounding first-order derivatives) We have the following estimates:\\
\begin{align}\label{Lipg_mu}
&\left|\p_\mu g(x,\mu,\alpha)(\wt{x}) - \p_\mu g(x',\mu',\alpha')(\wt{x}')\right|\leq \max\{\Lambda_g,\wb{l}_g\} \left(W_1(\mu,\mu')+|x-x'|+|\alpha-\alpha'|+|\wt{x}-\wt{x}'|\right);\\\label{Lipg_x}
&\left|\p_x g(x,\mu,\alpha) - \p_x g(x',\mu',\alpha')\right|\leq \max\{\Lambda_g,\wb{l}_g\} \left(W_1(\mu,\mu')+|x-x'|+|\alpha-\alpha'|\right);\\\label{Lipg_alpha}
&\left|\p_\alpha g(x,\mu,\alpha) - \p_\alpha g(x',\mu',\alpha')\right|\leq \max\{\Lambda_g,\wb{l}_g\} \left(W_1(\mu,\mu')+|x-x'|+|\alpha-\alpha'|\right);
\end{align}
\begin{align}\no
&\sup_{x\in\R^{d_x},\,\mu\in\mathcal{P}_2(\R^{d_x}),\,\alpha\in\R^{d_\alpha},\,\wt{x}\in\R^{d_x}}\frac{\left|\p_x g(x,\mu,\alpha)\right|}{1+|x|+ \|\mu\|_1+|\alpha|}
\vee\frac{\left|\p_\alpha g(x,\mu,\alpha)\right|}{1+|x|+ \|\mu\|_1+|\alpha|}\vee\frac{|\p_{\mu} g(x,\mu,\alpha)(\wt{x})|}{1+|x|+ \|\mu\|_1+|\alpha|+|\wt{x}|}\\\label{LGgDerivatives}
\leq&\ L_g:=\max\Big\{|\p_\mu g(0,\delta_0,0)(0)|,|\p_x g(0,\delta_0,0)|,|\p_\alpha g(0,\delta_0,0)|,\Lambda_g,\wb{l}_g\Big\};
\end{align}
and 
\begin{align}\label{LGkDerivatives}
\sup_{x\in\R^{d_x},\,\mu\in\mathcal{P}_2(\R^{d_x}),\,\wt{x}\in\R^{d_x}}\frac{\left|\p_x k(x,\mu)\right|}{1+|x|+ \|\mu\|_1}\vee\frac{|\p_\mu k(x,\mu)(\wt{x})|}{1+|x|+ \|\mu\|_1+|\wt{x}|}\leq L_k:=\max\Big\{|\p_\mu k(0,\delta_0)(0)|,|\p_x k(0,\delta_0)|,\Lambda_k\Big\}.
\end{align}  
\end{prpstn}
Define, a constant 
\begin{align}\label{k_0}
k_0:=4\max\left\{\wb{L}_k,L^*_0\right\}>0,
\end{align}
where $L^*_0$ will be defined later in \eqref{L_star_0} and it is a positive constant depending only on $\lambda_k,\ \Lambda_k,\ \lambda_g,\ \Lambda_g,\ \lambda_f$ and $\Lambda_f$, and 
\begin{align}\label{wbL_k}
\wb{L}_k:=3\Lambda_k.
\end{align}
In fact, we can show later in Remark \ref{remark_lower_bdd_L_star_0} that $\wb{L}_k\leq L^*_0$ and thus $k_0=4 L^*_0$.
 We further define a cone set:
\begin{align}\label{c_k_0}
c_{k_0}:=&\ \left\{(x,\mu,z)\in \R^{d_x}\times \mathcal{P}_2(\R^{d_x})\times \R^{d_x}:|z|\leq \frac{1}{2}k_0\left(1+|x|+\|\mu\|_1\right)\right\},
\end{align}
for the constant $k_0>0$ as defined in \eqref{k_0}. In fact, we can show in Theorem \ref{Thm6_1} (local existence of solution) and Theorem \ref{GlobalSol} (global existence of solution) that $\big(X^{t,m}_s(x),X^{t,m}_s\ot m, Z^{t,m}_s(x)\big)\in c_{k_0}$, where $\big(X^{t,m}_s(x),Z^{t,m}_s(x)\big)$ is
the solution pair of the FBODE system \eqref{MPIT}; this property is used to ensure the unique existence of the optimal control $\wh{\alpha}(x,\mu,z)$ stated in Proposition \ref{A5}, in which this additional growth confinement will be stated in its hypothesis is indeed justifiable.
\begin{prpstn}\label{h3}
For each fixed $(x,\mu,z)\in \R^{d_x}\times \mathcal{P}_2(\R^{d_x})\times \R^{d_x}$, we have $\displaystyle\lim_{|\alpha|\to \infty}f(x,\mu,\alpha)\cdot z+g(x,\mu,\alpha)\to \infty$. 
\end{prpstn}
\begin{prpstn}\label{A5}
Suppose further that 
\begin{align}\label{p_aa_f}
\sup_{(x,\mu,\alpha)\in \R^{d_x}\times \mathcal{P}_{2}(\mathbb{R}^{d_x})\times \R^{d_\alpha}}\|\p_\alpha\p_\alpha f(x,\mu,\alpha)\|_{\mathcal{L}(\R^{d_\alpha}\times \R^{d_\alpha};\R^{d_x})}\cdot (1+|x|+\|\mu\|_1) \leq \frac{1}{40 \max\{\wb{L}_k,L^*_0\}}\lambda_g.
\end{align} 
Then, for each $(x,\mu,z)\in c_{k_0}$, the function $f(x,\mu,\alpha)\cdot z  + g(x,\mu,\alpha)$ has a unique minimizer $\wh{\alpha}(x,\mu,z)\in\R^{d_\alpha}$, which satisfies the first order condition 
\begin{align}\label{first_order_condition}
\p_\alpha f(x,\mu,\wh{\alpha})\cdot z  + \p_\alpha g(x,\mu,\wh{\alpha})=0,
\end{align} and it is first-order differentiable in $x$, $z\in\R^{d_x}$ and $L$-differentiable in $\mu\in\mathcal{P}_2(\R^{d_x})$ with jointly Lipschitz continuous first-order derivatives such that
\small\begin{align}\label{eq_5_25_new}
\partial_x \wh{\alpha}(x,\mu,z)=&\ -\Big(\p_\alpha\p_\alpha f(x,\mu,\wh{\alpha})\cdot z+\p_\alpha\p_\alpha g(x,\mu,\wh{\alpha})\Big)^{-1} \left(\p_x\p_\alpha f(x,\mu,\wh{\alpha})\cdot z+\p_x\p_\alpha g(x,\mu,\wh{\alpha})\right),\\
\partial_z \wh{\alpha}(x,\mu,z)=&\ -\Big(\p_\alpha\p_\alpha f(x,\mu,\wh{\alpha})\cdot z+\p_\alpha\p_\alpha g(x,\mu,\wh{\alpha})\Big)^{-1} \p_\alpha f(x,\mu,\wh{\alpha}),\\\label{eq_5_27_new}
\p_\mu \wh{\alpha}(x,\mu,z)(\wt{x})=&\ -\Big(\p_\alpha\p_\alpha f(x,\mu,\wh{\alpha})\cdot z+\p_\alpha\p_\alpha g(x,\mu,\wh{\alpha})\Big)^{-1}\left(\p_\mu\p_\alpha f(x,\mu,\wh{\alpha})(\wt{x})\cdot z+\p_\mu\p_\alpha g(x,\mu,\wh{\alpha})(\wt{x})\right).
\end{align}\normalsize
In addition, we have the following estimates:\\ 
(i) for all $(x,\mu,z)\in c_{k_0}$ and an arbitrary $\xi\in\R^{d_x}$,
\begin{align}\label{positive_h}
&\Big(\Lambda_g+\frac{1}{20}\lambda_g\Big)|\xi|^2\geq \xi^\top\Big(\p_\alpha\p_\alpha f(x,\mu,\wh{\alpha})\cdot z+\p_\alpha\p_\alpha g(x,\mu,\wh{\alpha})\Big)\xi\geq \frac{19}{20}\lambda_g |\xi|^2;
\end{align}
(ii) for all $(x,\mu,z)\in c_{k_0}$,
\begin{align}\label{p_xalpha}
&\Big\|\p_{x}\wh{\alpha}(x,\mu,z)\Big\|_{\mathcal{L}(\R^{d_x};\R^{d_\alpha})}\vee\Big\|\p_{\mu}\wh{\alpha}(x,\mu,z)(\wt{x})\Big\|_{\mathcal{L}(\R^{d_x};\R^{d_\alpha})}\leq\frac{20\big(\wb{l}_g+\frac{1}{2}k_0\cdot\wb{l}_f\big) }{19\lambda_g},\\\label{p_zalpha}
&\Big\|\p_{z}\wh{\alpha}(x,\mu,z)\Big\|_{\mathcal{L}(\R^{d_x};\R^{d_\alpha})}\leq \frac{20\Lambda_f}{19\lambda_g};
\end{align}\normalsize
(iii) for any $(x,\mu,z)\in c_{k_0}$, and any $(x',\mu',z')\in c_{k_0}$ such that $\theta (x,\mu,z)+(1-\theta)(x',\mu',z')\in c_{k_0}$ for all $\theta\in[0,1]$
\begin{align}\label{lipalpha_old}
&|\wh{\alpha}(x,\mu,z)-\wh{\alpha}(x',\mu',z')|\leq\  L_\alpha\big(|x-x'|+W_1(\mu,\mu')+|z-z'|\big),\\\label{linalpha}
&|\wh{\alpha}(x,\mu,z)|\leq\  L_\alpha\cdot(1+|x|+\|\mu\|_1+|z|),
\end{align}
where 
\begin{align}\label{L_alpha}
L_\alpha:=\max\Big\{\frac{20\Lambda_f}{19\lambda_g},\frac{20\big(\wb{l}_g+\frac{1}{2}k_0\cdot\wb{l}_f\big) }{19\lambda_g},\wh{\alpha}(0,\delta_0,0)\Big\};
\end{align}
(iv) for any function $\Gamma:(x,\mu)\in\R^{d_x}\times\mc{P}_2(\R^{d_x})\mapsto \Gamma(x,\mu)\in\R^{d_x}$ which is differentiable in $x$ and is $L$-differentiable in $\mu$ with jointly continuous derivatives such that $\big\|\p_{x}\Gamma(x,\mu)\big\|_{\mathcal{L}(\R^{d_x};\R^{d_x})}\vee\big\|\p_{\mu}\Gamma(x,\mu)(\wt{x})\big\|_{\mathcal{L}(\R^{d_x};\R^{d_x})}\leq L_\Gamma$ for some positive constant $L_\Gamma\leq \frac{1}{2}k_0$, then, $|\Gamma(x,\mu)|\leq L_\Gamma(1+|x|+\|\mu\|_1)$ and
\begin{align}\label{lipalpha}
&|\wh \alpha(x,\mu,\Gamma(x,\mu))-\wh \alpha(x',\mu',\Gamma(x',\mu'))|\leq L_\alpha(1+L_\Gamma)(|x-x'|+W_1(\mu,\mu')).
\end{align}
\end{prpstn}

\begin{prpstn}\label{A6}
For the cone $c_{k_0}$ defined in \eqref{c_k_0} and any $\xi\in\R^{d_x}$, we have the following inequality:
\begin{align}\label{c_a7}
&\sup_{(x,\mu,z)\in c_{k_0}}\xi^\top\Big(\sum_{i=1}^{d_\alpha}\p_z \wh{\alpha}_i(x,\mu,z)\otimes \p_{\alpha_i} f(x,\mu,\wh{\alpha})\Big)\xi\leq -\frac{\lambda_f^2}{\Lambda_g+\frac{1}{2}k_0\cdot \wb{l}_f}|\xi|^2,
\end{align}\normalsize
where $\wh{\alpha}(x,\mu,z)$ is the unique minimizer solving the first order condition \eqref{first_order_condition} and we denote $a\otimes b:=(a_ib_j)_{n\times m}$ is an $n\times m$ matrix for arbitrary vectors $a\in\R^n$ and $b\in \R^m$.
\end{prpstn}

\subsection{Discussion on Our Assumptions}\label{subsec:dis}
In this subsection, we would like to provide some motivations and explanations regarding Assumptions ${\bf (a1)}$-${\bf (a3)}$ and the property \eqref{c_a7}, by comparing them with the coercivity and monotonicity conditions used in Cardaliaguet-Delarue-Lasry-Lions' monograph \cite{cardaliaguet2019master} and the displacement $\lambda$-convexity introduced in Gangbo and M\'esz\'aros \cite{gangbo2020global}. 

To simplify the calculus and computations, we consider a special case that the drift function is simply $f(x,\mu,\alpha)=\alpha$ and the running cost function is $g(x,\mu,\alpha)=\lambda\alpha^2+g_1(x,\mu)$ with a positive constant $\lambda$; one can see a more complicated example in Section \ref{sec:nonLQ} which is not in this special form. Then \eqref{positive_f}, \eqref{bdd_d1_f}, \eqref{bdd_d2_f} and \eqref{bdd_d2_g_2} are automatically satisfied with $\lambda_f=1$, $\Lambda_f=1$ and $\wb{l}_f=\wb{l}_g=0$. In addition, the optimal control $\wh{\alpha}$ can be solve explicitly via the first order condition \eqref{first_order_condition} as $\wh{\alpha}(x,\mu,z)=-\frac{1}{2\lambda}z$, and thus the left hand side of \eqref{c_a7} is equal to $-\frac{1}{2\lambda}|\xi|^2$. Here, we can see that the term $\Big(\sum_{i=1}^{d_\alpha}\p_z \wh{\alpha}_i(x,\mu,z)\otimes \p_{\alpha_i} f(x,\mu,\wh{\alpha})\Big)$ in \eqref{c_a7} is corresponding to the term $D^2_{zz}H(x,z)$ in (2.4) in Cardaliaguet-Delarue-Lasry-Lions' monograph \cite{cardaliaguet2019master} (due to the separable structure of the Hamiltonian is assumed in \cite{cardaliaguet2019master}, there is no mean field term $\mu$ in $H(x,z)$, while we can treat the non-separable case in this work), and thus \eqref{c_a7} is corresponding to the coercivity condition (2.4) in Cardaliaguet-Delarue-Lasry-Lions' monograph \cite{cardaliaguet2019master}. 

Also note that the function $G(x,\mu,\alpha)$ defined in \eqref{def_G} is corresponding to the function $F(x,\mu)$ in (2.5) in Cardaliaguet-Delarue-Lasry-Lions' monograph \cite{cardaliaguet2019master} (due to the separable structure of the running cost function $g(x,\mu,\alpha)=\lambda\alpha^2+g_1(x,\mu)$ is assumed in \cite{cardaliaguet2019master}, there is no control term $\alpha$ in $F(x,\mu)$). Therefore,  the estimate \eqref{positive_g_mu} assumed in Assumption ${\bf (a2)}$ generalizes the Lasry-Lions monotonicity assumption (see (2.5) in Cardaliaguet-Delarue-Lasry-Lions' monograph \cite{cardaliaguet2019master}), which is, for any $\mu,\,\mu'\in\mathcal{P}_2(\R^{d_x})$, 
\begin{align}\label{eq_141}
&\int_{\R^{d_x}} \Big(F(x,\mu')-F(x,\mu)\Big)d\big(\mu'-\mu\big)(x)\geq 0.
\end{align}\normalsize
Moreover, for any $m\in\mathcal{P}_2(\R^{d_x})$ and any $X,\,\wt{X}\in L^{2,d_x}_{m}$, by taking $\mu'=(X+\theta\wt{X})\ot m$ with $\theta>0$, $\mu=X\ot m$ and letting $\theta\to 0^+$ after dividing \eqref{positive_g_mu} by $\theta^2$, the estimate \eqref{positive_g_mu} implies that
\begin{align}\no
&\int_{\R^{d_x}}\int_{\R^{d_x}}(\wt{X}(\wt{x})\cdot \p_{x})(\wt{X}(\wh{x})\cdot \p_\mu)G\big(x,\mu,\alpha\big)\bigg|_{x=X(\wt{x}),\mu=X\ot m}\big(X(\wh{x})\big)dm(\wh{x})dm(\wt{x})\\\no
\geq &-l_g \lim_{\theta\to 0^+}\frac{1}{\theta^2}\left(\int_{\R} yd((X+\theta\wt{X})\ot m-X\ot m)(y)\right)^2\geq -l_g \lim_{\theta\to 0^+}\frac{1}{\theta^2}\left\|\theta\wt{X}\right\|_{L^{2,d_x}_m}^2=-l_g\left\|\wt{X}\right\|_{L^{2,d_x}_m}^2.
\end{align}
In fact, the estimate \eqref{positive_g_mu} assumed in Assumption ${\bf (a2)}$ can be replaced by
\small\begin{align}\no
&\int_{\R^{d_x}}\int_{\R^{d_x}}(\wt{X}(\wt{x})\cdot \p_{x})(\wt{X}(\wh{x})\cdot \p_\mu)G\big(x,\mu,\alpha\big)\bigg|_{x=X(\wt{x}),\mu=X\ot m}\big(X(\wh{x})\big)dm(\wh{x})dm(\wt{x})\\\no
=&\ \int_{\R^{d_x}}\int_{\R^{d_x}}\wt{X}(\wt{x})^\top\Big(\p_x\p_\mu g\big(X(\wt{x}),X\ot m,\alpha\big)\big(X(\wh{x})\big)\Big)\cdot\wt{X}(\wh{x})dm(\wh{x})dm(\wt{x})\\\no
&+\int_{\R^{d_x}}\int_{\R^{d_x}}\wt{X}(\wt{x})^\top\Big(\p_{x}\p_\mu g\big(X(\wh{x}),X\ot m,\alpha\big)\big(X(\wt{x})\big)\Big)\cdot\wt{X}(\wh{x})dm(\wh{x})dm(\wt{x})\\\no
&+\int_{\R^{d_x}}\int_{\R^{d_x}}\int_{\R^{d_x}}\wt{X}(\wt{x})^\top\Big(\p_\mu\p_\mu g\big(X(y),X\ot m,\alpha\big)\big(X(\wt{x}),X(\wh{x})\big)\Big)\cdot\wt{X}(\wh{x}) dm(y) dm(\wh{x})dm(\wt{x})\\\label{positive_g_mu_1}
\geq&\  -l_g\left\|\wt{X}\right\|_{L^{2,d_x}_m}^2,
\end{align}\normalsize
since we only use \eqref{positive_g_mu_1} in our proof of \eqref{displacement} in Theorem \ref{Crucial_Estimate} but not \eqref{positive_g_mu}. 
Also note that the estimate \eqref{positive_g_x} assumed in Assumption ${\bf (a2)}$ means that $G(x,\mu,\alpha)$ is strictly convex in $x$. 
Moreover, \eqref{positive_g_x} and \eqref{positive_g_mu} assumed in Assumption ${\bf (a2)}$ can be replaced by 
\small\begin{align}\no
&\int_{\R^{d_x}}\int_{\R^{d_x}}(\wt{X}(\wt{x})\cdot \p_{x})(\wt{X}(\wt{x})\cdot \p_x{})G\big(x,\mu,\alpha\big)\bigg|_{x=X(\wt{x}),\mu=X\ot m}dm(\wt{x})\\\label{dispalcement_lambda}
&+\int_{\R^{d_x}}\int_{\R^{d_x}}(\wt{X}(\wt{x})\cdot \p_{x})(\wt{X}(\wh{x})\cdot \p_\mu)G\big(x,\mu,\alpha\big)\bigg|_{x=X(\wt{x}),\mu=X\ot m}\big(X(\wh{x})\big)dm(\wh{x})dm(\wt{x})\geq \lambda\left\|\wt{X}\right\|_{L^{2,d_x}_m}^2,
\end{align}\normalsize
with $\lambda=\lambda_g>0$, which is $\lambda$-convexity, introduced in \cite{gangbo2020global}, of $\wb{G}(X):=\int_{\R^{d_x}}g(\wt{x},\mu,\alpha)d\mu(\wt{x})\big|_{\mu=X\ot m}$ on $X\in L^{2,d_x}_m$ for any fixed $\alpha\in\R^{d_\alpha}$ and $m\in\mc{P}_2(\R^{d_x})$, and is equivalent to the displacement $\lambda$-convexity of $\int_{\R^{d_x}}g(\wt{x},\mu,\alpha)d\mu(\wt{x})$ with $\lambda=\lambda_g$ by Lemma 3.6 in \cite{gangbo2020global}. The reasons of this replacement are as follows: i) \eqref{dispalcement_lambda} can imply \eqref{positive_g_x} by  
taking $\wt{X}(\wt{x})=\eps^{-d_x}\wh{X}(\eps^{-1} \wt{x})$ and then passing to the limit $\eps\to 0$ after multiplying \eqref{dispalcement_lambda} by $\eps^{d_x}$ (also see Proposition B.6 in \cite{gangbo2020global}); 
ii) \eqref{dispalcement_lambda} can be used to show \eqref{displacement} instead \eqref{positive_g_mu_1} in Theorem \ref{Crucial_Estimate}. On the other hand, \eqref{positive_g_x} and \eqref{positive_g_mu} can imply \eqref{dispalcement_lambda} with $\lambda=\lambda_g-l_g$; also refer to Definition 2.3 and Lemma 2.6 in \cite{gangbo2022mean} for its relationship with displacement monotonicity.

Likewise, the function $K(x,\mu)$ defined in \eqref{def_K} is corresponding to the function $G(x,\mu)$ in (2.5) in Cardaliaguet-Delarue-Lasry-Lions' monograph \cite{cardaliaguet2019master}, and thus the estimate \eqref{positive_k_mu} assumed in Assumption ${\bf (a3)}$ generalizes the Lasry-Lions monotonicity assumption (see (2.5) in Cardaliaguet-Delarue-Lasry-Lions' monograph \cite{cardaliaguet2019master}), which is
\small\begin{align}
&\int_{\R^{d_x}} \Big(G(x,\mu')-G(x,\mu)\Big)d\big(\mu'-\mu\big)(x)\geq 0,\ \forall \mu,\,\mu'\in\mathcal{P}_2(\R^{d_x}).
\end{align}
Again, for any $m\in\mathcal{P}_2(\R^{d_x})$ and any $X,\,\wt{X}\in L^{2,d_x}_{m}$, by taking $\mu'=(X+\theta\wt{X})\ot m$ with $\theta>0$, $\mu=X\ot m$ and letting $\theta\to 0^+$ after dividing \eqref{positive_k_mu} by $\theta^2$, the estimate \eqref{positive_k_mu} implies that
\begin{align}\no
&\int_{\R^{d_x}}\int_{\R^{d_x}}(\wt{X}(\wt{x})\cdot \p_{x})(\wt{X}(\wh{x})\cdot \p_\mu)K\big(x,\mu\big)\bigg|_{x=X(\wt{x}),\mu=X\ot m}\big(X(\wh{x})\big)dm(\wh{x})dm(\wt{x})\\\no
\geq &-l_k \lim_{\theta\to 0^+}\frac{1}{\theta^2}\left(\int_{\R} yd((X+\theta\wt{X})\ot m-X\ot m)(y)\right)^2\geq -l_k \lim_{\theta\to 0^+}\frac{1}{\theta^2}\left\|\theta\wt{X}\right\|_{L^{2,d_x}_m}^2=-l_k\left\|\wt{X}\right\|_{L^{2,d_x}_m}^2.
\end{align}
Again, the estimate \eqref{positive_k_mu} assumed in Assumption ${\bf (a3)}$ can be replaced by
\small\begin{align}\no
&\int_{\R^{d_x}}\int_{\R^{d_x}}(\wt{X}(\wt{x})\cdot \p_{x})(\wt{X}(\wh{x})\cdot \p_\mu)K\big(x,\mu\big)\bigg|_{x=X(\wt{x}),\mu=X\ot m}\big(X(\wh{x})\big)dm(\wh{x})dm(\wt{x})\\\no
=&\ \int_{\R^{d_x}}\int_{\R^{d_x}}\wt{X}(\wt{x})^\top\Big(\p_x\p_\mu k\big(X(\wt{x}),X\ot m\big)\big(X(\wh{x})\big)\Big)\cdot\wt{X}(\wh{x})dm(\wh{x})dm(\wt{x})\\\no
&+\int_{\R^{d_x}}\int_{\R^{d_x}}\wt{X}(\wt{x})^\top\Big(\p_{x}\p_\mu k\big(X(\wh{x}),X\ot m\big)\big(X(\wt{x})\big)\Big)\cdot\wt{X}(\wh{x})dm(\wh{x})dm(\wt{x})\\\no
&+\int_{\R^{d_x}}\int_{\R^{d_x}}\int_{\R^{d_x}}\wt{X}(\wt{x})^\top\Big(\p_\mu\p_\mu k\big(X(y),X\ot m\big)\big(X(\wt{x}),X(\wh{x})\big)\Big)\cdot\wt{X}(\wh{x}) dm(y) dm(\wh{x})dm(\wt{x})\\\label{positive_k_mu_1}
\geq&\  -l_k\left\|\wt{X}\right\|_{L^{2,d_x}_m}^2,
\end{align}\normalsize
since we only use \eqref{positive_k_mu_1} in our proof of \eqref{displacement} in Theorem \ref{Crucial_Estimate} but not \eqref{positive_k_mu}. 
Note again that the estimate \eqref{positive_k} assumed in Assumption ${\bf (a3)}$ means that $K(x,\mu)$ is strictly convex in $x$. 
Again, \eqref{positive_k} and \eqref{positive_k_mu} assumed in Assumption ${\bf (a3)}$ can also be replaced by 
\small\begin{align}\no
&\int_{\R^{d_x}}\int_{\R^{d_x}}(\wt{X}(\wt{x})\cdot \p_{x})(\wt{X}(\wt{x})\cdot \p_x{})K\big(x,\mu\big)\bigg|_{x=X(\wt{x}),\mu=X\ot m}dm(\wt{x})\\\label{dispalcement_lambda_k}
&+\int_{\R^{d_x}}\int_{\R^{d_x}}(\wt{X}(\wt{x})\cdot \p_{x})(\wt{X}(\wh{x})\cdot \p_\mu)K\big(x,\mu\big)\bigg|_{x=X(\wt{x}),\mu=X\ot m}\big(X(\wh{x})\big)dm(\wh{x})dm(\wt{x})\geq \lambda\left\|\wt{X}\right\|_{L^{2,d_x}_m}^2
\end{align}\normalsize
with $\lambda=\lambda_k>0$.

In summary, \eqref{c_a7} has a close relationship with the coercivity condition (2.4) of the Hamiltonian in Cardaliaguet-Delarue-Lasry-Lions' monograph \cite{cardaliaguet2019master}, and \eqref{positive_g_mu} and \eqref{positive_k_mu} have a close relationship with the monotonicity condition (2.5) in Cardaliaguet-Delarue-Lasry-Lions' monograph \cite{cardaliaguet2019master}. In a certain sense, we generalize the coercivity condition (2.4) of the Hamiltonian and the monotonicity condition (2.5) in Cardaliaguet-Delarue-Lasry-Lions' monograph \cite{cardaliaguet2019master} to the case with non-separable Hamiltonian and general drift function, and \eqref{positive_g_mu} and \eqref{positive_k_mu} generalize the monotonicity condition (2.5) in Cardaliaguet-Delarue-Lasry-Lions' monograph \cite{cardaliaguet2019master}. In addition, \eqref{c_a7} motivates our assumptions \eqref{positive_f}, \eqref{bdd_d2_f} and \eqref{bdd_d2_g_1}.

\section{Bellman Equation: First Order Condition and Pontryagin Maximum Principle}\label{sec:Bellman}
Applying the principle of dynamic programming, see \cite{bayraktar2018randomized,bensoussan2017interpretation,bensoussan2015master,carmona2018probabilistic} for example, 
 one can show that the value function $v(t,m)$ defined in \eqref{valfun} is expected to satisfy the Bellman equation:
\begin{align}\label{bellman}
\begin{cases}
\p_t v(t,m) + \int_{\R^{d_x}} h(x,m,\p_m v (t,m)(x))dm(x)= 0,\\
v(T,m)=\int_{\R^{d_x}} k(x,m)dm(x),
\end{cases}
\end{align} 
where the Hamiltonian $h:\R^{d_x}\times \mathcal{P}_2(\R^{d_x})\times \R^{d_x} \to\R$ is given by
\begin{equation}\label{hamilton1}
h(x,\mu,z):= \min_{\alpha\in \R^{d_\alpha}} \left(f(x,\mu,\alpha)\cdot z  + g(x,\mu,\alpha)\right).
\end{equation}
If, for each $x\in\R^{d_x}$, $\mu\in \mathcal{P}_2(\R^{d_x})$ and $z\in\R^{d_x}$, the function $f(x,\mu,\alpha)\cdot z  + g(x,\mu,\alpha)$ has a unique minimizer $\wh{\alpha}(x,\mu,z)\in\R^{d_\alpha}$ and is differentiable in $\alpha\in\R^{d_\alpha}$ (also refer to Proposition \ref{A5}), then the minimizer $\wh{\alpha}(x,\mu,z)$ satisfies the usual first order condition:
\begin{align}\label{foc}
\partial_\alpha f(x,\mu,\widehat{\alpha})\cdot z  + \partial_\alpha g(x,\mu,\widehat{\alpha})=0.
\end{align}
By substituting the optimal control $\wh{\alpha}$ in \eqref{bellman}, we obtain
\begin{align}\label{nomin_1}
\begin{cases}
\p_t v(t,m) + \int_{\R^{d_x}} \big(f\left(x,m,\widehat{\alpha}\left(x,m, \p_m v (t,m)(x)\right)\right)\cdot  \p_m v  (t,m) (x) \\
\ \ \ \ \ \ \ \ \ \ \ \ \ \ \ \ \ \ \ \ \ + g\left(x,m,\widehat{\alpha}\left(x,m, \p_m v (t,m)(x)\right)\right)\big)dm(x)= 0,\\
v(T,m)=\int_{\R^{d_x}} k(x,m)dm(x).
\end{cases}
\end{align}\normalsize
\begin{definition}\label{def_bellman_sol}
(Classical Solution to the Bellman equation). A function $v:(t,\mu)\in [0,T]\times\mathcal{P}_2(\R^{d_x})\mapsto v(t,\mu)\in \R$ is said to be a classical solution of \eqref{bellman} if it satisfies \eqref{bellman} and, is differentiable in $t$ and $L$-differentiable in $\mu$ so that the corresponding derivatives $\p_t v(t,\mu)$ and $\p_\mu v(t,\mu)(\cdot)$ exist.
\end{definition}
\begin{theorem}\label{Verification}
(Verification theorem). Let $v(t,\mu)$ be a classical solution in the sense of Definition \ref{def_bellman_sol} of \eqref{bellman}, and suppose that 
there is a unique $\widehat{\alpha}:= \widehat{\alpha}\left(x,\mu,z\right)\in\R^{d_\alpha}$ such that 
\begin{align}\no
&\displaystyle\min_{\alpha\in \R^{d_\alpha}} \left(f(x,\mu,\alpha)\cdot \p_\mu v(t,\mu)(x)  + g(x,\mu,\alpha)\right)\\\label{minH}
=&\  f(x,\mu,\wh{\alpha}(x,\mu,\p_\mu v(t,\mu)(x)))\cdot \p_\mu v(t,\mu)(x)  + g(x,\mu,\wh{\alpha}(x,\mu,\p_\mu v(t,\mu)(x))),
\end{align}
then $v(t,m)$ is the value function defined by \eqref{valfun}.
\end{theorem}
See Appendix Section \ref{proofVerification} for a proof.

In conclusion, under the validity of Proposition \ref{A5}, i.e. the unique existence of  minimizer $\wh{\alpha}$ of the Hamiltonian, based on this new version of verification theorem, it is equivalent to solving the control problem \eqref{SDE1}-\eqref{valfun} by tackling the associated  Bellman equation \eqref{nomin_1}.

We now relate the following mean field version of Pontryagin maximum principle \eqref{MPIT} to the Bellman equation \eqref{bellman}. 
To avoid too much cumbersome of notation, we use $\alpha$ in place of $\wh{\alpha}$ to denote the optimal feedback control which satisfies the first order condition \eqref{first_order_condition} in the rest of this article. Firstly, we define the forward-backward ``deterministic'' process $(X_s^{t,m}(x),Z_s^{t,m}(x))_{x\in\R^{d_x},s\in[t,T]}$ as follows:
for a given $m \in \mathcal{P}_2(\R^{d_x})$ as the initial distribution, and a time $t \in [0,T]$ as the starting time, for all $s\in[t,T]$ and $x\in\R^{d_x}$,\small
\begin{align}
\label{MPIT}
\begin{cases}       
\dfrac{d}{ds}X^{t,m}_s(x) = f\left(X^{t,m}_s(x),X^{t,m}_s\ot m,\alpha(X^{t,m}_s(x),X^{t,m}_s\ot m,Z^{t,m}_s(x))\right),\\
\ \ \ \,X^{t,m}_t(x)= x;\\
\dfrac{d}{ds}Z^{t,m}_s(x) = -\displaystyle \int_{\R^{d_x}}\Bigg(\p_\mu f\left(X^{t,m}_s(\wt{x}),X^{t,m}_s\ot m,\alpha\left(X^{t,m}_s(\wt{x}),X^{t,m}_s\ot m, Z^{t,m}_s(\wt{x})\right)\right)(X^{t,m}_s(x))\cdot  Z^{t,m}_s(\wt{x})\\
\ \ \ \ \ \ \ \ \ \ \ \ \ \ \ \ \ \ \ \ \ \ + \p_\mu g\left(X^{t,m}_s(\wt{x}),X^{t,m}_s\ot m,\alpha\left(X^{t,m}_s(\wt{x}),X^{t,m}_s\ot m, Z^{t,m}_s(\wt{x}))\right)\right)(X^{t,m}_s(x))\Bigg)dm(\wt{x})\\
\ \ \ \ \ \ \ \ \ \ \ \ \ \ \ \ \ \ -\p_x f(X^{t,m}_s(x),X^{t,m}_s\ot m,\alpha(X^{t,m}_s(x),X^{t,m}_s\ot m,Z^{t,m}_s(x)))\cdot Z^{t,m}_s(x)\\
\ \ \ \ \ \ \ \ \ \ \ \ \ \ \ \ \ \ -\p_x g(X^{t,m}_s(x),X^{t,m}_s\ot m,\alpha(X^{t,m}_s(x),X^{t,m}_s\ot m,Z^{t,m}_s(x))),\\
\ \ \ \,Z^{t,m}_T(x)=\displaystyle\int_{\R^{d_x}} \p_\mu k(X^{t,m}_T(\wt{x}),X^{t,m}_T\ot m)(X^{t,m}_T(x))dm(\wt{x})+\p_x k(X^{t,m}_T(x),X^{t,m}_T\ot m),
\end{cases}
\end{align}\normalsize
and recall that as an optimal control
$\alpha(x,\mu,z)$ solves the first order condition \eqref{first_order_condition}, i.e.
\begin{align}\label{usolve}
\partial_\alpha f(x,\mu,\alpha)\cdot z  + \partial_\alpha g(x,\mu,\alpha)=0,\text{ for any }(x,\mu,z)\in c_{k_0}.
\end{align}
\begin{remark}\label{collect_x}
For any fixed $t\in[0,T]$ and $m\in\mc{P}_2(\R^{d_x})$, when we say that $(X_s^{t,m}(x),Z_s^{t,m}(x))_{x\in\R^{d_x},s\in[t,T]}$ solves \eqref{MPIT}, it means that $(X_s^{t,m}(x),Z_s^{t,m}(x))_{x\in\R^{d_x},s\in[t,T]}$ solves \eqref{MPIT} collectively for all $x\in\R^{d_x}$. 
\end{remark}
As expected, we have the following theorem about the connection between Maximum Principle and the Bellman Equation.
\begin{theorem}\label{MPBE}
(Maximum Principle and Bellman Equation). Assume that $f(x,\mu,\alpha)$, $g(x,\mu,\alpha)$ and $k(x,\mu)$ are second-order differentiable in $x\in\R^{d_x}$ and $\alpha\in\R^{d_\alpha}$ and second-order regularly differentiable in $\mu\in\mathcal{P}_2(\R^{d_x})$, and $\alpha(x,\mu,z)$ is first-order differentiable in $(x,z)$ and $L$-differentiable in $\mu$. Assume also that, given any $t \in [0,T]$, $m \in \mathcal{P}_2(\R^{d_x})$, the couple $(X_s^{t,m}(x),Z_s^{t,m}(x))_{x\in\R^{d_x},s\in[t,T]}$ solves \eqref{MPIT}, and it is differentiable in $(t,x)\in[0,T]\times \R^{d_x}$ and is $L$-differentiable in $m\in\mathcal{P}_2(\R^{d_x})$; also denote the corresponding derivatives by $\Big(\p_t X^{t,m}_s(x), \p_t Z^{t,m}_s(x)\Big)$, $\Big(\p_m \big(X^{t,m}_s(x)\big)(\wt{x}),\p_m \big(Z^{t,m}_s(x)\big)(\wt{x})\Big)$ and $\Big(\p_x\big(X^{t,m}_s(x)\big),\p_x\big(Z^{t,m}_s(x)\big)\Big)$, respectively. Moreover, assume that all derivatives are differentiable in $s$. Define
\begin{align}\nonumber
v(t,m):
=&\   \int_{\R^{d_x}} k(X^{t,m}_T(\wt{x}),X^{t,m}_T\ot m)dm(\wt{x})\\\label{definev}
&+ \int_t^T \int_{\R^{d_x}} g\bigg(X^{t,m}_s(\wt{x}),X^{t,m}_s\ot m,\alpha\big(X^{t,m}_s(\wt{x}), X^{t,m}_s\ot m,Z^{t,m}_s(\wt{x})\big)\bigg)d m(\wt{x})ds.
\end{align}
Then $v(t,m)$ is a classical solution to the Bellman equation \eqref{nomin_1}. 
\end{theorem}
\begin{remark}
The regularity of $(X^{t,m}_s(x), Z^{t,m}_s(x))$, namely their differentiability in $m$, $t$, $x$ and $s$, mentioned above can be proven in Theorem \ref{Thm6_2} and Theorem \ref{GlobalSol}.
\end{remark}
See Appendix Section \ref{App:MPBE} for a proof of Theorem \ref{MPBE}.

In accordance with Theorem \ref{MPBE}, to solve for the optimal control problem \eqref{SDE1}-\eqref{valfun}, as in the classical setting, it suffices to consider the existence, uniqueness, and classical regularity of the forward-backward system \eqref{MPIT} after Pontryagin maximum principle.

\section{Local Existence of FBODE \eqref{fbodesystem}}\label{sec:local}
\subsection{Main Results}
For any fixed $0\leq t\leq \wt{T}\leq T$ and $m\in \mathcal{P}_2(\R^{d_x})$, consider the following forward-backward ordinary differential equation (FBODE) system, now with possibly a generic terminal cost functional $p$:  for all $x\in\R^{d_x}$ and $s\in[t,\wt{T}]$,
\footnotesize\begin{align}\label{fbodesystem}
\begin{cases}       
\dfrac{d}{ds}X^{t,m}_s(x) = f\left(X^{t,m}_s(x),X^{t,m}_s\ot m,\alpha(X^{t,m}_s(x),X^{t,m}_s\ot m,Z^{t,m}_s(x))\right),\ 
X^{t,m}_t(x)= x,\\
\,\dfrac{d}{ds}Z^{t,m}_s(x) =-\p_x f(X^{t,m}_s(x),X^{t,m}_s\ot m,\alpha(X^{t,m}_s(x),X^{t,m}_s\ot m,Z^{t,m}_s(x)))\cdot Z^{t,m}_s(x)\\
\ \ \ \ \ \ \ \ \ \ \ \ \ \ \ \ \ \ -\p_x g(X^{t,m}_s(x),X^{t,m}_s\ot m,\alpha(X^{t,m}_s(x),X^{t,m}_s\ot m,Z^{t,m}_s(x)))\\
\ \ \ \ \ \ \ \ \ \ \ \ \ \ \ \ \ \ -\displaystyle \int_{\R^{d_x}}\Bigg(\p_\mu f\left(X^{t,m}_s(\wt{x}),X^{t,m}_s\ot m,\alpha\left(X^{t,m}_s(\wt{x}),X^{t,m}_s\ot m, Z^{t,m}_s(\wt{x})\right)\right)(X^{t,m}_s(x))\cdot  Z^{t,m}_s(\wt{x})\\
\ \ \ \ \ \ \ \ \ \ \ \ \ \ \ \ \ \ + \p_\mu g\left(X^{t,m}_s(\wt{x}),X^{t,m}_s\ot m,\alpha\left(X^{t,m}_s(\wt{x}),X^{t,m}_s\ot m, Z^{t,m}_s(\wt{x}))\right)\right)(X^{t,m}_s(x))\Bigg)dm(\wt{x}),\\ 
\ \ \ \ Z^{t,m}_{\wt{T}}(x)=p(X^{t,m}_{\wt{T}}(x),X^{t,m}_{\wt{T}}\ot m),
\end{cases}
\end{align}\normalsize
where the optimal control $\alpha(x,\mu,z)$ solves the first order condition \eqref{first_order_condition},
and the terminal function $p$, which can be arbitrary now, satisfies the following assumptions:\\ 
${\bf(P)}$ The terminal function $p:(x,\mu)\in \R^{d_x}\times \mathcal{P}_2(\R^{d_x})\mapsto p(x,\mu)\in \R^{d_x}$ is first-order differentiable in $x\in\R^{d_x}$ and $L$-differentiable in $\mu\in\mathcal{P}_2(\R^{d_x})$ with their first-order derivatives being Lipschitz continuous in their corresponding arguments and satisfying, for some finite positive constants $L_p$ and $\wb{L}_p$,
\begin{align}\label{p1}
&\sup_{x\in\R^{d_x},\,\mu\in\mathcal{P}_2(\R^{d_x}),\,\wt{x}\in\R^{d_x}} \left\|\p_x p(x,\mu)\right\|_{\mathcal{L}(\R^{d_x};\R^{d_x})}\vee\left\|\p_\mu p(x,\mu)(\wt{x})\right\|_{\mathcal{L}(\R^{d_x};\R^{d_x})}
\leq L_p,\\\label{p2}
&|p(x,\mu)|\leq \wb{L}_p\cdot(1+|x|+\|\mu\|_1).
\end{align}
The Assumption ${\bf(P)}$ echoes Assumption ${\bf(a3)}$ about the nature of the terminal data, and this structure will be preserved, upon enlarging the bound $L_p$ gradually, while solving the FBODE over consecutive sub-intervals backwardly. In addition, \eqref{p1} implies that, by using \eqref{eq_9_3} in Appendix \ref{App:prop}, for any $x,\,x'\in\R^{d_x}$ and $\mu,\,\mu'\in\mathcal{P}_2(\R^{d_x})$,
\begin{align}\label{lipp}
\left|p(x,\mu)-p(x',\mu')\right|\leq L_p \left(|x-x'|+W_1(\mu,\mu')\right).
\end{align}
Then, clearly, we see that \eqref{p2} is valid 
for any finite constant $\wb{L}_p\geq \max\{|p(0,\delta_0)|,L_p\}$. For the sake of convenience, we also define 
\begin{align}\label{L_B}
L_B:=L_f(1+L_\alpha+2L_pL_\alpha)\text{ and }\wb{L}_B:=L_f(1+L_\alpha+2\wb{L}_pL_\alpha).
\end{align}
For local (in time) solutions of \eqref{MPIT}, $p(x,\mu)$ is taken to be $\int_{\R^{d_x}} \p_\mu k( \wt{x} ,\mu)(x)d\mu(\wt{x})+\p_x k(x,\mu)$ and thus, by \eqref{bdd_d2_k_1}, one can set $L_p=\wb{L}_k$ defined in \eqref{wbL_k}. In the proof for global (in time) existence theorem~\ref{GlobalSol}, we shall glue consecutively together local (in time) solutions of \eqref{MPIT} so as to obtain a global (in time) solution of \eqref{MPIT} over $[0,T]$ of arbitrary length $T$; for a variety of terminal data $p$ in different small time intervals $[t_i,t_{i+1}]$, $i=1,2,...,N$, by using Theorem \ref{Crucial_Estimate} to be proven later in Section \ref{sec:global}, one can set $\wb{L}_p=\max\left\{\wb{L}_k,L^*_0\right\}=L^*_0$ uniformly for various terminal data $p$, provided that $L^*_0$ will be defined later in \eqref{L_star_0}. In these cases, by \eqref{cone_ZX}, one has, for all $(s,x,m)\in[t,{\wt{T}}] \times \R^{d_x}\times \mathcal{P}_2(\R^{d_x})$, $\left(X^{t,m}_s(x),X^{t,m}_s\ot m,Z^{t,m}_s(x)\right)\in  c_{k_0} $, where the cone $c_{k_0}$ was defined in \eqref{c_k_0} and the constant $k_0$ was defined in \eqref{k_0}, namely $k_0=4\wb{L}_p$. Thus, when \eqref{p_aa_f} is valid, by Proposition \ref{A5}, the optimal control $\alpha\left(X^{t,m}_s(x),X^{t,m}_s\ot m,Z^{t,m}_s(x)\right)$ is well-defined and satisfies the first order condition \eqref{first_order_condition}.

Define
\begin{align}\label{Gamma_1}
\text{(a norm) }\vertiii{\gamma }_{1,[t,T]}:=&\  \sup_{s\in[t,T],\,x\in \R^{d_x},\,\mu\in\mathcal{P}_2(\R^{d_x})}\frac{|\gamma(s,x,\mu)|}{1+|x|+\|\mu\|_1},\\\label{Gamma_2}
\text{(a semi-norm) }\vertiii{\gamma }_{2,[t,T]}:=&\  \sup_{s\in[t,T],x\in\R^{d_x},\,\mu\in\mathcal{P}_2(\R^{d_x}),\,\wt{x}\in\R^{d_x}} \left\|\p_x \gamma(s,x,\mu)\right\|_{\mathcal{L}(\R^{d_x};\R^{d_x})}\vee\left\|\p_\mu \gamma(s,x,\mu)(\wt{x})\right\|_{\mathcal{L}(\R^{d_x};\R^{d_x})}.
\end{align}
When there is no cause of ambiguity, we denote $\vertiii{\gamma }_{i,[t,T]}$ by $\vertiii{\gamma }_{i}$, for $i=1,\,2$ for simplicity of notations. Before stating our main results in this section, let us first recall the following concept.

\begin{definition}\label{decouple}(Decoupling field)
We call a function $\gamma:[t,\wt{T}]\times \R^{d_x}\times \mc{P}_2(\R^{d_x})\to \R^{d_x}$ a decoupling field for the FBODE system \eqref{fbodesystem} on $[t,\wt{T}]$ if $\gamma$ solves the following system for all $m\in\mc{P}_2(\R^{d_x})$, $x\in\R^{d_x}$ and $s\in[t,\wt{T}]$:
\small\begin{align}\no
\gamma (s,x,m)
= &\ p(X^{s,m}_{\wt{T}}(x),X^{s,m}_{\wt{T}}\ot m)
\\\no
&+ \int_s^{\wt{T}} \Bigg(\int_{\R^{d_x}}\bigg(\p_\mu f\left(X^{s,m }_\tau(\wt{x}),X^{s,m }_\tau\ot m,\alpha\left(X^{s,m }_\tau(\wt{x}),X^{s,m }_\tau\ot m, Z^{s,m }_\tau(\wt{x})\right)\right)(X^{s,m }_\tau(x))\cdot  Z^{s,m }_\tau(\wt{x})\\\no
&\ \ \ \ \ \ \ \ \ \ \ \ \ \ \ \ \ \ \ \ \ \ + \p_\mu g\left(X^{s,m }_\tau(\wt{x}),X^{s,m }_\tau\ot m,\alpha\left(X^{s,m }_\tau(\wt{x}),X^{s,m }_\tau\ot m, Z^{s,m }_\tau(\wt{x}))\right)\right)(X^{s,m }_\tau(x))\bigg)dm(\wt{x})\\\no
&\ \ \ \ \ \ \ \ \ \ \ \ \ \ +\p_x f(X^{s,m }_\tau(x),X^{s,m }_\tau\ot m,\alpha(X^{s,m }_\tau(x),X^{s,m }_\tau\ot m,Z^{s,m }_\tau(x)))\cdot Z^{s,m }_\tau(x)\\\label{gammaeq} 
&\ \ \ \ \ \ \ \ \ \ \ \ \ \ +\p_x g(X^{s,m }_\tau(x),X^{s,m }_\tau\ot m,\alpha(X^{s,m }_\tau(x),X^{s,m }_\tau\ot m,Z^{s,m }_\tau(x)))\Bigg)d\tau,
\end{align}\normalsize
where $Z^{s,m }_\tau(x):=\gamma\big(\tau,X^{s,m }_\tau(x),X^{s,m }_\tau\ot m\big)$,
\small\begin{align}
\label{xeq}       
X^{s,m}_\tau(x) = x+\int_s^\tau f\Big(X^{s,m }_{\wt{\tau}}(x),X^{s,m }_{\wt{\tau}}\ot m,\alpha\big(X^{s,m }_{\wt{\tau}}(x),X^{s,m }_{\wt{\tau}}\ot m,\gamma(\tau,X^{s,m }_{\wt{\tau}}(x),X^{s,m }_{\wt{\tau}}\ot m)\big)\Big)d\wt{\tau},
\end{align}\normalsize
and $\alpha(x,\mu,z)$ is the optimal control solved via the first order condition \eqref{first_order_condition}.
\end{definition}
\begin{remark}
In other words, the decoupling field $\gamma:=\gamma(s,x,m)$ has the following property: using the given $\gamma$, one can first solve for $X^{s,m}_\tau(x)$ in \eqref{xeq}, and then define $Z^{s,m }_\tau(x):=\gamma\big(\tau,X^{s,m }_\tau(x),X^{s,m }_\tau\ot m\big)$. Then, by using \eqref{coro_6_4}, it can be checked directly that the couple $\big(X^{t,m}_s(x),Z^{t,m}_s(x)\big)_{x\in\R^{d_x},s\in[t,\wt{T}]}$is a solution to the FBODE system \eqref{fbodesystem}.
\end{remark}


Our main results in this section are as follows.
\begin{theorem}\label{Thm6_1} (Local-in-time existence)
Assume that the drift function $f$ and the running cost $g$ satisfy Assumptions ${\bf(a1)}$ and ${\bf(a2)}$ respectively, the terminal function $p$ satisfies the Assumption ${\bf(P)}$ with $L_p\leq \wb{L}_p:=\max\left\{\wb{L}_k,L^*_0\right\}$ and the relationship \eqref{p_aa_f} among $f$, $g$ and $p$ is valid, then there exists a constant $\eps_1=\eps_1(\wb{L}_p; L_f,\Lambda_f,\wb{l}_f,L_g,\Lambda_g,\wb{l}_g,L_\alpha)>0$ (also see \eqref{eps_1}), such that, for any $0\leq t\leq \wt{T}\leq T$ with $\wt{T}-t\leq \eps_1$, there exists a unique continuous decoupling field $\gamma:(s,x,m)\in [t,\wt{T}]\times \R^{d_x}\times \mathcal{P}_2(\R^{d_x})\mapsto \gamma(s,x,m)\in \R^{d_x}$ for the FBODE system \eqref{fbodesystem} on $[t,\wt{T}]$ which satisfies $\vertiii{\gamma}_{1,[t,\wt{T}]}\leq 2\wb{L}_p$. Morover, for any $m\in \mathcal{P}_2(\R^{d_x})$, define $\big(X^{t,m}_s(x),Z^{t,m}_s(x)\big)$ by \eqref{xeq} and $Z^{t,m}_s(x)=\gamma(s,X^{t,m}_s(x),X^{t,m}_s\ot m)$, then the pair $(X_s^{t,m}(x),Z_s^{t,m}(x))_{x\in\R^{d_x},s\in[t,\wt{T}]}$ is a solution of FBODE system \eqref{fbodesystem}, each of them is continuous in $(t,m,x)\in[t,\wt{T}]\times\mc{P}_2(\R^{d_x})\times\R^{d_x}$ and is continuously differentiable in $s\in[t,\wt{T}]$; and also satisfies the following estimate:
\begin{align}\label{cone_ZX}
\left|Z^{t,m}_s(x)\right|\leq 2\wb{L}_p\cdot\left(1+\left|X^{t,m}_s(x)\right|+\left\|X^{t,m}_s\right\|_{L^{1,d_x}_m}\right),
\end{align}
which means that $(X_s^{t,m}(x),Z_s^{t,m}(x))$ satisfies the cone condition $(X_s^{t,m}(x),X_s^{t,m}\ot m, Z_s^{t,m}(x))\in c_{k_0}$ with the cone set $c_{k_0}$ defined in \eqref{c_k_0}. In addition, $X^{t,m}_s(x)$ has the flow property $X^{t,m }_\tau(x)=X^{s,X^{t,m }_s\ot m }_\tau\big(X^{t,m }_s(x)\big)$ for any $0\leq t\leq s\leq \tau\leq T$.
\end{theorem}
The proof of Theorem \ref{Thm6_1} will be given in Section \ref{subsec:7_2}.

\begin{theorem}\label{Thm6_2} (Regularity of the local solution)
Assume that the drift function $f$ and the running cost $g$ satisfy Assumptions ${\bf(a1)}$ and ${\bf(a2)}$ respectively, the terminal function $p$ satisfies the Assumption ${\bf(P)}$ with $L_p\leq \wb{L}_p:=\max\left\{\wb{L}_k,L^*_0\right\}$ and the relationship \eqref{p_aa_f} among $f$, $g$ and $p$ is valid, then there exists a constant $\eps_2=\eps_2(\wb{L}_p; L_f,\Lambda_f,\wb{l}_f,L_g,\Lambda_g,\wb{l}_g,L_\alpha)>0$, such that, for any $0\leq t\leq \wt{T}\leq T$ with $\wt{T}-t\leq \eps_2$, the continuous decoupling field $\gamma(s,x,\mu)$ for the FBODE system \eqref{fbodesystem} on $[t,\wt{T}]$ constructed in Theorem \ref{Thm6_1} has the following regularity:\\
(i) $\gamma(s,x,\mu)$ is differentiable in $s\in[\wt{T}-\eps_2,\wt{T}]$, $x\in\R^{d_x}$ and $L$-differentiable in $\mu\in\mc{P}_2(\R^{d_x})$, and its derivatives $\p_x\gamma(s,x,\mu)$ and $\p_\mu\gamma(s,x,\mu)(\wt{x})$ are jointly continuous in their corresponding arguments $(s,x,\mu)\in [\wt{T}-\eps_2,\wt{T}]\times \R^{d_x} \times \mathcal{P}_2(\R^{d_x})$ and $(s,x,\mu, \wt{x})\in [\wt{T}-\eps_2,\wt{T}]\times \R^{d_x} \times \mathcal{P}_2(\R^{d_x})\times \R^{d_x}$ respectively; they also satisfy the following estimates:
\begin{align}\label{p_x_gamma_bdd}
\sup_{s\in[\wt{T}-\eps_2,\wt{T}],\,x\in\R^{d_x},\,\mu\in\mathcal{P}_2(\R^{d_x})} \left\|\p_x\gamma(s,x,\mu)\right\|_{\mathcal{L}(\R^{d_x};\R^{d_x})}\leq 2L_p,\\
\label{p_m_gamma_bdd}
\sup_{s\in[\wt{T}-\eps_2,\wt{T}],\,x\in\R^{d_x},\,\mu\in\mathcal{P}_2(\R^{d_x}),\,\wt{x}\in\R^{d_x}} \left\|\p_\mu \gamma(s,x,\mu)(\wt{x})\right\|_{\mathcal{L}(\R^{d_x};\R^{d_x})}\leq 2L_p.
\end{align}
Moreover, the solution pair $\big(X^{t,m}_s(x),Z^{t,m}_s(x)=\gamma(s,X^{t,m}_s(x),X^{t,m}_s\ot m)\big)_{x\in\R^{d_x},s\in[\wt{T}-\eps_2,\wt{T}]}$ of FBODE system \eqref{fbodesystem} constructed in Theorem \ref{Thm6_1} through $\gamma$ has the following regularity:\\
(ii) $X^{t,m}_s(x)$ is differentiable in $x\in\R^{d_x}$ and $L$-differentiable in $m\in\mc{P}_2(\R^{d_x})$, and its derivatives $\p_x\big(X^{t,m}_s(x)\big)$ and $\p_m\big(X^{t,m}_s(x)\big)(y)$ are jointly continuous in their corresponding arguments $(t,m,x)\in [\wt{T}-\eps_2,\wt{T}]\times \mathcal{P}_2(\R^{d_x})\times \R^{d_x}$ and $(t,m,x,y)\in [\wt{T}-\eps_2,\wt{T}]\times \mathcal{P}_2(\R^{d_x})\times \R^{d_x}\times \R^{d_x}$ respectively; they are also continuously differentiable in $s\in[t,\wt{T}]$ and satisfy the following estimates:
\begin{align}\label{p_x_X_bdd}
&\left\|\p_x\big(X^{t,m}_s(x)\big)\right\|_{\mc{L}(\R^{d_x};\R^{d_x})}\leq \exp\big(L_B(s-t)\big),\\\label{p_m_X_bdd}
&\left\|\p_m\big(X^{t,m}_s(x)\big)(y)\right\|_{\mc{L}(\R^{d_x};\R^{d_x})}\leq L_B(s-t)\Big(L_B(s-t)\exp\Big(2L_B(s-t)\Big)+1\Big)\exp\Big(2L_B(s-t)\Big);
\end{align}
(iii) $X^{t,m}_s(x)$ is differentiable in $t\in[\wt{T}-\eps_2,\wt{T}]$; its derivative $\p_t\big(X^{t,m}_s(x)\big)$ is jointly continuous in its arguments $(t,m,x)\in [\wt{T}-\eps_2,\wt{T}]\times \mathcal{P}_2(\R^{d_x})\times \R^{d_x}$; it is also continuously differentiable in $s\in[t,\wt{T}]$ and satisfies the following estimates:
\begin{align}\label{p_t_X_bdd}
\left|\p_t\big(X^{t,m}_s(x)\big)\right|\leq \wb{L}_B\big(1+|x|+\|m\|_1\big)\Big(L_B(s-t)\exp\Big(2L_B(s-t)\Big)+1\Big)\exp\Big(L_B(s-t)\Big);
\end{align}
here $L_B$ and $\wb{L}_B$ are constants defined in \eqref{L_B}.\\
In particular, $Z^{t,m}_s(x)=\gamma(s,X^{t,m}_s(x),X^{t,m}_s\ot m)$ is differentiable in $t\in[\wt{T}-\eps_2,\wt{T}]$ and $x\in\R^{d_x}$, and $L$-differentiable in $m\in\mc{P}_2(\R^{d_x})$; and its derivatives $\p_t\big(Z^{t,m}_s(x)\big)$, $\p_x\big(Z^{t,m}_s(x)\big)$ and $\p_m\big(Z^{t,m}_s(x)\big)(y)$ are jointly continuous in their corresponding arguments $(t,m,x)\in [\wt{T}-\eps_2,\wt{T}]\times \mathcal{P}_2(\R^{d_x})\times \R^{d_x}$ and $(t,m,x,y)\in [\wt{T}-\eps_2,\wt{T}]\times \mathcal{P}_2(\R^{d_x})\times \R^{d_x}\times \R^{d_x}$ respectively; they also continuously differentiable in $s\in[t,\wt{T}]$ and satisfy the following estimates:
\footnotesize\begin{align}\label{p_x_Z_bdd_n}
&\left\|\p_x\big(Z^{t,m}_s(x)\big)\right\|_{\mc{L}(\R^{d_x};\R^{d_x})}\leq 2L_p\exp\big(L_B(s-t)\big);\\\label{p_m_Z_bdd_n}
&\left\|\p_m\big(Z^{t,m}_s(x)\big)(y)\right\|_{\mc{L}(\R^{d_x};\R^{d_x})}\leq 4L_pL_B(s-t)\Big(L_B(s-t)\exp\Big(2L_B(s-t)\Big)+1\Big)\exp\Big(2L_B(s-t)\Big)+2L_p\exp\big(L_B(s-t)\big);\\\label{p_t_Z_bdd_n}
&\left|\p_t\big(Z^{t,m}_s(x)\big)\right|\leq 4L_p\wb{L}_B\big(1+|x|+\|m\|_1\big)\Big(L_B(s-t)\exp\Big(2L_B(s-t)\Big)+1\Big)\exp\Big(L_B(s-t)\Big).
\end{align}\normalsize
\end{theorem}
The proof of Theorem \ref{Thm6_2} will be given in Section \ref{subsec:7_3}.
\begin{remark}
The assumption on the validity of \eqref{p_aa_f} is only used to guarantee the unique solvability of the optimal control $\alpha(x,\mu,z)$, which is only sufficient but not necessary, yet it eases the involved calculus; see Proposition \ref{A5}. If the optimal control can be solved explicitly and is regular enough, the assumption on the validity of \eqref{p_aa_f} in Theorem \ref{Thm6_1} and \ref{Thm6_2} can be removed.
\end{remark}

In the rest of this section, we use $T$ instead $\wt{T}$ without cause of ambiguity since $\wt{T}$ is fixed in the whole section.

\subsection{Proof of Local Existence Theorem \ref{Thm6_1}}\label{subsec:7_2}
Let $C([t,T]\times \R^{d_x}\times\mathcal{P}_2(\R^{d_x});\R^{d_x})$ denotes the vector space of continuous functions each of which maps from $[t,T]\times \R^{d_x}\times\mathcal{P}_2(\R^{d_x})$ to $\R^{d_x}$.
By equipping this vector space $C([t,T]\times \R^{d_x}\times\mathcal{P}_2(\R^{d_x});\R^{d_x})$ with the norm $\vertiii{\cdot}_1$ defined in \eqref{Gamma_1}, then it becomes a Banach space. We next consider
a convex subset of this Banach space for the domain and range for an iteration map:
\small\begin{align}\no
\mathcal{I}_1:=\bigg\{\gamma(s,x,\mu)\in C([t,T]\times \R^{d_x}\times\mathcal{P}_2(\R^{d_x});\R^{d_x}):\ \gamma\text{ is differentiable in }x\text{ and $L$-differentiable in }\mu&\\\no
\text{ with its derivatives being continuous in their}&\\
\text{arguments and }\vertiii{\gamma}_1\leq 2\wb{L}_p, \vertiii{\gamma}_2\leq 2L_p&\bigg\}.
\end{align}\normalsize

\textbf{Step 1}. Set the initial seed $\gamma^{(0)}(s,x,\mu):= p(x,\mu)$ which clearly belongs to $\mathcal{I}_1$ by using \eqref{p1} and \eqref{p2}. Assume that $\gamma^{(n-1)}(s,x,\mu)\in \mathcal{I}_1$ and define $\gamma^{(n)}(s,x,\mu)$, for $n=1,2,...$, iteratively as follows:\\
(i) solve the following forward ODE equation with a mean-field term and an exogenous driving term $\gamma^{(n-1)}$ in the integral representation for $X^{t,m,(n)}_s(x)$:
\footnotesize\begin{align}
\label{fodesystem}       
X^{t,m,(n)}_s(x) = x+\int_t^s f\Big(X^{t,m,(n)}_\tau(x),X^{t,m,(n)}_\tau\ot m,\alpha\big(X^{t,m,(n)}_\tau(x),X^{t,m,(n)}_\tau\ot m,\gamma^{(n-1)}(\tau,X^{t,m,(n)}_\tau(x),X^{t,m,(n)}_\tau\ot m)\big)\Big)d\tau;
\end{align}\normalsize
(ii) define the next step backward solution approximation $\gamma^{(n)}$ explicitly in terms of $\gamma^{(n-1)}$ and $X^{t,m,(n)}_s(x)$ as follows:
\footnotesize\begin{align}\no
&\gamma^{(n)} (s,x,m)\\\no
= &\ \,p(X^{s,m,(n)}_T(x),X^{s,m,(n)}_T\ot m)\\\no
&+ \int_s^T \Bigg(\int_{\R^{d_x}}\bigg(\p_\mu f\left(X^{s,m,(n)}_\tau(\wt{x}),X^{s,m,(n)}_\tau\ot m,\alpha\left(X^{s,m,(n)}_\tau(\wt{x}),X^{s,m,(n)}_\tau\ot m, Z^{s,m,(n)}_\tau(\wt{x})\right)\right)(X^{s,m,(n)}_\tau(x))\cdot  Z^{s,m,(n)}_\tau(\wt{x})\\\no
&\ \ \ \ \ \ \ \ \ \ \ \ \ \ \ \ \ \ \ \ \ \ + \p_\mu g\left(X^{s,m,(n)}_\tau(\wt{x}),X^{s,m,(n)}_\tau\ot m,\alpha\left(X^{s,m,(n)}_\tau(\wt{x}),X^{s,m,(n)}_\tau\ot m, Z^{s,m,(n)}_\tau(\wt{x}))\right)\right)(X^{s,m,(n)}_\tau(x))\bigg)dm(\wt{x})\\\no
&\ \ \ \ \ \ \ \ \ \ \ \ \ \ +\p_x f(X^{s,m,(n)}_\tau(x),X^{s,m,(n)}_\tau\ot m,\alpha(X^{s,m,(n)}_\tau(x),X^{s,m,(n)}_\tau\ot m,Z^{s,m,(n)}_\tau(x)))\cdot Z^{s,m,(n)}_\tau(x)\\\label{bodesystem} 
&\ \ \ \ \ \ \ \ \ \ \ \ \ \ +\p_x g(X^{s,m,(n)}_\tau(x),X^{s,m,(n)}_\tau\ot m,\alpha(X^{s,m,(n)}_\tau(x),X^{s,m,(n)}_\tau\ot m,Z^{s,m,(n)}_\tau(x)))\Bigg)d\tau,
\end{align}\normalsize
where $Z^{s,m,(n)}_\tau(x):=\gamma^{(n-1)}\big(\tau,X^{s,m,(n)}_\tau(x),X^{s,m,(n)}_\tau\ot m\big)$. Here, the $\gamma^{(n)} (s,x,m)$ on the left hand side of \eqref{bodesystem} is an approximation of $Z^{s,m}_s(x)$ indeed. It is worth noting that $\alpha(X^{s,m,(n)}_\tau(x),X^{s,m,(n)}_\tau\ot m,Z^{s,m,(n)}_\tau(x))$ is well-defined by the validity of Proposition \ref{A5} since $\vertiii{\gamma^{(n-1)}}_1\leq 2\wb{L}_p$ and thus $(X^{s,m,(n)}_\tau(x),X^{s,m,(n)}_\tau\ot m,\\Z^{s,m,(n)}_\tau(x))$ belongs to the cone $c_{k_0}$ defined in \eqref{c_k_0}, and this is a key point in designing such an iteration mapping \eqref{bodesystem}. In the following steps $2$-$6$, we shall show that $X^{t,m,(n)}_s(x)$ is well-defined, $\gamma^{(n)}(s,x,\mu)$, which was defined by \eqref{bodesystem}, belongs to $\mathcal{I}_1$ and the iteration mapping $\gamma^{(n-1)}\mapsto \gamma^{(n)}$ is a contraction under the norm $\vertiii{\cdot}_1$.

\textbf{Step 2}. In this step, we first prove the following lemma to show that $X^{t,m,(n)}_s(x)$ is well-defined and satisfies the estimates \eqref{eq_6_14_1}-\eqref{Bp_tX}.

\begin{lemma}\label{lem6_3}
Under the assumptions of Theorem \ref{Thm6_1}, then, for each $t\in[0 ,T]$, $m\in\mathcal{P}_2(\R^{d_x})$ and $\gamma\in \mathcal{I}_1$, the equation 
\small\begin{align}
\label{fodesystem_lem}       
X^{t,m}_s(x) = x+\int_t^s f\Big(X^{t,m}_\tau(x),X^{t,m}_\tau\ot m,\alpha\big(X^{t,m}_\tau(x),X^{t,m}_\tau\ot m,\gamma(\tau,X^{t,m}_\tau(x),X^{t,m}_\tau\ot m)\big)\Big)d\tau;
\end{align}\normalsize 
has a unique solution $X^{t,m}_\cdot(\cdot):(s,x)\in[t,T]\times \R^{d_x}\mapsto X^{t,m}_s(x)\in \R^{d_x}$, which belongs to $C\big([t,T]\times \R^{d_x}; \R^{d_x}\big)$. Moreover, $X^{t,m}_s(x)$ is Lipschitz continuous and is also differentiable in $t\in[0 ,T]$ and $x\in\R^{d_x}$, and $L$-differentiable in $m\in\mathcal{P}_2(\R^{d_x})$; its derivatives $\p_t\big(X^{t,m}_s(x)\big)$, $\p_x\big(X^{t,m}_s(x)\big)$ and $\p_m\big(X^{t,m}_s(x)\big)(\wt{x})$ are jointly continuous in their corresponding arguments $(t,m,x)\in[0,T]\times \mathcal{P}_2(\R^{d_x})\times \R^{d_x}$ and $(t,m,x,\wt{x})\in[0,T]\times \mathcal{P}_2(\R^{d_x})\times \R^{d_x}\times \R^{d_x}$ respectively,
and they are also continuously differentiable in $s\in[t,T]$. Furthermore, the unique solution $X^{t,m}_s(x)$ satisfies the following estimates:\\
(i) for $t\in[0,T]$, $m\in\mc{P}_2(\R^{d_x})$, $x\in\R^{d_x}$,
\begin{align}\no
\sup_{s\in[t,T]}\left\|X^{t,m}_s \right\|_{L^{1,d_x}_m}
\leq&\ \Big(\left\|m\right\|_1+\wb{L}_B(T-t)\Big)\exp\Big(2\wb{L}_B(T-t)\Big)\text{ and }\\\label{eq_6_14_1}
\sup_{s\in[t,T]}\left|X^{t,m}_s (x)\right|\leq&\ \Big(\left|x\right|+\wb{L}_B\big(1+\sup_{s\in[t,T]}\left\|X^{t,m}_s \right\|_{L^{1,d_x}_m}\big)(T-t)\Big)
\exp\Big(\wb{L}_B(T-t)\Big).
\end{align}\normalsize
(ii) for $0\leq t\leq t'\leq s\leq T$, $m,\,m'\in\mc{P}_2(\R^{d_x})$, $x,\,x'\in\R^{d_x}$,
\small\begin{align}\label{XLipinx}     
\left|X^{t,m}_s(x')-X^{t,m}_s(x) \right|\leq&\ \left|x'-x\right|\exp\Big(L_B(s-t)\Big);
\\\label{XLipinm}     
\left|X^{t,m'}_s(x)-X^{t,m}_s(x) \right|
\leq&\ W_1(m',m)L_B(s-t)\exp\Big(2L_B(s-t)\Big)\Big(1+L_B(s-t)\exp\Big(2L_B(s-t)\Big)\Big);
\\\label{XLipint}    
\left|X^{t',m}_s(x)-X^{t,m}_s(x) \right|\leq&\ (t'-t)\wb{L}_B\exp\Big(L_B(s-t')\Big)\bigg(1+\sup_{\tau\in[t,t']}\left| X^{t,m}_\tau(x) \right|+\sup_{\tau\in[t,t']}\left\| X^{t,m}_\tau \right\|_{L^{1,d_x}_m}\\\no
&\ \ \ \ \ \ \ \ \ \ \ \ \ \ \ \ \ \ \ \ \ \ \ \ \ \ \ \ \ \ \ \ \ \ +L_B(s-t')\Big(1+2\sup_{\tau\in[t,t']}\left\| X^{t,m}_\tau \right\|_{L^{1,d_x}_m} \Big)\exp\Big(2L_B(s-t')\Big)\bigg).
\end{align}\normalsize
(iii) for $t\in[0,T]$, $m\in\mc{P}_2(\R^{d_x})$, $x,\,y\in\R^{d_x}$, $s\in[t,T]$,
\begin{align}\label{Bp_xX}
&\left\|\p_x\big(X^{t,m}_s(x)\big)\right\|_{\mc{L}(\R^{d_x};\R^{d_x})}\leq \exp\Big(L_B(s-t)\Big);
\\\label{Bp_mX}
&\left\|\p_m\big(X^{t,m}_s(x)\big)(y)\right\|_{\mc{L}(\R^{d_x};\R^{d_x})}\leq L_B(s-t)\exp\Big(2L_B(s-t)\Big)\Big(1+L_B(s-t)\exp\Big(2L_B(s-t)\Big)\Big);
\\\label{Bp_tX}
&\left|\p_t\big(X^{t,m}_s(x)\big)\right|\leq \wb{L}_B\big(1+|x|+\|m\|_1\big)\Big(L_B(s-t)\exp\Big(2L_B(s-t)\Big)+1\Big)\exp\Big(L_B(s-t)\Big).
\end{align}
\end{lemma}
\begin{remark}
By the statement 1 in Proposition \ref{property_push_forward} and the uniqueness of solutions of \eqref{fodesystem_lem}, we have the following flow property of $X^{t,m }_s(x)$, for any $0\leq t\leq s\leq \tau\leq T$,
\begin{align}\label{coro_6_4}
X^{t,m }_\tau(x)=X^{s,X^{t,m }_s\ot m }_\tau\big(X^{t,m }_s(x)\big).
\end{align} 
\end{remark}
\begin{remark}
Note that $X^{t,m}_s(x)$ is Lipschitz continuous in $m\in\mc{P}_2(\R^{d_x})$ under $W_1$-metric by \eqref{XLipinm}, which is stronger than that under $W_2$-metric since $W_1(\mu,\nu)\leq W_2(\mu,\nu)$ for any $\mu,\,\nu\in\mc{P}_2(\R^{d_x})\subset \mc{P}_1(\R^{d_x})$ by \eqref{eq_2_2}.
\end{remark}

\begin{proof}[Proof of Lemma \ref{lem6_3}]
Step 1 (Local existence).
For any fixed $(t,m)\in[0 ,T]\times \mathcal{P}_2(\R^{d_x})$, define the working space for the iteration as
\begin{align*}
\mathcal{I}_2^\delta:=\bigg\{X^{t,m}_\cdot(\cdot)\in C\big([t,t+\delta]\times \R^{d_x}; \R^{d_x}\big):\ \sup_{s\in[t,t+\delta],\, x\in\R^{d_x}}\frac{|X^{t,m}_s(x)|}{1+|x|+\|m\|_1}\leq 2\bigg\},
\end{align*}
where $\delta$ is a small positive constant to be determined later and the expression of supremum is a proper norm.
Set the initial seed $X^{t,m,(0)}_s(x)\equiv x\in \mathcal{I}_2^\delta$. Assume that $X^{t,m,(n-1)}_\cdot(\cdot)\in \mathcal{I}_2^\delta$ and define $X^{t,m,(n)}_\cdot(\cdot)$, for $n=1,2,...$, iteratively as follows: for $s\in[t,t+\delta]$,
\footnotesize\begin{align}\no
X^{t,m,(n)}_s(x) = x+\int_t^s f\Big(X^{t,m,(n-1)}_\tau(x),X^{t,m,(n-1)}_\tau\ot m,\alpha\big(X^{t,m,(n-1)}_\tau(x),X^{t,m,(n-1)}_\tau\ot m,\gamma(\tau,X^{t,m,(n-1)}_\tau(x),X^{t,m,(n-1)}_\tau\ot m)\big)\Big)d\tau.
\end{align}\normalsize 
Clearly, $X^{t,m,(n)}_\cdot(\cdot)\in C\big([t,t+\delta]\times \R^{d_x}; \R^{d_x}\big)$ since $f(x,\mu,\alpha)$, $\alpha(x,\mu,z)$ (for $(x,\mu,z)\in c_{k_0}$), $\gamma(s,x,\mu)$ are continuous in their corresponding arguments and $X^{t,m,(n-1)}_s(x)$ is continuous in $(s,x)$ by the inductive hypothesis. Then, for small enough $\delta>0$ that will be determined below, $X^{t,m,(n)}_\cdot(\cdot)\in \mathcal{I}_2^\delta$ by checking that it truly fulfills the properties required in $\mathcal{I}_2^\delta$ in the following:\\\small
\begin{align*}
\sup_{s\in[t,t+\delta]}\left|X^{t,m,(n)}_s(x) \right|
\leq&\ \left|x\right|+\delta \cdot L_f(1+L_\alpha)\cdot \sup_{s\in[t,t+\delta]}\left(1+|X^{t,m,(n-1)}_s(x)|+\|X^{t,m,(n-1)}_s\|_{L^{1,d_x}_m}\right)\\
&+\delta \cdot L_fL_\alpha\cdot\sup_{s\in[t,t+\delta]}|\gamma(s,X^{t,m,(n-1)}_s(x),X^{t,m,(n-1)}_s\ot m)|\ \ \ (\text{by using \eqref{ligf} and \eqref{linalpha}})\\
\leq&\ \left|x\right|+\delta \cdot \wb{L}_B\cdot \sup_{s\in[t,t+\delta]}\left(1+|X^{t,m,(n-1)}_s(x)|+\|X^{t,m,(n-1)}_s\|_{L^{1,d_x}_m}\right)\ \ \ (\text{by using $\gamma\in \mathcal{I}_1$})\\
\leq&\ \left|x\right|+\delta \cdot \wb{L}_B\cdot \left(5+2|x|+6\|m\|_1\right)\ \ \ (\text{ by using $X^{t,m,(n-1)}_\cdot(\cdot)\in \mathcal{I}_2^\delta$})\\
\leq&\ 2(1+|x|+\|m\|_1),
\end{align*}\normalsize
by selecting a $\delta\leq \frac{1}{3\wb{L}_B}$.
For a small enough $\delta$ to be determined later, we next show that the iteration mapping $X^{t,m,(n-1)}_\cdot(\cdot)\in \mathcal{I}_2^\delta\mapsto X^{t,m,(n)}_\cdot(\cdot)\in \mathcal{I}_2^\delta$ is contractive: 
\footnotesize\begin{align*}
&\sup_{s\in[t,t+\delta]} \left|X^{t,m,(n+1)}_s(x)-X^{t,m,(n)}_s(x)\right|\\
\leq&\ L_B\delta\cdot\sup_{\tau\in[t,t+\delta]}  \bigg(\left|X^{t,m,(n)}_\tau(x)-X^{t,m,(n-1)}_\tau(x) \right|+\left\|X^{t,m,(n)}_\tau-X^{t,m,(n-1)}_\tau \right\|_{L^{1,d_x}_m}\bigg)\ \ (\text{by using \eqref{Lipf}, \eqref{lipalpha} and $\gamma\in \mathcal{I}_1$})\\
\leq&\ L_B\delta\cdot\sup_{\tau\in[t,t+\delta]}  \bigg(\left|X^{t,m,(n)}_\tau(x)-X^{t,m,(n-1)}_\tau(x) \right|+\sup_{x\in\R^{d_x}}\frac{|X^{t,m,(n)}_\tau(x)-X^{t,m,(n-1)}_\tau(x) |}{1+|x|+\|m\|_1}\cdot(1+2\|m\|_1)\bigg)\\
\leq &\ L_B\delta\cdot3(1+|x|+\|m\|_1)\cdot\sup_{\tau\in[t,t+\delta],\,x\in\R^{d_x}}\frac{|X^{t,m,(n)}_\tau(x)-X^{t,m,(n-1)}_\tau(x) |}{1+|x|+\|m\|_1}\\
\leq &\ \frac{1}{2}(1+|x|+\|m\|_1)\cdot\sup_{\tau\in[t,t+\delta],\,x\in\R^{d_x}}\frac{|X^{t,m,(n)}_\tau(x)-X^{t,m,(n-1)}_\tau(x) |}{1+|x|+\|m\|_1},
\end{align*}\normalsize
which holds for any $\delta\leq  \frac{1}{6L_B}$.
Therefore, by Banach Fixed Point Theorem, for any $$0<\delta\leq  \delta_0:=\min\left\{\frac{1}{3\wb{L}_B},\frac{1}{6L_B}\right\},$$
which does not depend on the value of $(t,m)$, there exists a unique $X^{t,m}_\cdot(\cdot)\in \mathcal{I}_2^\delta$ satisfying \eqref{fodesystem_lem}.\\
Step 2 ({\it a priori} estimate and global existence) Assume that for some $T_*<T$, $X^{t,m}_\cdot(\cdot)\in C\big([t,T_*]\times \R^{d_x}; \R^{d_x}\big)$ is a solution to \eqref{fodesystem_lem}. Then, for $s\in[t,T_*]$,
\begin{align}\no
\left|X^{t,m}_s(x) \right|\leq&\ \left|x\right|+L_f(1+L_\alpha)\cdot \int_t^s\left(1+|X^{t,m}_\tau(x)|+\|X^{t,m}_\tau\|_{L^{1,d_x}_m}\right)d\tau\\\no
&+L_fL_\alpha\cdot\int_t^s|\gamma(\tau,X^{t,m}_\tau(x),X^{t,m}_\tau\ot m)|d\tau\ \ (\text{by using \eqref{ligf} and \eqref{linalpha}})\\\label{eq_6_12_1}
\leq&\ \left|x\right|+\wb{L}_B\cdot \int_t^s\left(1+|X^{t,m}_\tau(x)|+\|X^{t,m}_\tau\|_{L^{1,d_x}_m}\right)d\tau\ \ (\text{by using $\gamma\in \mathcal{I}_1$}),
\end{align}\normalsize
which implies, by integrating \eqref{eq_6_12_1} with respect to $x$,
\begin{align}\label{eq_6_13_1}
&\left\|X^{t,m}_s\right\|_{L^{1,d_x}_m}\leq\left\|m\right\|_1+\wb{L}_B\cdot \int_t^s\left(1+2\|X^{t,m}_\tau\|_{L^{1,d_x}_m}\right)d\tau.
\end{align}\normalsize
By applying Gr\"{o}nwall's inequality to \eqref{eq_6_12_1} and \eqref{eq_6_13_1}, we have the following {\it a priori} estimates for different norms:
\footnotesize\begin{align}
\sup_{s\in[t,T_*]}\left\|X^{t,m}_s \right\|_{L^{1,d_x}_m}\leq&\ \Big(\left\|m\right\|_1+\wb{L}_B(T_*-t)\Big)\exp\Big(2\wb{L}_B(T_*-t)\Big)
\leq\Big(\left\|m\right\|_1+\wb{L}_B(T-t)\Big)\exp\Big(2\wb{L}_B(T-t)\Big),\\\no
\sup_{s\in[t,T_*]}\left|X^{t,m}_s (x)\right|\leq&\ \bigg(\left|x\right|+\wb{L}_B\Big(1+\Big(\left\|m\right\|_1+\wb{L}_B(T-t)\Big)\exp\Big(2\wb{L}_B(T-t)\Big)\Big)(T-t)\bigg)
\exp\Big(\wb{L}_B(T-t)\Big)\\\no
\leq &\ \max\bigg\{1,\Big(1+\max\Big\{1,\wb{L}_B(T-t)\Big\}\exp\Big(2\wb{L}_B(T-t)\Big)\Big)\wb{L}_B(T-t)\bigg\}\exp\Big(\wb{L}_B(T-t)\Big)\\\label{eq_6_14_1_n}
&\cdot (1+|x|+\|m\|_1),
\end{align}\normalsize
where the right hand side of \eqref{eq_6_14_1_n} is clearly independent of the choice $T_*$. By the local existence in Step 1 and the standard continuous induction argument, the solution $X^{t,m}_\cdot(\cdot)\in C\big([t,T_*]\times \R^{d_x}; \R^{d_x}\big)$ can then be uniquely extended beyond $T_*$. Therefore, there exists a unique solution $X^{t,m}_\cdot(\cdot)\in C\big([t,T]\times \R^{d_x}; \R^{d_x}\big)$ to \eqref{fodesystem_lem} over the whole $[t,T]$, for any $t\in[0 ,T]$.\\
Step 3 (Lipschitz continuity in $x\in\R^{d_x}$ of the forward solution $X^{t,m}_s(x)$) For $x,\,x'\in \R^{d_x}$,
\small\begin{align*}     
&\left|X^{t,m}_s(x')-X^{t,m}_s(x) \right|\\
\leq &\ \left|x'-x\right|+\int_t^s \Big|f\left(X^{t,m}_s(x'),X^{t,m}_s\ot m,\alpha(X^{t,m}_s(x'),X^{t,m}_s\ot m,\gamma(s,X^{t,m}_s(x'),X^{t,m}_s\ot m))\right)\\
&\ \ \ \ \ \ \ \ \ \ \ \ \ \ \ \ \ \ \ \ \ -f\left(X^{t,m}_s(x),X^{t,m}_s\ot m,\alpha(X^{t,m}_s(x),X^{t,m}_s\ot m,\gamma(s,X^{t,m}_s(x),X^{t,m}_s\ot m))\right)\Big|d\tau\\
\leq&\ \left|x'-x\right|+L_B\int_t^s \left|X^{t,m}_\tau(x')-X^{t,m}_\tau(x)\right|d\tau\ (\text{by using \eqref{Lipf}, \eqref{lipalpha} and $\gamma\in \mathcal{I}_1$}),
\end{align*}\normalsize
which implies \eqref{XLipinx} by Gr\"{o}nwall's inequality.
\\Step 4 (Lipschitz continuity in $m\in\mathcal{P}_2(\R^{d_x})$ of the forward solution $X^{t,m}_s(x)$) First, note that, by using \eqref{XLipinx}, for $m,\,m'\in \mathcal{P}_2(\R^{d_x})$,
\begin{align}\no
W_1\Big(X^{t,m'}_\tau\ot m',X^{t,m}_\tau\ot m\Big)\leq&\  W_1\Big(X^{t,m'}_\tau\ot m',X^{t,m'}_\tau\ot m\Big)+W_1\Big(X^{t,m'}_\tau\ot m,X^{t,m}_\tau\ot m\Big)\\\no
\leq &\  W_1(m',m)\cdot\exp\Big(L_B(\tau-t)\Big)+\Big\|X^{t,m'}_\tau-X^{t,m}_\tau \Big\|_{L^{1,d_x}_m}.
\end{align}
Thus, we have
\footnotesize\begin{align*}     
\left|X^{t,m'}_s(x)-X^{t,m}_s(x) \right|
\leq &\ \int_t^s \Big|f\left(X^{t,m'}_s(x),X^{t,m'}_s\ot m',\alpha(X^{t,m'}_s(x),X^{t,m'}_s\ot m',\gamma(s,X^{t,m'}_s(x),X^{t,m'}_s\ot m'))\right)\\
&\ \ \ \ \ \ \ \ \ -f\left(X^{t,m}_s(x),X^{t,m}_s\ot m,\alpha(X^{t,m}_s(x),X^{t,m}_s\ot m,\gamma(s,X^{t,m}_s(x),X^{t,m}_s\ot m))\right)\Big|d\tau\\
\leq&\ L_B\int_t^s \Big(\left|X^{t,m'}_\tau(x)-X^{t,m}_\tau(x)\right|+W_1\Big(X^{t,m'}_\tau\ot m',X^{t,m}_\tau\ot m\Big)\Big)d\tau\ \ \ (\text{by using \eqref{Lipf}, \eqref{lipalpha} and $\gamma\in \mathcal{I}_1$})\\
\leq&\ L_B\int_t^s \Big(\left|X^{t,m'}_\tau(x)-X^{t,m}_\tau(x)\right|+ W_1(m',m)\cdot\exp\Big(L_B(\tau-t)\Big)+\Big\|X^{t,m'}_\tau-X^{t,m}_\tau \Big\|_{L^{1,d_x}_m}\Big)d\tau,
\end{align*}\normalsize
which implies, by integrating with respect to $x$,
\begin{align*}     
\left\|X^{t,m'}_s-X^{t,m}_s \right\|_{L^{1,d_x}_m}\leq&\ L_BW_1(m',m)\cdot\exp\Big(L_B(s-t)\Big)(s-t)+2L_B\int_t^s \left\|X^{t,m'}_\tau-X^{t,m}_\tau\right\|_{L^{1,d_x}_m}d\tau.
\end{align*}
By Gr\"{o}nwall's inequality,
\begin{align}\label{XLipinm_int}  
\left\|X^{t,m'}_s-X^{t,m}_s\right\|_{L^{1,d_x}_m}\leq&\  L_BW_1(m',m)\cdot\exp\Big(3L_B(s-t)\Big)(s-t),\ s\in[t,T].
\end{align}
Therefore, \footnotesize$\left|X^{t,m'}_s(x)-X^{t,m}_s(x) \right|
\leq L_B\int_t^s \Big(\left|X^{t,m'}_\tau(x)-X^{t,m}_\tau(x)\right|+ W_1(m',m)\cdot\exp\Big(L_B(\tau-t)\Big)+ W_1(m',m)\cdot L_B(\tau-t)\exp\Big(3L_B(\tau-t)\Big)\Big)d\tau$\normalsize,
which implies, by Gr\"{o}nwall's inequality, for any $s\in[t,T]$,
\begin{align}\no   
\left|X^{t,m'}_s(x)-X^{t,m}_s(x) \right|
\leq&\ W_1(m',m)L_B(s-t)\exp\Big(2L_B(s-t)\Big)\Big(1+L_B(s-t)\exp\Big(2L_B(s-t)\Big)\Big),
\end{align}
as claimed in \eqref{XLipinm}.\\
Step 5 (Lipschitz continuity in $t\in[0,T]$ of the forward solution $X^{t,m}_s(x)$) For $0 \leq t\leq t'\leq s\leq T$,
\footnotesize\begin{align*}     
&\left|X^{t',m}_s(x)-X^{t,m}_s(x) \right|\\
\leq &\ \int_t^{t'}\left|f\left(X^{t,m}_s(x),X^{t,m}_s\ot m,\alpha(X^{t,m}_s(x),X^{t,m}_s\ot m,\gamma(s,X^{t,m}_s(x),X^{t,m}_s\ot m))\right)\right|d\tau\\
&+\int_{t'}^s \Big|f\big(X^{t',m}_s(x),X^{t',m}_s\ot m,\alpha(X^{t',m}_s(x),X^{t',m}_s\ot m,\gamma(s,X^{t',m}_s(x),X^{t',m}_s\ot m))\big)\\
&\ \ \ \ \ \ \ \ \ \ \ \ \ -f\big(X^{t,m}_s(x),X^{t,m}_s\ot m,\alpha(X^{t,m}_s(x),X^{t,m}_s\ot m,\gamma(s,X^{t,m}_s(x),X^{t,m}_s\ot m))\big)\Big|d\tau\\
\leq&\ \wb{L}_B\int_t^{t'}\Big(1+\left| X^{t,m}_\tau(x) \right|+\left\| X^{t,m}_\tau \right\|_{L^{1,d_x}_m}\Big)d\tau+L_B\int_{t'}^s \Big(\left|X^{t',m}_\tau(x)-X^{t,m}_\tau(x)\right|+\left\|X^{t',m}_\tau-X^{t,m}_\tau\right\|_{L^{1,d_x}_m}\Big)d\tau,
\end{align*}\normalsize
where the last inequality follows by using \eqref{ligf}, \eqref{Lipf}, \eqref{linalpha}, \eqref{lipalpha} and $\gamma\in \mathcal{I}_1$.
Again, by integrating with respect to $x$, we obtain
\begin{align*}     
\left\|X^{t',m}_s-X^{t,m}_s \right\|_{L^{1,d_x}_m}\leq&\ \wb{L}_B\int_t^{t'}\Big(1+2\left\| X^{t,m}_\tau \right\|_{L^{1,d_x}_m}\Big)d\tau+2L_B\int_{t'}^s \left\|X^{t',m}_\tau-X^{t,m}_\tau\right\|_{L^{1,d_x}_m}d\tau,
\end{align*}
which implies, by Gr\"{o}nwall's inequality,
\footnotesize\begin{align} \no 
\left\|X^{t',m}_s(x)-X^{t,m}_s(x) \right\|_{L^{1,d_x}_m}
\leq&\  \wb{L}_B(t'-t)\cdot\Big(1+2\sup_{\tau\in[t,t']}\left\| X^{t,m}_\tau \right\|_{L^{1,d_x}_m} \Big)\exp\Big(2L_B(s-t')\Big)\\\label{XLinint_int}
\leq &\  \wb{L}_B(t'-t)\cdot\Big(1+2\Big(\left\|m\right\|_1+\wb{L}_B(t'-t)\Big)\exp\Big(2\wb{L}_B(t'-t)\Big)\Big)\exp\Big(2L_B(s-t')\Big),
\end{align}\normalsize
where we had substituted \eqref{eq_6_14_1} into the second factor in the first line to obtain the last inequality. Therefore,
\footnotesize\begin{align*}     
\left|X^{t',m}_s(x)-X^{t,m}_s(x) \right|
\leq&\ \,\wb{L}_B(t'-t)\cdot\Big(1+\sup_{\tau\in[t,t']}\left| X^{t,m}_\tau(x) \right|+\sup_{\tau\in[t,t']}\left\| X^{t,m}_\tau \right\|_{L^{1,d_x}_m}\Big)\\
&+L_B(s-t')\cdot \sup_{\tau\in[t',s]}\left\|X^{t',m}_\tau-X^{t,m}_\tau\right\|_{L^{1,d_x}_m}+L_B\int_{t'}^s \left|X^{t',m}_\tau(x)-X^{t,m}_\tau(x)\right|d\tau\\
\leq&\ \,\wb{L}_B(t'-t)\cdot\Big(1+\sup_{\tau\in[t,t']}\left| X^{t,m}_\tau(x) \right|+\sup_{\tau\in[t,t']}\left\| X^{t,m}_\tau \right\|_{L^{1,d_x}_m}\Big)\\
&+L_B(s-t')\cdot \wb{L}_B(t'-t)\cdot\Big(1+2\sup_{\tau\in[t,t']}\left\| X^{t,m}_\tau \right\|_{L^{1,d_x}_m} \Big)\exp\Big(2L_B(s-t')\Big)\\
&+L_B\int_{t'}^s \left|X^{t',m}_\tau(x)-X^{t,m}_\tau(x)\right|d\tau,
\end{align*}\normalsize
which implies, by Gr\"{o}nwall's inequality,
\footnotesize\begin{align*}     
&\left|X^{t',m}_s(x)-X^{t,m}_s(x) \right|\\
\leq&\ \wb{L}_B(t'-t)\cdot\bigg(1+\sup_{\tau\in[t,t']}\left| X^{t,m}_\tau(x) \right|+\sup_{\tau\in[t,t']}\left\| X^{t,m}_\tau \right\|_{L^{1,d_x}_m}+L_B(s-t')\Big(1+2\sup_{\tau\in[t,t']}\left\| X^{t,m}_\tau \right\|_{L^{1,d_x}_m} \Big)\exp\Big(2L_B(s-t')\Big)\bigg)\\
&\cdot \exp\Big(L_B(s-t')\Big),
\end{align*}\normalsize
where $\displaystyle\sup_{\tau\in[t,t']}\left| X^{t,m}_\tau(x) \right|$ and $\displaystyle\sup_{\tau\in[t,t']}\left\| X^{t,m}_\tau \right\|_{L^{1,d_x}_m}$ are bounded by \eqref{eq_6_14_1}. This establishes \eqref{XLipint}.\\
Step 6 (Continuous Differentiability in $s\in[t,T]$ of the forward solution $X^{t,m}_s(x)$) By \eqref{fodesystem_lem}, $X^{t,m}_s(x)$ is differentiable in $s\in[t,T]$ and 
\begin{align}
\begin{cases}       
\dfrac{d}{ds}X^{t,m}_s(x) = f\left(X^{t,m}_s(x),X^{t,m}_s\ot m,\alpha(X^{t,m}_s(x),X^{t,m}_s\ot m,\gamma(s,X^{t,m}_s(x),X^{t,m}_s\ot m))\right),\\
X^{t,m}_t(x)= x.\\
\end{cases}
\end{align}
Thanks to the continuity of $f(x,\mu,\alpha)$, $\alpha(x,\mu,z)$ (for $(x,\mu,z)\in c_{k_0}$), $\gamma(s,x,\mu)$ and $X^{t,m}_s(x)$, and so $\dfrac{d}{ds}X^{t,m}_s(x)$ is continuous in $(t,x,m)\in[0,T]\times \R^{d_x}\times\mathcal{P}_2(\R^{d_x})$ and $s\in[t,T]$.\\
Step 7 (Continuous differentiability in $x\in\R^{d_x}$ of the forward solution $X^{t,m}_s(x)$) Recall the notations defined in Appendix \eqref{notation_1} and Table \ref{notation_2}, and replace $Z^{t,m}_s(x)$ by $\gamma(s,X^{t,m}_s(x),X^{t,m}_s\ot m)$ in these notations. Consider the following linear forward ODE system
\small\begin{align}
\label{p_xFODESYSTEMX}
\begin{cases}       
\displaystyle\frac{d}{ds}\wb{\mc{X}}^{t,m}_s(x)= \bigg(\big(\p_x f\big)^{t,m}_s(x)+\big(\p_\alpha f\big)^{t,m}_s(x)\Big(\big(\p_x  \alpha\big)^{t,m}_s(x)+\big(\p_z \alpha\big)^{t,m}_s(x)\big(\p_x\gamma\big)^{t,m}_s(x)\Big)\bigg)\wb{\mc{X}}^{t,m}_s(x),\\
	\wb{\mc{X}}^{t,m}_t(x) =\mathcal{I}_{d_x\times d_x},
\end{cases}
\end{align}\normalsize
where $\big(\p_x\gamma\big)^{t,m}_s(x):=\p_x\gamma(s,x,\mu)\Big|_{(x,\mu)=(X^{t,m}_s(x),X^{t,m}_s\ot m)}$ and $\mathcal{I}_{d_x\times d_x}$ is the $d_x\times d_x$ identity matrix. As one may note that the initial-value problem \eqref{p_xFODESYSTEMX} is the dynamics governing the evolution of the Jacobian flow $\p_x\big(X^{t,m}_s(x)\big)$. Here, we shall first solve for this $\wb{\mc{X}}^{t,m}_s(x)$, and then verify the claim accordingly. By an argument similar to that from Step 1 to Step 6, the linear ODE system \eqref{p_xFODESYSTEMX} has a unique solution $\wb{\mc{X}}^{t,m}_s(x)$ which is continuous in $(t,x,m)\in[0,T]\times \R^{d_x}\times\mathcal{P}_2(\R^{d_x})$ and is continuously differentiable in $s\in[t,T]$; however, it is not necessarily Lipschitz continuous in either $t$, $x$ or $m$ since we assume that $\p_x \gamma(s,x,\mu)$ is only continuous in $x$ and $\mu$ but nothing more.  
Moreover, since we have
\scriptsize\begin{align}\no
\left\|\wb{\mc{X}}^{t,m}_s(x)\right\|_{\mc{L}(\R^{d_x};\R^{d_x})}
=&\ \bigg\|\mathcal{I}_{d_x\times d_x}+\int_t^s \bigg(\big(\p_x f\big)^{t,m}_\tau(x)+\big(\p_\alpha f\big)^{t,m}_\tau(x)\Big(\big(\p_x  \alpha\big)^{t,m}_\tau(x)+\big(\p_z \alpha\big)^{t,m}_\tau(x) \big(\p_x\gamma\big)^{t,m}_\tau(x)\Big)\bigg)\wb{\mc{X}}^{t,m}_\tau(x)d\tau\bigg\|_{\mc{L}(\R^{d_x};\R^{d_x})}\\\label{eq_6_23}
\leq&\  1+L_B\int_t^s \left\|\wb{\mc{X}}^{t,m}_\tau(x)\right\|_{\mc{L}(\R^{d_x};\R^{d_x})}d\tau\ \ \ \ \ \ \ \ \text{ (by using \eqref{bdd_d1_f}, \eqref{p_xalpha}, \eqref{p_zalpha} and $\gamma\in \mathcal{I}_1$)},
\end{align}\normalsize
therefore, by Gr\"{o}nwall's inequality, we obtain
\footnotesize\begin{align}\label{eq_6_34}
\left\|\wb{\mc{X}}^{t,m}_s(x)\right\|_{\mc{L}(\R^{d_x};\R^{d_x})}\leq \exp\Big(L_B(s-t)\Big).
\end{align}\normalsize
Note that, we have, for any $x,\,\wt{x}\in\R^{d_x}$,
\scriptsize\begin{align*}
&\left|X^{t,m}_s(x+\wt{x})-X^{t,m}_s(x)-\wb{\mc{X}}^{t,m}_s(x)\wt{x}\right|\\
=&\ \bigg|\int_t^s f^{t,m}_\tau(x+\wt{x})-f^{t,m}_\tau(x)-\big(\p_x f\big)^{t,m}_\tau(x)\wb{\mc{X}}^{t,m}_\tau(x)\wt{x}-\big(\p_\alpha f\big)^{t,m}_\tau(x)\Big(\big(\p_x  \alpha\big)^{t,m}_\tau(x)+\big(\p_z \alpha\big)^{t,m}_\tau(x)\big(\p_x\gamma\big)^{t,m}_\tau(x)\Big)\wb{\mc{X}}^{t,m}_\tau(x)\wt{x}d\tau\bigg|\\
\leq &\ \int_t^s\bigg|\Big(\big(\p_x f\big)^{t,m}_\tau(x)+\big(\p_\alpha f\big)^{t,m}_\tau(x)\Big(\big(\p_x  \alpha\big)^{t,m}_\tau(x)+\big(\p_z \alpha\big)^{t,m}_\tau(x)\big(\p_x\gamma\big)^{t,m}_\tau(x)\Big)\Big)\Big(X^{t,m}_\tau(x+\wt{x})-X^{t,m}_\tau(x)-\wb{\mc{X}}^{t,m}_\tau(x)\wt{x}\Big)\bigg|d\tau+R_1(\wt{x})\\
\leq &\ L_B \int_t^s \Big|X^{t,m}_\tau(x+\wt{x})-X^{t,m}_\tau(x)-\wb{\mc{X}}^{t,m}_\tau(x)\wt{x}\Big| d\tau+R_1(\wt{x}),
\end{align*}\normalsize
where the last inequality is obtained by using \eqref{bdd_d1_f}, \eqref{p_xalpha}, \eqref{p_zalpha} and $\gamma\in \mathcal{I}_1$. The remainder term $R_1(\wt{x})$ has the following expression
\footnotesize\begin{align}\no
R_1(\wt{x}):=&\ \int_t^s\int_0^1\bigg(\Big|\p_x f(\Theta_1)-\big(\p_x f\big)^{t,m}_\tau(x)\Big|+\Big|\p_\alpha f(\Theta_1)\p_x  \alpha(\Theta_2)-\big(\p_\alpha f\big)^{t,m}_\tau(x)\big(\p_x  \alpha\big)^{t,m}_\tau(x)\Big|\\\no
&\ \ \ \ \ \ \ \ \ \ \ \ +\Big|\p_\alpha f(\Theta_1)\p_z  \alpha(\Theta_2)\p_x\gamma(\Theta_3)-\big(\p_\alpha f\big)^{t,m}_\tau(x)\big(\p_z  \alpha\big)^{t,m}_\tau(x)\big(\p_x\gamma\big)^{t,m}_\tau(x)\Big|\bigg)d\theta \Big|X^{t,m}_\tau(x+\wt{x})-X^{t,m}_\tau(x)\Big|d\tau\\\no
\leq &\  \left|\wt{x}\right|\exp\Big(L_B(T-t)\Big)\int_t^s\int_0^1\bigg(\Big|\p_x f(\Theta_1)-\big(\p_x f\big)^{t,m}_\tau(x)\Big|+\Big|\p_\alpha f(\Theta_1)\p_x  \alpha(\Theta_2)-\big(\p_\alpha f\big)^{t,m}_\tau(x)\big(\p_x  \alpha\big)^{t,m}_\tau(x)\Big|\\\label{R_1}
&\ \ \ \ \ \ \ \ \ \ \ \ \ \ \ \ \ \ \ \ \ \ \ \ \ \ \ \ \ \ \ \ \ \ \ \ \ \ \ +\Big|\p_\alpha f(\Theta_1)\p_z  \alpha(\Theta_2)\p_x\gamma(\Theta_3)-\big(\p_\alpha f\big)^{t,m}_\tau(x)\big(\p_z  \alpha\big)^{t,m}_\tau(x)\big(\p_x\gamma\big)^{t,m}_\tau(x)\Big|\bigg)d\theta d\tau,
\end{align}\normalsize
since $X^{t,m}_s(x)$ is Lipschitz continuous in $x$ by Step 3 (see \eqref{XLipinx}), 
where the arguments $\Theta_1:=\Big(\theta X^{t,m}_\tau(x+\wt{x})+(1-\theta)X^{t,m}_\tau(x),X^{t,m}_\tau\ot m, \alpha\big(\theta X^{t,m}_\tau(x+\wt{x})+(1-\theta)X^{t,m}_\tau(x),X^{t,m}_\tau\ot m,\gamma(\tau,\theta X^{t,m}_\tau(x+\wt{x})+(1-\theta)X^{t,m}_\tau(x),X^{t,m}_\tau\ot m)\big)\Big)$, $\Theta_2:=\Big(\theta X^{t,m}_\tau(x+\wt{x})+(1-\theta)X^{t,m}_\tau(x),X^{t,m}_\tau\ot m,\gamma(\tau,\theta X^{t,m}_\tau(x+\wt{x})+(1-\theta)X^{t,m}_\tau(x),X^{t,m}_\tau\ot m)\Big)$ and $\Theta_3:=\Big(\tau,\theta X^{t,m}_\tau(x+\wt{x})+(1-\theta)X^{t,m}_\tau(x),X^{t,m}_\tau\ot m\Big)$. Since $\p_x f(x,\mu,\alpha)$, $\p_\alpha f(x,\mu,\alpha)$, $\p_x \alpha(x,\mu,z)$, $\p_z \alpha(x,\mu,z)$ and $\p_x \gamma(s,x,\mu)$ are uniformly bounded and are also continuous in their corresponding arguments, thus the integrand in \eqref{R_1} converges to $0$ as $\big|\wt{x}\big|\to 0$ and we can use Lebesgue's dominated convergence theorem to interchange the order of the limit taking $\big|\wt{x}\big|\to 0$ and the integration in the integral of \eqref{R_1}, which implies that $R_1(\wt{x})=o\Big(\big|\wt{x}\big|\Big)$ with the small-$o\Big(\big|\wt{x}\big|\Big)$ meaning $o\Big(\big|\wt{x}\big|\Big)/\big|\wt{x}\big|\to 0$ as $\big|\wt{x}\big|\to 0$.
By using Gr\"{o}nwall's inequality, we have $$\left|X^{t,m}_s(x+\wt{x})-X^{t,m}_s(x)-\wb{\mc{X}}^{t,m}_s(x)\wt{x}\right|=o\Big(\big|\wt{x}\big|\Big).$$ Therefore, $X^{t,m}_s(x)$ is differentiable in $x\in\R^{d_x}$ and $\p_x\big(X^{t,m}_s(x)\big)=\wb{\mc{X}}^{t,m}_s(x)$, which is continuous in $(t,x,m)\in[0,T]\times \R^{d_x}\times\mathcal{P}_2(\R^{d_x})$ and is continuously differentiable in $s\in[t,T]$. Hence, the estimate \eqref{eq_6_34} is indeed the desired estimate of \eqref{Bp_xX}.\\
Step 8 (Continuous $L$-Differentiability in $m\in\mathcal{P}_2(\R^{d_x})$ of the forward solution $X^{t,m}_s(x)$) Again, recall the notations defined in Appendix \eqref{notation_1} and Table \ref{notation_2}, and also replace $Z^{t,m}_s(x)$ by $\gamma(s,X^{t,m}_s(x),X^{t,m}_s\ot m)$ in these notations. Consider the following linear forward ODE system
\small\begin{align}
\label{p_mFODESYSTEMX}
\begin{cases}       
\displaystyle\frac{d}{ds}\mc{X}^{t,m}_s(x,y)=&\ \big(\p_x f\big)^{t,m}_s(x) \mc{X}^{t,m}_s(x,y)+\big(\p_\mu f\big)^{t,m}_s(x,y)  \p_{y}\big(X^{t,m}_s(y)\big)\\
&\displaystyle+\int_{\R^{d_x}} \big(\p_\mu f\big)^{t,m}_s(x,\wh{x}) \mc{X}^{t,m}_s(\wh{x},y) \ dm(\wh{x})\\
&\displaystyle+ \big(\p_\alpha f\big)^{t,m}_s(x)  \bigg(\big(\p_x  \alpha\big)^{t,m}_s(x) \mc{X}^{t,m}_s(x,y)+\big(\p_\mu \alpha\big)^{t,m}_s(x,y)  \p_{y}\big(X^{t,m}_s(y)\big)\\
&\ \ \ \ \ \ \ \ \ \ \ \ \ \ \ \ \ \ \ \ \ \ \ \ \displaystyle+\int_{\R^{d_x}}\big(\p_\mu \alpha\big)^{t,m}_s(x,\wh{x})  \mc{X}^{t,m}_s(\wh{x},y)\ dm(\wh{x})\bigg)\\
&\displaystyle+ \big(\p_\alpha f\big)^{t,m}_s(x)  \big(\p_z \alpha\big)^{t,m}_s(x)\bigg(\big(\p_x  \gamma\big)^{t,m}_s(x) \mc{X}^{t,m}_s(x,y)+\big(\p_\mu \gamma\big)^{t,m}_s(x,y)  \p_{y}\big(X^{t,m}_s(y)\big)\\
&\ \ \ \ \ \ \ \ \ \ \ \ \ \ \ \ \ \ \ \ \ \ \ \ \ \ \ \ \ \ \ \ \ \ \ \ \displaystyle+\int_{\R^{d_x}}\big(\p_\mu \gamma\big)^{t,m}_s(x,\wh{x})  \mc{X}^{t,m}_s(\wh{x},y)\ dm(\wh{x})\bigg),\\
	\mc{X}^{t,m}_t(x,y) =0,\\
\end{cases}
\end{align}\normalsize
where \footnotesize $\big(\p_x\gamma\big)^{t,m}_s(x):=\p_x\gamma(s,x,\mu)\Big|_{(x,\mu)=(X^{t,m}_s(x),X^{t,m}_s\ot m)}$\normalsize and \footnotesize $\big(\p_\mu \gamma\big)^{t,m}_s(x,y):=\p_\mu\gamma(s,x,\mu)\big(\wt{x}\big)\Big|_{(x,\mu,\wt{x})=(X^{t,m}_s(x),X^{t,m}_s\ot m,X^{t,m}_s(y))}$\normalsize. Again, as one may note that the initial-value problem \eqref{p_mFODESYSTEMX} is the dynamics governing the evolution of the Jacobian flow $\p_m\big(X^{t,m}_s(x)\big)(y)$. By an argument similar to that from Step 1 to Step 6, the linear ODE system \eqref{p_mFODESYSTEMX} has a unique solution $\mc{X}^{t,m}_s(x,y)$ which is continuous in $(t,x,m,y)\in[0,T]\times \R^{d_x}\times\mathcal{P}_2(\R^{d_x})\times\R^{d_x}$ and is continuously differentiable in $s\in[t,T]$; however, it is not necessarily Lipschitz continuous in either $t$, $x$, $y$ or $m$ since we assume that $\p_x \gamma(s,x,\mu)$ and $\p_\mu \gamma(s,x,\mu)(y)$ are only continuous in $x$, $y$ and $\mu$.  
Moreover, since we have
\footnotesize\begin{align}\no
&\left\|\mc{X}^{t,m}_s(x,y)\right\|_{\mc{L}(\R^{d_x};\R^{d_x})}\\\no
=&\ \bigg\|\int_t^s \Bigg(\bigg(\big(\p_x f\big)^{t,m}_\tau(x)+\big(\p_\alpha f\big)^{t,m}_\tau(x)\Big(\big(\p_x  \alpha\big)^{t,m}_\tau(x)+\big(\p_z \alpha\big)^{t,m}_\tau(x) \big(\p_x\gamma\big)^{t,m}_\tau(x)\Big)\bigg)\mc{X}^{t,m}_\tau(x,y)\\\no
&\ \ \ \ \ \ \ \ +\int_{\R^{d_x}}\Big(\big(\p_\mu f\big)^{t,m}_\tau(x,\wh{x})+\big(\p_\alpha f\big)^{t,m}_\tau(x)\big(\p_\mu \alpha\big)^{t,m}_\tau(x,\wh{x})+\big(\p_\alpha f\big)^{t,m}_\tau(x)\big(\p_z \alpha\big)^{t,m}_\tau(x)\big(\p_\mu \gamma\big)^{t,m}_\tau(x,\wh{x})\Big) \mc{X}^{t,m}_\tau(\wh{x},y)dm(\wh{x})\\\no
&\ \ \ \ \ \ \ \ +\Big(\big(\p_\mu f\big)^{t,m}_\tau(x,y)+ \big(\p_\alpha f\big)^{t,m}_\tau(x)\big(\p_\mu \alpha\big)^{t,m}_\tau(x,y)+\big(\p_\alpha f\big)^{t,m}_\tau(x)\big(\p_z \alpha\big)^{t,m}_\tau(x)\big(\p_\mu \gamma\big)^{t,m}_\tau(x,y) \Big)\p_{y}\big(X^{t,m}_\tau(y)\big)\Bigg)d\tau\bigg\|_{\mc{L}(\R^{d_x};\R^{d_x})}\\\no
\leq&\  L_B\int_t^s \bigg(\left\|\mc{X}^{t,m}_\tau(x,y)\right\|_{\mc{L}(\R^{d_x};\R^{d_x})}+\int_{\R^{d_x}}\left\|\mc{X}^{t,m}_\tau(\wh{x},y)\right\|_{\mc{L}(\R^{d_x};\R^{d_x})}dm(\wh{x})+\left\|\p_{y}\big(X^{t,m}_\tau(y)\big)\right\|_{\mc{L}(\R^{d_x};\R^{d_x})}\bigg)d\tau\\\no
&\text{ (by using \eqref{bdd_d1_f}, \eqref{p_xalpha}, \eqref{p_zalpha} and $\gamma\in \mathcal{I}_1$)},
\end{align}\normalsize
therefore, by first integrating with respect to $x$ and then using Gr\"{o}nwall's inequality and \eqref{Bp_xX}, we obtain
\footnotesize\begin{align}\label{eq_6_26_1}
\int_{\R^{d_x}}\left\|\mc{X}^{t,m}_s(x,y)\right\|_{\mc{L}(\R^{d_x};\R^{d_x})}dm(x)\leq L_B(s-t)\exp\Big(3L_B(s-t)\Big).
\end{align}
Therefore,
\begin{align}\label{eq_6_27_1}
\left\|\mc{X}^{t,m}_s(x,y)\right\|_{\mc{L}(\R^{d_x};\R^{d_x})}\leq L_B(s-t)\exp\Big(2L_B(s-t)\Big)\Big(1+L_B(s-t)\exp\Big(2L_B(s-t)\Big)\Big).
\end{align}
To show that $X^{t,m}_s(x)$ is $L$-differentiable at $m\in\mc{P}_2(\R^{d_x})$ and $\p_m\big(X^{t,m}_s(x)\big)(y)=\mc{X}^{t,m}_s(x,y)$, it suffices to show that, for any $Y\in L^{2,d_x}_m$, 
\begin{align*}
\bigg|X^{t,Y\ot m}_s(x)-X^{t,m}_s(x)-\int_{\R^{d_x}}\mc{X}^{t,m}_s(x,y)\big(Y(y)-y\big)dm(y)\bigg|=o\Big(\|Y-Id\|_{L^{2,d_x}_m}\Big),
\end{align*}
where the small-$o$ function $o\Big(\|Y-Id\|_{L^{2,d_x}_m}\Big)$ means that the term tends to $0$ as $\|Y-Id\|_{L^{2,d_x}_m}$ goes to $0$.
Denote $m_Y:=Y\ot m\in\mc{P}_2(\R^{d_x})$ and clearly $W_1(m_Y,m)\leq W_2(m_Y,m)\leq \|Y-Id\|_{L^{2,d_x}_m}$. Note that, since all $\p_x f(x,\mu,\alpha)$, $\p_\mu f(x,\mu,\alpha)(\wt{x})$, $\p_\alpha f(x,\mu,\alpha)$, $\p_x \alpha(x,\mu,z)$, $\p_\mu \alpha(x,\mu,z)(\wt{x})$, $\p_z \alpha(x,\mu,z)$, $\p_x \gamma(s,x,\mu)$ and $\p_\mu \gamma(s,x,\mu)(\wt{x})$ are uniformly bounded and are also continuous in their corresponding arguments, and $X^{t,m}_s(x)$ is Lipschitz continuous, continuously differentiable in $x$ by Step 3 and 7, and it is also Lipschitz continuous in $m$ by Step 4, we have, for any $Y\in L^{2,d_x}_m$,
\footnotesize\begin{align}\no
&\ \bigg|X^{t,m_Y}_s(x)-X^{t,m}_s(x)-\int_{\R^{d_x}}\mc{X}^{t,m}_s(x,y)\big(Y(y)-y\big)dm(y)\bigg|\\\no
=&\ \Bigg|\int_t^s \Bigg(f^{t,m_Y}_\tau(x)-f^{t,m}_\tau(x)-\big(\p_x f\big)^{t,m}_\tau(x)\int_{\R^{d_x}}\mc{X}^{t,m}_\tau(x,y)\big(Y(y)-y\big)dm(y)\\\no
&\ \ \ \ \ \ \ \ \ \ -\int_{\R^{d_x}}\big(\p_\mu f\big)^{t,m}_\tau(x,y)  \p_{y}\big(X^{t,m}_\tau(y)\big)\big(Y(y)-y\big)dm(y)\\\no
&\ \ \ \ \ \ \ \ \ \ -\int_{\R^{d_x}} \int_{\R^{d_x}}\big(\p_\mu f\big)^{t,m}_\tau(x,\wh{x}) \mc{X}^{t,m}_\tau(\wh{x},y) \ \big(Y(y)-y\big)dm(\wh{x})dm(y)\\\no
&\ \ \ \ \ \ \ \ \ \ -\big(\p_\alpha f\big)^{t,m}_\tau(x)  \int_{\R^{d_x}}\bigg(\big(\p_x  \alpha\big)^{t,m}_\tau(x) \mc{X}^{t,m}_\tau(x,y)+\big(\p_\mu \alpha\big)^{t,m}_\tau(x,y)  \p_{y}\big(X^{t,m}_\tau(y)\big)\\\no
&\ \ \ \ \ \ \ \ \ \ \ \ \ \ \ \ \ \ \ \ \ \ \ \ \ \ \ \ \ \ \ \ \ \ \ \ \ \ \ +\int_{\R^{d_x}}\big(\p_\mu \alpha\big)^{t,m}_\tau(x,\wh{x})  \mc{X}^{t,m}_\tau(\wh{x},y)\ dm(\wh{x})\bigg)\big(Y(y)-y\big)dm(y)\\\no
&\ \ \ \ \ \ \ \ \ \ - \big(\p_\alpha f\big)^{t,m}_\tau(x)  \big(\p_z \alpha\big)^{t,m}_\tau(x)\int_{\R^{d_x}}\bigg(\big(\p_x  \gamma\big)^{t,m}_\tau(x) \mc{X}^{t,m}_\tau(x,y)+\big(\p_\mu \gamma\big)^{t,m}_\tau(x,y)  \p_{y}\big(X^{t,m}_\tau(y)\big)\\\no
&\ \ \ \ \ \ \ \ \ \ \ \ \ \ \ \ \ \ \ \ \ \ \ \ \ \ \ \ \ \ \ \ \ \ \ \ \ \ \ \ \ \ \ \ \ \ \ \ \ \ \ \ \ \ +\int_{\R^{d_x}}\big(\p_\mu \gamma\big)^{t,m}_\tau(x,\wh{x})  \mc{X}^{t,m}_\tau(\wh{x},y)\ dm(\wh{x})\bigg)\big(Y(y)-y\big)dm(y)\Bigg)d\tau\Bigg|\\\no
\leq&\ \int_t^s \Bigg|\bigg(\big(\p_x f\big)^{t,m}_\tau(x)+\big(\p_\alpha f\big)^{t,m}_\tau(x)\big(\p_x  \alpha\big)^{t,m}_\tau(x)+\big(\p_\alpha f\big)^{t,m}_\tau(x)  \big(\p_z \alpha\big)^{t,m}_\tau(x)\big(\p_x  \gamma\big)^{t,m}_\tau(x)\bigg)\\\no
&\ \ \ \ \ \ \ \ \cdot\bigg(X^{t,m_Y}_\tau(x)-X^{t,m}_\tau(x)-\int_{\R^{d_x}}\mc{X}^{t,m}_\tau(x,y)\big(Y(y)-y\big)dm(y)\bigg)\\\no
&\ \ \ \ \ \ \ \ \ \ -\int_{\R^{d_x}}\bigg(\big(\p_\mu f\big)^{t,m}_\tau(x,\wh{x})+\big(\p_\alpha f\big)^{t,m}_\tau(x)\big(\p_\mu \alpha\big)^{t,m}_\tau(x,\wh{x})+\big(\p_\alpha f\big)^{t,m}_\tau(x)  \big(\p_z \alpha\big)^{t,m}_\tau(x)\big(\p_\mu \gamma\big)^{t,m}_\tau(x,\wh{x})\bigg)\\\no
& \ \ \ \ \ \ \ \ \ \ \ \ \ \ \ \ \ \ \ \ \cdot\bigg(X^{t,m_Y}_\tau(\wh{x})-X^{t,m}_\tau(\wh{x})-\int_{\R^{d_x}}\mc{X}^{t,m}_\tau(\wh{x},y)\big(Y(y)-y\big)dm(y)\bigg)dm(\wh{x})\Bigg|d\tau+o\Big(\|Y-Id\|_{L^{2,d_x}_m}\Big)\\\no
\leq &\ L_B\int_t^s \bigg|X^{t,m_Y}_\tau(x)-X^{t,m}_\tau(x)-\int_{\R^{d_x}}\mc{X}^{t,m}_\tau(x,y)\big(Y(y)-y\big)dm(y)\bigg|d\tau\\\label{eq_7_43_new}
&\ +L_B\int_t^s \int_{\R^{d_x}}\bigg|X^{t,m_Y}_\tau(\wh{x})-X^{t,m}_\tau(\wh{x})-\int_{\R^{d_x}}\mc{X}^{t,m}_\tau(\wh{x},y)\big(Y(y)-y\big)dm(y)\bigg|dm(\wh{x})d\tau+o\Big(\|Y-Id\|_{L^{2,d_x}_m}\Big),
\end{align}\normalsize
where the last inequality is obtained by using \eqref{bdd_d1_f}, \eqref{p_xalpha}, \eqref{p_zalpha} and $\gamma\in \mathcal{I}_1$, the small-$o\Big(\|Y-Id\|_{L^{2,d_x}_m}\Big)$ means $o\Big(\|Y-Id\|_{L^{2,d_x}_m}\Big)/\|Y-Id\|_{L^{2,d_x}_m}\to 0$ as $\|Y-Id\|_{L^{2,d_x}_m}\to 0$, and this small-$o$ nature of $o\Big(\|Y-Id\|_{L^{2,d_x}_m}\Big)$ can be again deduced by using \eqref{XLipinm} together with Lebesgue's dominated convergence theorem through an argument similar to that for \eqref{R_1}.
By integrating with respect to $x$ and using Gr\"{o}nwall's inequality, we have 
\begin{align}\label{eq_7_44_new}
\int_{\R^{d_x}}\bigg|X^{t,m_Y}_s(\wh{x})-X^{t,m}_s(\wh{x})-\int_{\R^{d_x}}\mc{X}^{t,m}_s(\wh{x},y)\big(Y(y)-y\big)dm(y)\bigg|dm(\wh{x})=o\Big(\|Y-Id\|_{L^{2,d_x}_m}\Big),
\end{align}
 and then, by substituting \eqref{eq_7_44_new} into \eqref{eq_7_43_new} and using Gr\"{o}nwall's inequality again, we obtain $$\bigg|X^{t,m_Y}_s(x)-X^{t,m}_s(x)-\int_{\R^{d_x}}\mc{X}^{t,m}_s(x,y)\big(Y(y)-y\big)dm(y)\bigg|=o\Big(\|Y-Id\|_{L^{2,d_x}_m}\Big).$$ Therefore, $X^{t,m}_s(x)$ is $L$-differentiable in $m\in\mc{P}_2(\R^{d_x})$ and $\p_m\big(X^{t,m}_s(x)\big)(y)=\mc{X}^{t,m}_s(x,y)$, which is continuous in $(t,x,m,y)\in[0,T]\times \R^{d_x}\times\mathcal{P}_2(\R^{d_x})\times\R^{d_x}$ and is continuously differentiable in $s\in[t,T]$. Hence, the estimate \eqref{eq_6_27_1} is indeed the desired estimate of \eqref{Bp_mX}.\\
Step 9 (Continuous differentiability in $t\in[0,T]$ of the forward solution $X^{t,m}_s(x)$) Consider the following linear forward ODE system
\small\begin{align}
\label{p_tFODESYSTEMX}
\begin{cases}       
\displaystyle\frac{d}{ds}\wt{\mc{X}}^{t,m}_s(x)=\big(\p_x f\big)^{t,m}_s(x) \wt{\mc{X}}^{t,m}_s(x)+\int_{\R^{d_x}} \big(\p_\mu f\big)^{t,m}_s(x,\wh{x}) \wt{\mc{X}}^{t,m}_s(\wh{x}) \ dm(\wh{x})\\
\ \ \ \ \ \ \ \ \ \ \ \ \ \ \ \ \ \ \ \displaystyle+ \big(\p_\alpha f\big)^{t,m}_s(x)  \bigg(\big(\p_x  \alpha\big)^{t,m}_s(x) \wt{\mc{X}}^{t,m}_s(x)+\int_{\R^{d_x}}\big(\p_\mu \alpha\big)^{t,m}_s(x,\wh{x})  \wt{\mc{X}}^{t,m}_s(\wh{x})\ dm(\wh{x})\bigg)\\
\ \ \ \ \ \ \ \ \ \ \ \ \ \ \ \ \ \ \ \displaystyle+ \big(\p_\alpha f\big)^{t,m}_s(x)  \big(\p_z \alpha\big)^{t,m}_s(x)\bigg(\big(\p_x  \gamma\big)^{t,m}_s(x) \wt{\mc{X}}^{t,m}_s(x)+\int_{\R^{d_x}}\big(\p_\mu \gamma\big)^{t,m}_s(x,\wh{x})  \wt{\mc{X}}^{t,m}_s(\wh{x})\ dm(\wh{x})\bigg),\\
\displaystyle\wt{\mc{X}}^{t,m}_t(x)=-f^{t,m}_t(x)=-f\Big(x,m,\alpha\big(x,m,\gamma(t,x,m)\big)\Big).
\end{cases}
\end{align}\normalsize
Again, as one may also note that the initial-value problem \eqref{p_tFODESYSTEMX} is the dynamics governing the evolution of the Jacobian flow $\p_t\big(X^{t,m}_s(x)\big)$. By the argument similar to that from Step 1 to Step 6, the linear ODE system \eqref{p_tFODESYSTEMX} has a unique solution $\wt{\mc{X}}^{t,m}_s(x)$ which is continuous in $(t,x,m)\in[0,T]\times \R^{d_x}\times\mathcal{P}_2(\R^{d_x})$ and is continuously differentiable in $s\in[t,T]$; however, it is not necessarily Lipschitz continuous in either $t$, $x$ or $m$, since $\p_x \gamma(s,x,\mu)$ and $\p_\mu \gamma(s,x,\mu)(y)$ are only continuous in $x$, $y$ and $\mu$. 
Moreover, since we have
\footnotesize\begin{align}\no
&\left|\wt{\mc{X}}^{t,m}_s(x)\right|\\\no
=&\ \bigg|-f\Big(x,m,\alpha\big(x,m,\gamma(t,x,m)\big)\Big)+\int_t^s \Bigg(\bigg(\big(\p_x f\big)^{t,m}_\tau(x)+\big(\p_\alpha f\big)^{t,m}_\tau(x)\Big(\big(\p_x  \alpha\big)^{t,m}_\tau(x)+\big(\p_z \alpha\big)^{t,m}_\tau(x) \big(\p_x\gamma\big)^{t,m}_\tau(x)\Big)\bigg)\wt{\mc{X}}^{t,m}_\tau(x)\\\no
&\ \ \ \ \ \ \ \ +\int_{\R^{d_x}}\Big(\big(\p_\mu f\big)^{t,m}_\tau(x,\wh{x})+\big(\p_\alpha f\big)^{t,m}_\tau(x)\big(\p_\mu \alpha\big)^{t,m}_\tau(x,\wh{x})+\big(\p_\alpha f\big)^{t,m}_\tau(x)\big(\p_z \alpha\big)^{t,m}_\tau(x)\big(\p_\mu \gamma\big)^{t,m}_\tau(x,\wh{x})\Big) \wt{\mc{X}}^{t,m}_\tau(\wh{x})dm(\wt{x})\Bigg)d\tau\bigg|\\\label{eq_7_46_new}
\leq&\  \wb{L}_B\big(1+|x|+\|m\|_1\big)+L_B\int_t^s \bigg(\left|\wt{\mc{X}}^{t,m}_\tau(x)\right|+\int_{\R^{d_x}}\left|\wt{\mc{X}}^{t,m}_\tau(\wh{x})\right|dm(\wh{x})\bigg)d\tau\ \\\no
&\ \text{ (by using \eqref{bdd_d1_f}, \eqref{ligf}, \eqref{p_xalpha}, \eqref{p_zalpha}, \eqref{linalpha} and $\gamma\in \mathcal{I}_1$)},
\end{align}\normalsize
therefore, by first integrating with respect to $x$ and then using Gr\"{o}nwall's inequality, we obtain
\begin{align}\label{eq_6_30_1}
\int_{\R^{d_x}}\left|\wt{\mc{X}}^{t,m}_s(x)\right|dm(x)\leq \wb{L}_B\big(1+|x|+\|m\|_1\big)\exp\Big(2L_B(s-t)\Big).
\end{align}
Substituting \eqref{eq_6_30_1} into \eqref{eq_7_46_new} and using Gr\"{o}nwall's inequality again, we finally obtain
\begin{align}\label{eq_6_31_1}
\left|\wt{\mc{X}}^{t,m}_s(x)\right|\leq \wb{L}_B\big(1+|x|+\|m\|_1\big)\Big(L_B(s-t)\exp\Big(2L_B(s-t)\Big)+1\Big)\exp\Big(L_B(s-t)\Big).
\end{align}
By an argument similar to Steps 7 and 8, one can show that $X^{t,m}_s(x)$ is differentiable in $t\in[0,T]$ and $\p_t\big(X^{t,m}_s(x)\big)=\wt{\mc{X}}^{t,m}_s(x)$, which is continuous in $(t,x,m)\in[0,T]\times \R^{d_x}\times\mathcal{P}_2(\R^{d_x})$ and is continuously differentiable in $s\in[t,T]$. Hence, the estimate \eqref{eq_6_31_1} is indeed the desired estimate of \eqref{Bp_tX}.
\end{proof}

For the sake of convenience, we introduce the following notations similar to those defined in Appendix \eqref{notation_1} and Table \ref{notation_2}:
\begin{align}\label{notation_1_1}
&\begin{cases}
\alpha^{t,m,(n)}_s(x):= \alpha(X^{t,m,(n)}_s(x), X^{t,m,(n)}_s\ot m,Z^{t,m,(n)}_s(x)),\\ 
f^{t,m,(n)}_s(x):=f(X^{t,m,(n)}_s(x),X^{t,m,(n)}_s\ot m,\alpha^{t,m,(n)}_s(x)),\\
g^{t,m,(n)}_s(x):=g(X^{t,m,(n)}_s(x),X^{t,m,(n)}_s\ot m,\alpha^{t,m,(n)}_s(x)),\\ 
k^{t,m,(n)}_s(x):=k(X^{t,m,(n)}_s(x),X^{t,m,(n)}_s\ot m),\\
p^{t,m,(n)}_s(x):=p(X^{t,m,(n)}_s(x),X^{t,m,(n)}_s\ot m);
\end{cases}\ \ \ \  
\end{align} 
\begin{align}\label{notation_2_1}
\begin{cases}
\big(\p_x\alpha\big)^{t,m,(n)}_s(y):= \p_x \alpha(x,\mu,z)\Big|_{(x,\mu,z)=(X^{t,m,(n)}_s(y), X^{t,m,(n)}_s\ot m,Z^{t,m,(n)}_s(y))},\\ 
\big(\p_\mu\alpha\big)^{t,m,(n)}_s(y,\wt{y}):= \p_\mu \alpha(x,\mu,z)(\wt{x})\Big|_{(x,\mu,z,\wt{x})=(X^{t,m,(n)}_s(y), X^{t,m,(n)}_s\ot m,Z^{t,m,(n)}_s(y),X^{t,m,(n)}_s(\wt{y}))},\\ 
\big(\p_z\alpha\big)^{t,m,(n)}_s(y):= \p_z \alpha(x,\mu,z)\Big|_{(x,\mu,z)=(X^{t,m,(n)}_s(y), X^{t,m,(n)}_s\ot m,Z^{t,m,(n)}_s(y))},\\ 
\end{cases}
\end{align} 
and so are the derivatives of $f$, $g$, $k$ and $p$, for example, $$\big(\p_x f\big)^{t,m,(n)}_s(y):= \p_x f(x,\mu,\alpha)\Big|_{(x,\mu,\alpha)=(X^{t,m,(n)}_s(y), X^{t,m,(n)}_s\ot m,\alpha^{t,m,(n)}_s(y))}.$$

\textbf{Step 3} (Aims to prove the claim that $\vertiii{\gamma^{(n)}}_1\leq 2\wb{L}_p$). 
For any $t\in[0,T]$ such that\small
\begin{align}\label{eps_3}
T-t\leq \min\bigg\{&\frac{1}{6\wb{L}_B},\frac{\wb{L}_p}{50\Big(2\wb{L}_p\Lambda_f+L_g(1+L_\alpha+2\wb{L}_pL_\alpha)\Big)}\bigg\}=:\eps_3,
\end{align}\normalsize
we have 
\begin{align}\label{small_1}
\wb{L}_B(T-t)\leq \frac{1}{6},\ \exp\Big(\wb{L}_B(T-t)\Big)\leq \frac{6}{5},\ \exp\Big(2\wb{L}_B(T-t)\Big)\leq \frac{3}{2},
\end{align}
and thus, for any $\gamma^{(n-1)}\in \mathcal{I}_1$, $s\in[t,T]$, we obtain
\small\begin{align}\no
&\big|\gamma^{(n)} (s,x,m)\big|\\\no
\leq &\ \wb{L}_p\bigg(1+\big|X^{s,m,(n)}_T(x)\big|+\big\|X^{s,m,(n)}_T\big\|_{L^{1,d_x}_m}\bigg)\\\no
&+ \int_s^T \Bigg|\int_{\R^{d_x}}\bigg(\big(\p_\mu f\big)^{s,m,(n)}_\tau(\wt{x},x)\cdot  Z^{s,m,(n)}_\tau(\wt{x})+ \big(\p_\mu g\big)^{s,m,(n)}_\tau(\wt{x},x)\bigg)dm(\wt{x})\\\no
&\ \ \ \ \ \ \ \ \ \ \ \ \ \ +\big(\p_x f\big)^{s,m,(n)}_\tau(x)\cdot Z^{s,m,(n)}_\tau(x)+\big(\p_x g\big)^{s,m,(n)}_\tau(x)\Bigg|d\tau\text{  (by using \eqref{p2})}\\\no
\leq &\ \wb{L}_p\bigg(1+\big|X^{s,m,(n)}_T(x)\big|+\big\|X^{s,m,(n)}_T\big\|_{L^{1,d_x}_m}\bigg)\\\no
&+ \int_s^T \bigg(\Lambda_f\Big(\big\|Z^{s,m,(n)}_\tau\big\|_{L^{1,d_x}_m}+\big|Z^{s,m,(n)}_\tau(x)\big|\Big)\\\no
&\ \ \ \ + L_g\left(1+2\|X^{s,m,(n)}_\tau\|_{L^{1,d_x}_m}+\int_{\R^{d_x}}\Big|\alpha^{s,m,(n)}_\tau(\wt{x})\Big|dm(\wt{x})+\left|X^{s,m,(n)}_\tau(x)\right|\right)\\\no
&\ \ \ \ + L_g\left(1+|X^{s,m,(n)}_\tau(x)|+\|X^{s,m,(n)}_\tau\|_{L^{1,d_x}_m}+\Big|\alpha^{s,m,(n)}_\tau(x)\Big|\right)\bigg)d\tau\\\no
&\text{(by using \eqref{bdd_d1_f} and \eqref{LGgDerivatives})}\\\no
\leq &\ \wb{L}_p\bigg(1+\big|X^{s,m,(n)}_T(x)\big|+\big\|X^{s,m,(n)}_T\big\|_{L^{1,d_x}_m}\bigg)\\\no
&+ \int_s^T \bigg(\Big(2\wb{L}_p\Lambda_f+L_g(1+L_\alpha+2\wb{L}_pL_\alpha)\Big)\left(2+|X^{s,m,(n)}_\tau(x)|+3\|X^{s,m,(n)}_\tau\|_{L^{1,d_x}_m}\right)+L_g\left|X^{s,m,(n)}_\tau(x)\right|\bigg)d\tau\\\no
&\text{(by using \eqref{linalpha} and $\gamma^{(n-1)}\in \mathcal{I}_1$)}\\\no
\leq&\ \frac{9}{5}\wb{L}_p\bigg(1+|x|+\|m\|_1\bigg)+5\Big(2\wb{L}_p\Lambda_f+L_g(1+L_\alpha+2\wb{L}_pL_\alpha)\Big)(T-t)\left(1+|x|+\|m\|_1\right)+\frac{6}{5}L_g(T-t)(1+|x|+\|m\|_1)\\\no
&\text{(by using \eqref{eq_6_14_1} and \eqref{small_1})}\\\no
\leq &\ 2\wb{L}_p\bigg(1+|x|+\|m\|_1\bigg),
\end{align}\normalsize
which implies that $\vertiii{\gamma^{(n)}}_1\leq 2\wb{L}_p$.

\textbf{Step 4} (Aims to prove the claim that $\vertiii{\gamma^{(n)}}_2\leq 2L_p$). Since $\gamma^{(0)}(s,x,\mu)= p(x,\mu)$ belongs to $\mathcal{I}_1$ and, by Lemma \ref{lem6_3}, we know that the forward process $X^{s,m,(1)}_\tau(x)$ is differentiable in $x\in\R^{d_x}$ and $L$-differentiable in $m\in\mathcal{P}_2(\R^{d_x})$ with its derivatives being continuous in their corresponding arguments, therefore $\gamma^{(1)}(s,x,\mu)$ is also differentiable in $x\in\R^{d_x}$ and $L$-differentiable in $\mu\in\mathcal{P}_2(\R^{d_x})$ with its derivatives being continuous in their corresponding arguments. Inductively, we also know that $\gamma^{(n)}(s,x,\mu)$ is differentiable in $x\in\R^{d_x}$ and $L$-differentiable in $\mu\in\mathcal{P}_2(\R^{d_x})$ with its derivatives being continuous in their corresponding arguments, since all $p(x,\mu)$, $\p_\mu f(x,\mu,\alpha)(\wt{x})$, $\p_\mu g(x,\mu,\alpha)(\wt{x})$, $\p_x f(x,\mu,\alpha)$, $\p_x g(x,\mu,\alpha)$, $\alpha(x,\mu,z)$, and by induction hypothesis, $\gamma^{(n-1)}(s,x,\mu)$ are differentiable in $x,\,\wt{x},\,z\in\R^{d_x}$, $\alpha\in\R^{d_\alpha}$ and $L$-differentiable in $\mu\in\mathcal{P}_2(\R^{d_x})$ with its derivatives being continuous in their corresponding arguments.

Recall the notations introduced in \eqref{notation_1_1} and \eqref{notation_2_1}. For any $t\in[0,T]$ such that\small
\begin{align}\label{eps_4}
T-t\leq \min\bigg\{&\eps_3,
\frac{7L_p}{73\big(3(2\wb{L}_p\wb{l}_f+\Lambda_g)+3(2\wb{L}_p\wb{l}_f+\wb{l}_g)L_\alpha(1+2L_p)+6\Lambda_fL_p\big)}\bigg\}=:\eps_4,
\end{align}\normalsize
and, for any $\gamma^{(n-1)}\in \mathcal{I}_1$, $s\in[t,T]$, we have
\begin{align*}
&\|\p_x\gamma^{(n)} (s,x,m)\|_{\mc{L}(\R^{d_x};\R^{d_x})}\\
=&\ \bigg\|\big(\p_x p\big)^{s,m,(n)}_T(x)\p_x\big(X^{s,m,(n)}_T(x)\big)\\
&+\int_s^T\bigg(\int_{\R^{d_x}}\Big(\big(\p_{\wt{x}}\p_\mu f\big)^{s,m,(n)}_\tau(\wt{x},x)\cdot  Z^{s,m,(n)}_\tau(\wt{x})+ \big(\p_{\wt{x}}\p_\mu g\big)^{s,m,(n)}_\tau(\wt{x},x)\Big)dm(\wt{x})\\
&\ \ \ \ \ \ \ \ \ \ +\big(\p_x f\big)^{s,m,(n)}_\tau(x)\cdot \big(\p_x\gamma^{(n-1)}\big)^{s,m,(n)}_\tau(x)+\Big(\big(\p_x\p_x f\big)^{s,m,(n)}_\tau(x)\cdot Z^{s,m,(n)}_\tau(x)+\big(\p_x\p_x g\big)^{s,m,(n)}_\tau(x)\Big)\\
&\ \ \ \ \ \ \ \ \ \ +\Big(\big(\p_\alpha\p_x f\big)^{s,m,(n)}_\tau(x)\cdot Z^{s,m,(n)}_\tau(x)+\big(\p_\alpha\p_x g\big)^{s,m,(n)}_\tau(x)\Big)\\
&\ \ \ \ \ \ \ \ \ \ \ \ \ \ \cdot\Big(\big(\p_x\alpha\big)^{s,m,(n)}_\tau(x)+\big(\p_z\alpha\big)^{s,m,(n)}_\tau(x)\big(\p_x\gamma^{(n-1)}\big)^{s,m,(n)}_\tau(x)\Big)\bigg)\p_x\big(X^{s,m,(n)}_\tau(x)\big)d\tau\bigg\|_{\mc{L}(\R^{d_x};\R^{d_x})}\\
\leq &\  L_p\left\|\p_x\big(X^{s,m,(n)}_T(x)\big)\right\|_{\mc{L}(\R^{d_x};\R^{d_x})}\\
&+\big(4\wb{L}_p\wb{l}_f+2\Lambda_g+2L_p\Lambda_f+(2\wb{L}_p\wb{l}_f+\wb{l}_g)L_\alpha(1+2L_p)\big)\int_s^T\left\|\p_x\big(X^{s,m,(n)}_\tau(x)\big)\right\|_{\mc{L}(\R^{d_x};\R^{d_x})}d\tau\\
&\text{(by using \eqref{bdd_d1_f}, \eqref{bdd_d2_f}, \eqref{bdd_d2_g_1}, \eqref{bdd_d2_g_2}, \eqref{p_xalpha}, \eqref{p_zalpha}, \eqref{p1} and $\gamma^{(n-1)}\in \mathcal{I}_1$)}\\
\leq &\  \frac{6}{5}\Big(L_p+\big(4\wb{L}_p\wb{l}_f+2\Lambda_g+2L_p\Lambda_f+(2\wb{L}_p\wb{l}_f+\wb{l}_g)L_\alpha(1+2L_p)\big)(T-s)\Big)\\
&\text{(by using \eqref{Bp_xX} and \eqref{small_1})}\\
\leq &\ 2L_p,
\end{align*}
where $\big(\p_x\gamma^{(n-1)}\big)^{s,m,(n)}_\tau(x):=\p_x\gamma^{(n-1)}(s,x,\mu)\Big|_{(x,\mu)=(X^{s,m,(n)}_\tau(x),X^{s,m,(n)}_\tau\ot m)}$.
In addition, we also have
\footnotesize\begin{align*}
&\|\p_\mu\gamma^{(n)} (s,x,m)(y)\|_{\mc{L}(\R^{d_x};\R^{d_x})}\\
=&\ \bigg\|\big(\p_x p\big)^{s,m,(n)}_T(x)\p_m\big(X^{s,m,(n)}_T(x)\big)(y)+\big(\p_\mu p\big)^{s,m,(n)}_T(x,y)\p_y\big(X^{s,m,(n)}_T(y)\big)\\
&+\int_{\R^{d_x}}\big(\p_\mu p\big)^{s,m,(n)}_T(x,\wt{x})\p_{m}\big(X^{s,m,(n)}_T(\wt{x})\big)(y)dm(\wt{x})\\
&+\int_s^T\bigg(\int_{\R^{d_x}}\bigg(\big(\p_x\p_\mu f\big)^{s,m,(n)}_\tau(\wt{x},x)\cdot  Z^{s,m,(n)}_\tau(\wt{x})+ \big(\p_x\p_\mu g\big)^{s,m,(n)}_\tau(\wt{x},x)\bigg)\p_m\big(X^{s,m,(n)}_\tau(\wt{x})\big)(y)dm(\wt{x})\\
&\ \ \ \ \ \ \ \ \ \ +\int_{\R^{d_x}}\bigg(\big(\p_{\wt{x}}\p_\mu f\big)^{s,m,(n)}_\tau(\wt{x},x)\cdot  Z^{s,m,(n)}_\tau(\wt{x})+ \big(\p_{\wt{x}}\p_\mu g\big)^{s,m,(n)}_\tau(\wt{x},x)\bigg)dm(\wt{x})\p_m\big(X^{s,m,(n)}_\tau(x)\big)(y)\\
&\ \ \ \ \ \ \ \ \ \ +\int_{\R^{d_x}}\int_{\R^{d_x}}\bigg(\big(\p_\mu\p_\mu f\big)^{s,m,(n)}_\tau(\wt{x},x,\wh{x})\cdot  Z^{s,m,(n)}_\tau(\wt{x})+ \big(\p_\mu\p_\mu g\big)^{s,m,(n)}_\tau(\wt{x},x,\wh{x})\bigg)\p_m\big(X^{s,m,(n)}_\tau(\wh{x})\big)(y)dm(\wh{x})dm(\wt{x})\\
&\ \ \ \ \ \ \ \ \ \ +\int_{\R^{d_x}}\bigg(\big(\p_\mu\p_\mu f\big)^{s,m,(n)}_\tau(\wt{x},x,y)\cdot  Z^{s,m,(n)}_\tau(\wt{x})+ \big(\p_\mu\p_\mu g\big)^{s,m,(n)}_\tau(\wt{x},x,y)\bigg)dm(\wt{x})\p_y\big(X^{s,m,(n)}_\tau(y)\big)\\
&\ \ \ \ \ \ \ \ \ \ +\int_{\R^{d_x}}\bigg(\big(\p_\alpha\p_\mu f\big)^{s,m,(n)}_\tau(\wt{x},x)\cdot  Z^{s,m,(n)}_\tau(\wt{x})+ \big(\p_\alpha\p_\mu g\big)^{s,m,(n)}_\tau(\wt{x},x)\bigg)\p_m\big(\alpha^{s,m,(n)}_\tau(\wt{x})\big)(y)dm(\wt{x})
\\
&\ \ \ \ \ \ \ \ \ \ +\int_{\R^{d_x}} \big(\p_\mu f\big)^{s,m,(n)}_\tau(\wt{x},x)\cdot  \p_m\big(Z^{s,m,(n)}_\tau(\wt{x})\big)(y)dm(\wt{x})+\big(\p_x f\big)^{s,m,(n)}_\tau(x)\cdot \p_m\big(Z^{s,m,(n)}_\tau(x)\big)(y)\\
&\ \ \ \ \ \ \ \ \ \ +\bigg(\big(\p_x\p_x f\big)^{s,m,(n)}_\tau(x)\cdot Z^{s,m,(n)}_\tau(x)+\big(\p_x\p_x g\big)^{s,m,(n)}_\tau(x)\bigg)\p_m\big(X^{s,m,(n)}_\tau(x)\big)(y)\\
&\ \ \ \ \ \ \ \ \ \ +\int_{\R^{d_x}}\bigg(\big(\p_\mu\p_x f\big)^{s,m,(n)}_\tau(x,\wt{x})\cdot Z^{s,m,(n)}_\tau(x)+\big(\p_\mu\p_x g\big)^{s,m,(n)}_\tau(x,\wt{x})\bigg)\p_m\big(X^{s,m,(n)}_\tau(\wt{x})\big)(y)dm(\wt{x})\\
&\ \ \ \ \ \ \ \ \ \ +\bigg(\big(\p_\mu\p_x f\big)^{s,m,(n)}_\tau(x,y)\cdot Z^{s,m,(n)}_\tau(x)+\big(\p_\mu\p_x g\big)^{s,m,(n)}_\tau(x,y)\bigg)\p_y\big(X^{s,m,(n)}_\tau(y)\big)\\
&\ \ \ \ \ \ \ \ \ \ +\bigg(\big(\p_\alpha\p_x f\big)^{s,m,(n)}_\tau(x)\cdot Z^{s,m,(n)}_\tau(x)+\big(\p_\alpha\p_x g\big)^{s,m,(n)}_\tau(x)\bigg)\p_m\big(\alpha^{s,m,(n)}_\tau(x)\big)(y)\bigg)d\tau\bigg\|_{\mc{L}(\R^{d_x};\R^{d_x})}\\
\leq &\  L_p\bigg(\left\|\p_y\big(X^{s,m,(n)}_T(y)\big)\right\|_{\mc{L}(\R^{d_x};\R^{d_x})}+\left\|\p_m\big(X^{s,m,(n)}_T(x)\big)(y)\right\|_{\mc{L}(\R^{d_x};\R^{d_x})}+\int_{\R^{d_x}}\left\|\p_m\big(X^{s,m,(n)}_T(\wt{x})\big)(y)\right\|_{\mc{L}(\R^{d_x};\R^{d_x})}dm(\wt{x})\bigg)\\
&+\big(2(2\wb{L}_p\wb{l}_f+\Lambda_g)+2(2\wb{L}_p\wb{l}_f+\wb{l}_g)L_\alpha(1+2L_p)+4\Lambda_fL_p\big)\int_s^T\left\|\p_y\big(X^{s,m,(n)}_\tau(y)\big)\right\|_{\mc{L}(\R^{d_x};\R^{d_x})}d\tau\\
&+\big(2(2\wb{L}_p\wb{l}_f+\Lambda_g)+(2\wb{L}_p\wb{l}_f+\wb{l}_g)L_\alpha(1+2L_p)+2\Lambda_fL_p\big)\int_s^T\left\|\p_m\big(X^{s,m,(n)}_\tau(x)\big)(y)\right\|_{\mc{L}(\R^{d_x};\R^{d_x})}d\tau\\
&+\big(3(2\wb{L}_p\wb{l}_f+\Lambda_g)+3(2\wb{L}_p\wb{l}_f+\wb{l}_g)L_\alpha(1+2L_p)+6\Lambda_fL_p\big)\int_s^T\int_{\R^{d_x}}\left\|\p_m\big(X^{s,m,(n)}_\tau(\wt{x})\big)(y)\right\|_{\mc{L}(\R^{d_x};\R^{d_x})}dm(\wt{x})d\tau\\
&\text{(by using \eqref{bdd_d1_f}, \eqref{bdd_d2_f}, \eqref{bdd_d2_g_1}, \eqref{bdd_d2_g_2}, \eqref{p_xalpha}, \eqref{p_zalpha}, \eqref{p1} and $\gamma^{(n-1)}\in \mathcal{I}_1$))}\\
\leq &\  \frac{6}{5}\Big(L_p+\big(2(2\wb{L}_p\wb{l}_f+\Lambda_g)+2(2\wb{L}_p\wb{l}_f+\wb{l}_g)L_\alpha(1+2L_p)+4\Lambda_fL_p\big)(T-s)\Big)\\
&+\frac{5}{16}\Big(L_p+\big(2(2\wb{L}_p\wb{l}_f+\Lambda_g)+(2\wb{L}_p\wb{l}_f+\wb{l}_g)L_\alpha(1+2L_p)+2\Lambda_fL_p\big)(T-s)\Big)\\
&+\frac{5}{16}\Big(L_p+\big(3(2\wb{L}_p\wb{l}_f+\Lambda_g)+3(2\wb{L}_p\wb{l}_f+\wb{l}_g)L_\alpha(1+2L_p)+6\Lambda_fL_p\big)(T-s)\Big)\text{  (by using \eqref{Bp_xX}, \eqref{Bp_mX} and \eqref{small_1})}\\
\leq &\ \frac{73}{40}\Big(L_p+\big(3(2\wb{L}_p\wb{l}_f+\Lambda_g)+3(2\wb{L}_p\wb{l}_f+\wb{l}_g)L_\alpha(1+2L_p)+6\Lambda_fL_p\big)(T-s)\Big)\\
\leq &\  2L_p,
\end{align*}\normalsize
where 
\small\begin{align*}
&\p_m\big(\alpha^{s,m,(n)}_\tau(x)\big)(y)=\big(\p_x  \alpha\big)^{s,m,(n)}_\tau(x)\p_m\big(X^{s,m,(n)}_\tau(x)\big)(y)+\big(\p_\mu \alpha\big)^{s,m,(n)}_\tau(x,y)\p_{y}\big(X^{s,m,(n)}_\tau(y)\big)\\
&\ \ \ \ \ \ \ \ \ \ \ \ \ \ \ \ \ \ \ \ \ \ \ \ \ \ \ \ +\displaystyle\int_{\R^{d_x}}\big(\p_\mu \alpha\big)^{s,m,(n)}_\tau(x,\wh{x})\p_m\big(X^{s,m,(n)}_\tau(\wh{x})\big)(y)dm(\wh{x})+\big(\p_z \alpha\big)^{s,m,(n)}_\tau(x)\p_m\big(Z^{s,m,(n)}_\tau(x)\big)(y),\\
&\p_m\big(Z^{s,m,(n)}_\tau(x)\big)(y)=\big(\p_x  \gamma^{(n-1)}\big)^{s,m,(n)}_\tau(x)\p_m\big(X^{s,m,(n)}_\tau(x)\big)(y)+\big(\p_\mu \gamma^{(n-1)}\big)^{s,m,(n)}_\tau(x,y)\p_{y}\big(X^{s,m,(n)}_\tau(y)\big)\\
&\ \ \ \ \ \ \ \ \ \ \ \ \ \ \ \ \ \ \ \ \ \ \ \ \ \ \ \ +\displaystyle\int_{\R^{d_x}}\big(\p_\mu \gamma^{(n-1)}\big)^{s,m,(n)}_\tau(x,\wh{x})\p_m\big(X^{s,m,(n)}_\tau(\wh{x})\big)(y)dm(\wh{x}),
\end{align*}
such that
\begin{align*}
&\big(\p_x\gamma^{(n-1)}\big)^{s,m,(n)}_\tau(x):=\p_x\gamma^{(n-1)}(s,x,\mu)\Big|_{(x,\mu)=(X^{s,m,(n)}_\tau(x),X^{s,m,(n)}_\tau\ot m)},\\
&\big(\p_\mu \gamma^{(n-1)}\big)^{s,m,(n)}_\tau(x,y):=\p_\mu\gamma(s,x,\mu)\big(\wt{x}\big)\Big|_{(x,\mu,\wt{x})=(X^{s,m,(n)}_\tau(x),X^{s,m,(n)}_\tau\ot m,X^{s,m,(n)}_\tau(y))}.
\end{align*}\normalsize
Therefore, we have $\vertiii{\gamma^{(n)}}_2\leq 2L_p$.

\textbf{Step 5} (showing the iteration mapping $\gamma^{(n-1)}\in\mc{I}_1\mapsto \gamma^{(n)}\in\mc{I}_1$ is contractive under the norm $\vertiii{\cdot}_1$). Recall the notations introduced in \eqref{notation_1_1} and \eqref{notation_2_1}. For any $t\in[0,T]$ such that\small
\begin{align}\no
T-t\leq \min\bigg\{&\eps_4,\frac{1}{2\big(8L_p L_fL_\alpha+5\big(\Lambda_f+\big(2\wb{L}_p\wb{l}_f+\max\{\Lambda_g,\wb{l}_g\}\big)L_\alpha\big)\big)},\\\label{eps_1}
&\frac{1}{2\sqrt{L_fL_\alpha}\sqrt{34L_p\Lambda_f+\big(34L_pL_\alpha+21+17L_\alpha\big)\big(2\wb{L}_p\wb{l}_f+\max\{\Lambda_g,\wb{l}_g\}\big)}}\bigg\}=:\eps_1,
\end{align}\normalsize
and, for any $\gamma^{(n-1)},\,\gamma^{(n)}\in \mathcal{I}_1$, $s\in[t,T]$, $\tau\in[s,T]$, $m\in\mc{P}_2(\R^{d_x})$ and $x\in\R^{d_x}$, we have, 
\begin{align}\no
&\Big|Z^{s,m,(n+1)}_\tau(x)-Z^{s,m,(n)}_\tau(x)\Big|\\\no
=&\ \Big|\gamma^{(n)}\big(\tau,X^{s,m,(n+1)}_\tau(x),X^{s,m,(n+1)}_\tau\ot m\big)-\gamma^{(n-1)}\big(\tau,X^{s,m,(n)}_\tau(x),X^{s,m,(n)}_\tau\ot m\big)\Big|\\\no
\leq&\ \Big|\gamma^{(n)}\big(\tau,X^{s,m,(n+1)}_\tau(x),X^{s,m,(n+1)}_\tau\ot m\big)-\gamma^{(n)}\big(\tau,X^{s,m,(n)}_\tau(x),X^{s,m,(n)}_\tau\ot m\big)\Big|\\\no
&+\Big|\gamma^{(n)}\big(\tau,X^{s,m,(n)}_\tau(x),X^{s,m,(n)}_\tau\ot m\big)-\gamma^{(n-1)}\big(\tau,X^{s,m,(n)}_\tau(x),X^{s,m,(n)}_\tau\ot m\big)\Big|\\\no
\leq&\  2L_p\bigg(\Big| X^{s,m,(n+1)}_\tau(x) -X^{s,m,(n)}_\tau(x)\Big|+\Big\| X^{s,m,(n+1)}_\tau -X^{s,m,(n)}_\tau\Big\|_{L^{1,d_x}_m}\bigg)\\\label{6_35}
&+\vertiii{\gamma^{(n)}-\gamma^{(n-1)}}_1\bigg(1+\Big| X^{s,m,(n)}_\tau(x)\Big|+\Big\| X^{s,m,(n)}_\tau\Big\|_{L^{1,d_x}_m}\bigg)\\\no
&\text{(by using the definition of $\vertiii{\cdot}_1$ and $\vertiii{\cdot}_2$ in \eqref{Gamma_1} and \eqref{Gamma_2} respectively and the very definition of $\mathcal{I}_1$)};
\end{align}\normalsize
for simplicity of notations, denote $\Theta(\theta):=\Big(\theta X^{s,m,(n+1)}_\tau(x)+(1-\theta)X^{s,m,(n)}_\tau(x),\theta \big(X^{s,m,(n+1)}_\tau\ot m\big)+(1-\theta)\big(X^{s,m,(n)}_\tau\ot m\big)\Big)$ and $\Psi_\beta(y):=\beta X^{s,m,(n+1)}_\tau(y)+(1-\beta)X^{s,m,(n)}_\tau(y)$, and then 
\footnotesize\begin{align}\no
&\Bigg|\big(\p_x f\big)^{s,m,(n+1)}_\tau(x)\cdot Z^{s,m,(n+1)}_\tau(x)-\big(\p_x f\big)^{s,m,(n)}_\tau(x)\cdot Z^{s,m,(n)}_\tau(x)\Bigg|\\\no
=&\ \Bigg|\p_x f\Big(\Theta(1),\alpha\big(\Theta(1),\gamma^{(n)}(\tau,\Theta(1))\big)\Big)\cdot \gamma^{(n)}(\tau,\Theta(1))-\p_x f\Big(\Theta(0),\alpha\big(\Theta(0),\gamma^{(n)}(\tau,\Theta(0))\big)\Big)\cdot \gamma^{(n)}(\tau,\Theta(0))\\\no
&+ \p_x f\Big(\Theta(0),\alpha\big(\Theta(0),\gamma^{(n)}(\tau,\Theta(0))\big)\Big)\cdot \gamma^{(n)}(\tau,\Theta(0))-\p_x f\Big(\Theta(0),\alpha\big(\Theta(0),\gamma^{(n-1)}(\tau,\Theta(0))\big)\Big)\cdot \gamma^{(n-1)}(\tau,\Theta(0))\Bigg|\\\no
=&\ \Bigg|\int_0^1 \Bigg(\bigg(\p_x\p_x f\Big(\Theta(\theta),\alpha\big(\Theta(\theta),\gamma^{(n)}(\tau,\Theta(\theta))\big)\Big)\cdot \gamma^{(n)}(\tau,\Theta(\theta))\bigg)\cdot \Big(X^{s,m,(n+1)}_\tau(x)-X^{s,m,(n)}_\tau(x)\Big)\\\no
&\ \ \ \ \ \ \ \ +\int_{\R^{d_x}}\bigg(\int_0^1\p_\mu\p_x f\Big(\Theta(\theta),\alpha\big(\Theta(\theta),\gamma^{(n)}(\tau,\Theta(\theta))\big)\Big)(\Psi_\beta(y))d\beta\cdot \gamma^{(n)}(\tau,\Theta(\theta))\bigg)\cdot \Big(X^{s,m,(n+1)}_\tau(y)-X^{s,m,(n)}_\tau(y)\Big)dm(y)\\\no
&\ \ \ \ \ \ \ \ +\bigg(\p_\alpha\p_x f\Big(\Theta(\theta),\alpha\big(\Theta(\theta),\gamma^{(n)}(\tau,\Theta(\theta))\big)\Big)\cdot \gamma^{(n)}(\tau,\Theta(\theta))\bigg)\\\no
&\ \ \ \ \ \ \ \ \ \ \ \ \cdot \bigg(\p_x\alpha\big(\Theta(\theta),\gamma^{(n)}(\tau,\Theta(\theta))\big)\cdot\Big(X^{s,m,(n+1)}_\tau(x)-X^{s,m,(n)}_\tau(x)\Big)\\\no
&\ \ \ \ \ \ \ \ \ \ \ \ \ \ \ \ +\int_{\R^{d_x}}\int_0^1\p_\mu\alpha\big(\Theta(\theta),\gamma^{(n)}(\tau,\Theta(\theta))\big)(\Psi_\beta(y))d\beta\cdot\Big(X^{s,m,(n+1)}_\tau(y)-X^{s,m,(n)}_\tau(y)\Big)dm(y)\\\no
&\ \ \ \ \ \ \ \ \ \ \ \ \ \ \ \ +\p_z\alpha\big(\Theta(\theta),\gamma^{(n)}(\tau,\Theta(\theta))\big)\cdot\p_x \gamma^{(n)}(\tau,\Theta(\theta))\cdot\Big(X^{s,m,(n+1)}_\tau(x)-X^{s,m,(n)}_\tau(x)\Big)\\\no
&\ \ \ \ \ \ \ \ \ \ \ \ \ \ \ \ +\p_z\alpha\big(\Theta(\theta),\gamma^{(n)}(\tau,\Theta(\theta))\big)\cdot\int_{\R^{d_x}}\int_0^1\p_\mu \gamma^{(n)}(\tau,\Theta(\theta))(\Psi_\beta(y))d\beta\cdot\Big(X^{s,m,(n+1)}_\tau(y)-X^{s,m,(n)}_\tau(y)\Big)dm(y)\bigg)\\\no
&\ \ \ \ \ \ \ \ +\p_x f\Big(\Theta(\theta),\alpha\big(\Theta(\theta),\gamma^{(n)}(\tau,\Theta(\theta))\big)\Big)\cdot\p_x \gamma^{(n)}(\tau,\Theta(\theta))\cdot\Big(X^{s,m,(n+1)}_\tau(x)-X^{s,m,(n)}_\tau(x)\Big)\\\no
&\ \ \ \ \ \ \ \ +\p_x f\Big(\Theta(\theta),\alpha\big(\Theta(\theta),\gamma^{(n)}(\tau,\Theta(\theta))\big)\Big)\cdot\int_{\R^{d_x}}\int_0^1\p_\mu \gamma^{(n)}(\tau,\Theta(\theta))(\Psi_\beta(y))d\beta\cdot\Big(X^{s,m,(n+1)}_\tau(y)-X^{s,m,(n)}_\tau(y)\Big)dm(y)\Bigg)d\theta\\\no
&+\int_0^1 \Bigg(\bigg(\p_\alpha\p_x f\Big(\Theta(0),\alpha\big(\Theta(0),\theta \gamma^{(n)}(\tau,\Theta(0))+(1-\theta)\gamma^{(n-1)}(\tau,\Theta(0))\big)\Big)\cdot \Big(\theta \gamma^{(n)}(\tau,\Theta(0))+(1-\theta)\gamma^{(n-1)}(\tau,\Theta(0))\Big)\bigg)\\\no
&\ \ \ \ \ \ \ \ \ \ \cdot \p_z \alpha\Big(\Theta(0),\theta \gamma^{(n)}\big(\tau,\Theta(0)\big)+(1-\theta)\gamma^{(n-1)}\big(\tau,\Theta(0)\big)\Big)\cdot \Big(\gamma^{(n)}\big(\tau,\Theta(0)\big)-\gamma^{(n-1)}\big(\tau,\Theta(0)\big)\Big)\\\no
&\ \ \ \ \ \ \ \ +\p_x f\Big(\Theta(0),\alpha\big(\Theta(0),\theta \gamma^{(n)}(\tau,\Theta(0))+(1-\theta)\gamma^{(n-1)}(\tau,\Theta(0))\big)\Big)\cdot \Big(\gamma^{(n)}\big(\tau,\Theta(0)\big)-\gamma^{(n-1)}\big(\tau,\Theta(0)\big)\Big)\Bigg)d\theta\Bigg|
\\\no
\leq&\ \Big(2\wb{L}_p\wb{l}_f(1+L_\alpha+2L_pL_\alpha)+2L_p\Lambda_f\Big)\cdot\Big(\left|  X^{s,m,(n+1)}_\tau(x) -X^{s,m,(n)}_\tau(x)\right|+\left\|  X^{s,m,(n+1)}_\tau -X^{s,m,(n)}_\tau\right\|_{L^{1,d_x}_m}\Big)\\\label{eq_7_54_new}
&+\Big(2\wb{L}_p\wb{l}_fL_\alpha+\Lambda_f\Big)\cdot\vertiii{\gamma^{(n)}-\gamma^{(n-1)}}_1\bigg(1+\Big| X^{s,m,(n)}_\tau(x)\Big|+\Big\| X^{s,m,(n)}_\tau\Big\|_{L^{1,d_x}_m}\bigg)\\\no
&\text{ (by using \eqref{bdd_d1_f}, \eqref{bdd_d2_f}, \eqref{p_xalpha}, \eqref{p_zalpha} and $\gamma^{(n-1)}, \gamma^{(n)}\in \mathcal{I}_1$)};
\end{align}\normalsize
similarly, we also have
\small\begin{align}\no
&\bigg|\big(\p_\mu f\big)^{s,m,(n+1)}_\tau(\wt{x},x)\cdot Z^{s,m,(n+1)}_\tau(\wt{x})-\big(\p_\mu f\big)^{s,m,(n)}_\tau(\wt{x},x)\cdot Z^{s,m,(n)}_\tau(\wt{x})\bigg|\\\no
\leq&\ \Big(2\wb{L}_p\wb{l}_f(1+L_\alpha+2L_pL_\alpha)+2L_p\Lambda_f\Big)\cdot\Big(\left|  X^{s,m,(n+1)}_\tau(\wt{x}) -X^{s,m,(n)}_\tau(\wt{x})\right|+\left\|  X^{s,m,(n+1)}_\tau -X^{s,m,(n)}_\tau\right\|_{L^{1,d_x}_m}\Big)\\\no
&+\Big(2\wb{L}_p\wb{l}_fL_\alpha+\Lambda_f\Big)\cdot\vertiii{\gamma^{(n)}-\gamma^{(n-1)}}_1\bigg(1+\Big| X^{s,m,(n)}_\tau(\wt{x})\Big|+\Big\| X^{s,m,(n)}_\tau\Big\|_{L^{1,d_x}_m}\bigg)\\\label{eq_7_55_new}
&+2\wb{L}_p\wb{l}_f\left|  X^{s,m,(n+1)}_\tau(x) -X^{s,m,(n)}_\tau(x)\right|;
\end{align}\normalsize
furthermore, we have,  
\small\begin{align}\no
&\Big| X^{t,m,(n+1)}_s(x) -X^{t,m,(n)}_s(x)\Big|\\\no
\leq&\   L_B\int_t^s \bigg(\Big| X^{t,m,(n+1)}_\tau(x) -X^{t,m,(n)}_\tau(x)\Big|+\left\|X^{t,m,(n+1)}_\tau-X^{t,m,(n)}_\tau\right\|_{L^{1,d_x}_m}\bigg)d\tau\\\label{eq_6_44}
&+L_fL_\alpha\vertiii{\gamma^{(n)}-\gamma^{(n-1)}}_1\bigg(1+\sup_{\tau\in[t,s]}\Big| X^{t,m,(n)}_\tau(x)\Big|+\sup_{\tau\in[t,s]}\Big\| X^{t,m,(n)}_\tau\Big\|_{L^{1,d_x}_m}\bigg)(s-t)\\\no
&\text{ (by using \eqref{Lipf}, \eqref{lipalpha}, \eqref{fodesystem} and \eqref{6_35})}.
\end{align}\normalsize
A direct integration with respect to $x$ yields
\begin{align*}
\left\|X^{t,m,(n+1)}_s-X^{t,m,(n)}_s\right\|_{L^{1,d_x}_m}
\leq&\   2L_B\int_t^s \left\|X^{t,m,(n+1)}_\tau-X^{t,m,(n)}_\tau\right\|_{L^{1,d_x}_m}d\tau\\
&+L_fL_\alpha\vertiii{\gamma^{(n)}-\gamma^{(n-1)}}_1\bigg(1+2\sup_{\tau\in[t,s]}\Big\| X^{t,m,(n)}_\tau\Big\|_{L^{1,d_x}_m}\bigg)(s-t),
\end{align*}
and hence, by Gr\"{o}nwall's inequality,
\begin{align}\no
\left\|X^{t,m,(n+1)}_s-X^{t,m,(n)}_s\right\|_{L^{1,d_x}_m}\leq&\  L_fL_\alpha\vertiii{\gamma^{(n)}-\gamma^{(n-1)}}_1\Big(1+2\sup_{\tau\in[t,s]}\Big\| X^{t,m,(n)}_\tau\Big\|_{L^{1,d_x}_m}\Big)(s-t)\exp\big(2L_B(s-t)\big)\\\label{eq_6_35}
\leq &\  \frac{3}{2}\Big(\frac{3}{2}+3\|m\|_1\Big)L_fL_\alpha(s-t)\vertiii{\gamma^{(n)}-\gamma^{(n-1)}}_1,
\end{align}
by using \eqref{eq_6_14_1} and \eqref{small_1} in the last inequality.
Substituting \eqref{eq_6_35} into \eqref{eq_6_44}, we obtain
\begin{align}\no
&\Big| X^{t,m,(n+1)}_s(x) -X^{t,m,(n)}_s(x)\Big|\\\no
\leq&\   L_B\int_t^s \Big| X^{t,m,(n+1)}_\tau(x) -X^{t,m,(n)}_\tau(x)\Big|d\tau\\\no
&+L_fL_\alpha\vertiii{\gamma^{(n)}-\gamma^{(n-1)}}_1\bigg(1+\sup_{\tau\in[t,s]}\Big| X^{t,m,(n)}_\tau(x)\Big|+\sup_{\tau\in[t,s]}\Big\| X^{t,m,(n)}_\tau\Big\|_{L^{1,d_x}_m}\\\no
&\ \ \ \ \ \ \ \ \ \ \ \ \ \ \ \ \ \ \ \ \ \ \ \ \ \ \ \ \ \ \ \ \ \ \ \ +L_B\Big(1+2\sup_{\tau\in[t,s]}\Big\| X^{t,m,(n)}_\tau\Big\|_{L^{1,d_x}_m}\Big)(s-t)\exp\big(2L_B(s-t)\big)\bigg)(s-t).
\end{align}
Again, by Gr\"{o}nwall's inequality,
\begin{align}\no
\Big| X^{t,m,(n+1)}_s(x) -X^{t,m,(n)}_s(x)\Big|
\leq&\  \bigg(1+\sup_{\tau\in[t,s]}\Big| X^{t,m,(n)}_\tau(x)\Big|+\sup_{\tau\in[t,s]}\Big\| X^{t,m,(n)}_\tau\Big\|_{L^{1,d_x}_m}\\\no
&\ \ +L_B\Big(1+2\sup_{\tau\in[t,s]}\Big\| X^{t,m,(n)}_\tau\Big\|_{L^{1,d_x}_m}\Big)(s-t)\exp\big(2L_B(s-t)\big)\bigg)(s-t)\\\no
&\cdot L_fL_\alpha\vertiii{\gamma^{(n)}-\gamma^{(n-1)}}_1\exp\big(L_B(s-t)\big)\\\no
\leq &\ \frac{6}{5}\bigg(1+\frac{6}{5}\Big(|x|+\frac{5}{24}+\frac{1}{4}\|m\|_1\Big)+\frac{3}{2}\Big(\|m\|_1+\frac{1}{6}\Big)+\frac{1}{4}\Big(\frac{3}{2}+3\|m\|_1\Big)\bigg)(s-t)\\\no
&\cdot L_fL_\alpha\vertiii{\gamma^{(n)}-\gamma^{(n-1)}}_1\\\label{eq_6_47}
=&\ \frac{6}{5}\bigg(\frac{15}{8}+\frac{6}{5}|x|+\frac{51}{20}\|m\|_1\bigg)L_fL_\alpha(s-t)\vertiii{\gamma^{(n)}-\gamma^{(n-1)}}_1,
\end{align}
by using \eqref{eq_6_14_1}, \eqref{small_1} and the assumption $L_p\leq \wb{L}_p$ in the last inequality.
Therefore, we have, for $t\leq s\leq T$, $m\in \mc{P}_2(\R^{d_x})$ and $x\in\R^{d_x}$,
\footnotesize\begin{align}\no
&\left|\gamma^{(n+1)} (s,x,m)-\gamma^{(n)} (s,x,m)\right|\\\no
\leq &\ \left|p^{s,m,(n+1)}_T-p^{s,m,(n)}_T\right|\\\no
&+ \int_s^T \bigg| \int_{\R^{d_x}}\bigg(\big(\p_\mu f\big)^{s,m,(n+1)}_\tau(\wt{x},x)\cdot Z^{s,m,(n+1)}_\tau(\wt{x})-\big(\p_\mu f\big)^{s,m,(n)}_\tau(\wt{x},x)\cdot Z^{s,m,(n)}_\tau(\wt{x})\\\no
&\ \ \ \ \ \ \ \ \ \ \ \ \ \ \ \ \ \ \ \ +\big(\p_\mu g\big)^{s,m,(n+1)}_\tau(\wt{x},x)-\big(\p_\mu g\big)^{s,m,(n)}_\tau(\wt{x},x)\bigg)dm(\wt{x})\\\no
&\ \ \ \ \ \ \ \ \ \ \ +\big(\p_x f\big)^{s,m,(n+1)}_\tau(x)\cdot Z^{s,m,(n+1)}_\tau(x)-\big(\p_x f\big)^{s,m,(n)}_\tau(x)\cdot Z^{s,m,(n)}_\tau(x)\\\no
&\ \ \ \ \ \ \ \ \ \ \ +\big(\p_x g\big)^{s,m,(n+1)}_\tau(x)-\big(\p_x g\big)^{s,m,(n)}_\tau(x)\bigg|d\tau\\\no
\leq&\  L_p \Big(\left| X^{s,m,(n+1)}_T(x)-X^{s,m,(n)}_T(x) \right|+\left\| X^{s,m,(n+1)}_T-X^{s,m,(n)}_T \right\|_{L^{1,d_x}_m}\Big)\\\no
&+\int_s^T \Bigg(\Big(\big(2\wb{L}_p\wb{l}_f+\max\{\Lambda_g,\wb{l}_g\}\big)(1+L_\alpha+2L_pL_\alpha)+2L_p\Lambda_f\Big)\cdot\Big(\left|  X^{s,m,(n+1)}_\tau(x) -X^{s,m,(n)}_\tau(x)\right|+3\left\|  X^{s,m,(n+1)}_\tau -X^{s,m,(n)}_\tau\right\|_{L^{1,d_x}_m}\Big)\\\no
&+\Big(\big(2\wb{L}_p\wb{l}_f+\max\{\Lambda_g,\wb{l}_g\}\big)L_\alpha+\Lambda_f\Big)\cdot\vertiii{\gamma^{(n)}-\gamma^{(n-1)}}_1\bigg(2+\Big| X^{s,m,(n)}_\tau(x)\Big|+3\Big\| X^{s,m,(n)}_\tau\Big\|_{L^{1,d_x}_m}\bigg)\\\no
& +\big(2\wb{L}_p\wb{l}_f+\max\{\Lambda_g,\wb{l}_g\}\big)\left|  X^{s,m,(n+1)}_\tau(x) -X^{s,m,(n)}_\tau(x)\right|\Bigg)d\tau\\\no
&\text{ (by using \eqref{bdd_d2_g_1}, \eqref{bdd_d2_g_2}, \eqref{p_xalpha}, \eqref{p_zalpha}, \eqref{lipp}, \eqref{eq_7_54_new}, \eqref{eq_7_55_new}, and $\gamma^{(n-1)},\, \gamma^{(n)}\in \mathcal{I}_1$)}\\\no
=&\  L_p \Big(\left| X^{s,m,(n+1)}_T(x)-X^{s,m,(n)}_T(x) \right|+\left\| X^{s,m,(n+1)}_T-X^{s,m,(n)}_T \right\|_{L^{1,d_x}_m}\Big)\\\no
& +\big(\Lambda_f+\big(2\wb{L}_p\wb{l}_f+\max\{\Lambda_g,\wb{l}_g\}\big)L_\alpha\big) \int_s^T   \Bigg(2L_p\bigg(\Big| X^{s,m,(n+1)}_\tau(x) -X^{s,m,(n)}_\tau(x)\Big|+3\Big\| X^{s,m,(n+1)}_\tau -X^{s,m,(n)}_\tau\Big\|_{L^{1,d_x}_m}\bigg)\\\no
&\ \ \ \ \ \ \ \ \ \ \ \ \ \ \ \ \ \ \ \ \ \ \ \ \ \ \ \ \ \ \ \ \ \ \ \ \ \ \ \ \ \ \ \ \ \ \ \ \ \ \ \ \ \ \ \ \ \ \ \ +\vertiii{\gamma^{(n)}-\gamma^{(n-1)}}_1\bigg(2+\Big| X^{s,m,(n)}_\tau(x)\Big|+3\Big\| X^{s,m,(n)}_\tau\Big\|_{L^{1,d_x}_m}\bigg)\Bigg)d\tau\\\no
&+\big(2\wb{L}_p\wb{l}_f+\max\{\Lambda_g,\wb{l}_g\}\big)\int_s^T \bigg((1+L_\alpha)\Big(\left|  X^{s,m,(n+1)}_\tau(x) -X^{s,m,(n)}_\tau(x)\right|+3\left\|  X^{s,m,(n+1)}_\tau -X^{s,m,(n)}_\tau\right\|_{L^{1,d_x}_m}\Big)\\\no
&\ \ \ \ \ \ \ \ \ \ \ \ \ \ \ \ \ \ \ \ \ \ \ \ \ \ \ \ \ \ \ \ \ \ \ \ \ \ \ \ \ +\left|  X^{s,m,(n+1)}_\tau(x) -X^{s,m,(n)}_\tau(x)\right|\bigg)d\tau
\\\no
\leq&\  L_p \Bigg(\frac{6}{5}\bigg(\frac{15}{8}+\frac{6}{5}|x|+\frac{51}{20}\|m\|_1\bigg)+\frac{3}{2}\Big(\frac{3}{2}+3\|m\|_1\Big)\Bigg)L_fL_\alpha(T-t)\vertiii{\gamma^{(n)}-\gamma^{(n-1)}}_1\\\no
& +2L_p\big(\Lambda_f+\big(2\wb{L}_p\wb{l}_f+\max\{\Lambda_g,\wb{l}_g\}\big)L_\alpha\big)L_fL_\alpha (T-t)^2\Bigg(\frac{6}{5}\bigg(\frac{15}{8}+\frac{6}{5}|x|+\frac{51}{20}\|m\|_1\bigg)+\frac{9}{2}\Big(\frac{3}{2}+3\|m\|_1\Big)\Bigg)\vertiii{\gamma^{(n)}-\gamma^{(n-1)}}_1\\\no
&+\big(\Lambda_f+\big(2\wb{L}_p\wb{l}_f+\max\{\Lambda_g,\wb{l}_g\}\big)L_\alpha\big)(T-t)\vertiii{\gamma^{(n)}-\gamma^{(n-1)}}_1\bigg(2+\frac{6}{5}\Big(|x|+\frac{5}{24}+\frac{1}{4}\|m\|_1\Big)+\frac{9}{2}\Big(\|m\|_1+\frac{1}{6}\Big)\bigg)\\\no
&+(1+L_\alpha)\big(2\wb{L}_p\wb{l}_f+\max\{\Lambda_g,\wb{l}_g\}\big)L_fL_\alpha(T-t)^2\Bigg(\frac{6}{5}\bigg(\frac{15}{8}+\frac{6}{5}|x|+\frac{51}{20}\|m\|_1\bigg)+\frac{9}{2}\Big(\frac{3}{2}+3\|m\|_1\Big)\Bigg)\vertiii{\gamma^{(n)}-\gamma^{(n-1)}}_1\\\no
&+\big(2\wb{L}_p\wb{l}_f+\max\{\Lambda_g,\wb{l}_g\}\big)L_fL_\alpha(T-t)^2\frac{6}{5}\bigg(\frac{15}{8}+\frac{6}{5}|x|+\frac{51}{20}\|m\|_1\bigg)\vertiii{\gamma^{(n)}-\gamma^{(n-1)}}_1\\\no
&\text{ (by using \eqref{eq_6_14_1}, \eqref{small_1}, \eqref{eq_6_35} and \eqref{eq_6_47})}
\\\no
\leq&\  8L_p L_fL_\alpha(T-t)(1+|x|+\|m\|_1)\vertiii{\gamma^{(n)}-\gamma^{(n-1)}}_1\\\no
& +34L_p\big(\Lambda_f+\big(2\wb{L}_p\wb{l}_f+\max\{\Lambda_g,\wb{l}_g\}\big)L_\alpha\big)L_fL_\alpha (T-t)^2(1+|x|+\|m\|_1)\vertiii{\gamma^{(n)}-\gamma^{(n-1)}}_1\\\no
&+5\big(\Lambda_f+\big(2\wb{L}_p\wb{l}_f+\max\{\Lambda_g,\wb{l}_g\}\big)L_\alpha\big)(T-t)\vertiii{\gamma^{(n)}-\gamma^{(n-1)}}_1(1+|x|+\|m\|_1)\\\no
&+17(1+L_\alpha)\big(2\wb{L}_p\wb{l}_f+\max\{\Lambda_g,\wb{l}_g\}\big)L_fL_\alpha(T-t)^2(1+|x|+\|m\|_1)\vertiii{\gamma^{(n)}-\gamma^{(n-1)}}_1\\\no
&+4\big(2\wb{L}_p\wb{l}_f+\max\{\Lambda_g,\wb{l}_g\}\big)L_fL_\alpha(T-t)^2(1+|x|+\|m\|_1)\vertiii{\gamma^{(n)}-\gamma^{(n-1)}}_1\\\no
\leq &\ \frac{3}{4}(1+|x|+\|m\|_1)\vertiii{\gamma^{(n)}-\gamma^{(n-1)}}_1,
\end{align}\normalsize 
by using \eqref{eps_1} in the last inequality,
which implies that $\vertiii{\gamma^{(n+1)}-\gamma^{(n)}}_1\leq \frac{3}{4} \vertiii{\gamma^{(n)}-\gamma^{(n-1)}}_1$, that is, the iteration mapping $\gamma^{(n-1)}\in\mc{I}_1\mapsto \gamma^{(n)}\in\mc{I}_1$ is contractive under the norm $\vertiii{\cdot}_1$. 

\textbf{Step 6} For any $t\in[0,T]$ with $T-t\leq \eps_1$, where $\eps_1$ was defined in \eqref{eps_1}, it follows from Steps 1-5 that the iteration mapping $\gamma^{(n-1)}\in\mc{I}_1\mapsto \gamma^{(n)}\in\mc{I}_1$ is well-defined for all $n\in\mathbb{N}$ and contractive under the norm $\vertiii{\cdot}_1$, thus $\{\gamma^{(n)}\}_{n=0}^{\infty}$ is a Cauchy sequence in the Banach space $C([t,T]\times \R^{d_x}\times\mathcal{P}_2(\R^{d_x});\R^{d_x})$ equipped with the norm $\vertiii{\cdot}_1$, which warrants that there exists a unique $\gamma \in C([t,T]\times \R^{d_x}\times\mathcal{P}_2(\R^{d_x});\R^{d_x})$ (but may not in $\mc{I}_1$) satisfying $\vertiii{\gamma}_1\leq 2\wb{L}_p$, $\vertiii{\gamma^{(n)}-\gamma}_1\to 0$ as $n\to\infty$ and for any $(s,x,m)\in [t,T]\times \R^{d_x}\times\mathcal{P}_2(\R^{d_x})$,
\small\begin{align}\no
\gamma (s,x,m)
= &\ p(X^{s,m}_T(x),X^{s,m}_T\ot m)\\\no
&+ \int_s^T \Bigg(\int_{\R^{d_x}}\bigg(\p_\mu f\left(X^{s,m }_\tau(\wt{x}),X^{s,m }_\tau\ot m,\alpha\left(X^{s,m }_\tau(\wt{x}),X^{s,m }_\tau\ot m, Z^{s,m }_\tau(\wt{x})\right)\right)(X^{s,m }_\tau(x))\cdot  Z^{s,m }_\tau(\wt{x})\\\no
&\ \ \ \ \ \ \ \ \ \ \ \ \ \ \ \ \ \ \ \ \ \ + \p_\mu g\left(X^{s,m }_\tau(\wt{x}),X^{s,m }_\tau\ot m,\alpha\left(X^{s,m }_\tau(\wt{x}),X^{s,m }_\tau\ot m, Z^{s,m }_\tau(\wt{x}))\right)\right)(X^{s,m }_\tau(x))\bigg)dm(\wt{x})\\\no
&\ \ \ \ \ \ \ \ \ \ \ \ \ \ +\p_x f(X^{s,m }_\tau(x),X^{s,m }_\tau\ot m,\alpha(X^{s,m }_\tau(x),X^{s,m }_\tau\ot m,Z^{s,m }_\tau(x)))\cdot Z^{s,m }_\tau(x)\\\no
&\ \ \ \ \ \ \ \ \ \ \ \ \ \ +\p_x g(X^{s,m }_\tau(x),X^{s,m }_\tau\ot m,\alpha(X^{s,m }_\tau(x),X^{s,m }_\tau\ot m,Z^{s,m }_\tau(x)))\Bigg)d\tau,
\end{align}\normalsize
where $Z^{s,m }_\tau(x):=\gamma\big(\tau,X^{s,m }_\tau(x),X^{s,m }_\tau\ot m\big)$ and
\small\begin{align}\no
X^{s,m}_\tau(x) = x+\int_s^\tau f\Big(X^{s,m }_{\wt{\tau}}(x),X^{s,m }_{\wt{\tau}}\ot m,\alpha\big(X^{s,m }_{\wt{\tau}}(x),X^{s,m }_{\wt{\tau}}\ot m,\gamma(\tau,X^{s,m }_{\wt{\tau}}(x),X^{s,m }_{\wt{\tau}}\ot m)\big)\Big)d\wt{\tau},
\end{align}\normalsize
which means that $\gamma$ is a decoupling field for the FBODE system \eqref{fbodesystem} on $[t,T]$ in the sense of Definition \ref{decouple}.
By using \eqref{coro_6_4}, it can be checked that the couple $\big(X^{t,m}_s(x),Z^{t,m}_s(x)\big)_{x\in\R^{d_x},s\in[t,T]}$ solves the FBODE system \eqref{fbodesystem} and the estimate \eqref{cone_ZX} is valid since $\vertiii{\gamma}_1\leq 2\wb{L}_p$.

\begin{remark}
We emphasize here that we cannot obtain the contractive property of the iteration mapping $\gamma^{(n-1)}\in\mc{I}_1\mapsto \gamma^{(n)}\in\mc{I}_1$ under the norm $\vertiii{\cdot}_2$ due to the term $Z^{s,m,(n)}_\tau(x)=\gamma^{(n-1)}\big(\tau,X^{s,m,(n)}_\tau(x),X^{s,m,(n)}_\tau\ot m\big)$ in \eqref{bodesystem}. In addition, since the Banach space $C^1([t,T]\times \R^{d_x}\times\mathcal{P}_2(\R^{d_x});\R^{d_x})$ equipped with the norm $\vertiii{\cdot}_1+\vertiii{\cdot}_2$ is not a dual space of any Banach space, we cannot extract a weak or weak* convergent subsequence from $\{\gamma^{(n)}\}_{n=0}^\infty\subset \mc{I}_1$; thus, we cannot obtain $\gamma\in\mc{I}_1$ by the uniqueness of weak or weak* limit. Therefore, we have to show the differentiability of $\gamma(s,x,\mu)$ with respect to $x$ and $\mu$ separately, in Theorem \ref{Thm6_2}.
\end{remark}

\subsection{Proof of Regularity Theorem \ref{Thm6_2}}\label{subsec:7_3}
Recall the notations defined in Appendix \eqref{notation_1} and Table \ref{notation_2}, and we can rewrite \eqref{gammaeq} and \eqref{xeq} as follows:
\small\begin{align}\no
\gamma (s,x,m)
= &\ p^{s,m}_T(x)+ \int_s^T \Bigg(\int_{\R^{d_x}}\bigg(\big(\p_\mu f\big)^{s,m }_\tau(\wt{x},x)\cdot  Z^{s,m }_\tau(\wt{x}) + \big(\p_\mu g\big)^{s,m }_\tau(\wt{x},x)\bigg)dm(\wt{x})\\\label{gammaeq_1}
&\ \ \ \ \ \ \ \ \ \ \ \ \ \ \ \ \ \ \ \ \ \ \ \ \ \ \ \ \ \ \ +\big(\p_x f\big)^{s,m }_\tau(x)\cdot Z^{s,m }_\tau(x)+\big(\p_x g\big)^{s,m }_\tau(x)\Bigg)d\tau,
\end{align}\normalsize
where $Z^{s,m }_\tau(x)=\gamma\big(\tau,X^{s,m }_\tau(x),X^{s,m }_\tau\ot m\big)$ and
\small\begin{align}
\label{xeq_1}       
X^{t,m }_s(x) = x+\int_t^s f^{t,m }_\tau(x) d\tau.
\end{align}\normalsize
In addition, if the derivatives of $\gamma (s,x,m)$ with respect to $x$ and $m$ exist, they should satisfy the following systems:
\small\begin{align}\no
&\p_x\gamma (s,x,m)\\\no
= &\ \big(\p_x p\big)^{s,m}_T(x)\p_x\big(X^{s,m}_T(x)\big)\\\no
&+ \int_s^T \Bigg(\int_{\R^{d_x}}\bigg(\big(\p_{\wt{x}}\p_\mu f\big)^{s,m }_\tau(\wt{x},x)\cdot  Z^{s,m }_\tau(\wt{x}) + \big(\p_{\wt{x}}\p_\mu g\big)^{s,m }_\tau(\wt{x},x)\bigg)dm(\wt{x})\\\no
&\ \ \ \ \ \ \ \ \ \ \ \ +\Big(\big(\p_\alpha \p_x f\big)^{s,m }_\tau(x)\cdot Z^{s,m }_\tau(x)+\big(\p_\alpha\p_x g\big)^{s,m }_\tau(x)\Big)\Big(\big(\p_x\alpha\big)^{s,m }_\tau(x)+\big(\p_z\alpha\big)^{s,m }_\tau(x)\big(\p_x\gamma\big)^{s,m }_\tau(x)\Big)\\\label{p_xgamma_eq_1}   
&\ \ \ \ \ \ \ \ \ \ \ \ +\big(\p_x f\big)^{s,m }_\tau(x)\cdot \big(\p_x\gamma\big)^{s,m }_\tau(x)+\big(\p_x\p_x f\big)^{s,m }_\tau(x)\cdot Z^{s,m }_\tau(x)+\big(\p_x \p_x g\big)^{s,m }_\tau(x)\Bigg)\p_x\big(X^{s,m}_\tau(x)\big)d\tau,
\end{align}\normalsize
where $Z^{s,m }_\tau(x)=\gamma\big(\tau,X^{s,m }_\tau(x),X^{s,m }_\tau\ot m\big)$,
\small\begin{align}
\label{p_xX_eq_1}       
\p_x \big(X^{t,m }_s(x)\big) = \mathcal{I}_{d_x\times d_x}+\int_t^s \bigg(\big(\p_x f\big)^{t,m }_\tau(x)+\big(\p_\alpha f\big)^{t,m }_\tau(x)\Big(\big(\p_x \alpha\big)^{t,m }_\tau(x)+\big(\p_z \alpha\big)^{t,m }_\tau(x)\big(\p_x \gamma\big)^{t,m }_\tau(x)\Big)\bigg)\p_x \big(X^{t,m }_\tau(x)\big) d\tau,
\end{align}\normalsize
and $\mathcal{I}_{d_x\times d_x}$ is the $d_x\times d_x$ identity matrix;
\footnotesize\begin{align}\no
&\p_\mu\gamma (s,x,m)(y)\\\no
= &\ \big(\p_x p\big)^{s,m}_T(x)\p_m\big(X^{s,m}_T(x)\big)(y)+\int_{\R^{d_x}}\big(\p_\mu p\big)^{s,m}_T(x,\wh{x})\p_m\big(X^{s,m}_T(\wh{x})\big)(y)dm(\wh{x})+\big(\p_\mu p\big)^{s,m}_T(x,y)\p_y\big(X^{s,m}_T(y)\big)\\\no
&+ \int_s^T \Bigg(\int_{\R^{d_x}}\bigg(\big(\p_{\wt{x}}\p_\mu f\big)^{s,m }_\tau(\wt{x},x)\cdot  Z^{s,m }_\tau(\wt{x}) + \big(\p_{\wt{x}}\p_\mu g\big)^{s,m }_\tau(\wt{x},x)\bigg)dm(\wt{x})\\\no
&\ \ \ \ \ \ \ \ \ \ \ \ +\Big(\big(\p_\alpha \p_x f\big)^{s,m }_\tau(x)\cdot Z^{s,m }_\tau(x)+\big(\p_\alpha\p_x g\big)^{s,m }_\tau(x)\Big)\Big(\big(\p_x\alpha\big)^{s,m }_\tau(x)+\big(\p_z\alpha\big)^{s,m }_\tau(x)\big(\p_x\gamma\big)^{s,m }_\tau(x)\Big)\\\no 
&\ \ \ \ \ \ \ \ \ \ \ \ +\big(\p_x f\big)^{s,m }_\tau(x)\cdot \big(\p_x\gamma\big)^{s,m }_\tau(x)+\big(\p_x\p_x f\big)^{s,m }_\tau(x)\cdot Z^{s,m }_\tau(x)+\big(\p_x \p_x g\big)^{s,m }_\tau(x)\Bigg)\p_m\big(X^{s,m}_\tau(x)\big)(y)d\tau
\\\no
&+ \int_s^T \Bigg(\int_{\R^{d_x}}\bigg(\big(\p_x\p_\mu f\big)^{s,m }_\tau(\wt{x},x)\cdot  Z^{s,m }_\tau(\wt{x}) + \big(\p_x\p_\mu g\big)^{s,m }_\tau(\wt{x},x)+\big(\p_\mu\p_x f\big)^{s,m }_\tau(x,\wt{x})\cdot  Z^{s,m }_\tau(x) + \big(\p_\mu\p_x g\big)^{s,m }_\tau(x,\wt{x})\\\no
&\ \ \ \ \ \ \ \ \ \ \ \ \ \ \ \ \ \ \ \ +\int_{\R^{d_x}}\Big(\big(\p_\mu\p_\mu f\big)^{s,m }_\tau(\wh{x},x,\wt{x})\cdot  Z^{s,m }_\tau(\wh{x}) + \big(\p_\mu\p_\mu g\big)^{s,m }_\tau(\wh{x},x,\wt{x})\Big)dm(\wh{x})\\\no
&\ \ \ \ \ \ \ \ \ \ \ \ \ \ \ \ \ \ \ \ +\Big(\big(\p_\alpha \p_\mu f\big)^{s,m }_\tau(\wt{x},x)\cdot Z^{s,m }_\tau(\wt{x})+\big(\p_\alpha\p_\mu g\big)^{s,m }_\tau(\wt{x},x)\Big)\Big(\big(\p_x\alpha\big)^{s,m }_\tau(\wt{x})+\big(\p_z\alpha\big)^{s,m }_\tau(\wt{x})\big(\p_x\gamma\big)^{s,m }_\tau(\wt{x})\Big)\\\no
&\ \ \ \ \ \ \ \ \ \ \ \ \ \ \ \ \ \ \ \ +\int_{\R^{d_x}}\Big(\big(\p_\alpha \p_\mu f\big)^{s,m }_\tau(\wh{x},x)\cdot Z^{s,m }_\tau(\wh{x})+\big(\p_\alpha\p_\mu g\big)^{s,m }_\tau(\wh{x},x)\Big)\Big(\big(\p_\mu\alpha\big)^{s,m }_\tau(\wh{x},\wt{x})+\big(\p_z\alpha\big)^{s,m }_\tau(\wh{x})\big(\p_\mu\gamma\big)^{s,m }_\tau(\wh{x},\wt{x})\Big)dm(\wh{x})\\\no
&\ \ \ \ \ \ \ \ \ \ \ \ \ \ \ \ \ \ \ \ +\big(\p_\mu f\big)^{s,m }_\tau(\wt{x},x)\cdot\big(\p_x \gamma\big)^{s,m }_\tau(\wt{x})+\int_{\R^{d_x}}\big(\p_\mu f\big)^{s,m }_\tau(\wh{x},x)\cdot\big(\p_\mu \gamma\big)^{s,m }_\tau(\wh{x},\wt{x})dm(\wh{x})\\\no
&\ \ \ \ \ \ \ \ \ \ \ \ \ \ \ \ \ \ \ \ +\Big(\big(\p_\alpha \p_x f\big)^{s,m }_\tau(x)\cdot Z^{s,m }_\tau(x)+\big(\p_\alpha\p_x g\big)^{s,m }_\tau(x)\Big)\Big(\big(\p_\mu\alpha\big)^{s,m }_\tau(x,\wt{x})+\big(\p_z\alpha\big)^{s,m }_\tau(x)\big(\p_\mu\gamma\big)^{s,m }_\tau(x,\wt{x})\Big)\\\no 
&\ \ \ \ \ \ \ \ \ \ \ \ \ \ \ \ \ \ \ \ +\big(\p_x f\big)^{s,m }_\tau(x)\cdot \big(\p_\mu\gamma\big)^{s,m }_\tau(x,\wt{x})\bigg)\p_m\big(X^{s,m}_\tau(\wt{x})\big)(y)dm(\wt{x})\Bigg)d\tau\\\no
&+ \int_s^T \bigg(\big(\p_\mu\p_x f\big)^{s,m }_\tau(x,y)\cdot  Z^{s,m }_\tau(x) + \big(\p_\mu\p_x g\big)^{s,m }_\tau(x,y)\\\no
&\ \ \ \ \ \ \ \ \ \ \ \ +\int_{\R^{d_x}}\Big(\big(\p_\mu\p_\mu f\big)^{s,m }_\tau(\wh{x},x,y)\cdot  Z^{s,m }_\tau(\wh{x}) + \big(\p_\mu\p_\mu g\big)^{s,m }_\tau(\wh{x},x,y)\Big)dm(\wh{x})\\\no
&\ \ \ \ \ \ \ \ \ \ \ \ +\int_{\R^{d_x}}\Big(\big(\p_\alpha \p_\mu f\big)^{s,m }_\tau(\wh{x},x)\cdot Z^{s,m }_\tau(\wh{x})+\big(\p_\alpha\p_\mu g\big)^{s,m }_\tau(\wh{x},x)\Big)\Big(\big(\p_\mu\alpha\big)^{s,m }_\tau(\wh{x},y)+\big(\p_z\alpha\big)^{s,m }_\tau(\wh{x})\big(\p_\mu\gamma\big)^{s,m }_\tau(\wh{x},y)\Big)dm(\wh{x})\\\no
&\ \ \ \ \ \ \ \ \ \ \ \ +\int_{\R^{d_x}}\big(\p_\mu f\big)^{s,m }_\tau(\wh{x},x)\cdot\big(\p_\mu \gamma\big)^{s,m }_\tau(\wh{x},y)dm(\wh{x})+\big(\p_x f\big)^{s,m }_\tau(x)\cdot \big(\p_\mu\gamma\big)^{s,m }_\tau(x,y)\\\label{p_mgamma_eq_1}  
&\ \ \ \ \ \ \ \ \ \ \ \ +\Big(\big(\p_\alpha \p_x f\big)^{s,m }_\tau(x)\cdot Z^{s,m }_\tau(x)+\big(\p_\alpha\p_x g\big)^{s,m }_\tau(x)\Big)\Big(\big(\p_\mu\alpha\big)^{s,m }_\tau(x,y)+\big(\p_z\alpha\big)^{s,m }_\tau(x)\big(\p_\mu\gamma\big)^{s,m }_\tau(x,y)\Big)\bigg)\p_y\big(X^{s,m}_\tau(y)\big)d\tau,
\end{align}\normalsize
where $Z^{s,m }_\tau(x)=\gamma\big(\tau,X^{s,m }_\tau(x),X^{s,m }_\tau\ot m\big)$ and
\small\begin{align}\no
&\p_m \big(X^{t,m }_s(x)\big)(y)\\\no
=&\  \int_t^s \bigg(\big(\p_x f\big)^{t,m }_\tau(x)+\big(\p_\alpha f\big)^{t,m }_\tau(x)\Big(\big(\p_x \alpha\big)^{t,m }_\tau(x)+\big(\p_z \alpha\big)^{t,m }_\tau(x)\big(\p_x \gamma\big)^{t,m }_\tau(x)\Big)\bigg)\p_m \big(X^{t,m }_\tau(x)\big)(y) d\tau\\\no
&+\int_t^s\int_{\R^{d_x}}\bigg(\big(\p_\mu f\big)^{t,m }_\tau(x,\wt{x})+\big(\p_\alpha f\big)^{t,m }_\tau(x)\Big(\big(\p_\mu \alpha\big)^{t,m }_\tau(x,\wt{x})+\big(\p_z \alpha\big)^{t,m }_\tau(x)\big(\p_\mu \gamma\big)^{t,m }_\tau(x,\wh{x})\Big)\bigg)\p_m\big(X^{t,m }_\tau(\wh{x})\big)(y)dm(\wh{x})d\tau\\\label{p_mX_eq_1}   
&+\int_t^s\bigg(\big(\p_\mu f\big)^{t,m }_\tau(x,y)+\big(\p_\alpha f\big)^{t,m }_\tau(x)\Big(\big(\p_\mu \alpha\big)^{t,m }_\tau(x,y)+\big(\p_z \alpha\big)^{t,m }_\tau(x)\big(\p_\mu \gamma\big)^{t,m }_\tau(x,y)\Big)\bigg)\p_y\big(X^{t,m }_\tau(y)\big)d\tau.
\end{align}\normalsize

\textbf{Part 1}. (Differentiability in $x\in \R^{d_x}$ of $\gamma(s,x,m)$ and $X^{t,m}_s(x)$ )
Recall the notations defined in Appendix \eqref{notation_1} and Table \ref{notation_2} and consider the following FBODE system for $\big(\wb{\mc{X}}^{t,m }_s(x),\gamma_x (s,x,m)\big)$:
\small\begin{align}\no
&\gamma_x (s,x,m)\\\no
= &\ \big(\p_x p\big)^{s,m}_T(x)\wb{\mc{X}}^{s,m }_T(x)\\\no
&+ \int_s^T \Bigg(\int_{\R^{d_x}}\bigg(\big(\p_{\wt{x}}\p_\mu f\big)^{s,m }_\tau(\wt{x},x)\cdot  Z^{s,m }_\tau(\wt{x}) + \big(\p_{\wt{x}}\p_\mu g\big)^{s,m }_\tau(\wt{x},x)\bigg)dm(\wt{x})\\\no
&\ \ \ \ \ \ \ \ \ \ \ \ +\Big(\big(\p_\alpha \p_x f\big)^{s,m }_\tau(x)\cdot Z^{s,m }_\tau(x)+\big(\p_\alpha\p_x g\big)^{s,m }_\tau(x)\Big)\Big(\big(\p_x\alpha\big)^{s,m }_\tau(x)+\big(\p_z\alpha\big)^{s,m }_\tau(x)\gamma_x\big(\tau,X^{s,m }_\tau(x),X^{s,m }_\tau\ot m\big)\Big)\\\label{p_xgamma_eq_2}
&\ \ \ \ \ \ \ \ \ \ \ \ +\big(\p_x f\big)^{s,m }_\tau(x)\cdot \gamma_x\big(\tau,X^{s,m }_\tau(x),X^{s,m }_\tau\ot m\big)+\big(\p_x\p_x f\big)^{s,m }_\tau(x)\cdot Z^{s,m }_\tau(x)+\big(\p_x \p_x g\big)^{s,m }_\tau(x)\Bigg)\wb{\mc{X}}^{s,m }_\tau(x)d\tau,
\end{align}\normalsize
where the  couple $\big(X^{t,m}_s(x),Z^{t,m}_s(x)\big)$ is the unique solution of FBODE \eqref{fbodesystem} constructed in Theorem \ref{Thm6_1} and 
\footnotesize\begin{align}
\label{p_xX_eq_2}       
\wb{\mc{X}}^{t,m }_s(x)= \mathcal{I}_{d_x\times d_x}+\int_t^s \bigg(\big(\p_x f\big)^{t,m }_\tau(x)+\big(\p_\alpha f\big)^{t,m }_\tau(x)\Big(\big(\p_x \alpha\big)^{t,m }_\tau(x)+\big(\p_z \alpha\big)^{t,m }_\tau(x) \gamma_x\big(\tau,X^{t,m }_\tau(x),X^{t,m }_\tau\ot m\big)\Big)\bigg)\wb{\mc{X}}^{t,m }_\tau(x) d\tau.
\end{align}\normalsize
As one may note that the FBODE system \eqref{p_xgamma_eq_2}-\eqref{p_xX_eq_2} is the system governing the evolution of \\
$\Big(\p_x \big(X^{t,m}_s(x)\big),\p_x \gamma(s,x,m)\Big)$ (see \eqref{p_xgamma_eq_1}-\eqref{p_xX_eq_1}). Here, we shall first solve for this $\big(\wb{\mc{X}}^{t,m }_s(x),\gamma_x (s,x,m)\big)$ in Steps 1-4, and then verify the claim that $\Big(\p_x \big(X^{t,m}_s(x)\big),\p_x \gamma(s,x,m)\Big)=\big(\wb{\mc{X}}^{t,m }_s(x),\gamma_x (s,x,m)\big)$ in Step 5. Let $C\big([t,T]\times \R^{d_x}\times\mathcal{P}_2(\R^{d_x});\mc{L}(\R^{d_x};\R^{d_x})\big)$ denotes the vector space of functions in $(s,x,\mu)$, each of which maps from $[t,T]\times \R^{d_x}\times\mathcal{P}_2(\R^{d_x})$ to $\mc{L}(\R^{d_x};\R^{d_x})$ and is continuous in $(s,x,\mu)$; this vector space $C\big([t,T]\times \R^{d_x}\times\mathcal{P}_2(\R^{d_x});\mc{L}(\R^{d_x};\R^{d_x})\big)$ is equipped with the norm $\sup_{s\in[t,T],\,x\in\R^{d_x},\,\mu\in\mathcal{P}_2(\R^{d_x})} \left\|\cdot \right\|_{\mathcal{L}(\R^{d_x};\R^{d_x})}$, and itself can be easily seen as a Banach space. Define a convex subset of this Banach space for iteration:
\begin{align}\no
\mathcal{I}_2:=\bigg\{\gamma_x(s,x,\mu)\in C\big([t,T]\times \R^{d_x}\times\mathcal{P}_2(\R^{d_x});\mc{L}(\R^{d_x};\R^{d_x})\big):&\\\label{wbI_2}
\sup_{s\in[t,T],\,x\in\R^{d_x},\,\mu\in\mathcal{P}_2(\R^{d_x})} \left\|\gamma_x(s,x,\mu)\right\|_{\mathcal{L}(\R^{d_x};\R^{d_x})}\leq 2L_p&\bigg\}.
\end{align}

Step 1. Set $\gamma_x^{(0)}(s,x,\mu):=\p_x p(x,\mu)$ which clearly belongs to $\mathcal{I}_2$ by using \eqref{p1}. Assume that $\gamma_x^{(n-1)}(s,x,\mu)\in \mathcal{I}_2$ and define $\gamma_x^{(n)}(s,x,\mu)$, for $n=1,2,...$, iteratively as follows:\\
(i) solve the following linear forward ODE equation with a mean-field term and an exogenous driving term $\gamma_x^{(n-1)}$ in integral form for $\wb{\mc{X}}^{t,m,(n)}_s(x)$:
\scriptsize\begin{align}
\label{p_xX_eq_2_1}       
\wb{\mc{X}}^{t,m,(n)}_s(x)= \mathcal{I}_{d_x\times d_x}+\int_t^s \bigg(\big(\p_x f\big)^{t,m }_\tau(x)+\big(\p_\alpha f\big)^{t,m }_\tau(x)\Big(\big(\p_x \alpha\big)^{t,m }_\tau(x)+\big(\p_z \alpha\big)^{t,m }_\tau(x) \gamma_x^{(n-1)}\big(\tau,X^{t,m }_\tau(x),X^{t,m }_\tau\ot m\big)\Big)\bigg)\wb{\mc{X}}^{t,m,(n) }_\tau(x) d\tau.
\end{align}\normalsize
(ii) solve the following linear backward ODE equation with a mean-field term and an exogenous driving term $\wb{\mc{X}}^{t,m,(n)}_s(x)$ in integral form for $\gamma_x^{(n)}$:
\footnotesize\begin{align}\no
&\gamma_x^{(n)} (s,x,m)\\\no
= &\ \big(\p_x p\big)^{s,m}_T(x)\wb{\mc{X}}^{s,m,(n) }_T(x)\\\no
&+ \int_s^T \Bigg(\int_{\R^{d_x}}\bigg(\big(\p_{\wt{x}}\p_\mu f\big)^{s,m }_\tau(\wt{x},x)\cdot  Z^{s,m }_\tau(\wt{x}) + \big(\p_{\wt{x}}\p_\mu g\big)^{s,m }_\tau(\wt{x},x)\bigg)dm(\wt{x})\\\no
&\ \ \ \ \ \ \ \ \ \ \ \ +\Big(\big(\p_\alpha \p_x f\big)^{s,m }_\tau(x)\cdot Z^{s,m }_\tau(x)+\big(\p_\alpha\p_x g\big)^{s,m }_\tau(x)\Big)\Big(\big(\p_x\alpha\big)^{s,m }_\tau(x)+\big(\p_z\alpha\big)^{s,m }_\tau(x)\gamma_x^{(n)}\big(\tau,X^{s,m }_\tau(x),X^{s,m }_\tau\ot m\big)\Big)\\\label{p_xgamma_eq_2_1}   
&\ \ \ \ \ \ \ \ \ \ \ \ +\big(\p_x f\big)^{s,m }_\tau(x)\cdot \gamma_x^{(n)}\big(\tau,X^{s,m }_\tau(x),X^{s,m }_\tau\ot m\big)+\big(\p_x\p_x f\big)^{s,m }_\tau(x)\cdot Z^{s,m }_\tau(x)+\big(\p_x \p_x g\big)^{s,m }_\tau(x)\Bigg)\wb{\mc{X}}^{s,m,(n) }_\tau(x)d\tau,
\end{align}\normalsize
where the couple $\big(X^{t,m}_s(x),Z^{t,m}_s(x)\big)$ is the unique solution of FBODE \eqref{fbodesystem} constructed in Theorem \ref{Thm6_1} and $\alpha(X^{s,m}_\tau(x),X^{s,m}_\tau\ot m,Z^{s,m}_\tau(x))$ is well-defined by the validity of \eqref{cone_ZX}. In the following steps $2$-$6$, we shall show that $\wb{\mc{X}}^{t,m,(n)}_s(x)$ is well-defined by \eqref{p_xX_eq_2_1} , $\gamma_x^{(n)}(s,x,\mu)$ is well-defined by \eqref{p_xgamma_eq_2_1} and belongs to $\mathcal{I}_2$, and the iteration mapping $\gamma_x^{(n-1)}\mapsto \gamma_x^{(n)}$ is contractive under the norm $\sup_{s\in[t,T],\,x\in\R^{d_x},\,\mu\in\mathcal{P}_2(\R^{d_x})} \left\|\cdot \right\|_{\mathcal{L}(\R^{d_x};\R^{d_x})}$.

Step 2. Note that \eqref{p_xX_eq_2_1} is a linear forward ODE in integral form and the following term is bounded, after using \eqref{bdd_d1_f}, \eqref{p_xalpha}, \eqref{p_zalpha} and $\gamma_x^{(n-1)}\in \mathcal{I}_2$:
\small\begin{align}
\bigg\|\big(\p_x f\big)^{t,m }_\tau(x)+\big(\p_\alpha f\big)^{t,m }_\tau(x)\Big(\big(\p_x \alpha\big)^{t,m }_\tau(x)+\big(\p_z \alpha\big)^{t,m }_\tau(x) \gamma_x^{(n-1)}\big(\tau,X^{s,m }_\tau(x),X^{s,m }_\tau\ot m\big)\Big)\bigg\|_{\mc{L}(\R^{d_x};\R^{d_x})}\leq L_B,
\end{align}\normalsize
where $L_B$ was defined in \eqref{L_B}. Thus, by a classical method of solving linear ODE (an argument which is similar to steps 1 and 2 in the proof of Lemma \ref{lem6_3} but is simpler), one can show that $\wb{\mc{X}}^{t,m,(n)}_s(x)$ is well-defined for all $s\in[t,T]$ and $t\in[T-\eps_1,T]$, and it also satisfies the following estimates:
\begin{align}\label{eq_6_56_1}
&\left\|\wb{\mc{X}}^{t,m,(n)}_s(x)\right\|_{\mc{L}(\R^{d_x};\R^{d_x})}\leq \exp\big(L_B(s-t)\big),\\\no
&\left\|\wb{\mc{X}}^{t,m,(n+1) }_s(x)-\wb{\mc{X}}^{t,m,(n) }_s(x)\right\|_{\mc{L}(\R^{d_x};\R^{d_x})}\\\label{eq_6_57_1}
\leq &\ (s-t)\exp\big(2L_B(s-t)\big)\Lambda_fL_\alpha\sup_{\tau\in[t,s],\,x\in\R^{d_x},\,\mu\in\mathcal{P}_2(\R^{d_x})}\left\|\gamma_x^{(n)}\big(\tau,x,\mu\big)-\gamma_x^{(n-1)}\big(\tau,x,\mu\big)\right\|_{\mc{L}(\R^{d_x};\R^{d_x})}.
\end{align}

Step 3 ($\gamma_x^{(n)}\in\mc{I}_2$). Note that \eqref{p_xgamma_eq_2_1} is a linear backward ODE in the integral form. Since 
\small\begin{align}\no
&\left\|\gamma_x^{(n)} (s,x,m)\right\|_{\mc{L}(\R^{d_x};\R^{d_x})}\\\no
\leq &\ \left\|\big(\p_x p\big)^{s,m}_T(x)\right\|_{\mc{L}(\R^{d_x};\R^{d_x})}\left\|\wb{\mc{X}}^{s,m,(n) }_T(x)\right\|_{\mc{L}(\R^{d_x};\R^{d_x})}\\\no
&+ \int_s^T \Bigg\|\int_{\R^{d_x}}\bigg(\big(\p_{\wt{x}}\p_\mu f\big)^{s,m }_\tau(\wt{x},x)\cdot  Z^{s,m }_\tau(\wt{x}) + \big(\p_{\wt{x}}\p_\mu g\big)^{s,m }_\tau(\wt{x},x)\bigg)dm(\wt{x})\\\no
&\ \ \ \ \ \ \ \ \ \ \ \ +\Big(\big(\p_\alpha \p_x f\big)^{s,m }_\tau(x)\cdot Z^{s,m }_\tau(x)+\big(\p_\alpha\p_x g\big)^{s,m }_\tau(x)\Big)\Big(\big(\p_x\alpha\big)^{s,m }_\tau(x)+\big(\p_z\alpha\big)^{s,m }_\tau(x)\gamma_x^{(n)}\big(\tau,X^{s,m }_\tau(x),X^{s,m }_\tau\ot m\big)\Big)\\\no 
&\ \ \ \ \ \ \ \ \ \ \ \ +\big(\p_x f\big)^{s,m }_\tau(x)\cdot \gamma_x^{(n)}\big(\tau,X^{s,m }_\tau(x),X^{s,m }_\tau\ot m\big)+\big(\p_x\p_x f\big)^{s,m }_\tau(x)\cdot Z^{s,m }_\tau(x)+\big(\p_x \p_x g\big)^{s,m }_\tau(x)\Bigg\|_{\mc{L}(\R^{d_x};\R^{d_x})}\\\no
&\ \ \ \ \ \ \ \ \ \ \cdot\left\|\wb{\mc{X}}^{s,m,(n) }_\tau(x)\right\|_{\mc{L}(\R^{d_x};\R^{d_x})}d\tau\\\no
\leq &\  L_p\exp\big(L_B(T-s)\big)\\\no
&+\exp\big(L_B(T-s)\big)\int_s^T\bigg(2\big(2\wb{L}_p\wb{l}_f+\Lambda_g\big)+\big(2\wb{L}_p\wb{l}_f+\Lambda_g\big)L_\alpha\Big(1+\|\gamma_x^{(n)}\big(\tau,X^{s,m }_\tau(x),X^{s,m }_\tau\ot m\big)\|_{\mc{L}(\R^{d_x};\R^{d_x})}\Big)\\\no
&\ \ \ \ \ \ \ \ \ \ \ \ \ \ \ \ \ \ \ \ \ \ \ \ \ \ \ \ \ \ \ \ +\Lambda_f\|\gamma_x^{(n)}\big(\tau,X^{s,m }_\tau(x),X^{s,m }_\tau\ot m\big)\|_{\mc{L}(\R^{d_x};\R^{d_x})}\bigg)d\tau\\\no
&(\text{by using \eqref{bdd_d1_f}, \eqref{bdd_d2_f}, \eqref{bdd_d2_g_1}, \eqref{bdd_d2_g_2}, \eqref{p_xalpha}, \eqref{p_zalpha}, \eqref{p1} and \eqref{cone_ZX}}),
\end{align}\normalsize
by using Gr\"{o}nwall's inequality and \eqref{coro_6_4}, we have
\begin{align}\no
\left\|\gamma_x^{(n)} (s,x,m)\right\|_{\mc{L}(\R^{d_x};\R^{d_x})}\leq &\ \Big(L_p+\big(2+L_\alpha\big)\big(2\wb{L}_p\wb{l}_f+\Lambda_g\big)(T-s)\Big)\exp\big(L_B(T-s)\big)\\
&\cdot\exp\Big(\big(\Lambda_f+\big(2\wb{L}_p\wb{l}_f+\Lambda_g\big)L_\alpha\big)(T-s)\exp\big(L_B(T-s)\Big).
\end{align}
Thus, by a classical method of solving linear ODE (an argument which is similar to steps 1 and 2 in the proof of Lemma \ref{lem6_3} but is simpler), one can show that $\gamma_x^{(n)} (s,x,m)$ is well-defined for all $s\in[t,T]$ and $t\in[T-\eps_1,T]$. Moreover, for any $t\in[0,T]$ such that\small
\begin{align}\label{eps_5}
T-t\leq \min\bigg\{&\eps_1,\frac{L_p}{3\big(2+L_\alpha\big)\big(2\wb{L}_p\wb{l}_f+\Lambda_g\big)},\frac{5\ln(5/4)}{6\big(\Lambda_f+\big(2\wb{L}_p\wb{l}_f+\Lambda_g\big)L_\alpha\big)}\bigg\}=:\eps_5,
\end{align}\normalsize
we have $\left\|\gamma_x^{(n)} (s,x,m)\right\|_{\mc{L}(\R^{d_x};\R^{d_x})}\leq 2L_p$ and thus $\gamma_x^{(n)}\in\mc{I}_2$.

Step 4 (contractive property of the iteration mapping $\gamma_x^{(n-1)}(s,x,\mu)\in \mathcal{I}_2\mapsto\gamma_x^{(n)}(s,x,\mu)\in \mathcal{I}_2$). Since
\begin{align}\no
&\left\|\gamma_x^{(n+1)}\big(\tau,X^{s,m }_\tau(x),X^{s,m }_\tau\ot m\big)\wb{\mc{X}}^{s,m,(n+1) }_\tau(x)-\gamma_x^{(n)}\big(\tau,X^{s,m }_\tau(x),X^{s,m }_\tau\ot m\big)\wb{\mc{X}}^{s,m,(n) }_\tau(x)\right\|_{\mc{L}(\R^{d_x};\R^{d_x})}\\\no
\leq &\ \left\|\gamma_x^{(n+1)}\big(\tau,X^{s,m }_\tau(x),X^{s,m }_\tau\ot m\big)\wb{\mc{X}}^{s,m,(n+1) }_\tau(x)-\gamma_x^{(n)}\big(\tau,X^{s,m }_\tau(x),X^{s,m }_\tau\ot m\big)\wb{\mc{X}}^{s,m,(n+1) }_\tau(x)\right\|_{\mc{L}(\R^{d_x};\R^{d_x})}\\\no
&+\left\|\gamma_x^{(n)}\big(\tau,X^{s,m }_\tau(x),X^{s,m }_\tau\ot m\big)\wb{\mc{X}}^{s,m,(n+1) }_\tau(x)-\gamma_x^{(n)}\big(\tau,X^{s,m }_\tau(x),X^{s,m }_\tau\ot m\big)\wb{\mc{X}}^{s,m,(n) }_\tau(x)\right\|_{\mc{L}(\R^{d_x};\R^{d_x})}\\\no
\leq &\  \exp\big(L_B(\tau-s)\big)\left\|\gamma_x^{(n+1)}\big(\tau,X^{s,m }_\tau(x),X^{s,m }_\tau\ot m\big)-\gamma_x^{(n)}\big(\tau,X^{s,m }_\tau(x),X^{s,m }_\tau\ot m\big)\right\|_{\mc{L}(\R^{d_x};\R^{d_x})}\\\no
&+2L_p\left\|\wb{\mc{X}}^{s,m,(n+1) }_\tau(x)-\wb{\mc{X}}^{s,m,(n) }_\tau(x)\right\|_{\mc{L}(\R^{d_x};\R^{d_x})}\\\label{eq_6_60_1}
&(\text{by using \eqref{eq_6_56_1} and $\gamma_x^{(n)}\in\mc{I}_2$}),
\end{align}
we have
\footnotesize\begin{align}\no
&\left\|\gamma_x^{(n+1)} (s,x,m)-\gamma_x^{(n)} (s,x,m)\right\|_{\mc{L}(\R^{d_x};\R^{d_x})}\\\no
\leq &\ L_p\left\|\wb{\mc{X}}^{s,m,(n+1) }_T(x)-\wb{\mc{X}}^{s,m,(n) }_T(x)\right\|_{\mc{L}(\R^{d_x};\R^{d_x})}\\\no
&+ \int_s^T \Big(2\big(2\wb{L}_p\wb{l}_f+\Lambda_g\big)+\big(2\wb{L}_p\wb{l}_f+\wb{l}_g\big)L_\alpha\Big)\left\|\wb{\mc{X}}^{s,m,(n+1) }_\tau(x)-\wb{\mc{X}}^{s,m,(n) }_\tau(x)\right\|_{\mc{L}(\R^{d_x};\R^{d_x})}d\tau\\\no
&+ \int_s^T \Big(\Lambda_f+\big(2\wb{L}_p\wb{l}_f+\wb{l}_g\big)L_\alpha\Big)\bigg(\exp\big(L_B(\tau-s)\big)\left\|\gamma_x^{(n+1)}\big(\tau,X^{s,m }_\tau(x),X^{s,m }_\tau\ot m\big)-\gamma_x^{(n)}\big(\tau,X^{s,m }_\tau(x),X^{s,m }_\tau\ot m\big)\right\|_{\mc{L}(\R^{d_x};\R^{d_x})}\\\no
&\ \ \ \ \ \ \ \ \ \ \ \ \ \ \ \ \ \ \ \ \ \ \ \ \ \ \ \ \ \ \ \ \ \ \ \ \ \ \ \ \ \ \ \ +2L_p\left\|\wb{\mc{X}}^{s,m,(n+1) }_\tau(x)-\wb{\mc{X}}^{s,m,(n) }_\tau(x)\right\|_{\mc{L}(\R^{d_x};\R^{d_x})}\bigg)d\tau\\\no
&(\text{by using \eqref{bdd_d1_f}, \eqref{bdd_d2_f}, \eqref{bdd_d2_g_1}, \eqref{bdd_d2_g_2}, \eqref{p_xalpha}, \eqref{p_zalpha}, \eqref{p1},  \eqref{cone_ZX} and \eqref{eq_6_60_1}})\\\no
\leq &\ \bigg(L_p+\Big(2\big(2\wb{L}_p\wb{l}_f+\Lambda_g\big)+\big(2\wb{L}_p\wb{l}_f+\wb{l}_g\big)L_\alpha(1+2L_p)+2L_p\Lambda_f\Big)(T-s)\bigg)(T-s)\exp\big(2L_B(T-s)\big)\Lambda_fL_\alpha\\\no
&\cdot\sup_{\tau\in[s,T],\,x\in\R^{d_x},\,\mu\in\mathcal{P}_2(\R^{d_x})}\left\|\gamma_x^{(n)}\big(\tau,x,\mu\big)-\gamma_x^{(n-1)}\big(\tau,x,\mu\big)\right\|_{\mc{L}(\R^{d_x};\R^{d_x})}\\\no
&+ \Big(\Lambda_f+\big(2\wb{L}_p\wb{l}_f+\wb{l}_g\big)L_\alpha\Big)\exp\big(L_B(T-s)\big)\int_s^T \left\|\gamma_x^{(n+1)}\big(\tau,X^{s,m }_\tau(x),X^{s,m }_\tau\ot m\big)-\gamma_x^{(n)}\big(\tau,X^{s,m }_\tau(x),X^{s,m }_\tau\ot m\big)\right\|_{\mc{L}(\R^{d_x};\R^{d_x})}d\tau\\\no
&(\text{by using \eqref{eq_6_57_1}}),
\end{align}\normalsize
which implies, by Gr\"{o}nwall's inequality and \eqref{coro_6_4}, 
\footnotesize\begin{align}\no
&\left\|\gamma_x^{(n+1)} (s,x,m)-\gamma_x^{(n)} (s,x,m)\right\|_{\mc{L}(\R^{d_x};\R^{d_x})}\\\no
\leq &\ \bigg(L_p+\Big(2\big(2\wb{L}_p\wb{l}_f+\Lambda_g\big)+\big(2\wb{L}_p\wb{l}_f+\wb{l}_g\big)L_\alpha(1+2L_p)+2L_p\Lambda_f\Big)(T-s)\bigg)(T-s)\exp\big(2L_B(T-s)\big)\Lambda_fL_\alpha\\\no
&\cdot\sup_{\tau\in[s,T],\,x\in\R^{d_x},\,\mu\in\mathcal{P}_2(\R^{d_x})}\left\|\gamma_x^{(n)}\big(\tau,x,\mu\big)-\gamma_x^{(n-1)}\big(\tau,x,\mu\big)\right\|_{\mc{L}(\R^{d_x};\R^{d_x})}\exp\Big(\Big(\Lambda_f+\big(2\wb{L}_p\wb{l}_f+\wb{l}_g\big)L_\alpha\Big)(T-s)\exp\big(L_B(T-s)\big)\Big).
\end{align}\normalsize
Thus, for any $t\in[0,T]$ such that\small
\begin{align}\no
T-t\leq \min\Bigg\{&\eps_5,\frac{5}{18\big(\Lambda_f+\big(2\wb{L}_p\wb{l}_f+\wb{l}_g\big)L_\alpha\big)},\\\label{eps_6}
&\frac{2}{9\Lambda_fL_\alpha\Big(L_p+\big(2\big(2\wb{L}_p\wb{l}_f+\Lambda_g\big)+\big(2\wb{L}_p\wb{l}_f+\wb{l}_g\big)L_\alpha(1+2L_p)+2L_p\Lambda_f\big)/(6\wb{L}_B)\Big)}\Bigg\}=:\eps_6,
\end{align}\normalsize
using \eqref{small_1} and taking the supremum over $s$, $x$ and $\mu$, we have \begin{align}\no
&\sup_{s\in[t,T],\,x\in\R^{d_x},\,\mu\in\mathcal{P}_2(\R^{d_x})}\left\|\gamma_x^{(n+1)} (s,x,m)-\gamma_x^{(n)} (s,x,m)\right\|_{\mc{L}(\R^{d_x};\R^{d_x})}
\\
\leq &\ \frac{1}{2}\sup_{s\in[t,T],\,x\in\R^{d_x},\,\mu\in\mathcal{P}_2(\R^{d_x})}\left\|\gamma_x^{(n)}\big(s,x,\mu\big)-\gamma_x^{(n-1)}\big(s,x,\mu\big)\right\|_{\mc{L}(\R^{d_x};\R^{d_x})}.
\end{align}\normalsize
Therefore, by a simple application of Banach Fixed Point Theorem warrants, for any $t\in[0,T]$ with $T-t\leq \eps_6$, a unique $\gamma_x \in \mc{I}_2$ satisfying \eqref{p_xgamma_eq_2}.

Step 5. In this step, we shall show that $\p_x \gamma(s,x,m)=\gamma_x(s,x,m)$ and $\p_x\big(X^{t,m }_s(x)\big)=\wb{\mc{X}}^{t,m }_s(x)$. First, it follows from the assumptions that $\p_x f(x,\mu,\alpha)$, $\p_\alpha f(x,\mu,\alpha)$, $\p_x\alpha(x,\mu,z)$ and $\p_z\alpha(x,\mu,z)$ are uniformly bounded and are also Lipschitz continuous in their corresponding arguments. Furthermore, we have already shown that $X^{t,m}_s(x)$ and $\gamma(s,x,m)$ are continuous in $x$, and $\gamma_x(s,x,m)$ and $\wb{\mc{X}}^{s,m}_\tau(x)$ are uniformly bounded. Using \eqref{xeq_1} and \eqref{p_xX_eq_2}, we have
\small\begin{align*}
&\ \Big|X^{t,m }_s(x+y)-X^{t,m }_s(x)-\wb{\mc{X}}^{t,m }_s(x)y\Big|\\
\leq&\ \int_t^s \bigg|\Big(\big(\p_x f\big)^{t,m }_\tau(x)+\big(\p_\alpha f\big)^{t,m }_\tau(x)\big(\p_x \alpha\big)^{t,m }_\tau(x)\Big)\Big(X^{t,m }_\tau(x+y)-X^{t,m }_\tau(x)-\wb{\mc{X}}^{t,m }_\tau(x)y\Big)\bigg| d\tau\\
&\ +\int_t^s \bigg|\big(\p_\alpha f\big)^{t,m }_\tau(x)\big(\p_z \alpha\big)^{t,m }_\tau(x) \Big(\gamma^{t,m }_\tau(x+y)-\gamma^{t,m }_\tau(x)-\gamma_x\big(\tau,X^{t,m }_\tau(x),X^{t,m }_\tau\ot m\big)\wb{\mc{X}}^{t,m }_\tau(x)y\Big)\bigg| d\tau+o(|y|)\\
\leq&\ \Lambda_f(1+L_\alpha)\int_t^s \bigg|X^{t,m }_\tau(x+y)-X^{t,m }_\tau(x)-\wb{\mc{X}}^{t,m }_\tau(x)y\bigg| d\tau\\
&\ +\Lambda_fL_\alpha\int_t^s \bigg|\gamma^{t,m }_\tau(x+y)-\gamma^{t,m }_\tau(x)-\gamma_x\big(\tau,X^{t,m }_\tau(x),X^{t,m }_\tau\ot m\big)\wb{\mc{X}}^{t,m }_\tau(x)y\bigg| d\tau+o(|y|)\\\no
&(\text{by using \eqref{bdd_d1_f}, \eqref{p_xalpha} and \eqref{p_zalpha}}),
\end{align*}\normalsize
where the small-$o\Big(\big|y\big|\Big)$ means $o\Big(\big|y\big|\Big)/\big|y\big|\to 0$ as $\big|y\big|\to 0$, the remainder term above being $o\Big(\big|y\big|\Big)$ can be shown by using Lebesgue's dominated convergence theorem and a similar argument leading to the $o\Big(\big|\wt{x}\big|\Big)$ nature of the remainder term $R_1(\wt{x})$ defined in \eqref{R_1}, also see Step 7 of the proof of Lemma \ref{lem6_3} for details.
Thus, by using Gr\"{o}nwall's inequality,
\small\begin{align}\no
&\ \Big|X^{t,m }_s(x+y)-X^{t,m }_s(x)-\wb{\mc{X}}^{t,m }_s(x)y\Big|\\\label{eq_6_63}
\leq&\ \Lambda_fL_\alpha\exp\Big(\Lambda_f(1+L_\alpha)(s-t)\Big)\int_t^s \bigg|\gamma^{t,m }_\tau(x+y)-\gamma^{t,m }_\tau(x)-\gamma_x\big(\tau,X^{t,m }_\tau(x),X^{t,m }_\tau\ot m\big)\wb{\mc{X}}^{t,m }_\tau(x)y\bigg| d\tau+o(|y|).
\end{align}\normalsize
Next, since $\p_x p(x,\mu)$, $\p_x\p_x f(x,\mu,\alpha)$, $\p_\alpha\p_x f(x,\mu,\alpha)$, $\p_{\wt{x}}\p_\mu f(x,\mu,\alpha)(\wt{x})$, $\p_x\p_x g(x,\mu,\alpha)$, $\p_\alpha\p_x g(x,\mu,\alpha)$, $\p_{\wt{x}}\p_\mu g(x,\mu,\alpha)(\wt{x})$, $\p_x f(x,\mu,\alpha)$, $\p_x \alpha(x,\mu,z)$ and $\p_z \alpha(x,\mu,z)$ are uniformly bounded and are also Lipschitz continuous in their corresponding arguments, $X^{t,m}_s(x)$ and $\gamma(s,x,m)$ are continuous in $x$, and $\gamma_x(s,x,m)$ and $\wb{\mc{X}}^{s,m}_\tau(x)$ are uniformly bounded, so it follows from \eqref{gammaeq_1} and \eqref{p_xgamma_eq_2} that

\footnotesize\begin{align*}
&\ \Big|\gamma(s,x+y,m)-\gamma(s,x,m)-\gamma_x(s,x,m)y\Big|\\
=&\ \Bigg|p^{s,m}_T(x+y)-p^{s,m}_T(x)-\big(\p_x p\big)^{s,m}_T(x)\wb{\mc{X}}^{s,m }_T(x)y\\
&+\int_s^T \Bigg(\int_{\R^{d_x}}\bigg(\big(\p_\mu f\big)^{s,m }_\tau(\wt{x},x+y)\cdot  Z^{s,m }_\tau(\wt{x}) + \big(\p_\mu g\big)^{s,m }_\tau(\wt{x},x+y)-\big(\p_\mu f\big)^{s,m }_\tau(\wt{x},x)\cdot  Z^{s,m }_\tau(\wt{x}) - \big(\p_\mu g\big)^{s,m }_\tau(\wt{x},x)\\
&\ \ \ \ \ \ \ \ \ \ \ \ \ \ \ \ \ \ \ \ \ -\Big(\big(\p_{\wt{x}}\p_\mu f\big)^{s,m }_\tau(\wt{x},x)\cdot  Z^{s,m }_\tau(\wt{x}) + \big(\p_{\wt{x}}\p_\mu g\big)^{s,m }_\tau(\wt{x},x)\Big)\wb{\mc{X}}^{s,m }_\tau(x)y\bigg)dm(\wt{x})\\
&\ \ \ \ \ \ \ \ \ \ \ \ +\big(\p_x f\big)^{s,m }_\tau(x+y)\cdot Z^{s,m }_\tau(x+y)+\big(\p_x g\big)^{s,m }_\tau(x+y)-\big(\p_x f\big)^{s,m }_\tau(x)\cdot Z^{s,m }_\tau(x)-\big(\p_x g\big)^{s,m }_\tau(x)\\
&\ \ \ \ \ \ \ \ \ \ \ \ -\bigg(\Big(\big(\p_\alpha \p_x f\big)^{s,m }_\tau(x)\cdot Z^{s,m }_\tau(x)+\big(\p_\alpha\p_x g\big)^{s,m }_\tau(x)\Big)\Big(\big(\p_x\alpha\big)^{s,m }_\tau(x)+\big(\p_z\alpha\big)^{s,m }_\tau(x)\gamma_x\big(\tau,X^{s,m }_\tau(x),X^{s,m }_\tau\ot m\big)\Big)\\ 
&\ \ \ \ \ \ \ \ \ \ \ \ \ \ \ \ +\big(\p_x f\big)^{s,m }_\tau(x)\cdot \gamma_x\big(\tau,X^{s,m }_\tau(x),X^{s,m }_\tau\ot m\big)+\big(\p_x\p_x f\big)^{s,m }_\tau(x)\cdot Z^{s,m }_\tau(x)+\big(\p_x \p_x g\big)^{s,m }_\tau(x)\bigg)\wb{\mc{X}}^{s,m }_\tau(x)y\Bigg)d\tau\Bigg|\\
\leq &\ \Bigg|\big(\p_x p\big)^{s,m}_T(x)\Big(X^{s,m}_T(x+y)-X^{s,m}_T(x)-\wb{\mc{X}}^{s,m }_T(x)y\Big)\Bigg|\\
&+\int_s^T \Bigg(\int_{\R^{d_x}}\Bigg|\Big(\big(\p_{\wt{x}}\p_\mu f\big)^{s,m }_\tau(\wt{x},x)\cdot  Z^{s,m }_\tau(\wt{x}) + \big(\p_{\wt{x}}\p_\mu g\big)^{s,m }_\tau(\wt{x},x)\Big)\Big(X^{s,m }_\tau(x+y)-X^{s,m }_\tau(x)-\wb{\mc{X}}^{s,m }_\tau(x)y\Big)\Bigg|dm(\wt{x})\\
&\ \ \ \ \ \ \ \ \ \ \ \ +\Bigg|\bigg(\Big(\big(\p_\alpha \p_x f\big)^{s,m }_\tau(x)\cdot Z^{s,m }_\tau(x)+\big(\p_\alpha\p_x g\big)^{s,m }_\tau(x)\Big)\big(\p_x\alpha\big)^{s,m }_\tau(x)\\ 
&\ \ \ \ \ \ \ \ \ \ \ \ \ \ \ \ +\big(\p_x\p_x f\big)^{s,m }_\tau(x)\cdot Z^{s,m }_\tau(x)+\big(\p_x \p_x g\big)^{s,m }_\tau(x)\bigg)\Big(X^{s,m }_\tau(x+y)-X^{s,m }_\tau(x)-\wb{\mc{X}}^{s,m }_\tau(x)y\Big)\Bigg|\\
&\ \ \ \ \ \ \ \ \ \ \ \ +\bigg(\Big(\big(\p_\alpha \p_x f\big)^{s,m }_\tau(x)\cdot Z^{s,m }_\tau(x)+\big(\p_\alpha\p_x g\big)^{s,m }_\tau(x)\Big)\big(\p_z\alpha\big)^{s,m }_\tau(x)\\ 
&\ \ \ \ \ \ \ \ \ \ \ \ \ \ \ \ +\big(\p_x f\big)^{s,m }_\tau(x)^\top\bigg)\Big(\gamma^{s,m }_\tau(x+y)-\gamma^{s,m }_\tau(x)- \gamma_x\big(\tau,X^{s,m }_\tau(x),X^{s,m }_\tau\ot m\big)\wb{\mc{X}}^{s,m }_\tau(x)y\Big)\Bigg)d\tau\Bigg|+o\Big(\big|y\big|\Big)\\
\leq&\ \bigg(\Big(L_p+\big(2(2\wb{L}_p\wb{l}_f+\Lambda_g)+(2\wb{L}_p\wb{l}_f+\wb{l}_g)L_\alpha\big)(T-s)\Big)\Lambda_fL_\alpha\exp\Big(\Lambda_f(1+L_\alpha)(T-s)\Big)+(2\wb{L}_p\wb{l}_f+\wb{l}_g)L_\alpha+\Lambda_f\bigg)\\
&\ \cdot\int_s^T \bigg|\gamma^{s,m }_\tau(x+y)-\gamma^{s,m }_\tau(x)-\gamma_x\big(\tau,X^{s,m }_\tau(x),X^{s,m }_\tau\ot m\big)\wb{\mc{X}}^{s,m }_\tau(x)y\bigg| d\tau+o\Big(\big|y\big|\Big)\\
&(\text{by using \eqref{bdd_d2_f}, \eqref{bdd_d2_g_1}, \eqref{bdd_d2_g_2}, \eqref{p1}, \eqref{cone_ZX} and \eqref{eq_6_63}}),
\end{align*}\normalsize
where 
the remainder term being $o\Big(\big|y\big|\Big)$ can also be shown by using Lebesgue's dominated convergence theorem through an argument similar to that for \eqref{R_1}.
Thus, by Gr\"{o}nwall's inequality and \eqref{coro_6_4}, we have $$\Big|\gamma(s,x+y,m)-\gamma(s,x,m)-\gamma_x(s,x,m)y\Big|=o\Big(\big|y\big|\Big).$$ Therefore, $\gamma(s,x,m)$ is differentiable in $x\in\R^{d_x}$ and $\p_x\gamma(s,x,m)=\gamma_x(s,x,m)\in\mc{I}_2$, which is continuous in $s$, $x$ and $m$ and satisfies \eqref{p_x_gamma_bdd}. Moreover, by \eqref{eq_6_63}, we also have $$\Big|X^{t,m }_s(x+y)-X^{t,m }_s(x)-\wb{\mc{X}}^{t,m }_s(x)y\Big|=o\Big(\big|y\big|\Big).$$ Thus $X^{t,m }_s(x)$ is differentiable in $x\in\R^{d_x}$ and $\p_x\big(X^{t,m }_s(x)\big)=\wb{\mc{X}}^{t,m }_s(x)$, which is continuous in $t$, $x$ and $m$, is continuously differentiable in $s\in[t,T]$ and satisfies \eqref{p_x_X_bdd}.

\textbf{Part 2}. ($L$-Differentiability in $m\in\mathcal{P}_2(\R^{d_x})$ of $\gamma(s,x,m)$ and $X^{t,m}_s(x)$)
Consider the following FBODE system for $\big(\mc{X}^{t,m }_s(x,y),
\gamma_\mu (s,x,m)(y)\big)$:
\footnotesize\begin{align}\no
&\gamma_\mu (s,x,m)(y)\\\no
= &\ \big(\p_x p\big)^{s,m}_T(x)\mc{X}^{s,m}_T(x,y)+\int_{\R^{d_x}}\big(\p_\mu p\big)^{s,m}_T(x,\wh{x})\mc{X}^{s,m}_T(\wh{x},y)dm(\wh{x})+\big(\p_\mu p\big)^{s,m}_T(x,y)\p_y\big(X^{s,m}_T(y)\big)\\\no
&+ \int_s^T \Bigg(\int_{\R^{d_x}}\bigg(\big(\p_{\wt{x}}\p_\mu f\big)^{s,m }_\tau(\wt{x},x)\cdot  Z^{s,m }_\tau(\wt{x}) + \big(\p_{\wt{x}}\p_\mu g\big)^{s,m }_\tau(\wt{x},x)\bigg)dm(\wt{x})\\\no
&\ \ \ \ \ \ \ \ \ \ \ \ +\Big(\big(\p_\alpha \p_x f\big)^{s,m }_\tau(x)\cdot Z^{s,m }_\tau(x)+\big(\p_\alpha\p_x g\big)^{s,m }_\tau(x)\Big)\Big(\big(\p_x\alpha\big)^{s,m }_\tau(x)+\big(\p_z\alpha\big)^{s,m }_\tau(x)\big(\p_x\gamma\big)^{s,m }_\tau(x)\Big)\\\no 
&\ \ \ \ \ \ \ \ \ \ \ \ +\big(\p_x f\big)^{s,m }_\tau(x)\cdot \big(\p_x\gamma\big)^{s,m }_\tau(x)+\big(\p_x\p_x f\big)^{s,m }_\tau(x)\cdot Z^{s,m }_\tau(x)+\big(\p_x \p_x g\big)^{s,m }_\tau(x)\Bigg)\mc{X}^{s,m}_\tau(x,y)d\tau\\\no
&+ \int_s^T \Bigg(\int_{\R^{d_x}}\bigg(\big(\p_x\p_\mu f\big)^{s,m }_\tau(\wt{x},x)\cdot  Z^{s,m }_\tau(\wt{x}) + \big(\p_x\p_\mu g\big)^{s,m }_\tau(\wt{x},x)+\big(\p_\mu\p_x f\big)^{s,m }_\tau(x,\wt{x})\cdot  Z^{s,m }_\tau(x) + \big(\p_\mu\p_x g\big)^{s,m }_\tau(x,\wt{x})\\\no
&\ \ \ \ \ \ \ \ \ \ \ \ \ \ \ \ \ \ \ \ +\int_{\R^{d_x}}\Big(\big(\p_\mu\p_\mu f\big)^{s,m }_\tau(\wh{x},x,\wt{x})\cdot  Z^{s,m }_\tau(\wh{x}) + \big(\p_\mu\p_\mu g\big)^{s,m }_\tau(\wh{x},x,\wt{x})\Big)dm(\wh{x})\\\no
&\ \ \ \ \ \ \ \ \ \ \ \ \ \ \ \ \ \ \ \ +\Big(\big(\p_\alpha \p_\mu f\big)^{s,m }_\tau(\wt{x},x)\cdot Z^{s,m }_\tau(\wt{x})+\big(\p_\alpha\p_\mu g\big)^{s,m }_\tau(\wt{x},x)\Big)\Big(\big(\p_x\alpha\big)^{s,m }_\tau(\wt{x})+\big(\p_z\alpha\big)^{s,m }_\tau(\wt{x})\big(\p_x\gamma\big)^{s,m }_\tau(\wt{x})\Big)\\\no
&\ \ \ \ \ \ \ \ \ \ \ \ \ \ \ \ \ \ \ \ +\int_{\R^{d_x}}\Big(\big(\p_\alpha \p_\mu f\big)^{s,m }_\tau(\wh{x},x)\cdot Z^{s,m }_\tau(\wh{x})+\big(\p_\alpha\p_\mu g\big)^{s,m }_\tau(\wh{x},x)\Big)\Big(\big(\p_\mu\alpha\big)^{s,m }_\tau(\wh{x},\wt{x})+\big(\p_z\alpha\big)^{s,m }_\tau(\wh{x})\big(\gamma_\mu\big)^{s,m }_\tau(\wh{x},\wt{x})\Big)dm(\wh{x})\\
\no
&\ \ \ \ \ \ \ \ \ \ \ \ \ \ \ \ \ \ \ \ +\big(\p_\mu f\big)^{s,m }_\tau(\wt{x},x)\cdot\big(\p_x \gamma\big)^{s,m }_\tau(\wt{x})+\int_{\R^{d_x}}\big(\p_\mu f\big)^{s,m }_\tau(\wh{x},x)\cdot\big(\gamma_\mu\big)^{s,m }_\tau(\wh{x},\wt{x})dm(\wh{x})\\\no
&\ \ \ \ \ \ \ \ \ \ \ \ \ \ \ \ \ \ \ \ +\Big(\big(\p_\alpha \p_x f\big)^{s,m }_\tau(x)\cdot Z^{s,m }_\tau(x)+\big(\p_\alpha\p_x g\big)^{s,m }_\tau(x)\Big)\Big(\big(\p_\mu\alpha\big)^{s,m }_\tau(x,\wt{x})+\big(\p_z\alpha\big)^{s,m }_\tau(x)\big(\gamma_\mu\big)^{s,m }_\tau(x,\wt{x})\Big)\\\no 
&\ \ \ \ \ \ \ \ \ \ \ \ \ \ \ \ \ \ \ \ +\big(\p_x f\big)^{s,m }_\tau(x)\cdot \big(\gamma_\mu\big)^{s,m }_\tau(x,\wt{x})\bigg)\mc{X}^{s,m}_\tau(\wt{x},y)dm(\wt{x})\Bigg)d\tau\\\no
&+ \int_s^T \bigg(\big(\p_\mu\p_x f\big)^{s,m }_\tau(x,y)\cdot  Z^{s,m }_\tau(x) + \big(\p_\mu\p_x g\big)^{s,m }_\tau(x,y)\\\no
&\ \ \ \ \ \ \ \ \ \ \ \ +\int_{\R^{d_x}}\Big(\big(\p_\mu\p_\mu f\big)^{s,m }_\tau(\wh{x},x,y)\cdot  Z^{s,m }_\tau(\wh{x}) + \big(\p_\mu\p_\mu g\big)^{s,m }_\tau(\wh{x},x,y)\Big)dm(\wh{x})\\\no
&\ \ \ \ \ \ \ \ \ \ \ \ +\int_{\R^{d_x}}\Big(\big(\p_\alpha \p_\mu f\big)^{s,m }_\tau(\wh{x},x)\cdot Z^{s,m }_\tau(\wh{x})+\big(\p_\alpha\p_\mu g\big)^{s,m }_\tau(\wh{x},x)\Big)\Big(\big(\p_\mu\alpha\big)^{s,m }_\tau(\wh{x},y)+\big(\p_z\alpha\big)^{s,m }_\tau(\wh{x})\big(\gamma_\mu\big)^{s,m }_\tau(\wh{x},y)\Big)dm(\wh{x})\\\no
&\ \ \ \ \ \ \ \ \ \ \ \ +\int_{\R^{d_x}}\big(\p_\mu f\big)^{s,m }_\tau(\wh{x},x)\cdot\big(\gamma_\mu\big)^{s,m }_\tau(\wh{x},y)dm(\wh{x})+\big(\p_x f\big)^{s,m }_\tau(x)\cdot \big(\gamma_\mu\big)^{s,m }_\tau(x,y)\\\label{p_mgamma_eq_2}  
&\ \ \ \ \ \ \ \ \ \ \ \ +\Big(\big(\p_\alpha \p_x f\big)^{s,m }_\tau(x)\cdot Z^{s,m }_\tau(x)+\big(\p_\alpha\p_x g\big)^{s,m }_\tau(x)\Big)\Big(\big(\p_\mu\alpha\big)^{s,m }_\tau(x,y)+\big(\p_z\alpha\big)^{s,m }_\tau(x)\big(\gamma_\mu\big)^{s,m }_\tau(x,y)\Big)\bigg)\p_y\big(X^{s,m}_\tau(y)\big)d\tau,
\end{align}\normalsize
where $\big(\gamma_\mu\big)^{s,m }_\tau(x,y):=\gamma_\mu\big(\tau,X^{s,m }_\tau(x),X^{s,m }_\tau\ot m\big)\big(X^{s,m }_\tau(y)\big)$, the  couple $\big(X^{t,m}_s(x),Z^{t,m}_s(x)\big)$ is the unique solution of FBODE \eqref{fbodesystem} constructed in Theorem \ref{Thm6_1} and
\small\begin{align}\no
&\mc{X}^{t,m }_s(x,y)\\\no
=&\  \int_t^s \bigg(\big(\p_x f\big)^{t,m }_\tau(x)+\big(\p_\alpha f\big)^{t,m }_\tau(x)\Big(\big(\p_x \alpha\big)^{t,m }_\tau(x)+\big(\p_z \alpha\big)^{t,m }_\tau(x)\big(\p_x \gamma\big)^{t,m }_\tau(x)\Big)\bigg)\mc{X}^{t,m }_\tau(x,y) d\tau\\\no
&+\int_t^s\int_{\R^{d_x}}\bigg(\big(\p_\mu f\big)^{t,m }_\tau(x,\wt{x})+\big(\p_\alpha f\big)^{t,m }_\tau(x)\Big(\big(\p_\mu \alpha\big)^{t,m }_\tau(x,\wt{x})+\big(\p_z \alpha\big)^{t,m }_\tau(x)\big(\gamma_\mu\big)^{t,m }_\tau(x,\wt{x})\Big)\bigg)\mc{X}^{t,m }_\tau(\wt{x},y)dm(\wt{x})d\tau\\\label{p_mX_eq_2}   
&+\int_t^s\bigg(\big(\p_\mu f\big)^{t,m }_\tau(x,y)+\big(\p_\alpha f\big)^{t,m }_\tau(x)\Big(\big(\p_\mu \alpha\big)^{t,m }_\tau(x,y)+\big(\p_z \alpha\big)^{t,m }_\tau(x)\big(\gamma_\mu\big)^{t,m }_\tau(x,y)\Big)\bigg)\p_y\big(X^{t,m }_\tau(y)\big)d\tau.
\end{align}\normalsize
As one may also note that the FBODE system \eqref{p_mgamma_eq_2}-\eqref{p_mX_eq_2} is the system governing the evolution of \\
$\Big(\p_m \big(X^{t,m}_s(x)\big)(y),\p_\mu \gamma(s,x,m)(y)\Big)$ (see \eqref{p_mgamma_eq_1}-\eqref{p_mX_eq_1}). We can first solve for this $\big(\mc{X}^{t,m }_s(x,y),
\gamma_\mu (s,x,m)(y)\big)$ by an argument similar to Step 1-Step 4 in \textbf{Part 1}, and then verify that $\big(\p_m \big(X^{t,m}_s(x)\big)(y),\p_m \gamma(s,x,m)(y)\big)\\=\big(\mc{X}^{t,m }_s(x,y),
\gamma_\mu (s,x,m)(y)\big)$ as well as estimates \eqref{p_m_gamma_bdd} and \eqref{p_m_X_bdd} by an argument similar to Step 5 in \textbf{Part 1} and Step 8 in the proof of Lemma \ref{lem6_3}. We leave the details to interested readers. 

\textbf{Part 3}. (Differentiability in $t\in[T-\eps_2,T]$ of $\gamma(t,x,m)$ and $X^{t,m}_s(x)$) By \textbf{Part 1} and \textbf{Part 2}, we have $\gamma\in\mc{I}_1$, thus, by Lemma \ref{lem6_3}, $X^{t,m}_s(x)$ is continuously differentiable in $t\in[T-\eps_2,T]$, and its derivative $\p_t \big(X^{t,m}_s(x)\big)$ is continuously differentiable in $s\in[t,T]$ and satisfies \eqref{p_t_X_bdd}. In particular, the properties of $Z^{t,m}_s(x)=\gamma(s,X^{t,m}_s(x),X^{t,m}_s\ot m)\big)$, including but not limited to \eqref{p_x_Z_bdd_n}-\eqref{p_t_Z_bdd_n}, follow immediately from the properties of $\gamma(s,x,\mu)$ and $X^{t,m}_s(x)$. Moreover, the differentiability of $\gamma(s,x,m)$ in $s\in[T-\eps_2,T]$ follows immediately from \eqref{gammaeq}, the differentiability of $X^{s,m}_\tau(x)$, $Z^{s,m}_\tau(x)$ in $s$ and the differentiability of $p(x,\mu)$, $\p_\mu f(x,\mu,\alpha)(y)$, $\p_\mu g(x,\mu,\alpha)(y)$, $\p_x f(x,\mu,\alpha)$, $\p_x g(x,\mu,\alpha)$ and $\alpha(x,\mu,z)$ in their corresponding arguments.

\section{Crucial Estimate and Global-in-time Existence}\label{sec:global}
To obtain a global-in-time solution of \eqref{MPIT}, it suffices to show the global existence of decoupling field $\gamma$ for the FBODE system \eqref{MPIT} on $[0,T]$ which satisfies \eqref{gammaeq}, since the pair $\big(X^{t,m}_s(x),Z^{t,m}_s(x)\big)$ defined by \eqref{xeq} and $Z^{t,m}_s(x)=\gamma(s,X^{t,m}_s(x),X^{t,m}_s\ot m)$ is a solution of FBODE system \eqref{fbodesystem}, and the FBODE system \eqref{fbodesystem} is the system \eqref{MPIT} when $\wt{T}=T$ and $p(x,\mu)=\int_{\R^{d_x}} \p_\mu k( \wt{x} ,\mu)(x)d\mu(\wt{x})+\p_x k(x,\mu)$.

To extend the local-in-time solution of \eqref{gammaeq} to a global one over the whole time horizon $[0,T]$, we shall have to derive some important new {\it a priori} estimates, which play a crucial role in extending the existence lifespan and can be established under Assumption ${\bf(a1)}$-${\bf(a3)}$ and the following
\textbf{Hypotheses}:\\ 
$\bf{(h1)}$ $f(0,\delta_0,0)=0$, $\p_x g(0,\delta_0,0)=\p_\alpha g(0,\delta_0,0)=0$, $\p_\mu g(0,\delta_0,0)(0)=0$, $\p_x k(0,\delta_0)=0$ and $\p_\mu k(0,\delta_0)(0)=0$, where $\delta_0$ is the Dirac measure with a point mass at $0$.\\
$\bf{(h2)}$ Recall the constants defined in ${\bf(a1)}-{\bf(a3)}$, assume that $8\wb{l}_g\leq \lambda_g$ and $\wb{l}_f\leq \frac{1}{40 \max\{\wb{L}_k,L^*_0\}}\lambda_g$,
where $\wb{L}_k$ was defined in \eqref{wbL_k}, and the positive constant $L^*_0$ defined in \eqref{L_star_0} depends only on $\lambda_k,\ \Lambda_k,\ \lambda_g,\ \Lambda_g,\ \lambda_f$ and $\Lambda_f$ but not on the initial data $m_0$ and $\wb{l}_f$.\\

\begin{remark}
(i) Hypotheses $\bf{(h1)}$ and $\bf{(h2)}$ are satisfied in many interesting cases including Linear-Quadratic setting (see \cite{bensoussan2016linear}), also see an immediate non-linear quadratic examples in Section \ref{sec:nonLQ}. For the case with linear drift function $f$, we can even drop Hypothesis $\bf{(h1)}$, and $\wb{l}_f=0$ so that $\bf{(h2)}$ is automatically partially satisfied. \\
(ii) Under Assumptions ${\bf(a1)}$-${\bf(a3)}$ and Hypothesis ${\bf(h2)}$, we have the results stated in Proposition \ref{A1}-\ref{A6} since $\wb{l}_f\leq \frac{1}{40 \max\{\wb{L}_k,L^*_0\}}\lambda_g$ implies \eqref{p_aa_f}.
\\ (iii) The condition $\wb{l}_f\leq \frac{1}{40 \max\{\wb{L}_k,L^*_0\}}\lambda_g$ in Hypothesis $\bf{(h2)}$ means that the rate of the second order derivatives of drift function $f$ is relatively small, that stands for the case that the drift rate cannot have a large quadratic-growth effect; even for the classical control problems, the drift function with quadratic-growth may also lead to ill-posedness, and this philosophy can also be observed in a lot of applications, see the monograph \cite{bensoussan2018estimation} for instance. The inequalities assumed in Hypotheses $\bf{(h2)}$ are not optimal but they are convenient for our calculus.
\end{remark}

By using Hypothesis ${\bf(h1)}$, we have the following proposition which guides us with better linear-growth estimates.
\begin{prpstn}
Under Assumptions ${\bf(a1)}$-${\bf(a3)}$ and Hypotheses ${\bf(h1)}$-${\bf(h2)}$, then $\wh{\alpha}(0,\delta_0,0)=0$ and the following linear-growth estimates hold:
\begin{align}\label{ligfb}
&i)\ |f(x,\mu,\alpha)|\leq L_f \left(|x|+|\alpha| + \|\mu\|_1 \right);\\\label{LGgDerivativesb}
&ii)\ \sup_{x\in\R^{d_x},\,\mu\in\mathcal{P}_2(\R^{d_x}),\,\alpha\in\R^{d_\alpha},\,\wt{x}\in\R^{d_x}}\frac{\left|\p_x g(x,\mu,\alpha)\right|}{|x|+ \|\mu\|_1+|\alpha|}
\vee\frac{\left|\p_\alpha g(x,\mu,\alpha)\right|}{|x|+ \|\mu\|_1+|\alpha|}\vee\frac{|\p_{\mu} g(x,\mu,\alpha)(\wt{x})|}{|x|+ \|\mu\|_1+|\alpha|+|\wt{x}|}\leq L_g;\\\label{LGkDerivativesb}
&iii)\ \sup_{x\in\R^{d_x},\,\mu\in\mathcal{P}_2(\R^{d_x}),\,\wt{x}\in\R^{d_x}}\frac{\left|\p_x k(x,\mu)\right|}{|x|+ \|\mu\|_1}\vee\frac{|\p_\mu k(x,\mu)(\wt{x})|}{|x|+ \|\mu\|_1+|\wt{x}|}\leq L_k;\\\label{linalphab}
&iv)\ |\wh{\alpha}(x,\mu,z)|\leq L_\alpha (|x|+\|\mu\|_1+|z|),
\end{align}
where $L_f$, $L_g$, $L_k$ and $L_\alpha$ were defined in Proposition \ref{A1}, \ref{A2} and \ref{A5}  respectively.
\end{prpstn}
\begin{proof}
By using Hypothesis ${\bf(h1)}$ and the first order condition \eqref{first_order_condition}, $\alpha=0$ is a zero to the first-order condition \eqref{first_order_condition} at the point $(x,\mu,z)=(0,\delta_0,0)$, therefore in light of the uniqueness stated in Proposition \ref{A5}, we know that the unique minimizer $\alpha(0,\delta_0,0)=0$. By using Hypothesis ${\bf(h1)}$ together with the arguments leading to \eqref{ligf} \eqref{LGgDerivatives}, \eqref{LGkDerivatives} and \eqref{linalpha}, we can similarly deduce \eqref{ligfb}-\eqref{linalphab}; in a certain sense, the proof of \eqref{ligfb}-\eqref{linalphab} are slightly simpler that those for \eqref{ligf} \eqref{LGgDerivatives}, \eqref{LGkDerivatives} and \eqref{linalpha}.
\end{proof}

For the sake of convenience, we use the notations defined in Appendix \eqref{notation_1} and Table \ref{notation_2}. Now, we are ready to establish new {\it a priori} estimates.
For any fixed $m\in \mathcal{P}_2(\R^{d_x})$, it follows from Theorem \ref{Thm6_1} and \ref{Thm6_2} that there is a (local-in-time) solution pair $\big(X^{t,m}_s(x),Z^{t,m}_s(x)=\gamma(s,X^{t,m}_s(x),X^{t,m}_s\ot m)\big)$ to \eqref{MPIT}, which is continuously differentiable in $x\in\R^{d_x}$ and continuously $L$-differentiable in $m\in\mc{P}_2(\R^{d_x})$ with the corresponding derivatives being continuously differentiable in $s\in [t,T]$. Then the pairs $\big(\p_m\big(X^{t,m}_s(x)\big)(y),\p_m\big(Z^{t,m}_s(x)\big)(y)\big)$ and $\big(\p_x\big(X^{t,m}_s(x)\big)$, $\p_x\big(Z^{t,m}_s(x)\big)\big)$ satisfy the linear FBODE systems respectively:
\scriptsize\begin{align}\label{eq_8_5_new}
\left\{ \begin{aligned}
	\frac{d}{ds}\p_m\big(X^{t,m}_s(x)\big)(y) =&\  \p_m \big(f^{t,m}_s(x)\big)(y),\\
	\p_m \big(X^{t,m}_t(x)\big)(y) =&\ 0,\\
-\frac{d}{ds}\p_m\big(Z^{t,m}_s(x)\big)(y)=&\  \int_{\R^{d_x}}\p_m\big(X^{t,m}_s(\wt{x})\big)(y)^\top\bigg(\big(\p_x\p_\mu f\big)^{t,m}_s(\wt{x},x)\cdot  Z^{t,m}_s(\wt{x})+ \big(\p_x\p_\mu g\big)^{t,m}_s(\wt{x},x)\bigg)dm(\wt{x})\\
&+\p_m\big(X^{t,m}_s(x)\big)(y)^\top\int_{\R^{d_x}}\bigg(\big(\p_{\wt{x}}\p_\mu f\big)^{t,m}_s(\wt{x},x)\cdot  Z^{t,m}_s(\wt{x})+ \big(\p_{\wt{x}}\p_\mu g\big)^{t,m}_s(\wt{x},x)\bigg)dm(\wt{x})\\
&+\int_{\R^{d_x}}\int_{\R^{d_x}}\p_m\big(X^{t,m}_s(\wh{x})\big)(y)^\top\bigg(\big(\p_\mu\p_\mu f\big)^{t,m}_s(\wt{x},x,\wh{x})\cdot  Z^{t,m}_s(\wt{x})+ \big(\p_\mu\p_\mu g\big)^{t,m}_s(\wt{x},x,\wh{x})\bigg)dm(\wh{x})dm(\wt{x})\\
&+\int_{\R^{d_x}}\p_y\big(X^{t,m}_s(y)\big)^\top\bigg(\big(\p_\mu\p_\mu f\big)^{t,m}_s(\wt{x},x,y)\cdot  Z^{t,m}_s(\wt{x})+ \big(\p_\mu\p_\mu g\big)^{t,m}_s(\wt{x},x,y)\bigg)dm(\wt{x})\\
&+\int_{\R^{d_x}}\p_m\big(\alpha^{t,m}_s(\wt{x})\big)(y)^\top\bigg(\big(\p_\alpha\p_\mu f\big)^{t,m}_s(\wt{x},x)\cdot  Z^{t,m}_s(\wt{x})+ \big(\p_\alpha\p_\mu g\big)^{t,m}_s(\wt{x},x)\bigg)dm(\wt{x})\\
&+\int_{\R^{d_x}} \big(\p_\mu f\big)^{t,m}_s(\wt{x},x)\cdot  \p_m\big(Z^{t,m}_s(\wt{x})(y)\big)dm(\wt{x})+\big(\p_x f\big)^{t,m}_s(x)\cdot \p_m\big(Z^{t,m}_s(x)\big)(y)\\
&+\p_m\big(X^{t,m}_s(x)\big)(y)^\top\bigg(\big(\p_x\p_x f\big)^{t,m}_s(x)\cdot Z^{t,m}_s(x)+\big(\p_x\p_x g\big)^{t,m}_s(x)\bigg)\\
&+\int_{\R^{d_x}}\p_m\big(X^{t,m}_s(\wt{x})\big)(y)^\top\bigg(\big(\p_\mu\p_x f\big)^{t,m}_s(x,\wt{x})\cdot Z^{t,m}_s(x)+\big(\p_\mu\p_x g\big)^{t,m}_s(x,\wt{x})\bigg)dm(\wt{x})\\
&+\p_y\big(X^{t,m}_s(y)\big)^\top\bigg(\big(\p_\mu\p_x f\big)^{t,m}_s(x,y)\cdot Z^{t,m}_s(x)+\big(\p_\mu\p_x g\big)^{t,m}_s(x,y)\bigg)\\
&+\p_m\big(\alpha^{t,m}_s(x)\big)(y)^\top\bigg(\big(\p_\alpha\p_x f\big)^{t,m}_s(x)\cdot Z^{t,m}_s(x)+\big(\p_\alpha\p_x g\big)^{t,m}_s(x)\bigg),\\
\p_m\big(Z^{t,m}_T(x)\big)(y)=&\  \displaystyle\int_{\R^{d_x}} \big(\p_x\p_\mu k\big)^{t,m}_T(\wt{x},x)\cdot \p_m\big(X^{t,m}_T(\wt{x})\big)(y)dm(\wt{x})+\displaystyle\int_{\R^{d_x}} \big(\p_{\wt{x}}\p_\mu k\big)^{t,m}_T(\wt{x},x)dm(\wt{x})\cdot \p_m\big(X^{t,m}_T(x)\big)(y)\\
&+\displaystyle\int_{\R^{d_x}} \int_{\R^{d_x}}\big(\p_\mu\p_\mu k\big)^{t,m}_T(\wt{x},x,\wh{x})\cdot \p_m\big(X^{t,m}_T(\wh{x})\big)(y)dm(\wh{x})dm(\wt{x})\\
&+\displaystyle\int_{\R^{d_x}} \big(\p_\mu\p_\mu k\big)^{t,m}_T(\wt{x},x,y)\cdot \p_y\big(X^{t,m}_T(y)\big)dm(\wt{x})+\big(\p_x\p_x k\big)^{t,m}_T(x)\cdot \p_m\big(X^{t,m}_T(x)\big)(y)\\
&+\displaystyle\int_{\R^{d_x}}\big(\p_\mu\p_x k\big)^{t,m}_T(x,\wt{x})\cdot \p_m\big(X^{t,m}_T(\wt{x})\big)(y)dm(\wt{x})+\big(\p_\mu\p_x k\big)^{t,m}_T(x,y)\cdot \p_y\big(X^{t,m}_T(y)\big);
\end{aligned} \right.
\end{align}\normalsize 
\small\begin{align}\label{eq_8_6_new}
\left\{ \begin{aligned}
	\frac{d}{ds}\p_x\big(X^{t,m}_s(x)\big) =&\  \p_x \big(f^{t,m}_s(x)\big),\\
	\p_x \big(X^{t,m}_t(x)\big) =&\ \mathcal{I}_{d_x\times d_x},\\
	-\frac{d}{ds}\p_x \big(Z^{t,m}_s(x)\big)  =&\ \p_x\big(X^{t,m}_s(x)\big)^\top\int_{\R^{d_x}}\bigg(\big(\p_{\wt{x}}\p_\mu f\big)^{t,m}_s(\wt{x},x)\cdot  Z^{t,m}_s(\wt{x})+ \big(\p_{\wt{x}}\p_\mu g\big)^{t,m}_s(\wt{x},x)\bigg)dm(\wt{x})\\
&+\big(\p_x f\big)^{t,m}_s(x)\cdot \p_x\big(Z^{t,m}_s(x)\big)+\p_x\big(X^{t,m}_s(x)\big)^\top\bigg(\big(\p_x\p_x f\big)^{t,m}_s(x)\cdot Z^{t,m}_s(x)+\big(\p_x\p_x g\big)^{t,m}_s(x)\bigg)\\
&+\p_x\big(\alpha^{t,m}_s(x)\big)^\top\bigg(\big(\p_\alpha\p_x f\big)^{t,m}_s(x)\cdot Z^{t,m}_s(x)+\big(\p_\alpha\p_x g\big)^{t,m}_s(x)\bigg),\\
\p_x\big(Z^{t,m}_T(x)\big)=&\  \displaystyle\int_{\R^{d_x}} \big(\p_{\wt{x}}\p_\mu k\big)^{t,m}_T(\wt{x},x)dm(\wt{x})\cdot \p_x\big(X^{t,m}_T(x)\big)+\big(\p_x\p_x k\big)^{t,m}_T(x)\cdot \p_x\big(X^{t,m}_T(x)\big);
\end{aligned} \right.
\end{align}\normalsize
where $\mathcal{I}_{d_x\times d_x}$ is the $d_x\times d_x$ identity matrix, $\p_m \big(f^{t,m}_s(x)\big)(y)$ and $\p_x \big(f^{t,m}_s(x)\big)$ are computed in \eqref{calculus2}.

For the sake of convenience, we denote
\begin{align}\label{eq_7_6}
L^*:=\vertiii{\gamma }_{2,[t,T]}<\infty,
\end{align}
where $\vertiii{\cdot}_2$ was defined in \eqref{Gamma_2}.
Here, the finiteness of $L^*$ is guaranteed by the construction of the local-in-time solution, which exists at least for $t\in[0,T]$ with $T-t\leq \eps_2$, where $\eps_2$ was given in Theorem \ref{Thm6_2}. Furthermore, the function $\gamma$ in \eqref{eq_7_6} was constructed in Theorem \ref{Thm6_1}. Then, by using \eqref{eq_9_3} in Appendix, 
\begin{align}\label{eq_8_8_new}
&\big|Z^{t,m}_t(x)-Z^{t,m_0}_t(x_0)\big|=\big|\gamma(t,x,m)-\gamma(t,x_0,m_0)\big|
\leq \ L^*\Big(|x-x_0|+W_1(m,m_0)\Big).
\end{align}
By using Hypothesis ${\bf (h1)}$, one can easily check that, for any fixed $t\in[0,T]$, the pair $\Big(X^{t,\delta_0}_s(0)\equiv 0,Z^{t,\delta_0}_s(0)\equiv 0\Big)$ solves the FBODE \eqref{MPIT} with fixed $x=0$ and $m=\delta_0$. Therefore, we have $\big|Z^{t,m}_t(x)\big|=\big|\gamma(t,x,m)\big|\leq L^*\Big(|x|+\|m\|_1\Big)$, which further implies that $\big\|Z^{t,m}_t\big\|_{L^{1,d_x}_{m}}=\big\|\gamma(t,\cdot,m)\big\|_{L^{1,d_x}_{m}}\leq 2L^*\|m\|_1$ by taking an integration. Furthermore, we also have 
\begin{align}\label{eq_8_9_cone}
&\big|Z^{t,m}_s(x)\big|=\big|\gamma(s,X^{t,m}_s(x),X^{t,m}_s\ot m)\big|\leq L^*\Big(|X^{t,m}_s(x)|+\|X^{t,m}_s\|_{L^{1,d_x}_{m}}\Big),\\
&\big\|Z^{t,m}_s\big\|_{L^{1,d_x}_{m}}\leq 2 L^*\|X^{t,m}_s\|_{L^{1,d_x}_{m}}.
\end{align}

Now, we have the following crucial {\it a priori} estimates.
\begin{theorem}\label{Crucial_Estimate}
Under Assumptions ${\bf(a1)}$-${\bf(a3)}$ and Hypothesis ${\bf(h1)}$-${\bf(h2)}$, suppose that $\gamma(s,x,m)$ is a decoupling field for the FBODE system \eqref{fbodesystem} on $[t,T]$ with terminal data $p(x,\mu)=\int_{\R^{d_x}} \p_\mu k( \wt{x} ,\mu)(x)d\mu(\wt{x})+\p_x k(x,\mu)$, which is differentiable in 
$x\in\R^{d_x}$ and $L$-differentiable in $\mu\in\mc{P}_2(\R^{d_x})$, and suppose also that $L^*\leq 2\max\{L^*_0,\wb{L}_k\}$, we then have the following estimate:
\begin{align}\label{JFE}
L^*\leq L^*_0,
\end{align}
where $L^*$ is defined in \eqref{eq_7_6} and $L^*_0$ is defined in \eqref{L_star_0}.
\end{theorem}

\begin{remark}
Indeed, we shall show that the constant $L^*_0$ defined in \eqref{L_star_0} is always larger than or equal to $\wb{L}_k$; see Remark \ref{remark_lower_bdd_L_star_0} for details. In the statement of Theorem \ref{Crucial_Estimate}, the main reason for writing $\max\{L^*_0,\wb{L}_k\}$ (instead of just $L^*_0$) is to strongly emphasize that the quantity $L^*$  must be larger than or equal to $\wb{L}_k$. This fact is extremely important since the estimates of $\gamma$ must strongly depend on the terminal cost function $k$.
\end{remark}

\begin{proof}
Define the pair $\big(X^{t,m}_s(x),Z^{t,m}_s(x)\big)$ by \eqref{xeq} and $Z^{t,m}_s(x)=\gamma(s,X^{t,m}_s(x),X^{t,m}_s\ot m)$, then it is a solution pair of FBODE system \eqref{fbodesystem} by \eqref{gammaeq}. The differentiabilities of $\big(X^{t,m}_s(x),Z^{t,m}_s(x)\big)$ that were stated in Theorem \ref{Thm6_2} follows by the properties of $\gamma(s,x,\mu)$.

In principle, this proof is an unlifted version of the argument that we presented in Section \ref{sec:motivation}, in which the lifting procedure highly simplify the algebraic manipulations but creates more than necessary requirement on the differentiability of the coefficient functions. To release the regularity requirement, this proof, as an enhanced version of Section \ref{sec:motivation}, will derive estimates on the solution to the Jacobian flows \eqref{eq_8_5_new} and \eqref{eq_8_6_new} without using the lifting procedure. The unlifted/original problem allows us to have more flexibility on using various norms in our estimations; however, the calculus will be much more involved, because (i) there are actually two Jacobian flows (with respect to $\big(\p_m\big(X^{t,m}_s(x)\big)(y),\p_m\big(Z^{t,m}_s(x)\big)(y)\big)$ and $\big(\p_x\big(X^{t,m}_s(x)\big)$, $\p_x\big(Z^{t,m}_s(x)\big)\big)$) in the original setting instead of one Jacobian flow (with respect to $(D_YX^{t,Y}_s,D_YZ^{t,Y}_s)$) in the lifted Hilbert space, and (ii) expressions of various derivatives are highly simplified in the lifted Hilbert space notations.

This proof will be separated into five steps. In Step 1 we shall first provide some calculus identities and elementary estimates that are useful in deriving estimates on the solutions to the Jacobian flows. In Step 2 preliminary estimates, which are analogous to \eqref{lifted:estimate_3}, on different derivatives of $\gamma$ will be derived by using the backward ODE only. To control the upper bounds of these preliminary estimates, we shall derive our major estimates (e.g., \eqref{L_star_1_1} and \eqref{cru_est_2}) that are analogous to \eqref{lifted:estimate_2} by integrating the inner products of forward and backward flows in Step 3 and 4. Finally, in Step 5 we shall combine all the above mentioned estimates to obtain a uniform (in $T$) bound on the derivatives of $\gamma$.

It follows from \eqref{k_0} that $\frac{1}{2}k_0=2\max\{L^*_0,\wb{L}_k\}\geq L^*$, so using \eqref{c_k_0} and \eqref{eq_8_9_cone}, we know that $\big(X^{t,m}_s(x),X^{t,m}_s\ot m, Z^{t,m}_s(x)\big)\in c_{k_0}$, and hence the unique minimizer $\alpha^{t,m}_s(x):=\alpha\big(X^{t,m}_s(x),X^{t,m}_s\ot m, Z^{t,m}_s(x)\big)$ always exists according to Proposition \ref{A5}. Furthermore, Hypothesis ${\bf (h2)}$ implies that $\wb{l}_f L^* \leq \frac{1}{2}k_0\wb{l}_f \leq \frac{1}{20}\lambda_g$ and $\wb{l}_g \leq \frac{1}{8} \lambda_g$.
Denote $\Lambda_h:=\frac{1}{2}k_0\wb{l}_f+\Lambda_g\leq \frac{1}{20}\lambda_g+\Lambda_g$ and $\wb{l}_h:=\frac{1}{2}k_0\wb{l}_f+\wb{l}_g\leq \frac{1}{20}\lambda_g+\frac{1}{8}\lambda_g<\frac{1}{5}\lambda_g$. 

\textbf{Step 1} (Calculus identities and some elementary estimates) In this step, we shall first introduce a few calculus identities that will be useful in Step 3-4, as well as derive some elementary estimates that will be used in Step 2-5.
First of all, differentiating \eqref{first_order_condition} with respect to $x$, $\mu$ and $z$ respectively yields
\begin{align}\no
&\p_x\p_\alpha f(x,\mu,\alpha(x,\mu,z))\cdot z  + \p_x\p_\alpha g(x,\mu,\alpha(x,\mu,z))\\\label{p_x_foc}
&+\Big(\p_\alpha\p_\alpha f(x,\mu,\alpha(x,\mu,z))\cdot z  + \p_\alpha\p_\alpha g(x,\mu,\alpha(x,\mu,z))\Big)\cdot \p_x\alpha(x,\mu,z)=0,\\\no
&\p_\mu\p_\alpha f(x,\mu,\alpha(x,\mu,z))(\wt{x})\cdot z  + \p_\mu\p_\alpha g(x,\mu,\alpha(x,\mu,z))(\wt{x})\\\label{p_mu_foc}
&+\Big(\p_\alpha\p_\alpha f(x,\mu,\alpha(x,\mu,z))\cdot z  + \p_\alpha\p_\alpha g(x,\mu,\alpha(x,\mu,z))\Big)\cdot \p_\mu\alpha(x,\mu,z)(\wt{x})=0,\\\label{p_z_foc}
&\p_\alpha f(x,\mu,\alpha(x,\mu,z))+\Big(\p_\alpha\p_\alpha f(x,\mu,\alpha(x,\mu,z))\cdot z  + \p_\alpha\p_\alpha g(x,\mu,\alpha(x,\mu,z))\Big)\cdot \p_z\alpha(x,\mu,z)=0,
\end{align}
and hence,
\begin{align}\no
&\p_\alpha f(x,\mu,\alpha(x,\mu,z))\cdot \p_x\alpha(x,\mu,z)\\\label{eq_7_17_1}
&-\Big(\p_x\p_\alpha f(x,\mu,\alpha(x,\mu,z))\cdot z  + \p_x\p_\alpha g(x,\mu,\alpha(x,\mu,z))\Big)\cdot \p_z\alpha(x,\mu,z)=0,\\\no
&\p_\alpha f(x,\mu,\alpha(x,\mu,z))\cdot \p_\mu\alpha(x,\mu,z)(\wt{x})\\\label{eq_7_18_1}
&-\Big(\p_\mu\p_\alpha f(x,\mu,\alpha(x,\mu,z))(\wt{x})\cdot z  + \p_\mu\p_\alpha g(x,\mu,\alpha(x,\mu,z))(\wt{x})\Big)\cdot \p_z\alpha(x,\mu,z)=0.
\end{align}
Next, we are going to derive some elementary estimates and understand how these estimates depend on $L^*$ as follows. It follows from \eqref{bdd_d2_f}, \eqref{positive_g_alpha} and \eqref{eq_8_9_cone} that for any $\xi\in\R^{d_\alpha}$, 
\small\begin{align}\label{eq_8_11_new}
\xi^\top\Big((\p_\alpha\p_\alpha f)^{t,m}_s(x)\cdot Z^{t,m}_s(x)+(\p_\alpha\p_\alpha g)^{t,m}_s(x)\Big)\xi \geq \Big(\lambda_g-\wb{l}_fL^*\Big)|\xi|^2\geq \Big(\lambda_g-\frac{1}{2}k_0\wb{l}_f\Big)|\xi|^2\geq \frac{19}{20}\lambda_g|\xi|^2>0.
\end{align}\normalsize
To understand more on the role played by $L^*$ in our estimates, we prefer using the lower bound $\lambda_g-\wb{l}_fL^*$ to $\frac{19}{20}\lambda_g$ for estimating the lower bound of the term $(\p_\alpha\p_\alpha f)^{t,m}_s(x)\cdot Z^{t,m}_s(x)+(\p_\alpha\p_\alpha g)^{t,m}_s(x)$ that appears in \eqref{eq_5_25_new}-\eqref{eq_5_27_new}.
By \eqref{bdd_d1_f}, \eqref{bdd_d2_f}, \eqref{bdd_d2_g_2}, \eqref{eq_8_9_cone} and the first inequality in \eqref{eq_8_11_new}, we have
\begin{align}\label{p_xalpha_new}
&\Big\|(\p_x \alpha)^{t,m}_s(x)\Big\|_{\mathcal{L}(\R^{d_x};\R^{d_\alpha})}\vee\Big\|(\p_{\mu} \alpha)^{t,m}_s(x,\wt{x})\Big\|_{\mathcal{L}(\R^{d_x};\R^{d_\alpha})}\leq\frac{\wb{l}_h}{\lambda_g-\wb{l}_fL^*},\\\label{p_zalpha_new}
&\Big\|(\p_{z} \alpha)^{t,m}_s(x)\Big\|_{\mathcal{L}(\R^{d_x};\R^{d_\alpha})}\leq \frac{\Lambda_f}{\lambda_g-\wb{l}_fL^*}.
\end{align}\normalsize
Furthermore, applying \eqref{bdd_d1_f}, \eqref{p_xalpha_new} and \eqref{p_zalpha_new} to \eqref{calculus2} yields
\begin{align}\no
\Big|\p_m\big(\alpha^{t,m}_s(x)\big)(y)\Big|\leq&\ \frac{\wb{l}_h}{\lambda_g-\wb{l}_fL^*}\bigg(\Big|\p_m\big(X^{t,m}_s(x)\big)(y)\Big|+\Big|\p_{y}\big(X^{t,m}_s(y)\big)\Big|+\int_{\R^{d_x}}\Big|\p_m\big(X^{t,m}_s(\wt{x})\big)(y)\Big|dm(\wt{x})\bigg)\\\label{p_m_alpha_1}
&+\frac{\Lambda_f}{\lambda_g-\wb{l}_fL^*}\Big|\p_m\big(Z^{t,m}_s(x)\big)(y)\Big|;\\\label{p_x_alpha_1}
\Big|\p_x\big(\alpha^{t,m}_s(x)\big)\Big|\leq&\ \ \frac{\wb{l}_h}{\lambda_g-\wb{l}_fL^*}\Big|\p_x\big(X^{t,m}_s(x)\big)\Big|+\frac{\Lambda_f}{\lambda_g-\wb{l}_fL^*}\Big|\p_x\big(Z^{t,m}_s(x)\big)\Big|;\end{align}
\begin{align}\no
\Big|\p_m\big(f^{t,m}_s(x)\big)(y)\Big|\leq&\ \Lambda_f\bigg(\Big|\p_m\big(X^{t,m}_s(x)\big)(y)\Big|+\Big| \p_{y}\big(X^{t,m}_s(y)\big)\Big|+\int_{\R^{d_x}} \Big|\p_m\big(X^{t,m}_s(\wt{x})\big)(y)\Big| dm(\wt{x})\\\label{p_m_f_1}
&\ \ \ \ \ \ \ \ + \Big|\p_m\big(\alpha^{t,m}_s(x)\big)(y)\Big|\bigg);\\\label{p_x_f_1}
\Big|\p_x \big(f^{t,m}_s(x)\big)\Big|\leq&\ \Lambda_f\bigg(\Big|\p_x \big(X^{t,m}_s(x)\big)\Big|+ \Big|\p_x \big(\alpha^{t,m}_s(x)\big)\Big|\bigg).
\end{align}

\textbf{Step 2} (Preliminary estimates on derivatives of $\gamma$) 
For $i=1,2,3,...,d_x$, denote $\p_{m_i}\gamma(t,x,m)(y):=\p_{y_i}\frac{\delta}{\delta m}\gamma(t,x,m)(y)$. By using Assumptions ${\bf(a1)}$-${\bf(a3)}$ and Hypotheses ${\bf(h1)}$-${\bf(h2)}$, we can obtain, after using \eqref{eq_8_5_new}, for any $i=1,2,3,...,d_x$, 
\small\begin{align}\no
&\bigg|\p_{m_i}\gamma(t,x,m)(y)\bigg|^2=\bigg|\p_{m_i}\big(Z^{t,m}_t(x)\big)(y)\bigg|^2\\\no
=&\ \bigg|\p_{m_i}\big(Z^{t,m}_T(x)\big)(y)\bigg|^2-2\int_t^T \p_{m_i}\big(Z^{t,m}_s(x)\big)(y)\cdot \frac{d}{ds}\p_{m_i}\big(Z^{t,m}_s(x)\big)(y) ds\\\no
\leq &\ \bigg(3\Lambda_k\int_{\R^{d_x}}\Big|\p_{m_i}\big(X^{t,m}_T(\wt{x})\big)(y)\Big|dm(\wt{x})+2\Lambda_k\Big|\p_{m_i}\big(X^{t,m}_T(x)\big)(y)\Big|+2\Lambda_k\Big|\p_{y_i}\big(X^{t,m}_T(y)\big)\Big|\bigg)^2\\\no
&+2\int_t^T\bigg(3\Lambda_h \int_{\R^{d_x}}\Big|\p_{m_i}\big(X^{t,m}_s(\wt{x})\big)(y)\Big|dm(\wt{x})+2\Lambda_h\Big|\p_{m_i}\big(X^{t,m}_s(x)\big)(y)\Big|\\\no
&\ \ \ \ \ \ \ \ \ \ \ \ +2\Lambda_h\Big|\p_{y_i}\big(X^{t,m}_s(y)\big)\Big|+\wb{l}_h \int_{\R^{d_x}}\Big|\p_{m_i}\big(\alpha^{t,m}_s(\wt{x})\big)(y)\Big|dm(\wt{x})+\wb{l}_h \Big|\p_{m_i}\big(\alpha^{t,m}_s(x)\big)(y)\Big|\\\no
&\ \ \ \ \ \ \ \ \ \ \ \ +\Lambda_f\int_{\R^{d_x}}\Big|\p_{m_i}\big(Z^{t,m}_s(\wt{x})\big)(y)\Big|dm(\wt{x})+\Lambda_f\Big|\p_{m_i}\big(Z^{t,m}_s(x)\big)(y)\Big|\bigg)\cdot\Big|\p_{m_i}\big(Z^{t,m}_s(x)\big)(y)\Big|ds\\\no
&\text{(by using \eqref{bdd_d1_f}, \eqref{bdd_d2_f}, \eqref{bdd_d2_g_1}, \eqref{bdd_d2_g_2}, \eqref{bdd_d2_k_1} and \eqref{eq_8_9_cone})}\\\no
\leq &\ \bigg(3\Lambda_k \int_{\R^{d_x}}\Big|\p_{m_i}\big(X^{t,m}_T(\wt{x})\big)(y)\Big|dm(\wt{x})+2\Lambda_k\Big|\p_{m_i}\big(X^{t,m}_T(x)\big)(y)\Big|+2\Lambda_k\Big|\p_{y_i}\big(X^{t,m}_T(y)\big)\Big|\bigg)^2\\\no
&+2\int_t^T\bigg(\Big(3\Lambda_h+\frac{3\wb{l}_h^2}{\lambda_g-\wb{l}_fL^*}\Big) \int_{\R^{d_x}}\Big|\p_{m_i}\big(X^{t,m}_s(\wt{x})\big)(y)\Big|dm(\wt{x})+\Big(2\Lambda_h+\frac{\wb{l}_h^2}{\lambda_g-\wb{l}_fL^*}\Big)\Big|\p_{m_i}\big(X^{t,m}_s(x)\big)(y)\Big|\\\no
&\ \ \ \ \ \ \ \ \ \ \ \ +\Big(2\Lambda_h+\frac{2\wb{l}_h^2}{\lambda_g-\wb{l}_fL^*}\Big)\Big|\p_{y_i}\big(X^{t,m}_s(y)\big)\Big|+\Big(\Lambda_f+\wb{l}_h\frac{\Lambda_f}{\lambda_g-\wb{l}_fL^*}\Big)\int_{\R^{d_x}}\Big|\p_{m_i}\big(Z^{t,m}_s(\wt{x})\big)(y)\Big|dm(\wt{x})\\\no
&\ \ \ \ \ \ \ \ \ \ \ \ +\Big(\Lambda_f+\wb{l}_h\frac{\Lambda_f}{\lambda_g-\wb{l}_fL^*}\Big)\Big|\p_{m_i}\big(Z^{t,m}_s(x)\big)(y)\Big|\bigg)\cdot\Big|\p_{m_i}\big(Z^{t,m}_s(x)\big)(y)\Big|ds\\\no
&\text{(by using \eqref{p_m_alpha_1})}\\\no
\leq &\ 27\Lambda_k^2 \int_{\R^{d_x}}\Big|\p_{m_i}\big(X^{t,m}_T(\wt{x})\big)(y)\Big|^2dm(\wt{x})+12\Lambda_k^2\Big|\p_{m_i}\big(X^{t,m}_T(x)\big)(y)\Big|^2+12\Lambda_k^2\Big|\p_{y_i}\big(X^{t,m}_T(y)\big)\Big|^2\\\no
&+\int_t^T\Bigg(\Big(3\Lambda_h+\frac{3\wb{l}_h^2}{\lambda_g-\wb{l}_fL^*}\Big)\int_{\R^{d_x}}\Big|\p_{m_i}\big(X^{t,m}_s(\wt{x})\big)(y)\Big|^2dm(\wt{x})+\Big(2\Lambda_h+\frac{\wb{l}_h^2}{\lambda_g-\wb{l}_fL^*}\Big)\Big|\p_{m_i}\big(X^{t,m}_s(x)\big)(y)\Big|^2\\\no
&\ \ \ \ \ \ \ \ \ \ \ \ +\Big(2\Lambda_h+\frac{2\wb{l}_h^2}{\lambda_g-\wb{l}_fL^*}\Big)\Big|\p_{y_i}\big(X^{t,m}_s(y)\big)\Big|^2+\Big(\Lambda_f+\wb{l}_h\frac{\Lambda_f}{\lambda_g-\wb{l}_fL^*}\Big)\int_{\R^{d_x}}\Big|\p_{m_i}\big(Z^{t,m}_s(\wt{x})\big)(y)\Big|^2dm(\wt{x})\\\label{p_m_gamma}
&\ \ \ \ \ \ \ \ \ \ \ \ +\bigg(3\Lambda_f+3\wb{l}_h\frac{\Lambda_f}{\lambda_g-\wb{l}_fL^*}+7\Lambda_h+\frac{6\wb{l}_h^2}{\lambda_g-\wb{l}_fL^*}\bigg)\Big|\p_{m_i}\big(Z^{t,m}_s(x)\big)(y)\Big|^2\Bigg)ds
\end{align}\normalsize
and, after using \eqref{eq_8_6_new}, for any $i=1,2,3,...,d_x$,
\small\begin{align}\no
&\bigg|\p_{x_i}\gamma(t,x,m)\bigg|^2=\bigg|\p_{x_i}\big(Z^{t,m}_t(x)\big)\bigg|^2\\\no
=&\ \bigg|\p_{x_i}\big(Z^{t,m}_T(x)\big)\bigg|^2-2\int_t^T \p_{x_i}\big(Z^{t,m}_s(x)\big)\cdot \frac{d}{ds}\p_{x_i}\big(Z^{t,m}_s(x)\big) ds\\\no
\leq &\ 4\Lambda_k^2\Big|\p_{x_i}\big(X^{t,m}_T(x)\big)\Big|^2\\\no
&+2\int_t^T\bigg(2\Lambda_h\Big|\p_{x_i}\big(X^{t,m}_s(x)\big)\Big|+\wb{l}_h \Big|\p_{x_i}\big(\alpha^{t,m}_s(x)\big)\Big|+\Lambda_f\Big|\p_{x_i}\big(Z^{t,m}_s(x)\big)\Big|\bigg)\cdot\Big|\p_{x_i}\big(Z^{t,m}_s(x)\big)\Big|ds\\\no
&\text{(by using \eqref{bdd_d1_f}, \eqref{bdd_d2_f}, \eqref{bdd_d2_g_1}, \eqref{bdd_d2_g_2}, \eqref{bdd_d2_k_1} and \eqref{eq_8_9_cone})}\\\no
\leq &\ 4\Lambda_k^2\Big|\p_{x_i}\big(X^{t,m}_T(x)\big)\Big|^2\\\no
&+2\int_t^T\Bigg(\bigg(2\Lambda_h+\frac{\wb{l}_h^2}{\lambda_g-\wb{l}_fL^*}\bigg)\Big|\p_{x_i}\big(X^{t,m}_s(x)\big)\Big|+\bigg(\wb{l}_h\frac{\Lambda_f}{\lambda_g-\wb{l}_fL^*}+\Lambda_f\bigg)\Big|\p_{x_i}\big(Z^{t,m}_s(x)\big)\Big|\Bigg)\cdot\Big|\p_{x_i}\big(Z^{t,m}_s(x)\big)\Big|ds\\\no
&\text{(by using \eqref{p_x_alpha_1})}\\\no
\leq &\ 4\Lambda_k^2\Big|\p_{x_i}\big(X^{t,m}_T(x)\big)\Big|^2+\bigg(2\Lambda_h+\frac{\wb{l}_h ^2}{\lambda_g-\wb{l}_fL^*}\bigg)\int_t^T \Big|\p_{x_i}\big(X^{t,m}_s(x)\big)\Big|^2ds\\\label{p_x_gamma}
&+\bigg(2\wb{l}_h\frac{\Lambda_f}{\lambda_g-\wb{l}_fL^*}+2\Lambda_f+2\Lambda_h+ \frac{\wb{l}_h ^2}{\lambda_g-\wb{l}_fL^*}\bigg)\int_t^T\Big|\p_{x_i}\big(Z^{t,m}_s(x)\big)\Big|^2ds.
\end{align}\normalsize
Then, integrating \eqref{p_m_gamma} with respect to $x$, we have, for any $i=1,2,3,...,d_x$, 
\small\begin{align}\no
&\bigg\|\p_{m_i}\gamma(t,\cdot,m)(y)\bigg\|_{L^{2,d_x}_m}^2=\bigg\|\p_{m_i}\big(Z^{t,m}_t(\cdot)\big)(y)\bigg\|_{L^{2,d_x}_m}^2\\\no
\leq &\ 39\Lambda_k^2\Big\|\p_{m_i}\big(X^{t,m}_T(\cdot)\big)(y)\Big\|_{L^{2,d_x}_m}^2+12\Lambda_k^2\Big|\p_{y_i}\big(X^{t,m}_T(y)\big)\Big|^2\\\no
&+\int_t^T\Bigg(\Big(5\Lambda_h+\frac{4\wb{l}_h^2}{\lambda_g-\wb{l}_fL^*}\Big)\Big\|\p_{m_i}\big(X^{t,m}_s(\cdot)\big)(y)\Big\|_{L^{2,d_x}_m}^2+\Big(2\Lambda_h+\frac{2\wb{l}_h^2}{\lambda_g-\wb{l}_fL^*}\Big)\Big|\p_{y_i}\big(X^{t,m}_s(y)\big)\Big|^2\\\label{p_m_gamma_int}
&\ \ \ \ \ \ \ \ \ \ \ \ +\bigg(4\Lambda_f+4\wb{l}_h\frac{\Lambda_f}{\lambda_g-\wb{l}_fL^*}+7\Lambda_h+\frac{6\wb{l}_h^2}{\lambda_g-\wb{l}_fL^*}\bigg)\Big\|\p_{m_i}\big(Z^{t,m}_s(\cdot)\big)(y)\Big\|_{L^{2,d_x}_m}^2\Bigg)ds.
\end{align}\normalsize

To bound the right hand sides of \eqref{p_m_gamma}-\eqref{p_m_gamma_int}, we need the following major estimates.

\textbf{Step 3} (Major  estimate 1 of the following \eqref {L_star_1_1}) For any $i=1,2,3,...,d_x$, using \eqref{eq_8_6_new} yields
\scriptsize\begin{align*}
&\p_{x_i}\big(Z^{t,m}_T(x)\big)\cdot \p_{x_i}\big(X^{t,m}_T(x)\big)-\p_{x_i}\big(Z^{t,m}_t(x)\big)\cdot \p_{x_i}\big(X^{t,m}_t(x)\big) \\
=&\ \int_t^T \bigg(\p_{x_i}\big(Z^{t,m}_s(x)\big) \cdot \frac{d}{ds}\p_{x_i}\big(X^{t,m}_s(x)\big) +\p_{x_i}\big(X^{t,m}_s(x)\big) \cdot \frac{d}{ds}\p_{x_i}\big(Z^{t,m}_s(x)\big) \bigg)ds\\
=&\ \int_t^T \Bigg(\p_{x_i}\big(Z^{t,m}_s(x)\big) \cdot \bigg({\color{green}\big(\p_x f\big)^{t,m}_s(x)\cdot\p_{x_i} \big(X^{t,m}_s(x)\big)}+ \big(\p_\alpha f\big)^{t,m}_s(x)\cdot \Big({\color{green}\big(\p_x  \alpha\big)^{t,m}_s(x)\cdot\p_{x_i}\big(X^{t,m}_s(x)\big)}+\big(\p_z \alpha\big)^{t,m}_s(x)\cdot\p_{x_i}\big(Z^{t,m}_s(x)\big)\Big)\bigg)\\
&\ \ \ \ \ \ \ \ -\p_{x_i}\big(X^{t,m}_s(x)\big)^\top\int_{\R^{d_x}}\bigg(\big(\p_{\wt{x}}\p_\mu f\big)^{t,m}_s(\wt{x},x)\cdot  Z^{t,m}_s(\wt{x})+ \big(\p_{\wt{x}}\p_\mu g\big)^{t,m}_s(\wt{x},x)\bigg)dm(\wt{x})\p_{x_i}\big(X^{t,m}_s(x)\big)\\
&\ \ \ \ \ \ \ \ {\color{green}-\p_{x_i}\big(X^{t,m}_s(x)\big)^\top\big(\p_x f\big)^{t,m}_s(x)\cdot \p_{x_i}\big(Z^{t,m}_s(x)\big)}-\p_{x_i}\big(X^{t,m}_s(x)\big)^\top\bigg(\big(\p_x\p_x f\big)^{t,m}_s(x)\cdot Z^{t,m}_s(x)+\big(\p_x\p_x g\big)^{t,m}_s(x)\bigg)\p_{x_i}\big(X^{t,m}_s(x)\big)\\
&\ \ \ \ \ \ \ \ -\Big(\big(\p_x  \alpha\big)^{t,m}_s(x)\cdot\p_{x_i}\big(X^{t,m}_s(x)\big){\color{green}+\big(\p_z \alpha\big)^{t,m}_s(x)\cdot\p_{x_i}\big(Z^{t,m}_s(x)\big)}\Big)^\top\bigg(\big(\p_\alpha\p_x f\big)^{t,m}_s(x)\cdot Z^{t,m}_s(x)+\big(\p_\alpha\p_x g\big)^{t,m}_s(x)\bigg)\p_{x_i}\big(X^{t,m}_s(x)\big)\Bigg)ds\\
&\text{(by using \eqref{eq_7_17_1}, \eqref{calculus2} and we can cancel the terms in green)}\\
\leq&\ \int_t^T \Bigg(-\frac{\lambda_f^2}{\Lambda_h}\Big|\p_{x_i}\big(Z^{t,m}_s(x)\big) \Big| ^2-\big(\lambda_g-k_0\wb{l}_f-\frac{\wb{l}_h^2}{\lambda_g-\wb{l}_fL^*}\big)\Big|\p_{x_i}\big(X^{t,m}_s(x)\big)\Big|^2\Bigg)ds\\
&\text{(by using \eqref{bdd_d2_f}, \eqref{positive_g_x}, \eqref{bdd_d2_g_2}, \eqref{c_a7}, \eqref{eq_8_9_cone} and \eqref{p_xalpha_new}),}
\end{align*}\normalsize
which implies, by using \eqref{positive_k} on the left hand side,
\begin{align}\no
&\lambda_k\Big|\p_{x_i}\big(X^{t,m}_T(x)\big)\Big|^2-\p_{x_i}\big(Z^{t,m}_t(x)\big)_i\\\no
\leq&\ \int_t^T \Bigg(-\frac{\lambda_f^2}{\Lambda_h}\Big|\p_{x_i}\big(Z^{t,m}_s(x)\big) \Big| ^2-\bigg(\lambda_g-k_0\wb{l}_f-\frac{\wb{l}_h^2}{\lambda_g-\wb{l}_fL^*}\bigg)\Big|\p_{x_i}\big(X^{t,m}_s(x)\big)\Big|^2\Bigg)ds,\\\label{cru_est_1}
\leq&\ \int_t^T \Bigg(-\lambda_z\Big|\p_{x_i}\big(Z^{t,m}_s(x)\big) \Big| ^2-\lambda_x\Big|\p_{x_i}\big(X^{t,m}_s(x)\big)\Big|^2\Bigg)ds,
\end{align}\normalsize
where $\big(Z^{t,m}_t(x)\big)_i$ is the $i$-th component of $Z^{t,m}_t(x)$ and
\begin{align}\label{lambda_z}
\lambda_z:=&\ \frac{\lambda_f^2}{\Lambda_g+\lambda_g/20}\leq \frac{\lambda_f^2}{\Lambda_h},\\\label{lambda_x}
\lambda_x:=&\ \frac{17}{20}\lambda_g=\lambda_g-\frac{1}{10}\lambda_g-\frac{1}{20}\lambda_g\leq \lambda_g-k_0\wb{l}_f-\frac{\wb{l}_h^2}{\lambda_g-\wb{l}_fL^*}.
\end{align}
Thus, by using \eqref{p_x_gamma} and \eqref{cru_est_1}, we can obtain
\begin{align}\no
&\bigg|\p_{x_i}\gamma(t,x,m)\bigg|^2=\bigg|\p_{x_i}\big(Z^{t,m}_t(x)\big)\bigg|^2\\\no
\leq &\ \frac{4\Lambda_k^2}{\lambda_k}\lambda_k\Big|\p_{x_i}\big(X^{t,m}_T(x)\big)\Big|^2+\frac{1}{\lambda_z}\bigg(2\Lambda_h+\frac{\wb{l}_h ^2}{\lambda_g-\wb{l}_fL^*}\bigg)\lambda_z\int_t^T \Big|\p_{x_i}\big(X^{t,m}_s(x)\big)\Big|^2ds\\\no
&+\frac{1}{\lambda_x}\bigg(2\wb{l}_h\frac{\Lambda_f}{\lambda_g-\wb{l}_fL^*}+2\Lambda_f+2\Lambda_h+ \frac{\wb{l}_h ^2}{\lambda_g-\wb{l}_fL^*}\bigg)\lambda_x\int_t^T\Big|\p_{x_i}\big(Z^{t,m}_s(x)\big)\Big|^2ds\\\no
\leq &\ \max\bigg\{\frac{4\Lambda_k^2}{\lambda_k},\frac{1}{\lambda_z}\bigg(2\Lambda_h+\frac{\wb{l}_h ^2}{\lambda_g-\wb{l}_fL^*}\bigg),\frac{1}{\lambda_x}\bigg(2\wb{l}_h\frac{\Lambda_f}{\lambda_g-\wb{l}_fL^*}+2\Lambda_f+2\Lambda_h+ \frac{\wb{l}_h ^2}{\lambda_g-\wb{l}_fL^*}\bigg)\bigg\}\Big|\p_{x_i}\big(Z^{t,m}_t(x)\big)\Big|,
\end{align}\normalsize
which implies the following {\it a priori} estimate, for any $i=1,2,3,...,d_x$,
\begin{align}\no
&\bigg|\p_{x_i}\gamma(t,x,m)\bigg|=\bigg|\p_{x_i}\big(Z^{t,m}_t(x)\big)\bigg|\\\no
\leq &\ \max\bigg\{\frac{4\Lambda_k^2}{\lambda_k},\frac{1}{\lambda_z}\bigg(2\Lambda_h+\frac{\wb{l}_h ^2}{\lambda_g-\wb{l}_fL^*}\bigg),\frac{1}{\lambda_x}\bigg(2\wb{l}_h\frac{\Lambda_f}{\lambda_g-\wb{l}_fL^*}+2\Lambda_f+2\Lambda_h+ \frac{\wb{l}_h ^2}{\lambda_g-\wb{l}_fL^*}\bigg)\bigg\}\\\label{L_star_1}
\leq &\ \max\bigg\{\frac{4\Lambda_k^2}{\lambda_k},\frac{1}{\lambda_z}\bigg(2\Lambda_h+\frac{1}{20}\lambda_g\bigg),\frac{1}{\lambda_x}\bigg(\frac{5}{2}\Lambda_f+2\Lambda_h+\frac{1}{20}\lambda_g\bigg)\bigg\}=:L^*_1,
\end{align}\normalsize
since $\wb{l}_fL^*\leq \frac{1}{20} \lambda_g$, $\wb{l}_h<\frac{1}{5}\lambda_g$ and $\frac{\wb{l}_h}{\lambda_g-\wb{l}_fL^*}\leq \frac{1/5\lambda_g}{\lambda_g-\lambda_g/20}<\frac{1}{4}$. 
Therefore, we finally have
\begin{align}\label{L_star_1_1}
&\lambda_k\Big|\p_{x_i}\big(X^{t,m}_T(x)\big)\Big|^2+\int_t^T \Bigg(\lambda_z\Big|\p_{x_i}\big(Z^{t,m}_s(x)\big) \Big| ^2+\lambda_x\Big|\p_{x_i}\big(X^{t,m}_s(x)\big)\Big|^2\Bigg)ds\leq L^*_1.
\end{align}\normalsize

\textbf{Step 4} (Major estimate 2 of the following \eqref {cru_est_2}) For any $i=1,2,3,...,d_x$ and any $y\in\R^{d_x}$, using \eqref{eq_8_5_new} yields
\footnotesize\begin{align*}
&\p_{m_i}\big(Z^{t,m}_T(x)\big)(y)\cdot \p_{m_i}\big(X^{t,m}_T(x)\big)(y)-\p_{m_i}\big(Z^{t,m}_t(x)\big)(y)\cdot \p_{m_i}\big(X^{t,m}_t(x)\big)(y)\\
=&\ \int_t^T \bigg(\p_{m_i}\big(Z^{t,m}_s(x)\big)(y)\cdot \frac{d}{ds}\p_{m_i}\big(X^{t,m}_s(x)\big)(y)+\p_{m_i}\big(X^{t,m}_s(x)\big)(y)\cdot \frac{d}{ds}\p_{m_i}\big(Z^{t,m}_s(x)\big)(y)\bigg)ds\\
=&\ \int_t^T \Bigg(\p_{m_i}\big(Z^{t,m}_s(x)\big)(y)\cdot \bigg({\color{green}\big(\p_x f\big)^{t,m}_s(x) \p_{m_i}\big(X^{t,m}_s(x)\big)(y)}+\big(\p_\mu f\big)^{t,m}_s(x,y)  \p_{y_i}\big(X^{t,m}_s(y)\big)\\
&\ \ \ \ \ \ \  \ \ \ \ \ \ \  \ \ \ \ \ \ \ \ \ \ \ \ \ \ \ \ \ \ \ \ \ {\color{orange}+\int_{\R^{d_x}} \big(\p_\mu f\big)^{t,m}_s(x,\wh{x}) \p_{m_i}\big(X^{t,m}_s(\wh{x})\big)(y) dm(\wh{x})}\\
&\ \ \ \ \ \ \  \ \ \ \ \ \ \  \ \ \ \ \ \ \ \ \ \ \ \ \ \ \ \ \ \ \ \ \ + \big(\p_\alpha f\big)^{t,m}_s(x)\cdot \Big({\color{green}\big(\p_x \alpha\big)^{t,m}_s(x)\cdot\p_{m_i}\big(X^{t,m}_s(x)\big)(y)}+\big(\p_\mu \alpha\big)^{t,m}_s(x,y)\cdot \p_{y_i}\big(X^{t,m}_s(y)\big)\\
&\ \ \ \ \ \ \  \ \ \ \ \ \ \  \ \ \ \ \ \ \ \ \ \ \ \ \ \ \ \ \ \ \ \ \ \ \ \ \ \ \ \ \ \ \ \ \ \ \ \ \ \ \ \ \ \ \ {\color{orange}+\int_{\R^{d_x}}\big(\p_\mu \alpha\big)^{t,m}_s(x,\wh{x})\cdot \p_{m_i}\big(X^{t,m}_s(\wh{x})\big)(y)dm(\wh{x})}+\big(\p_z \alpha\big)^{t,m}_s(x)\cdot\p_{m_i}\big(Z^{t,m}_s(x)\big)(y)\Big)\bigg)\\
&\ \ \ \ \ \ \ \ -\int_{\R^{d_x}}\p_{m_i}\big(X^{t,m}_s(\wt{x})\big)(y)^\top\bigg(\big(\p_x\p_\mu f\big)^{t,m}_s(\wt{x},x)\cdot  Z^{t,m}_s(\wt{x})+ \big(\p_x\p_\mu g\big)^{t,m}_s(\wt{x},x)\bigg)dm(\wt{x})\p_{m_i}\big(X^{t,m}_s(x)\big)(y)\\
&\ \ \ \ \ \ \ \ -\p_{m_i}\big(X^{t,m}_s(x)\big)(y)^\top\int_{\R^{d_x}}\bigg(\big(\p_{\wt{x}}\p_\mu f\big)^{t,m}_s(\wt{x},x)\cdot  Z^{t,m}_s(\wt{x})+ \big(\p_{\wt{x}}\p_\mu g\big)^{t,m}_s(\wt{x},x)\bigg)dm(\wt{x})\p_{m_i}\big(X^{t,m}_s(x)\big)(y)\\
&\ \ \ \ \ \ \ \ -\int_{\R^{d_x}}\int_{\R^{d_x}}\p_{m_i}\big(X^{t,m}_s(\wh{x})\big)(y)^\top\bigg(\big(\p_\mu\p_\mu f\big)^{t,m}_s(\wt{x},x,\wh{x})\cdot  Z^{t,m}_s(\wt{x})+ \big(\p_\mu\p_\mu g\big)^{t,m}_s(\wt{x},x,\wh{x})\bigg)dm(\wh{x})dm(\wt{x})\p_{m_i}\big(X^{t,m}_s(x)\big)(y)\\
&\ \ \ \ \ \ \ \ -\int_{\R^{d_x}}\p_{y_i}\big(X^{t,m}_s(y)\big)^\top\bigg(\big(\p_\mu\p_\mu f\big)^{t,m}_s(\wt{x},x,y)\cdot  Z^{t,m}_s(\wt{x})+ \big(\p_\mu\p_\mu g\big)^{t,m}_s(\wt{x},x,y)\bigg)dm(\wt{x})\p_{m_i}\big(X^{t,m}_s(x)\big)(y)\\
&\ \ \ \ \ \ \ \ -\int_{\R^{d_x}}\bigg(\big(\p_x  \alpha\big)^{t,m}_s(\wt{x})\cdot\p_{m_i}\big(X^{t,m}_s(\wt{x})\big)(y)+\big(\p_\mu \alpha\big)^{t,m}_s(\wt{x},y)\cdot \p_{y_i}\big(X^{t,m}_s(y)\big)\\
&\ \ \ \ \ \ \ \ \ \ \ \ \ \  \ \ \ \ \ \ \ +\int_{\R^{d_x}}\big(\p_\mu \alpha\big)^{t,m}_s(\wt{x},\wh{x})\cdot \p_{m_i}\big(X^{t,m}_s(\wh{x})\big)(y)dm(\wh{x}){\color{orange}+\big(\p_z \alpha\big)^{t,m}_s(\wt{x})\cdot\p_{m_i}\big(Z^{t,m}_s(\wt{x})\big)(y)}\bigg)\\
&\ \ \ \ \ \ \ \ \ \ \ \ \ \ \ \ \cdot\bigg(\big(\p_\alpha\p_\mu f\big)^{t,m}_s(\wt{x},x)\cdot  Z^{t,m}_s(\wt{x})+ \big(\p_\alpha\p_\mu g\big)^{t,m}_s(\wt{x},x)\bigg)dm(\wt{x})\p_{m_i}\big(X^{t,m}_s(x)\big)(y)\\
&\ \ \ \ \ \ \ \ {\color{orange}-\p_{m_i}\big(X^{t,m}_s(x)\big)(y)^\top\int_{\R^{d_x}} \big(\p_\mu f\big)^{t,m}_s(\wt{x},x)\cdot  \p_{m_i}\big(Z^{t,m}_s(\wt{x})(y)\big)dm(\wt{x})}{\color{green}-\p_{m_i}\big(X^{t,m}_s(x)\big)(y)^\top\big(\p_x f\big)^{t,m}_s(x)\cdot \p_{m_i}\big(Z^{t,m}_s(x)\big)(y)}\\
&\ \ \ \ \ \ \ \ -\p_{m_i}\big(X^{t,m}_s(x)\big)(y)^\top\bigg(\big(\p_x\p_x f\big)^{t,m}_s(x)\cdot Z^{t,m}_s(x)+\big(\p_x\p_x g\big)^{t,m}_s(x)\bigg)\p_{m_i}\big(X^{t,m}_s(x)\big)(y)\\
&\ \ \ \ \ \ \ \ -\int_{\R^{d_x}}\p_{m_i}\big(X^{t,m}_s(\wt{x})\big)(y)^\top\bigg(\big(\p_\mu\p_x f\big)^{t,m}_s(x,\wt{x})\cdot Z^{t,m}_s(x)+\big(\p_\mu\p_x g\big)^{t,m}_s(x,\wt{x})\bigg)dm(\wt{x})\p_{m_i}\big(X^{t,m}_s(x)\big)(y)\\
&\ \ \ \ \ \ \ \ -\p_{y_i}\big(X^{t,m}_s(y)\big)^\top\bigg(\big(\p_\mu\p_x f\big)^{t,m}_s(x,y)\cdot Z^{t,m}_s(x)+\big(\p_\mu\p_x g\big)^{t,m}_s(x,y)\bigg)\p_{m_i}\big(X^{t,m}_s(x)\big)(y)\\
&\ \ \ \ \ \ \ \ -\bigg(\big(\p_x  \alpha\big)^{t,m}_s(x)\cdot\p_{m_i}\big(X^{t,m}_s(x)\big)(y)+\big(\p_\mu \alpha\big)^{t,m}_s(x,y)\cdot \p_{y_i}\big(X^{t,m}_s(y)\big)\\
&\ \ \ \ \ \ \ \ \ \ \ \ \ \  \ \ \ \ \ \ \ +\int_{\R^{d_x}}\big(\p_\mu \alpha\big)^{t,m}_s(x,\wh{x})\cdot \p_{m_i}\big(X^{t,m}_s(\wh{x})\big)(y)dm(\wh{x}){\color{green}+\big(\p_z \alpha\big)^{t,m}_s(x)\cdot\p_{m_i}\big(Z^{t,m}_s(x)\big)(y)}\bigg)\\
&\ \ \ \ \ \ \ \ \ \ \ \ \cdot\bigg(\big(\p_\alpha\p_x f\big)^{t,m}_s(x)\cdot Z^{t,m}_s(x)+\big(\p_\alpha\p_x g\big)^{t,m}_s(x)\bigg)\p_{m_i}\big(X^{t,m}_s(x)\big)(y)\Bigg)ds\\
&\text{(by using \eqref{calculus2}, \eqref{eq_7_17_1} and we can cancel the terms in green)}
\\
\leq &\ \int_t^T \Bigg(\Lambda_f\Big(1+\frac{\wb{l}_h}{\lambda_g-\wb{l}_fL^*}\Big)\Big|\p_{m_i}\big(Z^{t,m}_s(x)\big)(y)\Big|\Big| \p_{y_i}\big(X^{t,m}_s(y)\big)\Big|{\color{orange}+\p_{m_i}\big(Z^{t,m}_s(x)\big)(y)\cdot\int_{\R^{d_x}} \big(\p_\mu f\big)^{t,m}_s(x,\wh{x}) \p_{m_i}\big(X^{t,m}_s(\wh{x})\big)(y) dm(\wh{x})}\\
&\ \ \ \ \ \ \ \ {\color{orange}+\p_{m_i}\big(Z^{t,m}_s(x)\big)(y)\cdot\big(\p_\alpha f\big)^{t,m}_s(x)\int_{\R^{d_x}}\big(\p_\mu \alpha\big)^{t,m}_s(x,\wh{x})\p_{m_i}\big(X^{t,m}_s(\wh{x})\big)(y)dm(\wh{x})}\\
&\ \ \ \ \ \ \ \ -\lambda_z\Big|\p_{m_i}\big(Z^{t,m}_s(x)\big)(y)\Big|^2-\lambda_x\Big|\p_{m_i}\big(X^{t,m}_s(x)\big)(y)\Big|^2\\
&\ \ \ \ \ \ \ \ +\frac{3\wb{l}_h^2}{\lambda_g-\wb{l}_fL^*}\int_{\R^{d_x}}\Big|\p_{m_i}\big(X^{t,m}_s(\wt{x})\big)(y)\Big|dm(\wt{x})\Big|\p_{m_i}\big(X^{t,m}_s(x)\big)(y)\Big|\\
&\ \ \ \ \ \ \ \ -\int_{\R^{d_x}}\p_{m_i}\big(X^{t,m}_s(\wt{x})\big)(y)^\top\bigg(\big(\p_x\p_\mu f\big)^{t,m}_s(\wt{x},x)\cdot  Z^{t,m}_s(\wt{x})+ \big(\p_x\p_\mu g\big)^{t,m}_s(\wt{x},x)\bigg)dm(\wt{x})\p_{m_i}\big(X^{t,m}_s(x)\big)(y)\\
&\ \ \ \ \ \ \ \ -\int_{\R^{d_x}}\int_{\R^{d_x}}\p_{m_i}\big(X^{t,m}_s(\wh{x})\big)(y)^\top\bigg(\big(\p_\mu\p_\mu f\big)^{t,m}_s(\wt{x},x,\wh{x})\cdot  Z^{t,m}_s(\wt{x})+ \big(\p_\mu\p_\mu g\big)^{t,m}_s(\wt{x},x,\wh{x})\bigg)dm(\wh{x})dm(\wt{x})\p_{m_i}\big(X^{t,m}_s(x)\big)(y)\end{align*}
\begin{align*}
&\ \ \ \ \ \ \ \ -\int_{\R^{d_x}}\p_{m_i}\big(X^{t,m}_s(\wt{x})\big)(y)^\top\bigg(\big(\p_\mu\p_x f\big)^{t,m}_s(x,\wt{x})\cdot Z^{t,m}_s(x)+\big(\p_\mu\p_x g\big)^{t,m}_s(x,\wt{x})\bigg)dm(\wt{x})\p_{m_i}\big(X^{t,m}_s(x)\big)(y)\\
&\ \ \ \ \ \ \ \ +\Big(2\Lambda_h+\frac{2\wb{l}_h^2}{\lambda_g-\wb{l}_fL^*}\Big)\Big|\p_{y_i}\big(X^{t,m}_s(y)\big)\Big|\Big|\p_{m_i}\big(X^{t,m}_s(x)\big)(y)\Big|\\
&\ \ \ \ \ \ \ \ {\color{orange}-\int_{\R^{d_x}}\bigg(\big(\p_z \alpha\big)^{t,m}_s(\wt{x})\p_{m_i}\big(Z^{t,m}_s(\wt{x})\big)(y)\bigg)\cdot\bigg(\big(\p_\alpha\p_\mu f\big)^{t,m}_s(\wt{x},x)\cdot  Z^{t,m}_s(\wt{x})+ \big(\p_\alpha\p_\mu g\big)^{t,m}_s(\wt{x},x)\bigg)dm(\wt{x})\p_{m_i}\big(X^{t,m}_s(x)\big)(y)}\\
&\ \ \ \ \ \ \ \ {\color{orange}-\p_{m_i}\big(X^{t,m}_s(x)\big)(y)^\top\int_{\R^{d_x}} \big(\p_\mu f\big)^{t,m}_s(\wt{x},x)\cdot  \p_{m_i}\big(Z^{t,m}_s(\wt{x})(y)\big)dm(\wt{x})}\Bigg)ds\\
&\text{(by using \eqref{bdd_d1_f}, \eqref{bdd_d2_f}, \eqref{positive_g_x}, \eqref{c_a7}, \eqref{eq_8_9_cone}, \eqref{p_xalpha_new}, \eqref{lambda_z} and \eqref{lambda_x}),}
\end{align*}\normalsize
which implies, by using \eqref{positive_k_mu} and \eqref{bdd_d2_k_1} on the left hand side,
\footnotesize\begin{align}\no
&\lambda_k\Big|\p_{m_i}\big(X^{t,m}_T(x)\big)(y)\Big|^2-2\Lambda_k\Big|\p_{y_i}\big(X^{t,m}_T(y)\big)\Big|\Big|\p_{m_i}\big(X^{t,m}_T(x)\big)(y)\Big|\\\no
&+\int_{\R^{d_x}} \p_{m_i}\big(X^{t,m}_T(x)\big)(y)^\top\big(\p_x\p_\mu k\big)^{t,m}_T(\wt{x},x)\p_{m_i}\big(X^{t,m}_T(\wt{x})\big)(y)dm(\wt{x})\\\no
&+\int_{\R^{d_x}} \p_{m_i}\big(X^{t,m}_T(x)\big)(y)^\top\big(\p_\mu\p_x k\big)^{t,m}_T(x,\wt{x})\p_{m_i}\big(X^{t,m}_T(\wt{x})\big)(y)dm(\wt{x})\\\no
&+\int_{\R^{d_x}} \int_{\R^{d_x}}\p_{m_i}\big(X^{t,m}_T(x)\big)(y)^\top\big(\p_\mu\p_\mu k\big)^{t,m}_T(\wt{x},x,\wh{x}) \p_{m_i}\big(X^{t,m}_T(\wh{x})\big)(y)dm(\wh{x})dm(\wt{x})\\\no
\leq &\ \int_t^T \Bigg(\Lambda_f\Big(1+\frac{\wb{l}_h}{\lambda_g-\wb{l}_fL^*}\Big)\Big|\p_{m_i}\big(Z^{t,m}_s(x)\big)(y)\Big|\Big| \p_{y_i}\big(X^{t,m}_s(y)\big)\Big|{\color{orange}+\p_{m_i}\big(Z^{t,m}_s(x)\big)(y)\cdot\int_{\R^{d_x}} \big(\p_\mu f\big)^{t,m}_s(x,\wh{x}) \p_{m_i}\big(X^{t,m}_s(\wh{x})\big)(y) dm(\wh{x})}\\\no
&\ \ \ \ \ \ \ \ {\color{orange}+\p_{m_i}\big(Z^{t,m}_s(x)\big)(y)\cdot\big(\p_\alpha f\big)^{t,m}_s(x)\int_{\R^{d_x}}\big(\p_\mu \alpha\big)^{t,m}_s(x,\wh{x})\p_{m_i}\big(X^{t,m}_s(\wh{x})\big)(y)dm(\wh{x})}\\\no
&\ \ \ \ \ \ \ \ -\lambda_z\Big|\p_{m_i}\big(Z^{t,m}_s(x)\big)(y)\Big|^2-\lambda_x\Big|\p_{m_i}\big(X^{t,m}_s(x)\big)(y)\Big|^2\\\no
&\ \ \ \ \ \ \ \ +\frac{3\wb{l}_h^2}{\lambda_g-\wb{l}_fL^*}\int_{\R^{d_x}}\Big|\p_{m_i}\big(X^{t,m}_s(\wt{x})\big)(y)\Big|dm(\wt{x})\Big|\p_{m_i}\big(X^{t,m}_s(x)\big)(y)\Big|\\\no
&\ \ \ \ \ \ \ \ -\int_{\R^{d_x}}\p_{m_i}\big(X^{t,m}_s(\wt{x})\big)(y)^\top\bigg(\big(\p_x\p_\mu f\big)^{t,m}_s(\wt{x},x)\cdot  Z^{t,m}_s(\wt{x})+ \big(\p_x\p_\mu g\big)^{t,m}_s(\wt{x},x)\bigg)dm(\wt{x})\p_{m_i}\big(X^{t,m}_s(x)\big)(y)\\\no
&\ \ \ \ \ \ \ \ -\int_{\R^{d_x}}\int_{\R^{d_x}}\p_{m_i}\big(X^{t,m}_s(\wh{x})\big)(y)^\top\bigg(\big(\p_\mu\p_\mu f\big)^{t,m}_s(\wt{x},x,\wh{x})\cdot  Z^{t,m}_s(\wt{x})+ \big(\p_\mu\p_\mu g\big)^{t,m}_s(\wt{x},x,\wh{x})\bigg)dm(\wh{x})dm(\wt{x})\p_{m_i}\big(X^{t,m}_s(x)\big)(y)\\\no
&\ \ \ \ \ \ \ \ -\int_{\R^{d_x}}\p_{m_i}\big(X^{t,m}_s(\wt{x})\big)(y)^\top\bigg(\big(\p_\mu\p_x f\big)^{t,m}_s(x,\wt{x})\cdot Z^{t,m}_s(x)+\big(\p_\mu\p_x g\big)^{t,m}_s(x,\wt{x})\bigg)dm(\wt{x})\p_{m_i}\big(X^{t,m}_s(x)\big)(y)\\\no
&\ \ \ \ \ \ \ \ +\Big(2\Lambda_h+\frac{2\wb{l}_h^2}{\lambda_g-\wb{l}_fL^*}\Big)\Big|\p_{y_i}\big(X^{t,m}_s(y)\big)\Big|\Big|\p_{m_i}\big(X^{t,m}_s(x)\big)(y)\Big|\\\no
&\ \ \ \ \ \ \ \ {\color{orange}-\int_{\R^{d_x}}\bigg(\big(\p_z \alpha\big)^{t,m}_s(\wt{x})\p_{m_i}\big(Z^{t,m}_s(\wt{x})\big)(y)\bigg)\cdot\bigg(\big(\p_\alpha\p_\mu f\big)^{t,m}_s(\wt{x},x)\cdot  Z^{t,m}_s(\wt{x})+ \big(\p_\alpha\p_\mu g\big)^{t,m}_s(\wt{x},x)\bigg)dm(\wt{x})\p_{m_i}\big(X^{t,m}_s(x)\big)(y)}\\\label{eq_7_26_1}
&\ \ \ \ \ \ \ \ {\color{orange}-\p_{m_i}\big(X^{t,m}_s(x)\big)(y)^\top\int_{\R^{d_x}} \big(\p_\mu f\big)^{t,m}_s(\wt{x},x)\cdot  \p_{m_i}\big(Z^{t,m}_s(\wt{x})(y)\big)dm(\wt{x})}\Bigg)ds.
\end{align}\normalsize
Then, integrating \eqref{eq_7_26_1} with respective to $x$, using \eqref{eq_7_18_1} and then canceling the terms in orange, we have
\scriptsize\begin{align*}
&\lambda_k\Big\|\p_{m_i}\big(X^{t,m}_T(\cdot)\big)(y)\Big\|_{L^{2,d_x}_m}^2-2\Lambda_k\Big|\p_{y_i}\big(X^{t,m}_T(y)\big)\Big|\Big\|\p_{m_i}\big(X^{t,m}_T(\cdot)\big)(y)\Big\|_{L^{1,d_x}_m}\\\no
&+\int_{\R^{d_x}}\int_{\R^{d_x}} \p_{m_i}\big(X^{t,m}_T(x)\big)(y)^\top\big(\p_x\p_\mu k\big)^{t,m}_T(\wt{x},x) \p_{m_i}\big(X^{t,m}_T(\wt{x})\big)(y)dm(\wt{x})dm(x)\\\no
&+\int_{\R^{d_x}}\int_{\R^{d_x}} \p_{m_i}\big(X^{t,m}_T(x)\big)(y)^\top\big(\p_\mu\p_x k\big)^{t,m}_T(x,\wt{x}) \p_{m_i}\big(X^{t,m}_T(\wt{x})\big)(y)dm(\wt{x})dm(x)\\\no
&+\int_{\R^{d_x}}\int_{\R^{d_x}} \int_{\R^{d_x}}\p_{m_i}\big(X^{t,m}_T(x)\big)(y)^\top\big(\p_\mu\p_\mu k\big)^{t,m}_T(\wt{x},x,\wh{x})\p_{m_i}\big(X^{t,m}_T(\wh{x})\big)(y)dm(\wh{x})dm(\wt{x})dm(x)\\\no
\leq &\ \int_t^T \Bigg(\Lambda_f\Big(1+\frac{\wb{l}_h}{\lambda_g-\wb{l}_fL^*}\Big)\Big\|\p_{m_i}\big(Z^{t,m}_s(\cdot)\big)(y)\Big\|_{L^{1,d_x}_m}\Big| \p_{y_i}\big(X^{t,m}_s(y)\big)\Big|\\\no
&\ \ \ \ \ \ \ \ -\lambda_z\Big\|\p_{m_i}\big(Z^{t,m}_s(\cdot)\big)(y)\Big\|_{L^{2,d_x}_m}^2-\lambda_x\Big\|\p_{m_i}\big(X^{t,m}_s(\cdot)\big)(y)\Big\|_{L^{2,d_x}_m}^2\\\no
&\ \ \ \ \ \ \ \ +\frac{3\wb{l}_h^2}{\lambda_g-\wb{l}_fL^*} \Big\|\p_{m_i}\big(X^{t,m}_s(\cdot)\big)(y)\Big\|_{L^{1,d_x}_m}^2\\\no
&\ \ \ \ \ \ \ \ -\int_{\R^{d_x}}\int_{\R^{d_x}}\p_{m_i}\big(X^{t,m}_s(\wt{x})\big)(y)^\top\bigg(\big(\p_x\p_\mu f\big)^{t,m}_s(\wt{x},x)\cdot  Z^{t,m}_s(\wt{x})+ \big(\p_x\p_\mu g\big)^{t,m}_s(\wt{x},x)\bigg)dm(\wt{x})\p_{m_i}\big(X^{t,m}_s(x)\big)(y)dm(x)\\\no
&\ \ \ \ \ \ \ \ -\int_{\R^{d_x}}\int_{\R^{d_x}}\int_{\R^{d_x}}\p_{m_i}\big(X^{t,m}_s(\wh{x})\big)(y)^\top\bigg(\big(\p_\mu\p_\mu f\big)^{t,m}_s(\wt{x},x,\wh{x})\cdot  Z^{t,m}_s(\wt{x})+ \big(\p_\mu\p_\mu g\big)^{t,m}_s(\wt{x},x,\wh{x})\bigg)dm(\wh{x})dm(\wt{x})\p_{m_i}\big(X^{t,m}_s(x)\big)(y)dm(x)\\\no
&\ \ \ \ \ \ \ \ -\int_{\R^{d_x}}\int_{\R^{d_x}}\p_{m_i}\big(X^{t,m}_s(\wt{x})\big)(y)^\top\bigg(\big(\p_\mu\p_x f\big)^{t,m}_s(x,\wt{x})\cdot Z^{t,m}_s(x)+\big(\p_\mu\p_x g\big)^{t,m}_s(x,\wt{x})\bigg)dm(\wt{x})\p_{m_i}\big(X^{t,m}_s(x)\big)(y)dm(x)\\\no
&\ \ \ \ \ \ \ \ +\Big(2\Lambda_h+\frac{2\wb{l}_h^2}{\lambda_g-\wb{l}_fL^*}\Big)\Big|\p_{y_i}\big(X^{t,m}_s(y)\big)\Big|\Big\|\p_{m_i}\big(X^{t,m}_s(x)\big)(y)\Big\|_{L^{1,d_x}_m}\Bigg)ds,
\end{align*}\normalsize
which implies, by using \eqref{positive_g_mu_1}, \eqref{positive_k_mu_1} and using Jensen's inequality to replace $L^1$-norm by $L^2$-norm,
\footnotesize\begin{align}\no
&\Big(\lambda_k-l_k-\varepsilon_1\Big)\Big\|\p_{m_i}\big(X^{t,m}_T(\cdot)\big)(y)\Big\|_{L^{2,d_x}_m}^2-\frac{1}{\varepsilon_1}\Lambda_k^2\Big|\p_{y_i}\big(X^{t,m}_T(y)\big)\Big|^2\\\no
\leq &\ \int_t^T \Bigg(\frac{1}{4\varepsilon_1}\bigg(\Lambda_f^2\Big(1+\frac{\wb{l}_h}{\lambda_g-\wb{l}_fL^*}\Big)^2+\Big(2\Lambda_h+\frac{2\wb{l}_h^2}{\lambda_g-\wb{l}_fL^*}\Big)^2\bigg)\Big| \p_{y_i}\big(X^{t,m}_s(y)\big)\Big|^2-\Big(\lambda_z-\varepsilon_1\Big)\Big\|\p_{m_i}\big(Z^{t,m}_s(\cdot)\big)(y)\Big\|_{L^{2,d_x}_m}^2\\\label{displacement}
&\ \ \ \ \ \ \ \ -\Big(\lambda_x-\frac{3\wb{l}_h^2}{\lambda_g-\wb{l}_fL^*}-\frac{3}{2}k_0\wb{l}_f-l_g-\varepsilon_1\Big)\Big\|\p_{m_i}\big(X^{t,m}_s(\cdot)\big)(y)\Big\|_{L^{2,d_x}_m}^2\Bigg)ds.
\end{align}\normalsize
Here, we set 
\begin{align}\label{vareps_1}
\varepsilon_1:=\min\Big\{\frac{1}{4}\lambda_k,\frac{1}{2}\lambda_z,\frac{1}{40}\lambda_g\Big\},
\end{align}
and thus we can define the following positive constants, under \eqref{lambda_z}, \eqref{lambda_x}, the Hypotheses ${\bf(h2)}$ and the assumptions that $l_g\leq \frac{1}{2}\lambda_g$ and $l_k\leq \frac{1}{2}\lambda_k$,
\begin{align}\no
&\wb{\lambda}_k:=\frac{1}{4}\lambda_k\leq\lambda_k-l_k-\varepsilon_1,\ \wb{\lambda}_z:=\frac{1}{2}\lambda_z\leq \lambda_z-\varepsilon_1,\\\label{def_wb_lambda}
&\wb{\lambda}_x:=\frac{1}{40}\lambda_g\leq \lambda_x-\frac{3}{20}\lambda_g-\frac{3}{20}\lambda_g-\frac{1}{2}\lambda_g-\varepsilon_1\leq \lambda_x-\frac{3\wb{l}_h^2}{\lambda_g-\wb{l}_fL^*}-\frac{3}{2}k_0\wb{l}_f-l_g-\varepsilon_1.
\end{align}
Therefore, we have 
\footnotesize\begin{align}\no
&\wb{\lambda}_k\Big\|\p_{m_i}\big(X^{t,m}_T(\cdot)\big)(y)\Big\|_{L^{2,d_x}_m}^2+\int_t^T\Bigg(\wb{\lambda}_z\Big\|\p_{m_i}\big(Z^{t,m}_s(\cdot)\big)(y)\Big\|_{L^{2,d_x}_m}^2+\wb{\lambda}_x\Big\|\p_{m_i}\big(X^{t,m}_s(\cdot)\big)(y)\Big\|_{L^{2,d_x}_m}^2\Bigg)ds\\\label{cru_est_2}
\leq &\ \frac{1}{4\varepsilon_1}\bigg(\Lambda_f^2\Big(1+\frac{\wb{l}_h}{\lambda_g-\wb{l}_fL^*}\Big)^2+\Big(2\Lambda_h+\frac{2\wb{l}_h^2}{\lambda_g-\wb{l}_fL^*}\Big)^2\bigg)\int_t^T \Big| \p_{y_i}\big(X^{t,m}_s(y)\big)\Big|^2ds+\frac{1}{\varepsilon_1}\Lambda_k^2\Big|\p_{y_i}\big(X^{t,m}_T(y)\big)\Big|^2.
\end{align}\normalsize 

\textbf{Step 5} (Implications of major estimates obtained in Steps 3 and 4)
Thus, by using \eqref{p_m_gamma_int}, \eqref{L_star_1_1} and \eqref{cru_est_2}, we can obtain the following {\it a priori} estimate, for any $i=1,2,3,...,d_x$, 
\small\begin{align}\no
&\bigg\|\p_{m_i}\gamma(t,\cdot,m)(y)\bigg\|_{L^{2,d_x}_m}^2=\bigg\|\p_{m_i}\big(Z^{t,m}_t(\cdot)\big)(y)\bigg\|_{L^{2,d_x}_m}^2\\\no
\leq &\ \frac{39\Lambda_k^2}{\wb{\lambda}_k}\wb{\lambda}_k\Big\|\p_{m_i}\big(X^{t,m}_T(\cdot)\big)(y)\Big\|_{L^{2,d_x}_m}^2+12\Lambda_k^2\Big|\p_{y_i}\big(X^{t,m}_T(y)\big)\Big|^2\\\no
&+\int_t^T\Bigg(\Big(5\Lambda_h+\frac{4\wb{l}_h^2}{\lambda_g-\wb{l}_fL^*}\Big)\frac{1}{\wb{\lambda}_x}\wb{\lambda}_x\Big\|\p_{m_i}\big(X^{t,m}_s(\cdot)\big)(y)\Big\|_{L^{2,d_x}_m}^2+\Big(2\Lambda_h+\frac{2\wb{l}_h^2}{\lambda_g-\wb{l}_fL^*}\Big)\Big|\p_{y_i}\big(X^{t,m}_s(y)\big)\Big|^2\\\no
&\ \ \ \ \ \ \ \ \ \ \ \ +\bigg(4\Lambda_f+4\wb{l}_h\frac{\Lambda_f}{\lambda_g-\wb{l}_fL^*}+7\Lambda_h+\frac{6\wb{l}_h^2}{\lambda_g-\wb{l}_fL^*}\bigg)\frac{1}{\wb{\lambda}_z}\wb{\lambda}_z\Big\|\p_{m_i}\big(Z^{t,m}_s(\cdot)\big)(y)\Big\|_{L^{2,d_x}_m}^2\Bigg)ds\\\no
\leq &\  \Big(12+\frac{L^*_2}{\varepsilon_1}\Big)\Lambda_k^2\frac{1}{\lambda_k}\lambda_k\Big|\p_{y_i}\big(X^{t,m}_T(y)\big)\Big|^2\\\no
&+\int_t^T \bigg(2\Lambda_h+\frac{2\wb{l}_h^2}{\lambda_g-\wb{l}_fL^*}+\frac{L^*_2}{4\varepsilon_1}\Big(\Lambda_f^2\Big(1+\frac{\wb{l}_h}{\lambda_g-\wb{l}_fL^*}\Big)^2+\Big(2\Lambda_h+\frac{2\wb{l}_h^2}{\lambda_g-\wb{l}_fL^*}\Big)^2\Big)\bigg)\frac{1}{\lambda_x}\lambda_x\Big|\p_{y_i}\big(X^{t,m}_s(y)\big)\Big|^2 ds\\\label{L_star_3_1}
\leq &\ L^*_1L^*_3,
\end{align}\normalsize
where $L^*_1$ was defined in \eqref{L_star_1} and 
\footnotesize\begin{align}\label{L_star_2}
L^*_2:=\max\Bigg\{&\frac{39\Lambda_k^2}{\wb{\lambda}_k},\Big(5\Lambda_g+\frac{1}{4}\lambda_g\Big)\frac{1}{\wb{\lambda}_x},\bigg(5\Lambda_f+7\Lambda_g+\frac{13}{20}\lambda_g\bigg)\frac{1}{\wb{\lambda}_z}\Bigg\}\\\no
\geq \max\Bigg\{&\frac{39\Lambda_k^2}{\wb{\lambda}_k},\Big(5\Lambda_h+\frac{4\wb{l}_h^2}{\lambda_g-\wb{l}_fL^*}\Big)\frac{1}{\wb{\lambda}_x},\bigg(4\Lambda_f+4\wb{l}_h\frac{\Lambda_f}{\lambda_g-\wb{l}_fL^*}+7\Lambda_h+\frac{6\wb{l}_h^2}{\lambda_g-\wb{l}_fL^*}\bigg)\frac{1}{\wb{\lambda}_z}\Bigg\},
\end{align}\normalsize
and\footnotesize\begin{align}\label{L_star_3}
L^*_3:=\max\Bigg\{&\Big(12+\frac{L^*_2}{\varepsilon_1}\Big)\Lambda_k^2\frac{1}{\lambda_k},\bigg(2\Lambda_g+\frac{1}{5}\lambda_g+\frac{L^*_2}{4\varepsilon_1}\Big(\frac{25}{16}\Lambda_f^2+\Big(2\Lambda_g+\frac{1}{5}\lambda_g\Big)^2\Big)\bigg)\frac{1}{\lambda_x}\Bigg\}\\\no
\geq \max\Bigg\{&\Big(12+\frac{L^*_2}{\varepsilon_1}\Big)\Lambda_k^2\frac{1}{\lambda_k},\bigg(2\Lambda_h+\frac{2\wb{l}_h^2}{\lambda_g-\wb{l}_fL^*}+\frac{L^*_2}{4\varepsilon_1}\Big(\Lambda_f^2\Big(1+\frac{\wb{l}_h}{\lambda_g-\wb{l}_fL^*}\Big)^2+\Big(2\Lambda_h+\frac{2\wb{l}_h^2}{\lambda_g-\wb{l}_fL^*}\Big)^2\Big)\bigg)\frac{1}{\lambda_x}\Bigg\}.
\end{align}\normalsize
In addition, using \eqref{L_star_1_1} and \eqref{cru_est_2}, we have
\footnotesize\begin{align}\no
&\wb{\lambda_k}\Big\|\p_{m_i}\big(X^{t,m}_T(\cdot)\big)(y)\Big\|_{L^{2,d_x}_m}^2+\wb{\lambda}_z\int_t^T\Big\|\p_{m_i}\big(Z^{t,m}_s(\cdot)\big)(y)\Big\|_{L^{2,d_x}_m}^2ds+\wb{\lambda}_x\int_t^T\Big\|\p_{m_i}\big(X^{t,m}_s(\cdot)\big)(y)\Big\|_{L^{2,d_x}_m}^2ds\\\no
\leq &\ \frac{1}{\varepsilon_1}\Lambda_k^2\frac{1}{\lambda_k}\lambda_k\Big|\p_{y_i}\big(X^{t,m}_T(y)\big)\Big|^2+\frac{1}{4\varepsilon_1}\bigg(\Lambda_f^2\Big(1+\frac{\wb{l}_h}{\lambda_g-\wb{l}_fL^*}\Big)^2+\Big(2\Lambda_h+\frac{2\wb{l}_h^2}{\lambda_g-\wb{l}_fL^*}\Big)^2\bigg)\frac{1}{\lambda_x}\lambda_x\int_t^T \Big| \p_{y_i}\big(X^{t,m}_s(y)\big)\Big|^2ds\\\label{L_star_4_1}
\leq &\ L^*_1L^*_4,
\end{align}\normalsize
where 
\footnotesize\begin{align}\label{L_star_4}
L^*_4:=\max\Bigg\{&\frac{1}{\varepsilon_1}\Lambda_k^2\frac{1}{\lambda_k},\frac{1}{4\varepsilon_1}\bigg(\frac{25}{16}\Lambda_f^2+\Big(2\Lambda_g+\frac{1}{5}\lambda_g\Big)^2\bigg)\frac{1}{\lambda_x}\Bigg\}\\\no
\geq \max\Bigg\{&\frac{1}{\varepsilon_1}\Lambda_k^2\frac{1}{\lambda_k},\frac{1}{4\varepsilon_1}\bigg(\Lambda_f^2\Big(1+\frac{\wb{l}_h}{\lambda_g-\wb{l}_fL^*}\Big)^2+\Big(2\Lambda_h+\frac{2\wb{l}_h^2}{\lambda_g-\wb{l}_fL^*}\Big)^2\bigg)\frac{1}{\lambda_x}\Bigg\}.
\end{align}\normalsize

Now, we go back to \eqref{eq_7_26_1}. Using \eqref{bdd_d1_f}, \eqref{bdd_d2_f}, \eqref{bdd_d2_g_1}, \eqref{bdd_d2_g_2}, \eqref{bdd_d2_k_1}, \eqref{p_xalpha_new} and \eqref{p_zalpha_new}, we obtain
\footnotesize\begin{align}\no
&\lambda_k\Big|\p_{m_i}\big(X^{t,m}_T(x)\big)(y)\Big|^2-3\Lambda_k\int_{\R^{d_x}} \Big|\p_{m_i}\big(X^{t,m}_T(\wt{x})\big)(y)\Big|dm(\wt{x})\Big|\p_{m_i}\big(X^{t,m}_T(x)\big)(y)\Big|-2\Lambda_k\Big|\p_{y_i}\big(X^{t,m}_T(y)\big)\Big|\Big|\p_{m_i}\big(X^{t,m}_T(x)\big)(y)\Big|\\\no
\leq &\ \int_t^T \Bigg(\Lambda_f\Big(1+\frac{\wb{l}_h}{\lambda_g-\wb{l}_fL^*}\Big)\cdot\bigg(\Big|\p_{m_i}\big(Z^{t,m}_s(x)\big)(y)\Big|\Big| \p_{y_i}\big(X^{t,m}_s(y)\big)\Big|+\Big|\p_{m_i}\big(Z^{t,m}_s(x)\big)(y)\Big|\int_{\R^{d_x}} \Big|\p_{m_i}\big(X^{t,m}_s(\wh{x})\big)(y)\Big| dm(\wh{x})\\\no
&\ \ \ \ \ \ \ \ \ \ \ \ \ \ \ \ \ \ \ \ \ \ \ \ \ \ \ \ \ \ \ \ \ \ \ \ \ \ \ \ +\Big|\p_{m_i}\big(X^{t,m}_s(x)\big)(y)\Big|\int_{\R^{d_x}} \Big|\p_{m_i}\big(Z^{t,m}_s(\wh{x})\big)(y)\Big| dm(\wh{x})\bigg)\\\no
&\ \ \ \ \ \ \ \ -\lambda_z\Big|\p_{m_i}\big(Z^{t,m}_s(x)\big)(y)\Big|^2-\lambda_x\Big|\p_{m_i}\big(X^{t,m}_s(x)\big)(y)\Big|^2\\\no
&\ \ \ \ \ \ \ \ +\Big(3\Lambda_h+\frac{3\wb{l}_h^2}{\lambda_g-\wb{l}_fL^*}\Big)\int_{\R^{d_x}}\Big|\p_{m_i}\big(X^{t,m}_s(\wt{x})\big)(y)\Big|dm(\wt{x})\Big|\p_{m_i}\big(X^{t,m}_s(x)\big)(y)\Big|\\\label{eq_8_40_new}
&\ \ \ \ \ \ \ \ +\Big(2\Lambda_h+\frac{2\wb{l}_h^2}{\lambda_g-\wb{l}_fL^*}\Big)\Big|\p_{y_i}\big(X^{t,m}_s(y)\big)\Big|\Big|\p_{m_i}\big(X^{t,m}_s(x)\big)(y)\Big|\Bigg)ds.
\end{align}\normalsize
Here, it is worth noting that unlike the derivation of \eqref{cru_est_2}, the orange terms in \eqref{eq_7_26_1} do not cancel out with each other, so we have to bound them appropriately as above. Applying Young's inequality to \eqref{eq_8_40_new} yields 
\footnotesize\begin{align}\no
&\big(\lambda_k-2\varepsilon_2\big)\Big|\p_{m_i}\big(X^{t,m}_T(x)\big)(y)\Big|^2-\frac{9}{4\varepsilon_2}\Lambda_k^2\Big\|\p_{m_i}\big(X^{t,m}_T(\cdot)\big)(y)\Big\|_{L^{1,d_x}_m}^2-\frac{1}{\varepsilon_2}\Lambda_k^2\Big|\p_{y_i}\big(X^{t,m}_T(y)\big)\Big|^2\\\no
\leq &\ \int_t^T \Bigg(\frac{1}{4\varepsilon_2}\Lambda_f^2\Big(1+\frac{\wb{l}_h}{\lambda_g-\wb{l}_fL^*}\Big)^2\cdot\bigg(\Big| \p_{y_i}\big(X^{t,m}_s(y)\big)\Big|^2+ \Big\|\p_{m_i}\big(X^{t,m}_s(\cdot)\big)(y)\Big\|_{L^{1,d_x}_m}^2+ \Big\|\p_{m_i}\big(Z^{t,m}_s(\cdot)\big)(y)\Big\|_{L^{1,d_x}_m}^2\bigg)\\\no
&\ \ \ \ \ \ \ \ -\big(\lambda_z-2\varepsilon_2\big)\Big|\p_{m_i}\big(Z^{t,m}_s(x)\big)(y)\Big|^2-\Big(\lambda_x-3\varepsilon_2\Big)\Big|\p_{m_i}\big(X^{t,m}_s(x)\big)(y)\Big|^2\\\no
&\ \ \ \ \ \ \ \ +\frac{9}{4\varepsilon_2}\bigg(\Lambda_h+\frac{\wb{l}_h^2}{\lambda_g-\wb{l}_fL^*}\bigg)^2 \Big\|\p_{m_i}\big(X^{t,m}_s(\cdot)\big)(y)\Big\|_{L^{1,d_x}_m}^2+\frac{1}{\varepsilon_2}\Big(\Lambda_h+\frac{\wb{l}_h^2}{\lambda_g-\wb{l}_fL^*}\Big)^2\Big|\p_{y_i}\big(X^{t,m}_s(y)\big)\Big|^2\Bigg)ds,
\end{align}\normalsize
which implies, by using \eqref{L_star_1_1} and \eqref{L_star_4_1},
\footnotesize\begin{align}\no
&\big(\lambda_k-2\varepsilon_2\big)\Big|\p_{m_i}\big(X^{t,m}_T(x)\big)(y)\Big|^2+\big(\lambda_z-2\varepsilon_2\big)\int_t^T\Big|\p_{m_i}\big(Z^{t,m}_s(x)\big)(y)\Big|^2ds+\Big(\lambda_x-3\varepsilon_2\Big)\int_t^T\Big|\p_{m_i}\big(X^{t,m}_s(x)\big)(y)\Big|^2ds\\\no
\leq &\ \frac{9}{4\varepsilon_2}\Lambda_k^2\frac{1}{\wb{\lambda}_k}\wb{\lambda}_k\Big\|\p_{m_i}\big(X^{t,m}_T(\cdot)\big)(y)\Big\|_{L^{1,d_x}_m}^2+\frac{1}{\varepsilon_2}\Lambda_k^2\frac{1}{\lambda_k}\lambda_k\Big|\p_{y_i}\big(X^{t,m}_T(y)\big)\Big|^2\\\no
&+\int_t^T \Bigg(\frac{1}{4\varepsilon_2}\Lambda_f^2\Big(1+\frac{\wb{l}_h}{\lambda_g-\wb{l}_fL^*}\Big)^2\cdot\bigg(\frac{1}{\lambda_x}\lambda_x\Big| \p_{y_i}\big(X^{t,m}_s(y)\big)\Big|^2+ \frac{1}{\wb{\lambda}_x}\wb{\lambda}_x\Big\|\p_{m_i}\big(X^{t,m}_s(\cdot)\big)(y)\Big\|_{L^{1,d_x}_m}^2+ \frac{1}{\wb{\lambda}_z}\wb{\lambda}_z\Big\|\p_{m_i}\big(Z^{t,m}_s(\cdot)\big)(y)\Big\|_{L^{1,d_x}_m}^2\bigg)\\\no
&\ \ \ \ \ \ \ \ \ \ +\frac{9}{4\varepsilon_2}\bigg(\Lambda_h+\frac{\wb{l}_h^2}{\lambda_g-\wb{l}_fL^*}\bigg)^2 \frac{1}{\wb{\lambda}_x}\wb{\lambda}_x\Big\|\p_{m_i}\big(X^{t,m}_s(\cdot)\big)(y)\Big\|_{L^{1,d_x}_m}^2+\frac{1}{\varepsilon_2}\Big(\Lambda_h+\frac{\wb{l}_h^2}{\lambda_g-\wb{l}_fL^*}\Big)^2\frac{1}{\lambda_x}\lambda_x\Big|\p_{y_i}\big(X^{t,m}_s(y)\big)\Big|^2\Bigg)ds\\\label{L_star_5_1}
\leq&\  L^*_1L^*_4\big(1+L^*_5\big),
\end{align}\normalsize
where
\footnotesize\begin{align}\label{L_star_5}
L^*_5:=\max\Bigg\{&\frac{9}{4\varepsilon_2}\Lambda_k^2\frac{1}{\wb{\lambda}_k},\frac{25}{64\varepsilon_2\wb{\lambda}_z}\Lambda_f^2,\frac{1}{4\varepsilon_2\wb{\lambda}_x}\bigg(\frac{25}{16}\Lambda_f^2+9\bigg(\frac{1}{10}\lambda_g+\Lambda_g\bigg)^2\bigg)\Bigg\}\\\no
\geq \max\Bigg\{&\frac{9}{4\varepsilon_2}\Lambda_k^2\frac{1}{\wb{\lambda}_k},\frac{1}{4\varepsilon_2\wb{\lambda}_z}\Lambda_f^2\Big(1+\frac{\wb{l}_h}{\lambda_g-\wb{l}_fL^*}\Big)^2,\frac{1}{4\varepsilon_2\wb{\lambda}_x}\bigg(\Lambda_f^2\Big(1+\frac{\wb{l}_h}{\lambda_g-\wb{l}_fL^*}\Big)^2+9\bigg(\Lambda_h+\frac{\wb{l}_h^2}{\lambda_g-\wb{l}_fL^*}\bigg)^2\bigg)\Bigg\}.
\end{align}\normalsize
Here, we choose
\begin{align}\label{vareps_2}
\varepsilon_2:=\min\Big\{\frac{1}{4}\lambda_k,\frac{1}{4}\lambda_z,\frac{1}{5}\lambda_g\Big\},
\end{align}
and thus, $\lambda_k-2\varepsilon_2\geq \frac{1}{2}\lambda_k =2\wb{\lambda}_k>0$, $\lambda_z-2\varepsilon_2\geq \wb{\lambda}_z>0$ and $\lambda_x-3\varepsilon_2\geq \wb{\lambda}_x>0$ are all positive.
Thus, by using \eqref{p_m_gamma}, \eqref{L_star_1_1}, \eqref{L_star_4_1} and \eqref{L_star_5_1}, we can obtain the following {\it a priori} estimate, for $i=1,2,3,...,d_x$, 
\small\begin{align}\no
&\bigg|\p_{m_i}\gamma(t,x,m)(y)\bigg|^2=\bigg|\p_{m_i}\big(Z^{t,m}_t(x)\big)(y)\bigg|^2\\\no
\leq &\ 27\Lambda_k^2\frac{1}{\wb{\lambda}_k}\wb{\lambda}_k \Big\|\p_{m_i}\big(X^{t,m}_T(\cdot)\big)(y)\Big\|_{L^{2,d_x}_m}^2 +6\Lambda_k^2\frac{1}{\wb{\lambda}_k}\cdot 2\wb{\lambda}_k\Big|\p_{m_i}\big(X^{t,m}_T(x)\big)(y)\Big|^2+12\Lambda_k^2\frac{1}{\lambda_k}\lambda_k\Big|\p_{y_i}\big(X^{t,m}_T(y)\big)\Big|^2\\\no
&+\int_t^T\Bigg(\Big(3\Lambda_h+\frac{3\wb{l}_h^2}{\lambda_g-\wb{l}_fL^*}\Big) \frac{1}{\wb{\lambda}_x}\wb{\lambda}_x\Big\|\p_{m_i}\big(X^{t,m}_s(\cdot)\big)(y)\Big\|_{L^{2,d_x}_m}^2+\Big(2\Lambda_h+\frac{\wb{l}_h^2}{\lambda_g-\wb{l}_fL^*}\Big)\frac{1}{\wb{\lambda}_x}\wb{\lambda}_x\Big|\p_{m_i}\big(X^{t,m}_s(x)\big)(y)\Big|^2\\\no
&\ \ \ \ \ \ \ \ \ \ \ \ +\Big(2\Lambda_h+\frac{2\wb{l}_h^2}{\lambda_g-\wb{l}_fL^*}\Big)\frac{1}{\lambda_x}\lambda_x\Big|\p_{y_i}\big(X^{t,m}_s(y)\big)\Big|^2+\Big(\Lambda_h+\frac{\wb{l}_h^2}{\lambda_g-\wb{l}_fL^*}\Big) \frac{1}{\wb{\lambda}_z}\wb{\lambda}_z \Big\|\p_{m_i}\big(Z^{t,m}_s(\cdot)\big)(y)\Big\|_{L^{2,d_x}_m}^2 \\\no
&\ \ \ \ \ \ \ \ \ \ \ \ +\bigg(3\Lambda_f+3\wb{l}_h\frac{\Lambda_f}{\lambda_g-\wb{l}_fL^*}+7\Lambda_h+6\frac{\wb{l}_h^2}{\lambda_g-\wb{l}_fL^*}\bigg)\frac{1}{\wb{\lambda}_z}\wb{\lambda}_z\Big|\p_{m_i}\big(Z^{t,m}_s(x)\big)(y)\Big|^2\Bigg)ds\\\label{L_star_6_1}
\leq &\  \Big(L^*_4\big(2+L^*_5\big)+1\Big)L^*_1L^*_6,
\end{align}\normalsize
where
\footnotesize\begin{align}\no
L^*_6:=\max\Bigg\{&27\Lambda_k^2\frac{1}{\wb{\lambda}_k},\Big(\Lambda_g+\frac{1}{10}\lambda_g\Big) \frac{1}{\wb{\lambda}_x},\Big(\Lambda_g+\frac{1}{10}\lambda_g\Big) \frac{1}{\wb{\lambda}_z},\\\no
&12\Lambda_k^2\frac{1}{\lambda_k},\Big(2\Lambda_g+\frac{1}{5}\lambda_g\Big)\frac{1}{\lambda_x},6\Lambda_k^2\frac{1}{\wb{\lambda}_k},\Big(2\Lambda_g+\frac{3}{20}\lambda_g\Big)\frac{1}{\wb{\lambda}_x},\\\label{L_star_6}
&\bigg(\frac{15}{4}\Lambda_f+7\Lambda_g+\frac{13}{20}\lambda_g\bigg)\frac{1}{\wb{\lambda}_z}\Bigg\}\\\no
\geq \max\Bigg\{&27\Lambda_k^2\frac{1}{\wb{\lambda}_k},\Big(\Lambda_h+\frac{\wb{l}_h^2}{\lambda_g-\wb{l}_fL^*}\Big) \frac{1}{\wb{\lambda}_x},\Big(\Lambda_h+\frac{\wb{l}_h^2}{\lambda_g-\wb{l}_fL^*}\Big) \frac{1}{\wb{\lambda}_z},\\\no
&12\Lambda_k^2\frac{1}{\lambda_k},\Big(2\Lambda_h+\frac{2\wb{l}_h^2}{\lambda_g-\wb{l}_fL^*}\Big)\frac{1}{\lambda_x},6\Lambda_k^2\frac{1}{\wb{\lambda}_k},\Big(2\Lambda_h+\frac{\wb{l}_h^2}{\lambda_g-\wb{l}_fL^*}\Big)\frac{1}{\wb{\lambda}_x},\\\no
&\bigg(3\Lambda_f+3\wb{l}_h\frac{\Lambda_f}{\lambda_g-\wb{l}_fL^*}+7\Lambda_h+6\frac{\wb{l}_h^2}{\lambda_g-\wb{l}_fL^*}\bigg)\frac{1}{\wb{\lambda}_z}\Bigg\}.
\end{align}\normalsize

In conclusion, by \eqref{L_star_1} and \eqref{L_star_6_1}, we can obtain,
\begin{align}\no
&\sup_{x\in\R^{d_x},m\in\mc{P}_2(\R^{d_x}),y\in\R^{d_x}}\|\p_x\gamma(t,x,m)\|_{\mathcal{L}(\R^{d_x};\R^{d_x})}\vee \|\p_m\gamma(t,x,m)(y)\|_{\mathcal{L}(\R^{d_x};\R^{d_x})}\\\label{L_star_0}
\leq&\ d_x\cdot\max\Big\{L^*_1,\sqrt{\Big(L^*_4\big(2+L^*_5\big)+1\Big)L^*_1L^*_6}\Big\}=:L^*_0.
\end{align}
Since $\gamma(s,x,m)$ is also a decoupling field for the FBODE system \eqref{fbodesystem} on $[t,T]$, for any arbitrary $s\in[t,T]$, $(X^{s,m}_\tau(x),Z^{s,m}_\tau(x)=\gamma(\tau,X^{s,m}_\tau(x),X^{s,m}_\tau\ot m))_{x\in\R^{d_x},\tau\in[s,T]}$ is also a solution pair to \eqref{fbodesystem} over the time horizon $[s,T]$, so the above estimates are also valid for $\gamma(s,x,m)=\gamma(s,X^{s,m}_s(x),X^{s,m}_s\ot m)=Z^{s,m}_s(x)$; in other words, we still have
\begin{align}\no
&\sup_{x\in\R^{d_x},m\in\mc{P}_2(\R^{d_x}),y\in\R^{d_x}}\|\p_x\gamma(s,x,m)\|_{\mathcal{L}(\R^{d_x};\R^{d_x})}\vee \|\p_m\gamma(s,x,m)(y)\|_{\mathcal{L}(\R^{d_x};\R^{d_x})}\leq L^*_0.
\end{align}
Therefore, $L^*=\vertiii{\gamma }_{2,[t,T]}\leq L^*_0$.
We emphasize here that the positive constants $L^*_i,\,i=0,1,2,3,4,5,6$, which are defined in \eqref{L_star_1}, \eqref{L_star_2}, \eqref{L_star_3}, \eqref{L_star_4}, \eqref{L_star_5}, \eqref{L_star_6} and \eqref{L_star_0}, only depend on $\lambda_f$, $\Lambda_f$, $\lambda_g$, $\Lambda_g$, $\lambda_k$ and $\Lambda_k$ but not $L^*$, $\wb{l}_f$ and $T$.

\begin{remark}\label{remark_lower_bdd_L_star_0}
It is worth noting that for our choice of $L^*_0$, we always have $L^*_0\geq \wb{L}_k$. More precisely, it follows from \eqref{positive_k}, \eqref{bdd_d2_k_1} and \eqref{def_wb_lambda} that $\lambda_k\leq 2\Lambda_k$ and $\wb{\lambda}_k\leq \frac{1}{2}\Lambda_k$. Thus, by \eqref{L_star_1} and \eqref{L_star_6}, we have 
\begin{align}\label{lower_bdd_L_star_1}
L^*_1\geq \frac{4\Lambda_k^2}{\lambda_k}\geq 2\Lambda_k\text{ and } L^*_6\geq \frac{27\Lambda_k^2}{\wb{\lambda}_k}\geq 54\Lambda_k,
\end{align}
and hence, 
\begin{align}\label{lower_bdd_L_star_0}
L^*_0\geq d_x\cdot\sqrt{\Big(L^*_4\big(2+L^*_5\big)+1\Big)L^*_1L^*_6}\geq \sqrt{L^*_1L^*_6}\geq \sqrt{108}\Lambda_k\geq 3\Lambda_k=\wb{L}_k.
\end{align}
\end{remark}

\begin{theorem}\label{GlobalSol}
(Global existence in time of the FBODE system \eqref{MPIT}). Under Assumption ${\bf(a1)}$-${\bf(a3)}$ and Hypothesis ${\bf(h1)}$-${\bf(h2)}$, then, for any $T>0$, there exists a global decoupling field $\gamma(s,x,\mu)$ for the FBODE system \eqref{MPIT} on $[0,T]$ with terminal data $p(x,\mu)=\int_{\R^{d_x}} \p_\mu k( \wt{x} ,\mu)(x)d\mu(\wt{x})+\p_x k(x,\mu)$, which is differentiable in $s\in[0,T]$, $x\in\R^{d_x}$ and $L$-differentiable in $\mu\in\mc{P}_2(\R^{d_x})$, and its derivatives $\p_x\gamma(s,x,\mu)$ and $\p_\mu\gamma(s,x,\mu)(\wt{x})$ are jointly continuous in their corresponding arguments $(s,x,\mu)\in [t,T]\times \R^{d_x} \times \mathcal{P}_2(\R^{d_x})$ and $(s,x,\mu, \wt{x})\in [t,T]\times \R^{d_x} \times \mathcal{P}_2(\R^{d_x})\times \R^{d_x}$ respectively; it also satisfies $\vertiii{\gamma }_{2,[0,T]}\leq L^*_0$ with $L^*_0$ defined in \eqref{L_star_0}.
Moreover, for any $t\in[0,T]$ and $m\in\mc{P}_2(\R^{d_x})$, there exists a solution pair $\big(X^{t,m}_s(x),Z^{t,m}_s(x)\big)_{x\in\R^{d_x},s\in[t,T]}$ of the FBODE system \eqref{MPIT}, which is defined by solving for $X^{t,m}_s(x)$ in \eqref{xeq} and $Z^{t,m}_s(x):=\gamma(s,X^{t,m}_s(x),X^{t,m}_s\ot m)$, such that, for each $s\in[t ,T]$ and $x\in\R^{d_x}$,
\begin{align}\label{eq_7_43}
&\big|Z^{t,m}_s(x)-Z^{t,m_0}_s(x_0)\big|
\leq \ L^*_0\Big(\big|X^{t,m}_s(x)-X^{t,m_0}_s(x_0)\big|+W_1(X^{t,m}_s\ot m,X^{t,m_0}_s\ot m_0)\Big),\\\label{eq_7_44}
&\left|Z^{t,m}_s(x)\right|\leq L^*_0 \left(\left|X^{t,m}_s(x)\right|+\left\|X^{t,m}_s\right\|_{L^{1,d_x}_m}\right).
\end{align} 
Furthermore, both $X^{t,m}_s(x)$ and $Z^{t,m}_s(x)$ are differentiable in $t\in[0 ,T]$, $x\in\R^{d_x}$ and are $L$-differentiable in $m\in\mc{P}_2(\R^{d_x})$ with their corresponding derivatives $\Big(\p_t\big(X^{t,m}_s(x)\big),\p_t\big(Z^{t,m}_s(x)\big)\Big)$, $\Big(\p_x\big(X^{t,m}_s(x)\big),\p_x\big(Z^{t,m}_s(x)\big)\Big)$ and $\Big(\p_m\big(X^{t,m}_s(x)\big)(y), \p_m\big(Z^{t,m}_s(x)\big)(y)\Big)$ being continuous in their corresponding arguments, and these derivatives are also continuously differentiable in $s\in[t,T]$. In addition, the following estimates hold:
\begin{align}\label{Bp_xX_G}
&\left\|\p_x\big(X^{t,m}_s(x)\big)\right\|_{\mc{L}(\R^{d_x};\R^{d_x})}\leq \exp\Big(L_B'(s-t)\Big);
\\\label{Bp_mX_G}
&\left\|\p_m\big(X^{t,m}_s(x)\big)(y)\right\|_{\mc{L}(\R^{d_x};\R^{d_x})}\leq L_M^{(s-t)};
\\\label{Bp_tX_G}
&\left|\p_t\big(X^{t,m}_s(x)\big)\right|\leq L_B'\big(1+|x|+\|m\|_1\big)\Big(L_B'(s-t)\exp\Big(2L_B'(s-t)\Big)+1\Big)\exp\Big(L_B'(s-t)\Big);\\\label{Bp_xZ_G}
&\left\|\p_x\big(Z^{t,m}_s(x)\big)\right\|_{\mc{L}(\R^{d_x};\R^{d_x})}\leq L^*_0\exp\Big(L_B'(s-t)\Big);
\\\label{Bp_mZ_G}
&\left\|\p_m\big(Z^{t,m}_s(x)\big)(y)\right\|_{\mc{L}(\R^{d_x};\R^{d_x})}\leq 2L^*_0L_M^{(s-t)}+L^*_0\exp\Big(L_B'(s-t)\Big);
\\\label{Bp_tZ_G}
&\left|\p_t\big(Z^{t,m}_s(x)\big)\right|\leq 2L^*_0L_B'\big(1+|x|+\|m\|_1\big)\Big(L_B'(s-t)\exp\Big(2L_B'(s-t)\Big)+1\Big)\exp\Big(L_B'(s-t)\Big),
\end{align}
where 
\begin{align}\label{L_Bp}
L_B':=&\ \Lambda_f(1+L_\alpha+L^*_0L_\alpha),\\\label{L_M^{(s-t)}}
L_M^{(\tau)}:=&\ L_B'\tau\Big(L_B'\tau\cdot\exp\Big(2L_B'\tau\Big)+1\Big)\exp\Big(2L_B'\tau\Big).
\end{align}
\end{theorem}

\begin{proof}
Define $\wb{\eps}_2=:\eps_2(L^*_0; L_f,\Lambda_f,\wb{l}_f,L_g,\Lambda_g,\wb{l}_g,L_\alpha)$ where $\eps_2$ is given in Theorem \ref{Thm6_2}. Recall that $L^*_0$ was defined in \eqref{L_star_0}. Partition the interval $[0 ,T]$ into sub-intervals with endpoints $0  < t_1 <...<t_{N-1}:=T-\wb{\eps}_2<t_N:=T-\frac{1}{2}\wb{\eps}_2<t_{N+1} =:T$ where $t_{N-i}=T-\frac{i+1}{2}\wb{\eps}_2$, $i=0,1,...,N-1$ and $N$ is the smallest integer such that $N\geq 2T/\wb{\eps}_2-1$. Note that the sub-interval $[0 ,t_1=T-\frac{N}{2}\wb{\eps}_2]$ has a length not longer than those other intervals $[t_1,t_2]$,..., and $[t_{N-1},t_N]$, everyone of which has a common length of size $\frac{1}{2}\wb{\eps}_2$. Now, we are ready to paste the local solution $\gamma$ of \eqref{gammaeq} together to obtain a global one by using the idea depicted in Figure 1.

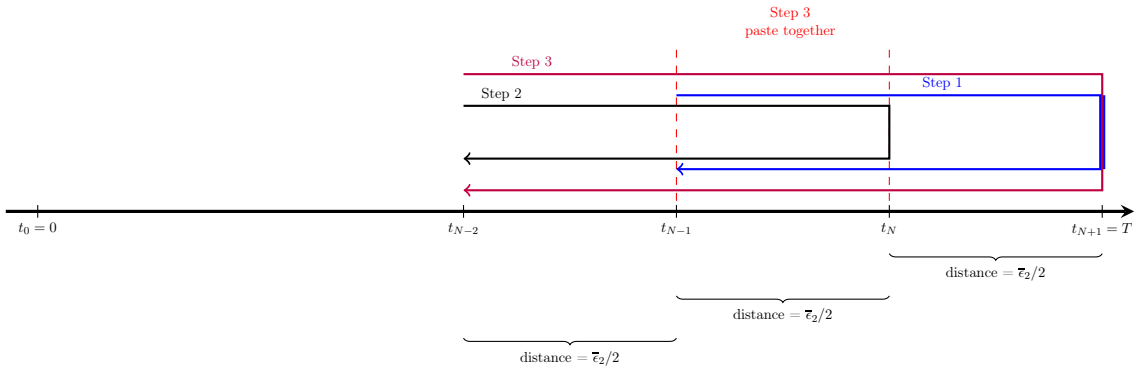
\begin{figure}[h!]\label{fig_1}
	\centering
	\begin{tikzpicture}[scale=1.4]
		\draw[-{stealth[scale=4]}, very thick, black] (-0.3, 0) -- (10.3, 0);
		\draw (0, -0.06) node[below, scale=0.5]{$t_0 = 0$} -- (0, 0.06);
		\draw (4, -0.06) node[below, scale=0.5]{$t_{N-2}$} -- (4, 0.06);
		\draw (6, -0.06) node[below, scale=0.5]{$t_{N-1}$} -- (6, 0.06);
		\draw (8, -0.06) node[below, scale=0.5]{$t_{N}$} -- (8, 0.06);
		\draw (10, -0.06) node[below, scale=0.5]{$t_{N+1}=T$} -- (10, 0.06);
		
		\draw [decorate, decoration={brace,mirror}]
		(8, -0.4) -- node[below=0.1cm, scale=0.5]{distance = $\wb{\eps}_2/2$} (10, -0.4);
		\draw [decorate, decoration={brace,mirror}]
		(6, -0.8) -- node[below=0.1cm, scale=0.5]{distance = $\wb{\eps}_2/2$} (8, -0.8);
		\draw [decorate, decoration={brace,mirror}]
		(4, -1.2) -- node[below=0.1cm, scale=0.5]{distance = $\wb{\eps}_2/2$} (6, -1.2);

\draw[color=red, dashed] (6, 0.1) -- (6, 1.6) node[above, xshift=1.5cm, align=center, scale=0.5, color=red]{Step 3\\paste together};
		\draw[color=red, dashed] (8, 0.1) -- (8, 1.6);
	
		\draw[->, color=black, thick] (4, 1.0) node[above, scale=0.5, xshift=1cm, color=black]{Step 2}--(8,1.0)--(8,0.5)--(4, 0.5);
           
		\draw[-, color=blue, thick] (6, 1.1) node[above, scale=0.5, xshift=7cm, color=blue]{Step 1}--(10,1.1);
\draw[-, color=blue, line width=0.8mm] (10,1.1) node[above, scale=0.5, xshift=7cm, color=blue]{}--(10,0.4);
\draw[->, color=blue, thick] (10,0.4) node[above, scale=0.5, xshift=7cm, color=blue]{}--(6, 0.4);
\draw[->, color=purple, thick] (4, 1.3) node[above, scale=0.5, xshift=1.8cm, color=purple]{Step 3} --(10,1.3)--(10,0.2)--(4, 0.2);
		
	\end{tikzpicture}
	\caption{Algorithm for pasting the local-in-time solutions.}
\end{figure}

Step 1. On the interval $ [t_{N-1}=T-\wb{\eps}_2,T]$, take $p(x,\mu)=\int_{\R^{d_x}} \p_\mu k(\wt{x},\mu)(x)d\mu(\wt{x})+\p_x k(x,\mu)$, and apply Theorem \ref{Thm6_1} and \ref{Thm6_2}, then, the equation \eqref{gammaeq} has a unique solution $\gamma^{(N)}(s,x,\mu)$ on the time interval $[t_{N-1},T]$ with the regularity $(i)$ stated in Theorem \ref{Thm6_2} and satisfying $\vertiii{\gamma^{(N)}}_{2,[t_{N-1},T]}\leq 2\max\{\wb{L}_k,L^*_0\}=2L^*_0$. Moreover, applying Theorem \ref{Crucial_Estimate}, we actually have a better estimate $\vertiii{\gamma^{(N)}}_{2,[t_{N-1},T]}\leq L^*_0$.

Step 2. On the interval $ [t_{N-2}=t_N-\wb{\eps}_2,t_N]$, take $p(x,\mu)=\gamma^{(N)}(t_N,x,\mu)$, and thus by using $L_p\leq L^*_0$ and applying Theorem \ref{Thm6_1} and \ref{Thm6_2}, then, the equation \eqref{gammaeq} has a unique solution $\gamma^{(N-1)}(s,x,\mu)$ on the time interval $[t_{N-2},t_N]$ with the regularity $(i)$ stated in Theorem \ref{Thm6_2} and satisfying $\vertiii{\gamma^{(N-1)}}_{2,[t_{N-2},t_N]}\leq 2L^*_0$. 

Step 3. By the local uniqueness of $\gamma$ that we have shown in Theorem \ref{Thm6_1}, for $s\in [t_{N-1},t_N]$, 
\small\begin{align}\no
\gamma^{(N-1)}(s,x,\mu)=\gamma^{(N)}(s,x,\mu),
\end{align}\normalsize
since both of them satisfy the following equation: 
for any $(s,x,m)\in [t_{N-1},t_N]\times \R^{d_x}\times\mathcal{P}_2(\R^{d_x})$,
\small\begin{align}\no
\gamma (s,x,m)
= &\ \gamma^{(N)}(t_N,X^{s,m}_{t_N}(x),X^{s,m}_{t_N}\ot m)\\\no
&+ \int_s^{t_N} \Bigg(\int_{\R^{d_x}}\bigg(\p_\mu f\left(X^{s,m }_\tau(\wt{x}),X^{s,m }_\tau\ot m,\alpha\left(X^{s,m }_\tau(\wt{x}),X^{s,m }_\tau\ot m, Z^{s,m }_\tau(\wt{x})\right)\right)(X^{s,m }_\tau(x))\cdot  Z^{s,m }_\tau(\wt{x})\\\no
&\ \ \ \ \ \ \ \ \ \ \ \ \ \ \ \ \ \ \ \ \ \ + \p_\mu g\left(X^{s,m }_\tau(\wt{x}),X^{s,m }_\tau\ot m,\alpha\left(X^{s,m }_\tau(\wt{x}),X^{s,m }_\tau\ot m, Z^{s,m }_\tau(\wt{x}))\right)\right)(X^{s,m }_\tau(x))\bigg)dm(\wt{x})\\\no
&\ \ \ \ \ \ \ \ \ \ \ \ \ \ +\p_x f(X^{s,m }_\tau(x),X^{s,m }_\tau\ot m,\alpha(X^{s,m }_\tau(x),X^{s,m }_\tau\ot m,Z^{s,m }_\tau(x)))\cdot Z^{s,m }_\tau(x)\\
&\ \ \ \ \ \ \ \ \ \ \ \ \ \ +\p_x g(X^{s,m }_\tau(x),X^{s,m }_\tau\ot m,\alpha(X^{s,m }_\tau(x),X^{s,m }_\tau\ot m,Z^{s,m }_\tau(x)))\Bigg)d\tau,
\end{align}\normalsize
where $Z^{s,m }_\tau(x):=\gamma\big(\tau,X^{s,m }_\tau(x),X^{s,m }_\tau\ot m\big)$ and
\small\begin{align}     
X^{s,m}_\tau(x) = x+\int_s^\tau f\Big(X^{s,m }_{\wt{\tau}}(x),X^{s,m }_{\wt{\tau}}\ot m,\alpha\big(X^{s,m }_{\wt{\tau}}(x),X^{s,m }_{\wt{\tau}}\ot m,\gamma(\tau,X^{s,m }_{\wt{\tau}}(x),X^{s,m }_{\wt{\tau}}\ot m)\big)\Big)d\wt{\tau}.
\end{align}\normalsize
Therefore, for any $t\in [t_{N-1},t_N]$, we paste together the local solutions $\gamma^{(N)}(s,x,\mu)$ and $\gamma^{(N-1)}(s,x,\mu)$ consecutively in order by extending $\gamma^{(N-1)}(s,x,\mu)$ as follows: for any $s\in[t_N,T]$,  $\gamma^{(N-1)}(s,x,\mu)=\gamma^{(N)}(s,x,\mu)$.
We still denote the extended solution over the time horizon up to  the terminal time $[t_{N-2},T]$ by $\gamma^{(N-1)}(s,x,\mu)$ without much cause of ambiguity, which thus solves the equation \eqref{gammaeq} over the time horizon up to  the terminal time $[t_{N-2},T]$. In addition, $\gamma^{(N-1)}(s,x,\mu)$ on the time interval $[t_{N-2},t_N]$ still has the regularity $(i)$ stated in Theorem \ref{Thm6_2} and satisfies $\vertiii{\gamma^{(N-1)}}_{2,[t_{N-2},T]}\leq 2L^*_0$. Moreover, applying Theorem \ref{Crucial_Estimate} again, we also have $\vertiii{\gamma^{(N-1)}}_{2,[t_{N-2},T]}\leq L^*_0$. 

Step 4. On the interval $[t_{N-3},t_{N-1}]$, we repeat Step 2-Step 3. More precisely, due to the uniform (in time) estimate \eqref{JFE} in Theorem \ref{Crucial_Estimate}, the minimal lifespan $\wb{\eps}_2$ is uniform in all sub-intervals, so we can do the same pasting procedure for all the rest of the other intervals $[t_{N-4},t_{N-2}]$,...,$[0 ,t_2]$ in a backward manner. As a result, we can obtain a global solution $\gamma(s,x,\mu)$ of the equation \eqref{gammaeq} on the whole time horizon $[0,T]$, so that it has the regularity $(i)$ stated in Theorem \ref{Thm6_2} and satisfies $\vertiii{\gamma}_{2,[0,T]}\leq L^*_0$ as well. 

Step 5. For any $t\in[0,T]$ and $m\in\mc{P}_2(\R^{d_x})$, define the pair $\big(X^{t,m}_s(x),Z^{t,m}_s(x)\big)$ by \eqref{xeq} and $Z^{t,m}_s(x):=\gamma(s,X^{t,m}_s(x),X^{t,m}_s\ot m)$ for all $s\in[t,T]$ and $x\in\R^{d_x}$, then it is a solution pair of the FBODE system \eqref{MPIT}, and \eqref{eq_7_43} and \eqref{eq_7_44} are valid. The regularity of $\big(X^{t,m}_s(x),Z^{t,m}_s(x)\big)$ stated in this theorem statement is directly follows from that of $\gamma$ stated there. In addition, by the same argument as that in the proof of \eqref{Bp_xX}-\eqref{Bp_tX}, we have \eqref{Bp_xX_G}-\eqref{Bp_tX_G} and then \eqref{Bp_xZ_G}-\eqref{Bp_tZ_G} by using $\vertiii{\gamma}_{2,[0,T]}\leq L^*_0$.
\end{proof}

\begin{theorem}\label{GlobalSol_Bellman}
(Global existence in time of the Bellman equation \eqref{bellman}). Under Assumption ${\bf(a1)}$-${\bf(a3)}$ and Hypothesis ${\bf(h1)}$-${\bf(h2)}$, then there exists a global solution $v(t,m)$, which is given by \eqref{definev}, of the Bellman equation \eqref{bellman}, which is differentiable in $t\in[0 ,T]$ and is $L$-differentiable in $m\in\mc{P}_2(\R^{d_x})$ with its corresponding derivatives $\p_t v(t,m)$ and $\p_m v(t,m)(x)$ being continuous in their corresponding arguments. In addition, $\p_t v(t,m)$ is differentiable in $t\in[0,T]$ and $L$-differentiable in $m\in\mc{P}_2(\R^{d_x})$, and $\p_m v(t,m)(x)$ is differentiable in $t\in[0,T]$, $x\in\R^{d_x}$ and $L$-differentiable in $m\in\mc{P}_2(\R^{d_x})$.
\end{theorem}
\begin{proof}
Under Assumption ${\bf(a1)}$-${\bf(a3)}$ and Hypothesis ${\bf(h1)}$-${\bf(h2)}$, by Theorem \ref{GlobalSol}, we then have a global solution pair $\big(X^{t,m}_s(x),Z^{t,m}_s(x)\big)_{x\in\R^{d_x},s\in[t,T]}$ of the FBODE system \eqref{MPIT} for any $t\in[0 ,T]$ and $m\in\mc{P}_2(\R^{d_x})$. Define $v(t,m)$ by \eqref{definev}, then, by Theorem \ref{MPBE}, this $v(t,m)$ is a global solution of the Bellman equation \eqref{bellman}, and its regularity directly follows from the regularity of $\big(X^{t,m}_s(x),Z^{t,m}_s(x)\big)$ by using \eqref{definev}, \eqref{p_t_v} and \eqref{p_m_v}. Moreover, by \eqref{first_order_condition}, \eqref{p_t_v} and \eqref{p_m_v}, 
\begin{align*}
&\p_t\p_tv(t,m)=-\int_{\R^{d_x}}\p_t\Big(Z^{t,m}_t(x)\Big)\cdot f(x,m,\alpha(x,m,Z^{t,m}_t(x)))
dm(x)\\
&\ \ \ \ \ \ \ \ \ \ \ \ \ \ =-\int_{\R^{d_x}}\p_t\gamma(t,x,m)\cdot f(x,m,\alpha(x,m,\gamma(t,x,m)))
dm(x);\\
&\p_m\Big(\p_tv(t,m)\Big)(y)=-\int_{\R^{d_x}}\bigg(\p_m\Big(Z^{t,m}_t(x)\Big)(y)\cdot f(x,m,\alpha(x,m,Z^{t,m}_t(x)))\\
&\ \ \ \ \ \ \ \ \ \ \ \ \ \ \ \ \ \ \ \ \ \ \ \ \ \ \ \ \ \ \ \ \ \ \ \ \ \ +Z^{t,m}_t(x)\cdot \p_m f(x,m,\alpha(x,m,Z^{t,m}_t(x)))(y)\\
&\ \ \ \ \ \ \ \ \ \ \ \ \ \ \ \ \ \ \ \ \ \ \ \ \ \ \ \ \ \ \ \ \ \ \ \ \ \ +\p_m g(x,m,\alpha(x,m,Z^{t,m}_t(x)))(y)\bigg)dm(x)\\
&\ \ \ \ \ \ \ \ \ \ \ \ \ \ \ \ \ \ \ \ \ \ \ =-\int_{\R^{d_x}}\bigg(\p_m\gamma(t,x,m)(y)\cdot f(x,m,\alpha(x,m,\gamma(t,x,m)))\\
&\ \ \ \ \ \ \ \ \ \ \ \ \ \ \ \ \ \ \ \ \ \ \ \ \ \ \ \ \ \ \ \ \ \ \ \ \ \ +\gamma(t,x,m)\cdot \p_m f(x,m,\alpha(x,m,\gamma(t,x,m)))(y)\\
&\ \ \ \ \ \ \ \ \ \ \ \ \ \ \ \ \ \ \ \ \ \ \ \ \ \ \ \ \ \ \ \ \ \ \ \ \ \ +\p_m g(x,m,\alpha(x,m,\gamma(t,x,m)))(y)\bigg)dm(x);\\
&\p_m\Big(\p_m v(t,m)(x)\Big)(y)=\p_m \Big(Z^{t,m}_t(x)\Big)(y)=\p_m \gamma(t,x,m)(y);\\ 
&\p_x\Big(\p_m v(t,m)(x)\Big)=\p_x \Big(Z^{t,m}_t(x)\Big)=\p_x \gamma(t,x,m);\\
&\p_t\Big(\p_m v(t,m)(x)\Big)=\p_t \Big(Z^{t,m}_t(x)\Big)=\p_t\gamma(t,x,m).
\end{align*}
Hence, the regularities of $\p_t v$ and $\p_m v$ also follow from that of $\gamma$ in Theorem \ref{GlobalSol} immediately.
\end{proof}

\section{Local and Global Uniqueness}\label{sec:unique}

\begin{theorem}\label{Thm6_3} (Local uniqueness of solutions to the FBODE system \eqref{fbodesystem})
Assume that the drift function $f$ and the running cost $g$ satisfy Assumptions ${\bf(a1)}$ and ${\bf(a2)}$ respectively, the terminal function $p$ satisfies Assumption ${\bf(P)}$ with $L_p\leq \wb{L}_p:=\max\left\{\wb{L}_k,L^*_0\right\}$ and the relationship \eqref{p_aa_f} among $f$, $g$ and $p$ is valid, then there exists a constant $\eps_7=\eps_7(\wb{L}_p; L_f,\Lambda_f,\wb{l}_f,L_g,\Lambda_g,\wb{l}_g,L_\alpha)>0$, such that, for any fixed $0\leq t\leq \wt{T}\leq T$ with $\wt{T}-t\leq \eps_7$ and $m\in \mathcal{P}_2(\R^{d_x})$, if there exist two solution pairs $\big(X^{t,m}_s(x),Z^{t,m}_s(x)\big)_{x\in\R^{d_x},s\in[t,\wt{T}]}$ and $\big(\wt{X}^{t,m}_s(x),\wt{Z}^{t,m}_s(x)\big)_{x\in\R^{d_x},s\in[t,\wt{T}]}$ of FBODE system \eqref{fbodesystem}, each of them is continuously differentiable in $s\in[t,\wt{T}]$, then they must be equal for all $s\in[t,\wt{T}]$ and $x\in\R^{d_x}$.
\end{theorem}
\begin{proof}Step 1 (Both solutions belong to the same cone $c_{k_0}$) 
By \eqref{bdd_d1_f}, \eqref{ligf}, \eqref{LGgDerivatives}, \eqref{linalpha}, \eqref{fbodesystem} and \eqref{p2}, we have
\footnotesize\begin{align}\label{eq_9_1}
\begin{cases}       
|X^{t,m}_s(x)-x| \leq (s-t)\cdot \sup_{\tau\in[t,s]}\bigg(L_f(1+L_\alpha)\Big(1+|X^{t,m}_\tau(x)|+\|X^{t,m}_\tau\|_{L^1_m}\Big)+L_fL_\alpha|Z^{t,m}_\tau(x)|\bigg),\\
|Z^{t,m}_s(x)| \leq \wb{L}_p\Big(1+|X^{t,m}_{\wt{T}}(x)|+\|X^{t,m}_{\wt{T}}\|_{L^1_m}\Big)\\
\ \ \ \ \ \ \ \ \ \ \ \ \ \ \ \ +(\wt{T}-s)\cdot\sup_{\tau\in[s,\wt{T}]}\bigg(\Lambda_f\Big(|Z^{t,m}_\tau(x)|+\|Z^{t,m}_\tau\|_{L^1_m}\Big)+L_g(1+L_\alpha)\Big(2+|X^{t,m}_\tau(x)|+3\|X^{t,m}_\tau\|_{L^1_m}\Big)\\
\ \ \ \ \ \ \ \ \ \ \ \ \ \ \ \ \ \ \ \ \ \ \ \ \ \ \ \ \ \ \ \ \ \ \ \ \ \ \ \ \ \ \ \ \ \ \ \ +L_gL_\alpha|Z^{t,m}_\tau(x)|+L_gL_\alpha\|Z^{t,m}_\tau\|_{L^1_m}+L_g|X^{t,m}_\tau(x)|\bigg),
\end{cases}
\end{align}\normalsize
and thus, by integrating with respect to $x$,
\begin{align}
\begin{cases}       
\|X^{t,m}_s\|_{L^1_m} \leq \|m\|_1+(s-t)\cdot \sup_{\tau\in[t,s]}\bigg(L_f(1+L_\alpha)\Big(1+2\|X^{t,m}_\tau\|_{L^1_m}\Big)+L_fL_\alpha\|Z^{t,m}_\tau\|_{L^1_m}\bigg),\\
\|Z^{t,m}_s\|_{L^1_m} \leq \wb{L}_p\Big(1+2\|X^{t,m}_{\wt{T}}\|_{L^1_m}\Big)\\
\ \ \ \ \ \ \ \ \ \ \ \ \ \ \ \ +(\wt{T}-s)\cdot\sup_{\tau\in[s,\wt{T}]}\bigg(2\Lambda_f \|Z^{t,m}_\tau\|_{L^1_m}+L_g(1+L_\alpha)\Big(2+4\|X^{t,m}_\tau\|_{L^1_m}\Big)\\
\ \ \ \ \ \ \ \ \ \ \ \ \ \ \ \ \ \ \ \ \ \ \ \ \ \ \ \ \ \ \ \ \ \ \ \ \ \ \ \ \ \ \ \ \ \ \ \ +2L_gL_\alpha\|Z^{t,m}_\tau\|_{L^1_m}+L_g\|X^{t,m}_\tau\|_{L^1_m}\bigg),
\end{cases}
\end{align}\normalsize
which implies that
\begin{align}
\begin{cases}       
\sup_{\tau\in[t,\wt{T}]}\|X^{t,m}_\tau\|_{L^1_m} \leq \frac{1}{1-2L_f(1+L_\alpha)\eps_7}\|m\|_1+\frac{L_f(1+L_\alpha)\eps_7}{1-2L_f(1+L_\alpha)\eps_7}+\frac{L_fL_\alpha\eps_7}{{1-2L_f(1+L_\alpha)\eps_7}}\sup_{\tau\in[t,\wt{T}]}\|Z^{t,m}_\tau\|_{L^1_m};\\
\sup_{\tau\in[t,\wt{T}]}\|Z^{t,m}_\tau\|_{L^1_m} \leq \frac{\wb{L}_p+2L_g(1+L_\alpha)\eps_7}{1-2(\Lambda_f+L_gL_\alpha)\eps_7}+\frac{2\wb{L}_p+4L_g(1+L_\alpha)\eps_7+L_g\eps_7}{1-2(\Lambda_f+L_gL_\alpha)\eps_7}\sup_{\tau\in[t,\wt{T}]}\|X^{t,m}_\tau\|_{L^1_m}.
\end{cases}
\end{align}\normalsize
Therefore, for any 
\footnotesize\begin{align}\label{eps_8}
\eps_7\leq\min\bigg\{\frac{1}{800L_f(1+L_\alpha)},\frac{1}{800(\Lambda_f+L_gL_\alpha)},\frac{1}{800L_g(1+L_\alpha)},\frac{1}{400L_g},\frac{1}{400(2\wb{L}_p+1)L_fL_\alpha}\bigg\}=:\eps_8,
\end{align}\normalsize
we have
\begin{align}
\begin{cases}       
\sup_{\tau\in[t,\wt{T}]}\|X^{t,m}_\tau\|_{L^1_m} \leq \frac{400}{399}\|m\|_1+\frac{1}{798}+\frac{1}{399(2\wb{L}_p+1)}\sup_{\tau\in[t,\wt{T}]}\|Z^{t,m}_\tau\|_{L^1_m};\\
\sup_{\tau\in[t,\wt{T}]}\|Z^{t,m}_\tau\|_{L^1_m} \leq \frac{400}{399}(\wb{L}_p+\frac{1}{400})+\frac{400}{399}(2\wb{L}_p+\frac{3}{400})\sup_{\tau\in[t,\wt{T}]}\|X^{t,m}_\tau\|_{L^1_m},
\end{cases}
\end{align}\normalsize
which implies that
\begin{align}\label{eq_9_6}
\begin{cases}       
\sup_{\tau\in[t,\wt{T}]}\|X^{t,m}_\tau\|_{L^1_m} \leq 1.006\|m\|_1+0.003;\\
\sup_{\tau\in[t,\wt{T}]}\|Z^{t,m}_\tau\|_{L^1_m} \leq 1.003(\wb{L}_p+0.004)(1.006+2.012\|m\|_1),
\end{cases}
\end{align}\normalsize
where in the rest of this proof, we use some constants in numerical form instead of the precise fractional expressions so as to simplify our presentation without cause of any numerical errors.  
Thus, by \eqref{eq_9_1}, \eqref{eps_8} and \eqref{eq_9_6}, 
\footnotesize\begin{align}
\begin{cases}       
\sup_{\tau\in[t,\wt{T}]}|X^{t,m}_\tau(x)| \leq \frac{800}{799}|x|+\frac1{799}\Big(1+1.006\|m\|_1+0.003\Big)+\frac{2}{799(2\wb{L}_p+1)}\sup_{\tau\in[t,\wt{T}]}|Z^{t,m}_\tau(x)|;\\
\sup_{\tau\in[t,\wt{T}]}|Z^{t,m}_\tau(x)| \leq \frac{800}{799}(\wb{L}_p+\frac3{800})\Big(1+\sup_{\tau\in[t,\wt{T}]}|X^{t,m}_\tau(x)|+1.006\|m\|_1+0.003\Big)+0.002(\wb{L}_p+0.004)(1.006+2.012\|m\|_1),
\end{cases}
\end{align}\normalsize
which implies that
\begin{align}\label{eq_9_8}
\begin{cases}       
\sup_{\tau\in[t,\wt{T}]}|X^{t,m}_\tau(x)| \leq 1.005|x|+0.005\Big(1+\|m\|_1\Big);\\
\sup_{\tau\in[t,\wt{T}]}|Z^{t,m}_\tau(x)| \leq 1.05(\wb{L}_p+0.008)(1+|x|+\|m\|_1).
\end{cases}
\end{align}\normalsize
Therefore, by \eqref{eq_9_1}, \eqref{eps_8}, \eqref{eq_9_6} and \eqref{eq_9_8}, 
\begin{align}\label{eq_9_9}
|X^{t,m}_s(x)-x| \leq 0.008(1+|x|+\|m\|_1),
\end{align}\normalsize
which implies, after integrating with respect to $x$,
\begin{align}\label{eq_9_10}
\|X^{t,m}_s\|_{L^1_m} \geq \|m\|_1-0.008(1+2\|m\|_1).
\end{align}
By \eqref{eq_9_8}-\eqref{eq_9_10}, we can obtain
\footnotesize\begin{align}
|Z^{t,m}_s(x)|\leq 1.1(\wb{L}_p+0.008)\Big(1+|X^{t,m}_s(x)|+\|X^{t,m}_s\|_{L^1_m}\Big)\leq 2\wb{L}_p\Big(1+|X^{t,m}_s(x)|+\|X^{t,m}_s\|_{L^1_m}\Big).
\end{align}\normalsize
Here, we can assume, without loss of generality, that $\wb{L}_p\geq \wb{L}_k\geq 1$. 
Thus we have $\big(X^{t,m}_s(x),X^{t,m}_s\ot m,Z^{t,m}_s(x)\big)\in c_{k_0}$ with $k_0=4\wb{L}_p$, and the same estimates also hold for $\big(\wt{X}^{t,m}_s(x),\wt X^{t,m}_s\ot m,\wt{Z}^{t,m}_s(x)\big)$. 
In addition, by using \eqref{eq_9_8}, \eqref{eq_9_10} and $\wb{L}_p\geq 1$, we can obtain, for all $|x|\leq 0.8(1+\|m\|_1)$, $|Z^{t,m}_s(x)|\leq 2\wb{L}_p\Big(1+\|X^{t,m}_s\|_{L^1_m}\Big)$ and $|\wt Z^{t,m}_s(x)|\leq 2\wb{L}_p\Big(1+\|\wt X^{t,m}_s\|_{L^1_m}\Big)$, and thus $\theta \big(X^{t,m}_s(x),X^{t,m}_s\ot m,Z^{t,m}_s(x)\big)+(1-\theta)\big(\wt{X}^{t,m}_s(x),\wt X^{t,m}_s\ot m,\wt{Z}^{t,m}_s(x)\big)\in c_{k_0}$ for all $\theta\in[0,1]$. On the other hand, for all $|x|\geq 0.8(1+\|m\|_1)$, by using \eqref{eq_9_9}, we have $X^{t,m}_s(x), \wt X^{t,m}_s(x)\geq 0.7(1+\|m\|_1)>0$ when $x\geq 0.8(1+\|m\|_1)$, and $X^{t,m}_s(x), \wt X^{t,m}_s(x)\leq 0.7(1+\|m\|_1)>0$ when $x\leq 0.8(1+\|m\|_1)$, which implies that $X^{t,m}_s(x)$ and $\wt X^{t,m}_s(x)$ always have the same sign; thus, we also have $\theta \big(X^{t,m}_s(x),X^{t,m}_s\ot m,Z^{t,m}_s(x)\big)+(1-\theta)\big(\wt{X}^{t,m}_s(x),\wt X^{t,m}_s\ot m,\wt{Z}^{t,m}_s(x)\big)\in c_{k_0}$ for all $\theta\in[0,1]$.

Since $\big(X^{t,m}_s(x),Z^{t,m}_s(x)\big)\in c_{k_0}$ and $\big(\wt{X}^{t,m}_s(x),\wt{Z}^{t,m}_s(x)\big)\in c_{k_0}$, by Proposition \ref{A5}, the optimal control $\alpha(x,\mu,z)$ is uniquely solvable in the lifespan $[t,\wt{T}]$ of these two solutions via the first order condition \eqref{first_order_condition}.

Step 2 (Uniqueness) Let $\wt{T}-t\leq \eps_7$, where $\eps_7$ is suitably small depending only on $\wb{L}_p,\ L_f,\ \Lambda_f,\ \wb{l}_f,\ L_g,\ \Lambda_g,\ \wb{l}_g$ and $L_\alpha$. Recall the notations defined in Appendix \eqref{notation_1} and Table \ref{notation_2}, and we use similar notations for 
\begin{align}
&\wt{f}^{t,m}_s(x):=f\left(\wt{X}^{t,m}_s(x),\wt{X}^{t,m}_s\ot m,\alpha(\wt{X}^{t,m}_s(x),\wt{X}^{t,m}_s\ot m,\wt{Z}^{t,m}_s(x))\right),\text{ and }\\
&\wt{g}^{t,m}_s(x):=g\left(\wt{X}^{t,m}_s(x),\wt{X}^{t,m}_s\ot m,\alpha(\wt{X}^{t,m}_s(x),\wt{X}^{t,m}_s\ot m,\wt{Z}^{t,m}_s(x))\right),
\end{align}
by replacing all $X^{t,m}_s(x)$'s and $Z^{t,m}_s(x)$'s by $\wt{X}^{t,m}_s(x)$'s and $\wt{Z}^{t,m}_s(x)$'s respectively; the same custom applies to $(\p_x f)^{t,m}_s(x)$, $(\p_x g)^{t,m}_s(x)$ and others.
Denote $\big(\Delta X^{t,m}_s(x),\Delta Z^{t,m}_s(x)\big):=\big(X^{t,m}_s(x),Z^{t,m}_s(x)\big)-\big(\wt{X}^{t,m}_s(x),\wt{Z}^{t,m}_s(x)\big)$, then, by \eqref{fbodesystem}, we have, 
\footnotesize\begin{align}\label{fbodesystem_dif}
\begin{cases}       
\dfrac{d}{ds}\Delta X^{t,m}_s(x) = \Delta f^{t,m}_s(x),\ 
\Delta X^{t,m}_t(x)= 0,\\
\,\dfrac{d}{ds}\Delta Z^{t,m}_s(x) =-  \Delta(\p_xf)^{t,m}_s(x)\cdot \wt{Z}^{t,m}_s(x)-  (\p_xf)^{t,m}_s(x)\cdot \Delta Z^{t,m}_s(x)-\Delta (\p_x g)^{t,m}_s(x)\\
\ \ \ \ \ \ \ \ \ \ \ \ \ \ \ \ \ \ \ \ \ -\displaystyle \int_{\R^{d_x}}\Bigg(\Delta  (\p_\mu f)^{t,m}_s(\wt{x},x)\cdot  \wt{Z}^{t,m}_s(\wt{x})+(\p_\mu f)^{t,m}_s(\wt{x},x)\cdot  \Delta  Z^{t,m}_s(\wt{x})+ \Delta(\p_\mu g)^{t,m}_s(\wt{x},x)\Bigg)dm(\wt{x}),\\ 
\ \ \ \ \Delta Z^{t,m}_{\wt{T}}(x)=p(X^{t,m}_{\wt{T}}(x),X^{t,m}_{\wt{T}}\ot m)-p(\wt{X}^{t,m}_{\wt{T}}(x),\wt{X}^{t,m}_{\wt{T}}\ot m),
\end{cases}
\end{align}\normalsize
where 
\footnotesize\begin{align}\no
&\Delta f^{t,m}_s(x):=f^{t,m}_s(x)-\wt{f}^{t,m}_s(x),\ \Delta (\p_xf)^{t,m}_s(x):=(\p_x f)^{t,m}_s(x)-(\p_x \wt{f})^{t,m}_s(x);\\\no 
&\Delta  (\p_\mu f)^{t,m}_s(\wt{x},x):=(\p_\mu f)^{t,m}_s(\wt{x},x)-(\p_\mu \wt{f})^{t,m}_s(\wt{x},x);\\\no
&\Delta (\p_x g)^{t,m}_s(x):=(\p_x g)^{t,m}_s(x)-(\p_x \wt{g})^{t,m}_s(x),\ \Delta (\p_\mu g)^{t,m}_s(x,y):=(\p_\mu g)^{t,m}_s(x,y)-(\p_\mu \wt{g})^{t,m}_s(x,y).
\end{align}
\normalsize
By \eqref{bdd_d2_f}, \eqref{Lipf}-\eqref{Lipg_x}, \eqref{lipalpha_old}, \eqref{lipp} and \eqref{eq_9_10}, we have, for a suitably small $\eps_7$ and for any $s\in[t,\wt{T}]$,
\footnotesize\begin{align*}
&|\Delta f^{t,m}_s(x)|\leq L_f(1+L_\alpha) \left(\left|\Delta X^{t,m}_s(x)\right|+\left\|\Delta X^{t,m}_s\right\|_{L^1_m}\right)+L_f L_\alpha\left|\Delta Z^{t,m}_s(x)\right|,\\
&|\Delta (\p_xf)^{t,m}_s(x)|\leq \frac{2\wb{l}_f(1+L_\alpha)}{1+|x|+\|m\|_1} \left(\left|\Delta X^{t,m}_s(x)\right|+\left\|\Delta X^{t,m}_s\right\|_{L^1_m}\right)+\frac{2\wb{l}_f L_\alpha}{1+|x|+\|m\|_1}\left|\Delta Z^{t,m}_s(x)\right|,\\
&|\Delta (\p_\mu f)^{t,m}_s(x,y)|\leq \frac{2\wb{l}_f(1+L_\alpha)}{1+|x|+\|m\|_1} \left(\left|\Delta X^{t,m}_s(x)\right|+\left\|\Delta X^{t,m}_s\right\|_{L^1_m}\right)+\frac{2\wb{l}_f}{{1+|x|+\|m\|_1}}\left|\Delta X^{t,m}_s(y)\right|+\frac{2\wb{l}_f L_\alpha}{1+|x|+\|m\|_1}\left|\Delta Z^{t,m}_s(x)\right|;\\
&|\Delta (\p_x g)^{t,m}_s(x)|\leq \max\{\Lambda_g,\wb{l}_g\}(1+L_\alpha) \left(\left|\Delta X^{t,m}_s(x)\right|+\left\|\Delta X^{t,m}_s\right\|_{L^1_m}\right)+\max\{\Lambda_g,\wb{l}_g\} L_\alpha\left|\Delta Z^{t,m}_s(x)\right|;\\
&|\Delta (\p_\mu g)^{t,m}_s(x,y)|\leq \max\{\Lambda_g,\wb{l}_g\}(1+L_\alpha) \left(\left|\Delta X^{t,m}_s(x)\right|+\left\|\Delta X^{t,m}_s\right\|_{L^1_m}\right)+\max\{\Lambda_g,\wb{l}_g\}\left|\Delta X^{t,m}_s(y)\right|+\max\{\Lambda_g,\wb{l}_g\} L_\alpha\left|\Delta Z^{t,m}_s(x)\right|;\\
&|\Delta Z^{t,m}_{\wt{T}}(x)|\leq L_p \left(\left|\Delta X^{t,m}_{\wt{T}}(x)\right|+\left\|\Delta X^{t,m}_{\wt{T}}\right\|_{L^1_m}\right).
\end{align*}\normalsize
Therefore,
\small\begin{align}\label{eq_9_15_1}
\begin{cases}       
|\Delta X^{t,m}_s(x)| \leq \int_t^s  \bigg(L_f(1+L_\alpha)\left(\left|\Delta X^{t,m}_\tau(x)\right|+\left\|\Delta X^{t,m}_\tau\right\|_{L^1_m}\right)+L_f L_\alpha\left|\Delta Z^{t,m}_\tau(x)\right|\bigg)\,d\tau,\\
\ \ \ \ \ \ \ \ \ \ \ \ \ \ \ \ \leq \eps_7\sup_{\tau\in[t,\wt{T}]}\bigg(L_f(1+L_\alpha)\left(\left|\Delta X^{t,m}_\tau(x)\right|+\left\|\Delta X^{t,m}_\tau\right\|_{L^1_m}\right)+L_f L_\alpha\left|\Delta Z^{t,m}_\tau(x)\right|\bigg)\\
|\Delta Z^{t,m}_s(x)| \leq L_p \left(\left|\Delta X^{t,m}_{\wt{T}}(x)\right|+\left\|\Delta X^{t,m}_{\wt{T}}\right\|_{L^1_m}\right)\\
\ \ \ \ \ \ \ \ \ \ \ \ \ \ \ \ \ \ +\eps_7\cdot\sup_{s\in[t,\wt{T}]}\bigg(\Big(4\wb{l}_f(1+L_\alpha)\wb{L}_p+\max\{\Lambda_g,\wb{l}_g\}(1+L_\alpha)\Big) \left(2\left|\Delta X^{t,m}_s(x)\right|+3\left\|\Delta X^{t,m}_s\right\|_{L^1_m}\right)\\
\ \ \ \ \ \ \ \ \ \ \ \ \ \ \ \ \ \ \ \ \ \ \ \ \ \ \ \ \ \ \ \ \ \ \ \ \ \ \ \ +\Big(4\wb{l}_f L_\alpha\wb{L}_p+2\Lambda_f+\max\{\Lambda_g,\wb{l}_g\} L_\alpha\Big)\left(\left|\Delta Z^{t,m}_s(x)\right|+\left\|\Delta Z^{t,m}_s\right\|_{L^1_m}\right)\bigg).
\end{cases}
\end{align}\normalsize
By an argument similar to establishing \eqref{eq_9_1}-\eqref{eq_9_8}, we can show that $\sup_{x\in\R^{d_x},s\in[t,\wt{T}]}|\Delta X^{t,m}_s(x)|=0$ and $\sup_{x\in\R^{d_x},s\in[t,\wt{T}]}|\Delta Z^{t,m}_s(x)|=0$ for a suitably small $\eps_7$ depending only on $\wb{L}_p,\ L_f,\ \Lambda_f,\ \wb{l}_f,\ L_g,\ \Lambda_g,\ \wb{l}_g$ and $L_\alpha$. More precisely, one can integrate \eqref{eq_9_15_1} with respect to $x$, and then eliminate $\sup_{s\in[t,\wt{T}]} \| \Delta X^{t,m}_s \|_{L^1_m}$ to obtain $\sup_{s\in[t,\wt{T}]} \| \Delta Z^{t,m}_s \|_{L^1_m}= O(\epsilon_7) \sup_{s\in[t,\wt{T}]} \| \Delta Z^{t,m}_s \|_{L^1_m}$, which implies $\sup_{s\in[t,\wt{T}]} \| \Delta Z^{t,m}_s \|_{L^1_m}= 0$ (and hence, $\sup_{s\in[t,\wt{T}]} \| \Delta X^{t,m}_s \|_{L^1_m}= 0)$,  if we choose $\epsilon_7$ to be sufficiently small so that $O(\epsilon_7) < 1$. Substituting $\sup_{s\in[t,\wt{T}]} \| \Delta X^{t,m}_s \|_{L^1_m}= \sup_{s\in[t,\wt{T}]} \| \Delta Z ^{t,m}_s \|_{L^1_m}= 0$ into \eqref{eq_9_15_1}, and repeating the same elimination argument similar to establishing \eqref{eq_9_8}, one can finally show that $\sup_{s\in[t,\wt{T}],x\in\R^{d_x}} | \Delta X ^{t,m}_s(x) | = \sup_{s\in[t,\wt{T}],x\in\R^{d_x}} | \Delta Z ^{t,m}_s(x) | = 0$ by choosing $\epsilon_7$ sufficiently small.
\end{proof}


\begin{theorem}\label{Thm6_4} (Global uniqueness of solution to the FBODE system \eqref{MPIT})
Under Assumption ${\bf(a1)}$-${\bf(a3)}$ and Hypothesis ${\bf(h1)}$-${\bf(h2)}$, then, for any fixed terminal time $T>0$ and initial distribution $m\in\mc{P}_2(\R^{d_x})$, if there exists a continuously differentiable, in $s\in[0,T]$, solution pair $\big(\wt{X}^{0,m}_s(x),\wt{Z}^{0,m}_s(x)\big)_{x\in\R^{d_x},s\in[0,T]}$ of FBODE system \eqref{MPIT} subject to the initial distribution $m$ at the initial time $t=0$, then it must be equal to the solution pair $\big(X^{0,m}_s(x),Z^{0,m}_s(x):=\gamma(s,X^{0,m}_s(x),X^{0,m}_s\ot m)\big)_{x\in\R^{d_x},s\in[0,T]}$, where the decoupling field $\gamma$ was constructed in Theorem \ref{GlobalSol}.
\end{theorem}
\begin{proof}
Under Assumption ${\bf(a1)}$-${\bf(a3)}$ and Hypothesis ${\bf(h1)}$-${\bf(h2)}$, according to Theorem \ref{GlobalSol}, we have a global decoupling field $\gamma(s,x,\mu)$ for the FBODE system \eqref{MPIT} on $[0,T]$ satisfying the equation \eqref{gammaeq}. Let $\wb{L}_p:=\max\left\{\wb{L}_k,L^*_0\right\}=L^*_0$ and $\eps_7=\eps_7(\wb{L}_p; L_f,\Lambda_f,\wb{l}_f,L_g,\Lambda_g,\wb{l}_g,L_\alpha)>0$ given in Theorem \ref{Thm6_3}. Divide the interval $[0 ,T]$ into $0  < t_1 <...<t_{N-1}:=T-\eps_7<t_N =:T$, where $t_{N-i}=T-i\cdot \eps_7$, $i=0,1,...,N-1$ and $N$ is the smallest integer such that $N\geq T/\eps_7$. Note that the interval $[0 ,t_1=T-(N-1)\eps_7]$ has a length not longer than those other intervals $[t_1,t_2]$,..., and $[t_{N-1},t_N]$, everyone of which has a common length of size $\eps_7$. \\

Step 1. On the interval $[t_{N-1},t_N=T]$, define $y:=\wt{X}^{0,m}_{t_{N-1}}(x)$, $\wt{m}:=\wt{X}^{0,m}_{t_{N-1}}\ot m$ and $p(x,\mu):=\int_{\R^{d_x}} \p_\mu k( \wt{x} ,\mu)(x)d\mu(\wt{x})+\p_x k(x,\mu)$, then $\big(X^{t_{N-1},\wt{m}}_{s}(y),Z^{t_{N-1},\wt{m}}_{s}(y)\big):=\big(\wt{X}^{0,m}_s(x),\wt{Z}^{0,m}_s(x)\big)$ \footnote{Since the initial measure at time $t_N$ is $\wt{m}=\wt{X}^{0,m}_{t_{N-1}}\ot m$ where $\wt{X}^{0,m}_{t_{N-1}}(\cdot)$ is a push-forward mapping, so the complement set $S^c$ of the set $S:=\{y:y=\wt{X}^{0,m}_{t_{N-1}}(x),x\in\R^{d_x}\}=\wt{X}^{0,m}_{t_{N-1}}(\R^{d_x})$ is $\wt{m}$-measure zero, that is $\wt{m}(S^c)=m(\emptyset)=0$, and thus we can ignore those elements $y\in S^c$.} is a solution pair of FBODE system \eqref{fbodesystem} with the terminal data $Z^{t_{N-1},\wt{m}}_{T}(y)=p(X^{t_{N-1},\wt{m}}_{T}(y),X^{t_{N-1},\wt{m}}_{T}\ot \wt{m})$ and initial data $X^{t_{N-1},\wt{m}}_{t_{N-1}}(y)=y$. Note that,  for any $0\leq s\leq \tau \leq T$, $\wh{x}\in\R^{d_x}$ and $\wh{m}\in\mc{P}_2(\R^{d_x})$, the decoupling field $\gamma$ solves \eqref{gammaeq} with \eqref{xeq}:\\
\small\begin{align}\no
\gamma (s,\wh{x},\wh{m})
= &\ p(\wh{X}^{s,\wh{m}}_T(\wh{x}),\wh{X}^{s,\wh{m}}_T\ot \wh{m})\\\no
&+ \int_s^T \Bigg(\int_{\R^{d_x}}\bigg(\p_\mu f\left(\wh{X}^{s,\wh{m}}_\tau(\wt{x}),\wh{X}^{s,\wh{m}}_\tau\ot \wh{m},\alpha\left(\wh{X}^{s,\wh{m}}_\tau(\wt{x}),\wh{X}^{s,\wh{m}}_\tau\ot \wh{m}, \wh{Z}^{s,\wh{m}}_\tau(\wt{x})\right)\right)(\wh{X}^{s,\wh{m}}_\tau(\wh{x}))\cdot  \wh{Z}^{s,\wh{m}}_\tau(\wt{x})\\\no
&\ \ \ \ \ \ \ \ \ \ \ \ \ \ \ \ \ \ \ \ \ \ + \p_\mu g\left(\wh{X}^{s,\wh{m}}_\tau(\wt{x}),\wh{X}^{s,\wh{m}}_\tau\ot \wh{m},\alpha\left(\wh{X}^{s,\wh{m}}_\tau(\wt{x}),\wh{X}^{s,\wh{m}}_\tau\ot \wh{m}, \wh{Z}^{s,\wh{m}}_\tau(\wt{x}))\right)\right)(\wh{X}^{s,\wh{m}}_\tau(\wh{x}))\bigg)dm(\wt{x})\\\no
&\ \ \ \ \ \ \ \ \ \ \ \ \ \ +\p_x f(\wh{X}^{s,\wh{m}}_\tau(\wh{x}),\wh{X}^{s,\wh{m}}_\tau\ot \wh{m},\alpha(\wh{X}^{s,\wh{m}}_\tau(\wh{x}),\wh{X}^{s,\wh{m}}_\tau\ot \wh{m},\wh{Z}^{s,\wh{m}}_\tau(\wh{x})))\cdot \wh{Z}^{s,\wh{m}}_\tau(\wh{x})\\\no
&\ \ \ \ \ \ \ \ \ \ \ \ \ \ +\p_x g(\wh{X}^{s,\wh{m}}_\tau(\wh{x}),\wh{X}^{s,\wh{m}}_\tau\ot \wh{m},\alpha(\wh{X}^{s,\wh{m}}_\tau(\wh{x}),\wh{X}^{s,\wh{m}}_\tau\ot \wh{m},\wh{Z}^{s,\wh{m}}_\tau(\wh{x})))\Bigg)d\tau,
\end{align}\normalsize
where $\wh{Z}^{s,\wh{m}}_\tau(\wh{x}):=\gamma\big(\tau,\wh{X}^{s,\wh{m}}_\tau(\wh{x}),\wh{X}^{s,\wh{m}}_\tau\ot \wh{m}\big)$ and
\small\begin{align} \no    
\wh{X}^{s,\wh{m}}_\tau(\wh{x}) = \wh{x}+\int_s^\tau f\Big(\wh{X}^{s,\wh{m}}_{\wt{\tau}}(\wh{x}),\wh{X}^{s,\wh{m}}_{\wt{\tau}}\ot \wh{m},\alpha\big(\wh{X}^{s,\wh{m}}_{\wt{\tau}}(\wh{x}),\wh{X}^{s,\wh{m}}_{\wt{\tau}}\ot \wh{m},\gamma(\tau,\wh{X}^{s,\wh{m}}_{\wt{\tau}}(\wh{x}),\wh{X}^{s,\wh{m}}_{\wt{\tau}}\ot \wh{m})\big)\Big)d\wt{\tau}.
\end{align}\normalsize
Then, one can check directly that $\big(\wh{X}^{t_{N-1},\wt{m}}_s(y),\wh{Z}^{t_{N-1},\wt{m}}_s(y)\big)$ is also a solution pair of the FBODE system \eqref{fbodesystem} on the interval $[t_{N-1},t_N=T]$ with the same initial and terminal data as $\big(X^{t_{N-1},\wt{m}}_{s}(y),Z^{t_{N-1},\wt{m}}_{s}(y)\big)$. By Theorem \ref{Thm6_3}, $\big(\wh{X}^{t_{N-1},\wt{m}}_s(y),\wh{Z}^{t_{N-1},\wt{m}}_s(y)\big)=\big(X^{t_{N-1},\wt{m}}_{s}(y),Z^{t_{N-1},\wt{m}}_{s}(y)\big)=\big(\wt{X}^{0,m}_s(x),\wt{Z}^{0,m}_s(x)\big)$ and thus $\wt{Z}^{0,m}_s(x)=\gamma(s,\wt{X}^{0,m}_s(x),\wt{X}^{0,m}_s\ot m)$ on the interval $s\in[t_{N-1},t_N=T]$.

Step 2. On the interval $[t_{N-2},t_{N-1}]$, define $y:=\wt{X}^{0,m}_{t_{N-2}}(x)$, $\wt{m}:=\wt{X}^{0,m}_{t_{N-2}}\ot m$ and $p(x,\mu):=\gamma(t_{N-1},x,\mu)$, then $\big(X^{t_{N-2},\wt{m}}_{s}(y),Z^{t_{N-2},\wt{m}}_{s}(y)\big):=\big(\wt{X}^{0,m}_s(x),\wt{Z}^{0,m}_s(x)\big)$ is a solution pair of FBODE system \eqref{fbodesystem} with terminal data $Z^{t_{N-2},\wt{m}}_{t_{N-1}}(y)=p(X^{t_{N-2},\wt{m}}_{t_{N-1}}(y),X^{t_{N-2},\wt{m}}_{t_{N-1}}\ot \wt{m})$ and initial data $X^{t_{N-2},\wt{m}}_{t_{N-2}}(y)=y$. Same as Step 1, one can also check directly that $\big(\wh{X}^{t_{N-2},\wt{m}}_s(y),\wh{Z}^{t_{N-2},\wt{m}}_s(y)\big)$ is also a solution pair of FBODE system \eqref{fbodesystem} on the interval $[t_{N-2},t_{N-1}]$ with the same initial and terminal data as $\big(X^{t_{N-2},\wt{m}}_{s}(y),Z^{t_{N-2},\wt{m}}_{s}(y)\big)$. Since $\vertiii{\gamma }_{2,[0,T]}\leq L^*_0$, by Theorem \ref{Thm6_3} again, $\big(\wh{X}^{t_{N-2},\wt{m}}_s(y),\wh{Z}^{t_{N-2},\wt{m}}_s(y)\big)=\big(X^{t_{N-2},\wt{m}}_{s}(y),Z^{t_{N-2},\wt{m}}_{s}(y)\big)=\big(\wt{X}^{0,m}_s(x),\wt{Z}^{0,m}_s(x)\big)$, and thus $\wt{Z}^{0,m}_s(x)=\gamma(s,\wt{X}^{0,m}_s(x),\wt{X}^{0,m}_s\ot m)$ on the interval $s\in[t_{N-2},t_{N-1}]$. We can repeat the same procedure for all the rest of the other intervals $[t_{N-3},t_{N-2}]$,...,$[0 ,t_1]$ in a backward manner. As a result, we can obtain $\wt{Z}^{0,m}_s(x)=\gamma(s,\wt{X}^{0,m}_s(x),\wt{X}^{0,m}_s\ot m)$ on the whole interval $s\in[0,T]$.
\end{proof}

\begin{remark}
Recall that we have used the monotonicity conditions \eqref{positive_g_x}, \eqref{positive_g_mu}, \eqref{positive_k} and \eqref{positive_k_mu} in the proof of global existence of continuously differentiable decoupling field $\gamma$, so it is not surprising that the global existence of such a decoupling field $\gamma$ can implies the global uniqueness of the solution to the FBODE system \eqref{MPIT}.  
\end{remark}

\begin{theorem}\label{Unique_B} (Global uniqueness of the solution to the Bellman equation \eqref{bellman})
Under Assumption ${\bf(a1)}$-${\bf(a3)}$ and Hypothesis ${\bf(h1)}$-${\bf(h2)}$, the Bellman equation \eqref{bellman} has at most one solution $v(t,m)$ that has the regularity stated in Theorem \ref{GlobalSol_Bellman}.
\end{theorem}
\begin{proof}
Suppose that there exist two solutions $v(t,m)$ and $\wb{v}(t,m)$ to the Bellman equation \eqref{bellman}, both of which have the regularity stated in Theorem \ref{GlobalSol_Bellman}. Define $\gamma(t,x,m):=\p_m v(t,m)(x)$, $\wb{\gamma}(t,x,m):=\p_m \wb{v}(t,m)(x)$, and
\begin{align*}
X^{t,m}_s(x):= x+\int_t^s f\Big(X^{t,m}_\tau(x),X^{t,m}_\tau\ot m,\alpha\big(X^{t,m}_\tau(x),X^{t,m}_\tau\ot m,\gamma(\tau,X^{t,m}_\tau(x),X^{t,m}_\tau\ot m)\big)\Big)d\tau,\\
\wb{X}^{t,m}_s(x):= x+\int_t^s f\Big(\wb{X}^{t,m}_\tau(x),\wb{X}^{t,m}_\tau\ot m,\alpha\big(\wb{X}^{t,m}_\tau(x),\wb{X}^{t,m}_\tau\ot m,\wb{\gamma}(\tau,\wb{X}^{t,m}_\tau(x),\wb{X}^{t,m}_\tau\ot m)\big)\Big)d\tau,
\end{align*}
$Z^{t,m}_s(x):=\gamma(s,X^{t,m}_s(x),X^{t,m}_s\ot m)$ and $\wb{Z}^{t,m}_s(x):=\wb{\gamma}(s,\wb{X}^{t,m}_s(x),\wb{X}^{t,m}_s\ot m)$. Then, by taking the $L$-derivative $\p_m$ to the Bellman equation \eqref{nomin_1} and using the first order condition \eqref{first_order_condition} and $\gamma(t,x,m)=\p_m v(t,m)(x)$, we have 
\footnotesize\begin{align}\no
&\p_t \gamma(t,x,m)+f(x,m,\alpha(x,m,\gamma(t,x,m)))\cdot \p_x\gamma(t,x,m)+\int_{\R^{d_x}}f(y,m,\alpha(y,m,\gamma(t,y,m)))\cdot \p_\mu\gamma(t,x,m)(y)dm(y)\\\no
=&\ -\p_x f(x,m,\alpha(x,m,\gamma(t,x,m)))\cdot \gamma(t,x,m)-\p_x g(x,m,\alpha(x,m,\gamma(t,x,m)))\\
&\ -\int_{\R^{d_x}}\p_\mu f(y,m,\alpha(y,m,\gamma(t,y,m)))(x)\cdot \gamma(t,y,m)dm(y)-\int_{\R^{d_x}}\p_\mu g(y,m,\alpha(y,m,\gamma(t,y,m)))(x)dm(y),
\end{align}\normalsize
which implies that, by taking $t=s$, $x=X^{t,m}_s(x)$ and $m=X^{t,m}_s\ot m$,
\footnotesize\begin{align}\no
&\p_s \gamma(s,X^{t,m}_s(x),X^{t,m}_s\ot m)+f(X^{t,m}_s(x),X^{t,m}_s\ot m,\alpha(X^{t,m}_s(x),X^{t,m}_s\ot m,\gamma(s,X^{t,m}_s(x),X^{t,m}_s\ot m)))\cdot \p_x\gamma(s,X^{t,m}_s(x),X^{t,m}_s\ot m)\\\no
&+\int_{\R^{d_x}}f(X^{t,m}_s(y),X^{t,m}_s\ot m,\alpha(X^{t,m}_s(y),X^{t,m}_s\ot m,\gamma(s,X^{t,m}_s(y),X^{t,m}_s\ot m)))\cdot \p_\mu\gamma(s,X^{t,m}_s(x),X^{t,m}_s\ot m)(X^{t,m}_s(y))dm(y)\\\no
=&\ -\p_x f(X^{t,m}_s(x),X^{t,m}_s\ot m,\alpha(X^{t,m}_s(x),X^{t,m}_s\ot m,\gamma(s,X^{t,m}_s(x),X^{t,m}_s\ot m)))\cdot \gamma(s,X^{t,m}_s(x),X^{t,m}_s\ot m)\\\no
&\ -\p_x g(X^{t,m}_s(x),X^{t,m}_s\ot m,\alpha(X^{t,m}_s(x),X^{t,m}_s\ot m,\gamma(s,X^{t,m}_s(x),X^{t,m}_s\ot m)))\\\no
&\ -\int_{\R^{d_x}}\p_\mu f(X^{t,m}_s(y),X^{t,m}_s\ot m,\alpha(X^{t,m}_s(y),X^{t,m}_s\ot m,\gamma(s,X^{t,m}_s(y),X^{t,m}_s(\ot m)))(X^{t,m}_s(x))\cdot \gamma(s,X^{t,m}_s(y),X^{t,m}_s\ot m)dm(y)\\\label{eq_9_17_1}
&\ -\int_{\R^{d_x}}\p_\mu g(X^{t,m}_s(y),X^{t,m}_s\ot m,\alpha(X^{t,m}_s(y),X^{t,m}_s\ot m,\gamma(s,X^{t,m}_s(y),X^{t,m}_s\ot m)))(x)dm(y).
\end{align}\normalsize 
And then applying the characteristics method to the derived equation \eqref{eq_9_17_1}, that means we can see that the left hand side of \eqref{eq_9_17_1} is exactly equal to $\frac{d}{ds}Z^{t,m}_s(x)$ and the corresponding characteristic equations (that both $\Big(X^{t,m}_s(x),Z^{t,m}_s(x)\Big)$ and $\Big(\wb{X}^{t,m}_s(x),\wb{Z}^{t,m}_s(x)\Big)$ satisfy) are exactly the same as that of FBODE system \eqref{MPIT}, which implies that $\Big(X^{t,m}_s(x),Z^{t,m}_s(x)\Big)=\Big(\wb{X}^{t,m}_s(x),\wb{Z}^{t,m}_s(x)\Big)$ by Theorem \ref{Thm6_4}. Thus, $\p_m v(t,m)(x)=\gamma(t,x,m)=Z^{t,m}_t(x)=\wb{Z}^{t,m}_t(x)=\wb{\gamma}(t,x,m)=\p_m \wb{v}(t,m)(x)$. Therefore, by using the Bellman equation \eqref{nomin_1} again, we have $\p_t \big(v(t,m)-\wb{v}(t,m)\big)=0$ for all $t\in[0,T]$ and $v(T,m)-\wb{v}(T,m)=0$, and hence, $v(t,m)=\wb{v}(t,m)$, for all $t$.
\end{proof}

\section{Master Equation}\label{sec:master}
For if the value function $v(t,m):[0 ,T]\times\mathcal{P}_2(\R^{d_x})\to\R$ possesses a linear functional derivative $u(t,x,m):=\frac{\delta}{\delta m}v(t,m)(x):[0 ,T]\times \R^{d_x}\times \mathcal{P}_2(\R^{d_x})\to \R$, then, taking the linear functional derivative of the Bellman equation \eqref{bellman}, we obtain the following master equation (see \cite{bensoussan2015master}): 
\begin{align}\label{master}
\begin{cases}
\p_t u(t,x,m) + h(x,m,\p_x u(t,x,m))+\int_{\R^{d_x}} \Big(\frac{\delta}{\delta m} h(\wt{x},m,\p_{\wt{x}} u(t,\wt{x},m))(x)\\
\ \ \ \ \ \ \ \ \ \ \ \ \ \ \ \ \ \ \ \ \ \ \ \ \ \ \ \ \ \ \ \ +\p_z h(\wt{x},m,\p_{\wt{x}} u(t,\wt{x},m))\cdot \p_m u(t,x,m)(\wt{x})\Big) dm(\wt{x})= 0,\\
u(T,x,m)=k(x,m)+\int_{\R^{d_x}} \frac{\delta}{\delta m}k(\wt{x},m)(x)dm(\wt{x}),
\end{cases}
\end{align}
where $h(x,\mu,z)$ is the Hamiltonian defined in \eqref{hamilton1}.
Since, by Theorem \ref{GlobalSol_Bellman} and \ref{Unique_B}, we have already had a global unique solution to the Bellman equation \eqref{bellman}, to obtain a global unique solution to the master equation \eqref{master}, it suffices to show the linear functional differentiability of both $v(t,m)$ and $\p_m v(t,m)$. Since $\p_m v (t,m)(x)=Z^{t,m}_t(x)$ (see \eqref{p_m_v} in Appendix) and we have \eqref{eq_7_43}, by Proposition \ref{prop_3_7}, $v(t,m)$ has a linear functional derivative $\frac{\delta }{\delta m}v(t,m)(\cdot)$ such that
\begin{equation}
\partial_m v(t,m)(x)=\partial_x \frac{\delta }{\delta m}v(t,m)(x).
\end{equation}

Now, we consider the linear functional differentiability in $m$ of $\p_m v (t,m)(x)$. Although $\p_m v (t,m)(x)=Z^{t,m}_t(x)$ is $L$-differentiable in $m$, the Fr\'echet derivative of its lifted version may not be uniformly Lipschitz in its argument and thus we cannot apply Proposition \ref{prop_3_7} to obtain the linear functional derivative of $\p_m v (t,m)(x)$ directly. Instead, we shall use an argument which is similar to that in the proof of $L$-differentiability of $Z^{t,m}_t(x)$, since $f(x,\mu,\alpha)$, $g(x,\mu,\alpha)$ and $k(x,\mu,\alpha)$ are assumed to be regularly differentiable in $\mu\in\mathcal{P}_2(\R^{d_x})$ up to  the second order in Assumptions $\bf{(a1)}$-$\bf{(a3)}$. 
\begin{theorem}
Under Assumption ${\bf(a1)}$-${\bf(a3)}$ and Hypothesis ${\bf(h1)}$-${\bf(h2)}$, the unique solution pair $\big(X^{t,m}_s(x),\\Z^{t,m}_s(x)\big)$ of FBODE system \eqref{MPIT} constructed in Theorem \ref{GlobalSol} has a linear functional derivative\\ $\Big(\frac{\delta}{\delta m}\big(X^{t,m}_s(x)\big)(y),\frac{\delta}{\delta m}\big(Z^{t,m}_s(x)\big)(y)\Big)$ such that
\begin{equation}\label{eq_8_3_1}
\Big(\p_m\big(X^{t,m}_s(x)\big)(y),\p_m\big(Z^{t,m}_s(x)\big)(y)\Big)= \Big(\p_y\frac{\delta}{\delta m}\big(X^{t,m}_s(x)\big)(y),\p_y\frac{\delta}{\delta m}\big(Z^{t,m}_s(x)\big)(y)\Big),\ \text{for all }y\in\R^{d_x}.
\end{equation}
\end{theorem}

\begin{proof}
Step 1. Define 
\begin{align}\label{eq_8_5}
\begin{cases}
\mathcal{D}_mX^{t,m}_s(x,y):=\int_0^1 \p_m \big(X^{t,m}_s(x)\big)(\theta 
y)\cdot y d\theta,\\
\mathcal{D}_m\gamma (s,x,m)(y):=\int_0^1 \p_m \gamma(s,x,m)(\theta y)\cdot y d\theta,
\end{cases}
\end{align}
where $\Big(\p_m \big(X^{t,m}_s(x)\big)(y),\p_m \gamma(s,x,m)(y)\Big)$ is the $L$-derivative of $\Big(X^{t,m}_s(x),\gamma(s,x,m)\Big)$ that solves \eqref{gammaeq}-\eqref{xeq} with the terminal condition $$p(X^{s,m}_T(x),X^{s,m}_T\ot m)=\int_{\R^{d_x}} \p_\mu k(X^{s,m}_T(\wt{x}),X^{s,m}_T\ot m)(X^{s,m}_T(x))dm(\wt{x})+\p_x k(X^{s,m}_T(x),X^{s,m}_T\ot m).$$ 
Recall the notations defined in Appendix \eqref{notation_1} and Table \ref{notation_2} and define 
\begin{align*}
&\llbracket \mathcal{D}_m\gamma\rrbracket^{s,m }_\tau(x,y):=\mathcal{D}_m\gamma\big(\tau,X^{s,m }_\tau(x),X^{s,m }_\tau\ot m\big)\big(X^{s,m }_\tau(y)\big),\\
&\llbracket \p_x\gamma\rrbracket^{s,m}_\tau(y):=\p_x \gamma(s,x,\mu)|_{x=X^{s,m}_\tau(y),\mu=X^{s,m}_\tau\ot m},\\
&\llbracket \p_\mu\gamma\rrbracket^{s,m}_\tau(x,\wt{x}):=\p_\mu \gamma(s,y,\mu)(\wt{y})|_{y=X^{s,m}_\tau(x),\mu=X^{s,m}_\tau\ot m,\wt{y}=X^{s,m}_\tau(\wt{x})}.
\end{align*} 
Since $f(x,\mu,\alpha)$, $g(x,\mu,\alpha)$ and $k(x,\mu,\alpha)$ are assumed to be regularly differentiable in $\mu\in\mathcal{P}_2(\R^{d_x})$ up to  the second order in Assumptions $\bf{(a1)}$-$\bf{(a3)}$, one can check directly that $\Big(\mathcal{D}_mX^{t,m}_s(x,y),\mathcal{D}_m\gamma (s,x,m)(y)\Big)$ defined by \eqref{eq_8_5} solves the following linear FBODE system with the following backward equation:
\footnotesize\begin{align}\no
&\mathcal{D}_m\gamma (s,x,m)(y)\\\no
= &\ \int_{\R^{d_x}} \bigg(\big(\p_x\p_\mu k\big)^{s,m}_T(\wt{x},x)+\big(\p_\mu\p_x k\big)^{s,m}_T(x,\wt{x})+\int_{\R^{d_x}}\big(\p_\mu\p_\mu k\big)^{s,m}_T(\wh{x},x,\wt{x})dm(\wh{x})\bigg)\cdot \mathcal{D}_mX^{s,m}_T(\wt{x},y)dm(\wt{x})\\
&+\bigg(\int_{\R^{d_x}} \big(\p_{\wt{x}}\p_\mu k\big)^{s,m}_T(\wt{x},x)dm(\wt{x})+\big(\p_x\p_x k\big)^{s,m}_T(x)\bigg)\cdot \mathcal{D}_mX^{s,m}_T(x,y)\\\no
&+\int_{\R^{d_x}} \bigg(\bigg(\frac{\delta}{\delta \mu}\p_\mu k\bigg)^{s,m}_T(\wt{x},x,y)-\bigg(\frac{\delta}{\delta \mu}\p_\mu k\bigg)^{s,m}_T(\wt{x},x,0)\bigg)dm(\wt{x})+\bigg(\frac{\delta}{\delta \mu}\p_x k\bigg)^{s,m}_T(x,y)-\bigg(\frac{\delta}{\delta \mu}\p_x k\bigg)^{s,m}_T(x,0)\\\no
&+ \int_s^T \Bigg(\int_{\R^{d_x}}\bigg(\big(\p_{\wt{x}}\p_\mu f\big)^{s,m }_\tau(\wt{x},x)\cdot  Z^{s,m }_\tau(\wt{x}) + \big(\p_{\wt{x}}\p_\mu g\big)^{s,m }_\tau(\wt{x},x)\bigg)dm(\wt{x})\\\no
&\ \ \ \ \ \ \ \ \ \ \ \ +\Big(\big(\p_\alpha \p_x f\big)^{s,m }_\tau(x)\cdot Z^{s,m }_\tau(x)+\big(\p_\alpha\p_x g\big)^{s,m }_\tau(x)\Big)\Big(\big(\p_x\alpha\big)^{s,m }_\tau(x)+\big(\p_z\alpha\big)^{s,m }_\tau(x)\cdot\llbracket \p_x\gamma\rrbracket^{s,m }_\tau(x)\Big)\\\no 
&\ \ \ \ \ \ \ \ \ \ \ \ +\big(\p_x f\big)^{s,m }_\tau(x)\cdot \llbracket \p_x\gamma\rrbracket^{s,m }_\tau(x)+\big(\p_x\p_x f\big)^{s,m }_\tau(x)\cdot Z^{s,m }_\tau(x)+\big(\p_x \p_x g\big)^{s,m }_\tau(x)\Bigg)\mathcal{D}_mX^{s,m}_\tau(x,y)d\tau
\\\no
&+ \int_s^T \Bigg(\int_{\R^{d_x}}\bigg(\big(\p_x\p_\mu f\big)^{s,m }_\tau(\wt{x},x)\cdot  Z^{s,m }_\tau(\wt{x}) + \big(\p_x\p_\mu g\big)^{s,m }_\tau(\wt{x},x)+\big(\p_\mu\p_x f\big)^{s,m }_\tau(x,\wt{x})\cdot  Z^{s,m }_\tau(x) + \big(\p_\mu\p_x g\big)^{s,m }_\tau(x,\wt{x})\\\no
&\ \ \ \ \ \ \ \ \ \ \ \ \ \ \ \ \ \ \ \ +\int_{\R^{d_x}}\Big(\big(\p_\mu\p_\mu f\big)^{s,m }_\tau(\wh{x},x,\wt{x})\cdot  Z^{s,m }_\tau(\wh{x}) + \big(\p_\mu\p_\mu g\big)^{s,m }_\tau(\wh{x},x,\wt{x})\Big)dm(\wh{x})\\\no
&\ \ \ \ \ \ \ \ \ \ \ \ \ \ \ \ \ \ \ \ +\Big(\big(\p_\alpha \p_\mu f\big)^{s,m }_\tau(\wt{x},x)\cdot Z^{s,m }_\tau(\wt{x})+\big(\p_\alpha\p_\mu g\big)^{s,m }_\tau(\wt{x},x)\Big)\Big(\big(\p_x\alpha\big)^{s,m }_\tau(\wt{x})+\big(\p_z\alpha\big)^{s,m }_\tau(\wt{x})\cdot\llbracket \p_x\gamma\rrbracket^{s,m }_\tau(\wt{x})\Big)\\\no
&\ \ \ \ \ \ \ \ \ \ \ \ \ \ \ \ \ \ \ \ +\int_{\R^{d_x}}\Big(\big(\p_\alpha \p_\mu f\big)^{s,m }_\tau(\wh{x},x)\cdot Z^{s,m }_\tau(\wh{x})+\big(\p_\alpha\p_\mu g\big)^{s,m }_\tau(\wh{x},x)\Big)\Big(\big(\p_\mu\alpha\big)^{s,m }_\tau(\wh{x},\wt{x})+\big(\p_z\alpha\big)^{s,m }_\tau(\wh{x})\cdot\llbracket \p_\mu\gamma\rrbracket^{s,m }_\tau(\wh{x},\wt{x})\Big)dm(\wh{x})\\\no
&\ \ \ \ \ \ \ \ \ \ \ \ \ \ \ \ \ \ \ \ +\big(\p_\mu f\big)^{s,m }_\tau(\wt{x},x)\cdot\llbracket \p_x\gamma\rrbracket^{s,m }_\tau(\wt{x})+\int_{\R^{d_x}}\big(\p_\mu f\big)^{s,m }_\tau(\wh{x},x)\cdot\llbracket \p_\mu\gamma\rrbracket^{s,m }_\tau(\wh{x},\wt{x})dm(\wh{x})\\\no
&\ \ \ \ \ \ \ \ \ \ \ \ \ \ \ \ \ \ \ \ +\Big(\big(\p_\alpha \p_x f\big)^{s,m }_\tau(x)\cdot Z^{s,m }_\tau(x)+\big(\p_\alpha\p_x g\big)^{s,m }_\tau(x)\Big)\Big(\big(\p_\mu\alpha\big)^{s,m }_\tau(x,\wt{x})+\big(\p_z\alpha\big)^{s,m }_\tau(x)\cdot\llbracket \p_\mu\gamma\rrbracket^{s,m }_\tau(x,\wt{x})\Big)\\\no 
&\ \ \ \ \ \ \ \ \ \ \ \ \ \ \ \ \ \ \ \ +\big(\p_x f\big)^{s,m }_\tau(x)\cdot \llbracket \p_\mu\gamma\rrbracket^{s,m }_\tau(x,\wt{x})\bigg)\mathcal{D}_mX^{s,m}_\tau(\wt{x},y)dm(\wt{x})\Bigg)d\tau\\\no
&+ \int_s^T \bigg(\big(\frac{\delta}{\delta \mu}\p_x f\big)^{s,m }_\tau(x,y)\cdot  Z^{s,m }_\tau(x) + \big(\frac{\delta}{\delta \mu}\p_x g\big)^{s,m }_\tau(x,y)\\\no
&\ \ \ \ \ \ \ \ \ \ \ \ +\int_{\R^{d_x}}\Big(\big(\frac{\delta}{\delta \mu}\p_\mu f\big)^{s,m }_\tau(\wh{x},x,y)\cdot  Z^{s,m }_\tau(\wh{x}) + \big(\frac{\delta}{\delta \mu}\p_\mu g\big)^{s,m }_\tau(\wh{x},x,y)\Big)dm(\wh{x})\\\no
&\ \ \ \ \ \ \ \ \ \ \ \ +\int_{\R^{d_x}}\Big(\big(\p_\alpha \p_\mu f\big)^{s,m }_\tau(\wh{x},x)\cdot Z^{s,m }_\tau(\wh{x})+\big(\p_\alpha\p_\mu g\big)^{s,m }_\tau(\wh{x},x)\Big)\Big(\big(\frac{\delta}{\delta \mu}\alpha\big)^{s,m }_\tau(\wh{x},y)+\big(\p_z\alpha\big)^{s,m }_\tau(\wh{x})\cdot\llbracket \mathcal{D}_m\gamma\rrbracket^{s,m }_\tau(\wh{x},y)\Big)dm(\wh{x})\\\no
&\ \ \ \ \ \ \ \ \ \ \ \ +\int_{\R^{d_x}}\big(\p_\mu f\big)^{s,m }_\tau(\wh{x},x)\cdot\llbracket \mathcal{D}_m\gamma\rrbracket^{s,m }_\tau(\wh{x},y)dm(\wh{x})+\big(\p_x f\big)^{s,m }_\tau(x)\cdot \llbracket \mathcal{D}_m\gamma\rrbracket^{s,m }_\tau(x,y)\\\no
&\ \ \ \ \ \ \ \ \ \ \ \ +\Big(\big(\p_\alpha \p_x f\big)^{s,m }_\tau(x)\cdot Z^{s,m }_\tau(x)+\big(\p_\alpha\p_x g\big)^{s,m }_\tau(x)\Big)\Big(\big(\frac{\delta}{\delta \mu}\alpha\big)^{s,m }_\tau(x,y)+\big(\p_z\alpha\big)^{s,m }_\tau(x)\cdot\llbracket \mathcal{D}_m\gamma\rrbracket^{s,m }_\tau(x,y)\Big)\bigg)d\tau\\\no
&-\int_s^T \bigg(\big(\frac{\delta}{\delta \mu}\p_x f\big)^{s,m }_\tau(x,0)\cdot  Z^{s,m }_\tau(x) + \big(\frac{\delta}{\delta \mu}\p_x g\big)^{s,m }_\tau(x,0)\\\no
&\ \ \ \ \ \ \ \ \ \ \ \ +\int_{\R^{d_x}}\Big(\big(\frac{\delta}{\delta \mu}\p_\mu f\big)^{s,m }_\tau(\wh{x},x,0)\cdot  Z^{s,m }_\tau(\wh{x}) + \big(\frac{\delta}{\delta \mu}\p_\mu g\big)^{s,m }_\tau(\wh{x},x,0)\Big)dm(\wh{x})\\\no
&\ \ \ \ \ \ \ \ \ \ \ \ +\int_{\R^{d_x}}\Big(\big(\p_\alpha \p_\mu f\big)^{s,m }_\tau(\wh{x},x)\cdot Z^{s,m }_\tau(\wh{x})+\big(\p_\alpha\p_\mu g\big)^{s,m }_\tau(\wh{x},x)\Big)\Big(\big(\frac{\delta}{\delta \mu}\alpha\big)^{s,m }_\tau(\wh{x},0)+\big(\p_z\alpha\big)^{s,m }_\tau(\wh{x})\cdot \llbracket \mathcal{D}_m\gamma\rrbracket^{s,m }_\tau(\wh{x},0)\Big)dm(\wh{x})\\\no
&\ \ \ \ \ \ \ \ \ \ \ \ +\int_{\R^{d_x}}\big(\p_\mu f\big)^{s,m }_\tau(\wh{x},x)\cdot\llbracket \mathcal{D}_m\gamma\rrbracket^{s,m }_\tau(\wh{x},0)dm(\wh{x})+\big(\p_x f\big)^{s,m }_\tau(x)\cdot \llbracket \mathcal{D}_m\gamma\rrbracket^{s,m }_\tau(x,0)\\\label{d_mgamma}  
&\ \ \ \ \ \ \ \ \ \ \ \ +\Big(\big(\p_\alpha \p_x f\big)^{s,m }_\tau(x)\cdot Z^{s,m }_\tau(x)+\big(\p_\alpha\p_x g\big)^{s,m }_\tau(x)\Big)\Big(\big(\frac{\delta}{\delta \mu}\alpha\big)^{s,m }_\tau(x,0)+\big(\p_z\alpha\big)^{s,m }_\tau(x)\cdot\llbracket \mathcal{D}_m\gamma\rrbracket^{s,m }_\tau(x,0)\Big)\bigg)d\tau,
\end{align}\normalsize
where the couple $\big(X^{t,m}_s(x),Z^{t,m}_s(x)=\gamma(s,X^{t,m}_s(x),X^{t,m}_s\ot m)\big)$ is the unique solution of FBODE \eqref{MPIT} constructed in Theorem \ref{GlobalSol}, and the forward equation:
\small\begin{align}\no
&\mathcal{D}_mX^{t,m }_s(x,y)\\\no
=&\  \int_t^s \bigg(\big(\p_x f\big)^{t,m }_\tau(x)+\big(\p_\alpha f\big)^{t,m }_\tau(x)\Big(\big(\p_x \alpha\big)^{t,m }_\tau(x)+\big(\p_z \alpha\big)^{t,m }_\tau(x)\cdot\llbracket \p_x\gamma\rrbracket^{t,m }_\tau(x)\Big)\bigg)\mathcal{D}_mX^{t,m }_\tau(x,y) d\tau\\\no
&+\int_t^s\int_{\R^{d_x}}\bigg(\big(\p_\mu f\big)^{t,m }_\tau(x,\wt{x})+\big(\p_\alpha f\big)^{t,m }_\tau(x)\Big(\big(\p_\mu \alpha\big)^{t,m }_\tau(x,\wt{x})+\big(\p_z \alpha\big)^{t,m }_\tau(x)\cdot\llbracket \p_\mu\gamma\rrbracket^{t,m }_\tau(x,\wt{x})\Big)\bigg)\mathcal{D}_mX^{t,m }_\tau(\wt{x},y)dm(\wt{x})d\tau\\\no
&+\int_t^s\bigg(\big(\frac{\delta}{\delta \mu} f\big)^{t,m }_\tau(x,y)+\big(\p_\alpha f\big)^{t,m }_\tau(x)\Big(\big(\frac{\delta}{\delta \mu} \alpha\big)^{t,m }_\tau(x,y)+\big(\p_z \alpha\big)^{t,m }_\tau(x)\cdot\llbracket \mathcal{D}_m\gamma\rrbracket^{t,m }_\tau(x,y)\Big)\bigg)d\tau\\\label{d_mX}   
&-\int_t^s\bigg(\big(\frac{\delta}{\delta \mu} f\big)^{t,m }_\tau(x,0)+\big(\p_\alpha f\big)^{t,m }_\tau(x)\Big(\big(\frac{\delta}{\delta \mu} \alpha\big)^{t,m }_\tau(x,0)+\big(\p_z \alpha\big)^{t,m }_\tau(x)\cdot\llbracket \mathcal{D}_m\gamma\rrbracket^{t,m }_\tau(x,0)\Big)\bigg)d\tau.
\end{align}\normalsize
As one may note that the linear functional derivative is uniquely defined up to a constant, and the problem \eqref{d_mgamma}-\eqref{d_mX} is the dynamics governing the evolution of $\Big(\frac{\delta}{\delta m}\big(X^{t,m}_s(x)\big)(y),\frac{\delta}{\delta m}\gamma(s,x,m)(y)\Big)$. Next, we shall verify that $\Big(\frac{\delta}{\delta m}\big(X^{t,m}_s(x)\big)(y),\frac{\delta}{\delta m}\gamma(s,x,m)(y)\Big)=\big(\mathcal{D}_mX^{t,m }_s(x,y),\mathcal{D}_m\gamma (s,x,m)(y)\big)$ up to a constant.

Step 2. We shall show that, for each $t\in[0,T]$, $x\in\R^{d_x}$ and $s\in[t,T]$, there exists a positive constant $C_{t,x,s}:\mathcal{P}_2(\mathbb{R}^{d})\rightarrow [0,\infty)$ which is bounded on any bounded subsets of $\mathcal{P}_2(\R^d)$, such that
\begin{align}\label{eq_8_6}
\begin{cases}
\big|\mathcal{D}_mX^{t,m}_s(x,y)\big|\leq C_{t,x,s}(m)\big(1+|y|^2\big),\\
\big|\mathcal{D}_m\gamma (s,x,m)(y)\big|\leq C_{t,x,s}(m)\big(1+|y|^2\big).
\end{cases}
\end{align}
By using \eqref{Bp_mX_G} and $\vertiii{\gamma}_2\leq L^*_0$, obviously we have
\begin{align*}
\big|\mathcal{D}_mX^{t,m}_s(x,y)\big|&\leq \int_0^1 \big|\p_m \big(X^{t,m}_s(x)\big)(ty)\cdot y \big|dt \leq L_M^{(s-t)}|y|,\\
\big|\mathcal{D}_m\gamma (s,x,m)(y)\big|&\leq \int_0^1 \big|\p_m \gamma(s,x,m)(ty)\cdot y\big|dt \leq L^*_0|y|.
\end{align*}
Therefore, \eqref{eq_8_6} is valid with $C_{t,x,s}(m)=\frac{1}{2}\max\{L_M^{(s-t)},L^*_0\}$ by noting the simple fact of $|y|\leq \frac{1}{2}(1+|y|^2)$.

Step 3. We shall show that, for all $m$, $m'\in \mathcal{P}_2(\mathbb{R}^{d})$,
\begin{align}\label{linD_X}
 &\lim_{\theta\to 0^+}\frac{X^{t,(1-\theta)m+\theta m'}_s(x)-X^{t,m}_s(x)}{\theta}= \int_{\mathbb{R}^{d}}\mathcal{D}_mX^{t,m}_s(x,y)d(m'-m)(y),\\\label{linD_gamma}
&\lim_{\theta\to 0^+}\frac{\gamma(s,x,(1-\theta)m+\theta m')-\gamma(s,x,m)}{\theta}= \int_{\mathbb{R}^{d}}\mathcal{D}_m\gamma (s,x,m)(y)d(m'-m)(y).
\end{align}
First, it follows from \eqref{eq_8_8_new}, \eqref{JFE}, \eqref{Bp_mX_G} and \eqref{eq_9_3} that 
\begin{align}\label{eq_8_9}
&\big|X^{t,(1-\theta)m+\theta m'}_s(x)-X^{t,m}_s(x)\big|\leq L_M^{(s-t)} W_1\big((1-\theta)m+\theta m',m\big)\leq L_M^{(s-t)}(\|m\|_1+\|m'\|_1)\theta,\\\label{eq_8_10}
&\big|\gamma(s,x,(1-\theta)m+\theta m')-\gamma(s,x,m)\big|\leq L^*_0W_1\big((1-\theta)m+\theta m',m\big)\leq L^*_0(\|m\|_1+\|m'\|_1)\theta.
\end{align}
Also note that, by the regularities of $f$ and $\alpha$, 
\begin{align}\no
&f\left(X^{t,m}_\tau(x),X^{t,m}_\tau\ot \big((1-\theta)m+\theta m'\big),\alpha^{t,m}_\tau(x)\right)-f\left(X^{t,m}_\tau(x),X^{t,m}_\tau\ot m,\alpha^{t,m}_\tau(x)\right)\\\label{eq_8_11}
&-\theta\int_{\mathbb{R}^{d}}\bigg(\frac{\delta}{\delta \mu} f\bigg)^{t,m}_\tau(x,y)d(m'-m)(y)=o(|\theta|);\\\no
&\alpha\left(X^{t,m}_\tau(x),X^{t,m}_\tau\ot \big((1-\theta)m+\theta m'\big), Z^{t,m}_\tau(x)\right)-\alpha\left(X^{t,m}_\tau(x),X^{t,m}_\tau\ot m, Z^{t,m}_\tau(x)\right)\\\label{eq_8_12}
&-\theta\int_{\mathbb{R}^{d}}\bigg(\frac{\delta}{\delta \mu} \alpha\bigg)^{t,m}_\tau(x,y)d(m'-m)(y)=o(|\theta|),
\end{align}
where the small-$o\Big(\big|\theta\big|\Big)$ means $o\Big(\big|\theta\big|\Big)/\big|\theta\big|\to 0$ as $\big|\theta\big|\to 0$.
Since $\p_x f(x,\mu,\alpha)$, $\p_\mu f(x,\mu,\alpha)(\wt{x})$, $\p_\alpha f(x,\mu,\alpha)$, $\p_x\alpha(x,\mu,z)$, $\p_\mu \alpha(x,\mu,z)(\wt{x})$ and $\p_z \alpha(x,\mu,z)$ are bounded and continuous in their corresponding arguments, by using \eqref{MPIT}, \eqref{d_mX}, \eqref{eq_8_9}-\eqref{eq_8_12} and Lebesgue's dominated convergence theorem, we have
\small\begin{align}\no
&\Bigg|X^{t,(1-\theta)m+\theta m'}_s(x)-X^{t,m}_s(x)- \theta\int_{\mathbb{R}^{d}}\mathcal{D}_mX^{t,m}_s(x,y)d(m'-m)(y)\Bigg|\\\no
=&\  \Bigg|\int_t^s \bigg(f^{t,(1-\theta)m+\theta m'}_\tau(x) -f^{t,m}_\tau(x)- \theta\int_{\mathbb{R}^{d}}\frac{d}{d\tau}\mathcal{D}_mX^{t,m}_\tau(x,y)d(m'-m)(y)\bigg)d\tau\Bigg|\\\no
\leq &\ \bigg|\int_t^s \bigg(\big(\p_x f\big)^{t,m }_\tau(x)+\big(\p_\alpha f\big)^{t,m }_\tau(x)\Big(\big(\p_x \alpha\big)^{t,m }_\tau(x)+\big(\p_z \alpha\big)^{t,m }_\tau(x)\cdot\llbracket \p_x\gamma\rrbracket^{t,m }_\tau(x)\Big)\bigg)\\\no
&\ \ \ \ \ \ \cdot\bigg(X^{t,(1-\theta)m+\theta m'}_\tau(x)-X^{t,m}_\tau(x)-\theta\int_{\mathbb{R}^{d}}\mathcal{D}_mX^{t,m}_\tau(x,y)d(m'-m)(y)\bigg) d\tau\bigg|\\\no
&+\bigg|\int_t^s\int_{\R^{d_x}}\bigg(\big(\p_\mu f\big)^{t,m }_\tau(x,\wt{x})+\big(\p_\alpha f\big)^{t,m }_\tau(x)\Big(\big(\p_\mu \alpha\big)^{t,m }_\tau(x,\wt{x})+\big(\p_z \alpha\big)^{t,m }_\tau(x)\cdot\llbracket \p_\mu\gamma\rrbracket^{t,m }_\tau(x,\wt{x})\Big)\bigg)\\\no
&\ \ \ \ \ \ \ \ \ \ \ \ \ \ \ \ \cdot\bigg(X^{t,(1-\theta)m+\theta m'}_\tau(\wt{x})-X^{t,m}_\tau(\wt{x})-\theta\int_{\mathbb{R}^{d}}\mathcal{D}_mX^{t,m}_\tau(\wt{x},y)d(m'-m)(y)\bigg)dm(\wt{x})d\tau\bigg|\\\no
&+\bigg|\int_t^s \big(\p_\alpha f\big)^{t,m }_\tau(x) \big(\p_z \alpha\big)^{t,m }_\tau(x)\bigg(\gamma\Big(\tau,X^{t,m}_\tau(x),X^{t,m}_\tau\ot\big((1-\theta)m+\theta m'\big)\Big)-\gamma\Big(\tau,X^{t,m}_\tau(x),X^{t,m}_\tau\ot m\Big)\\\no
&\ \ \ \ \ \ \ \ \ \ \ \ \ \ \ \ \ \ \ \ \ \ \ \ \ \ \ \ \ \ \ \ \ \ \ \ \ \ \ \ \ \ \ -\theta\int_{\R^{d_x}}\llbracket \mathcal{D}_m\gamma\rrbracket^{t,m}_\tau(x,y)d(m'-m)(y) \bigg)d\tau\bigg|+o(|\theta|)\\\no
\leq &\ L_B'\int_t^s \bigg|X^{t,(1-\theta)m+\theta m'}_\tau(x)-X^{t,m}_\tau(x)-\theta\int_{\mathbb{R}^{d}}\mathcal{D}_mX^{t,m}_\tau(x,y)d(m'-m)(y)\bigg|d\tau\\\no
&+L_B'\int_t^s\int_{\R^{d_x}} \bigg|X^{t,(1-\theta)m+\theta m'}_\tau(\wt{x})-X^{t,m}_\tau(\wt{x})-\theta\int_{\mathbb{R}^{d}}\mathcal{D}_mX^{t,m}_\tau(\wt{x},y)d(m'-m)(y) \bigg|dm(\wt{x})d\tau\\\no
&+\Lambda_fL_\alpha\int_t^s \bigg|\gamma\Big(\tau,X^{t,m}_\tau(x),X^{t,m}_\tau\ot\big((1-\theta)m+\theta m'\big)\Big)-\gamma\Big(\tau,X^{t,m}_\tau(x),X^{t,m}_\tau\ot m\Big)\\\label{eq_10_15_1}
&\ \ \ \ \ \ \ \ \ \ \ \ \ \ \ \ \ \ -\theta\int_{\R^{d_x}}\llbracket \mathcal{D}_m\gamma\rrbracket^{t,m}_\tau(x,y)d(m'-m)(y) \bigg|d\tau+o(|\theta|)\\\no
&\text{(by using \eqref{bdd_d1_f}, \eqref{p_xalpha}, \eqref{p_zalpha} and $\vertiii{\gamma}_2\leq L^*_0$),}
\end{align}\normalsize
where the constant $L_B'$ was defined in \eqref{L_Bp}.
Thus, by integrating \eqref{eq_10_15_1} with respect to $x$ and using Gr\"{o}nwall's inequality, we obtain
\footnotesize\begin{align}\no
&\int_{\R^{d_x}}\bigg|X^{t,(1-\theta)m+\theta m'}_s(x)-X^{t,m}_s(x)- \theta\int_{\mathbb{R}^{d}}\mathcal{D}_mX^{t,m}_s(x,y)d(m'-m)(y)\bigg|dm(x)\\\no
\leq &\ \Lambda_fL_\alpha\exp\Big(2L_B'(s-t)\Big)\int_t^s \int_{\R^{d_x}}\bigg|\gamma\Big(\tau,X^{t,m}_\tau(x),X^{t,m}_\tau\ot\big((1-\theta)m+\theta m'\big)\Big)-\gamma\Big(\tau,X^{t,m}_\tau(x),X^{t,m}_\tau\ot m\Big)\\\label{eq_8_13}
&\ \ \ \ \ \ \ \ \ \ \ \ \ \ \ \ \ \ \ \ \ \ \ \ \ \ \ \ \ \ \ \ \ \ \ \ \ \ \ \ \ \ \ -\theta\int_{\R^{d_x}}\llbracket \mathcal{D}_m\gamma\rrbracket^{t,m}_\tau(x,y)d(m'-m)(y) \bigg|dm(x)d\tau+o(|\theta|).
\end{align}\normalsize
Therefore, substituting \eqref{eq_8_13} into \eqref{eq_10_15_1}, and using Gr\"{o}nwall's inequality again, we have
\footnotesize\begin{align}\no
&\bigg|X^{t,(1-\theta)m+\theta m'}_s(x)-X^{t,m}_s(x)- \theta\int_{\mathbb{R}^{d}}\mathcal{D}_mX^{t,m}_s(x,y)d(m'-m)(y)\bigg|\\\no
\leq &\  \Bigg(\Lambda_fL_\alpha\int_t^s \bigg|\gamma\Big(\tau,X^{t,m}_\tau(x),X^{t,m}_\tau\ot\big((1-\theta)m+\theta m'\big)\Big)-\gamma\Big(\tau,X^{t,m}_\tau(x),X^{t,m}_\tau\ot m\Big)-\theta\int_{\R^{d_x}}\llbracket \mathcal{D}_m\gamma\rrbracket^{t,m}_\tau(x,y)d(m'-m)(y) \bigg|d\tau\\\no
&\ \ +L_B'\Lambda_fL_\alpha\exp\Big(2L_B'(s-t)\Big)(s-t)\int_t^s \int_{\R^{d_x}}\bigg|\gamma\Big(\tau,X^{t,m}_\tau(\wt{x}),X^{t,m}_\tau\ot\big((1-\theta)m+\theta m'\big)\Big)-\gamma\Big(\tau,X^{t,m}_\tau(\wt{x}),X^{t,m}_\tau\ot m\Big)\\\no
&\ \ \ \ \ \ \ \ \ \ \ \ \ \ \ \ \ \ \ \ \ \ \ \ \ \ \ \ \ \ \ \ \ \ \ \ \ \ \ \ \ \ \ \ \ \ \ \ \ \ \ \ \ \ \ \ \ \ \ \ \ \ -\theta\int_{\R^{d_x}}\llbracket \mathcal{D}_m\gamma\rrbracket^{t,m}_\tau(\wt{x},y)d(m'-m)(y) \bigg|dm(\wt{x})d\tau\Bigg)\exp\Big(L_B'(s-t)\Big)+o(|\theta|)\\\no
\leq &\  M_1(T)\int_t^s \bigg|\gamma\Big(\tau,X^{t,m}_\tau(x),X^{t,m}_\tau\ot\big((1-\theta)m+\theta m'\big)\Big)-\gamma\Big(\tau,X^{t,m}_\tau(x),X^{t,m}_\tau\ot m\Big)-\theta\int_{\R^{d_x}}\llbracket \mathcal{D}_m\gamma\rrbracket^{t,m}_\tau(x,y)d(m'-m)(y) \bigg|d\tau\\\no
&+M_2(T)(s-t)\int_t^s \int_{\R^{d_x}}\bigg|\gamma\Big(\tau,X^{t,m}_\tau(\wt{x}),X^{t,m}_\tau\ot\big((1-\theta)m+\theta m'\big)\Big)-\gamma\Big(\tau,X^{t,m}_\tau(\wt{x}),X^{t,m}_\tau\ot m\Big)\\\label{eq_8_14}
&\ \ \ \ \ \ \ \ \ \ \ \ \ \ \ \ \ \ \ \ \ \ \ \ \ \ \ \ \ \ \ \ \ \ -\theta\int_{\R^{d_x}}\llbracket \mathcal{D}_m\gamma\rrbracket^{t,m}_\tau(\wt{x},y)d(m'-m)(y) \bigg|dm(\wt{x})d\tau+o(|\theta|),
\end{align}\normalsize
where
\begin{align}\label{M_1}
&M_1(T):=\Lambda_fL_\alpha\exp\Big(L_B'T\Big),\\\label{M_2}
&M_2(T):=L_B'\Lambda_fL_\alpha\exp\Big(3L_B'T\Big).
\end{align}
Also note that 
\footnotesize\begin{align}\no
&\p_\mu f\left(X^{t,m}_\tau(\wt{x}),X^{t,m}_\tau\ot \big((1-\theta)m+\theta m'\big),\alpha^{t,m}_\tau(\wt{x})\right)\left(X^{t,m}_\tau(x)\right)\cdot Z^{t,m}_\tau(\wt{x})\\\no
&+\p_\mu g\left(X^{t,m}_\tau(\wt{x}),X^{t,m}_\tau\ot \big((1-\theta)m+\theta m'\big),\alpha^{t,m}_\tau(\wt{x})\right)\left(X^{t,m}_\tau(x)\right)\\\no
&-\p_\mu f\left(X^{t,m}_\tau(\wt{x}),X^{t,m}_\tau\ot m,\alpha^{t,m}_\tau(\wt{x})\right)\left(X^{t,m}_\tau(x)\right)\cdot Z^{t,m}_\tau(\wt{x})-\p_\mu g\left(X^{t,m}_\tau(\wt{x}),X^{t,m}_\tau\ot m,\alpha^{t,m}_\tau(\wt{x})\right)\left(X^{t,m}_\tau(x)\right)\\\label{eq_8_15}
&-\theta\int_{\mathbb{R}^{d}}\bigg(\bigg(\frac{\delta}{\delta \mu}\p_\mu f\bigg)^{t,m}_s(\wt{x},x,y)\cdot  Z^{t,m}_s(\wt{x})+ \bigg(\frac{\delta}{\delta \mu}\p_\mu g\bigg)^{t,m}_s(\wt{x},x,y)\bigg)d(m'-m)(y)=o(|\theta|);\\\no
&\p_x f\left(X^{t,m}_\tau(x),X^{t,m}_\tau\ot \big((1-\theta)m+\theta m'\big),\alpha^{t,m}_\tau(x)\right)\cdot Z^{t,m}_\tau(\wt{x})+\p_x g\left(X^{t,m}_\tau(x),X^{t,m}_\tau\ot \big((1-\theta)m+\theta m'\big),\alpha^{t,m}_\tau(x)\right)\\\no
&-\p_x f\left(X^{t,m}_\tau(x),X^{t,m}_\tau\ot m,\alpha^{t,m}_\tau(x)\right)\cdot Z^{t,m}_\tau(\wt{x})-\p_x g\left(X^{t,m}_\tau(x),X^{t,m}_\tau\ot m,\alpha^{t,m}_\tau(x)\right)\\\label{eq_8_16}
&-\theta\int_{\mathbb{R}^{d}}\bigg(\bigg(\frac{\delta}{\delta \mu}\p_x f\bigg)^{t,m}_s(x,y)\cdot  Z^{t,m}_s(x)+ \bigg(\frac{\delta}{\delta \mu}\p_x g\bigg)^{t,m}_s(x,y)\bigg)d(m'-m)(y)=o(|\theta|);\\\no
&\p_\mu k\left(X^{t,m}_\tau(\wt{x}),X^{t,m}_\tau\ot \big((1-\theta)m+\theta m'\big)\right)\left(X^{t,m}_\tau(x)\right)-\p_\mu k\left(X^{t,m}_\tau(\wt{x}),X^{t,m}_\tau\ot m\right)\left(X^{t,m}_\tau(x)\right)\\\label{eq_8_17}
&-\theta\int_{\mathbb{R}^{d}}\bigg(\frac{\delta}{\delta \mu}\p_\mu k\bigg)^{t,m}_s(\wt{x},x,y)d(m'-m)(y)=o(|\theta|;\\\no
&\p_x k\left(X^{t,m}_\tau(x),X^{t,m}_\tau\ot \big((1-\theta)m+\theta m'\big)\right)-\p_x k\left(X^{t,m}_\tau(x),X^{t,m}_\tau\ot m\right)\\\label{eq_8_18}
&-\theta\int_{\mathbb{R}^{d}}\bigg(\frac{\delta}{\delta \mu}\p_x k\bigg)^{t,m}_s(x,y)d(m'-m)(y)=o(|\theta|).
\end{align}\normalsize
Similarly, since all the second order derivatives of $f(x,\mu,\alpha)$, $g(x,\mu,\alpha)$ and $k(x,\mu)$, and $\p_x f(x,\mu,\alpha)$, $\p_\mu f(x,\mu,\alpha)(\wt{x})$, $\p_x\alpha(x,\mu,z)$, $\p_\mu \alpha(x,\mu,z)(\wt{x})$, $\p_z \alpha(x,\mu,z)$, $\p_x \gamma(s,x,\mu)$ and $\p_\mu \gamma(s,x,\mu)$ are bounded and continuous in their corresponding arguments, by using \eqref{MPIT}, \eqref{d_mgamma}, \eqref{eq_8_9}, \eqref{eq_8_10}, \eqref{eq_8_12}, \eqref{eq_8_15}-\eqref{eq_8_18} and Lebesgue's dominated convergence theorem, we have
\scriptsize\begin{align*}
&\Bigg|\gamma\Big(s,x,(1-\theta)m+\theta m'\Big)-\gamma\Big(s,x,m\Big)-\theta\int_{\R^{d_x}}\mathcal{D}_m\gamma (s,x,m)(y)d(m'-m)(y) \Bigg|\\
\leq &\ \Bigg|\int_{\R^{d_x}} \bigg(\big(\p_x\p_\mu k\big)^{s,m}_T(\wt{x},x)+\big(\p_\mu\p_x k\big)^{s,m}_T(x,\wt{x})+\int_{\R^{d_x}}\big(\p_\mu\p_\mu k\big)^{s,m}_T(\wh{x},x,\wt{x})dm(\wh{x})\bigg)\\
&\ \ \ \ \ \ \ \ \cdot\bigg(X^{s,(1-\theta)m+\theta m'}_T(\wt{x})-X^{s,m}_T(\wt{x})- \theta\int_{\mathbb{R}^{d}}\mathcal{D}_mX^{s,m}_T(\wt{x},y)d(m'-m)(y)\bigg)dm(\wt{x})\Bigg|\\
&+\Bigg|\bigg(\int_{\R^{d_x}} \big(\p_{\wt{x}}\p_\mu k\big)^{s,m}_T(\wt{x},x)dm(\wt{x})+\big(\p_x\p_x k\big)^{s,m}_T(x)\bigg)\\\no
&\ \ \ \ \cdot\bigg(X^{s,(1-\theta)m+\theta m'}_T(x)-X^{s,m}_T(x)- \theta\int_{\mathbb{R}^{d}}\mathcal{D}_mX^{s,m}_T(x,y)d(m'-m)(y)\bigg)\Bigg|\\
&+ \int_s^T \Bigg|\Bigg(\int_{\R^{d_x}}\bigg(\big(\p_{\wt{x}}\p_\mu f\big)^{s,m }_\tau(\wt{x},x)\cdot  Z^{s,m }_\tau(\wt{x}) + \big(\p_{\wt{x}}\p_\mu g\big)^{s,m }_\tau(\wt{x},x)\bigg)dm(\wt{x})\\\no
&\ \ \ \ \ \ \ \ \ \ \ \ +\Big(\big(\p_\alpha \p_x f\big)^{s,m }_\tau(x)\cdot Z^{s,m }_\tau(x)+\big(\p_\alpha\p_x g\big)^{s,m }_\tau(x)\Big)\Big(\big(\p_x\alpha\big)^{s,m }_\tau(x)+\big(\p_z\alpha\big)^{s,m }_\tau(x)\cdot\llbracket \p_x\gamma\rrbracket^{s,m }_\tau(x)\Big)\\\no 
&\ \ \ \ \ \ \ \ \ \ \ \ +\big(\p_x f\big)^{s,m }_\tau(x)\cdot \llbracket \p_x\gamma\rrbracket^{s,m }_\tau(x)+\big(\p_x\p_x f\big)^{s,m }_\tau(x)\cdot Z^{s,m }_\tau(x)+\big(\p_x \p_x g\big)^{s,m }_\tau(x)\Bigg)\\
&\ \ \ \ \ \ \ \ \ \ \cdot\bigg(X^{s,(1-\theta)m+\theta m'}_\tau(x)-X^{s,m}_\tau(x)- \theta\int_{\mathbb{R}^{d}}\mathcal{D}_mX^{s,m}_\tau(x,y)d(m'-m)(y)\bigg)\Bigg|d\tau\\
&+ \int_s^T \Bigg|\int_{\R^{d_x}}\bigg(\big(\p_x\p_\mu f\big)^{s,m }_\tau(\wt{x},x)\cdot  Z^{s,m }_\tau(\wt{x}) + \big(\p_x\p_\mu g\big)^{s,m }_\tau(\wt{x},x)+\big(\p_\mu\p_x f\big)^{s,m }_\tau(x,\wt{x})\cdot  Z^{s,m }_\tau(x) + \big(\p_\mu\p_x g\big)^{s,m }_\tau(x,\wt{x})\\\no
&\ \ \ \ \ \ \ \ \ \ \ \ \ \ \ \ \ \ \ \ +\int_{\R^{d_x}}\Big(\big(\p_\mu\p_\mu f\big)^{s,m }_\tau(\wh{x},x,\wt{x})\cdot  Z^{s,m }_\tau(\wh{x}) + \big(\p_\mu\p_\mu g\big)^{s,m }_\tau(\wh{x},x,\wt{x})\Big)dm(\wh{x})\\\no
&\ \ \ \ \ \ \ \ \ \ \ \ \ \ \ \ \ \ \ \ +\Big(\big(\p_\alpha \p_\mu f\big)^{s,m }_\tau(\wt{x},x)\cdot Z^{s,m }_\tau(\wt{x})+\big(\p_\alpha\p_\mu g\big)^{s,m }_\tau(\wt{x},x)\Big)\Big(\big(\p_x\alpha\big)^{s,m }_\tau(\wt{x})+\big(\p_z\alpha\big)^{s,m }_\tau(\wt{x})\cdot\llbracket \p_x\gamma\rrbracket^{s,m }_\tau(\wt{x})\Big)\\\no
&\ \ \ \ \ \ \ \ \ \ \ \ \ \ \ \ \ \ \ \ +\int_{\R^{d_x}}\Big(\big(\p_\alpha \p_\mu f\big)^{s,m }_\tau(\wh{x},x)\cdot Z^{s,m }_\tau(\wh{x})+\big(\p_\alpha\p_\mu g\big)^{s,m }_\tau(\wh{x},x)\Big)\Big(\big(\p_\mu\alpha\big)^{s,m }_\tau(\wh{x},\wt{x})+\big(\p_z\alpha\big)^{s,m }_\tau(\wh{x})\cdot\llbracket \p_\mu\gamma\rrbracket^{s,m }_\tau(\wh{x},\wt{x})\Big)dm(\wh{x})\\\no
&\ \ \ \ \ \ \ \ \ \ \ \ \ \ \ \ \ \ \ \ +\big(\p_\mu f\big)^{s,m }_\tau(\wt{x},x)\cdot\llbracket \p_x\gamma\rrbracket^{s,m }_\tau(\wt{x})+\int_{\R^{d_x}}\big(\p_\mu f\big)^{s,m }_\tau(\wh{x},x)\cdot\llbracket \p_\mu\gamma\rrbracket^{s,m }_\tau(\wh{x},\wt{x})dm(\wh{x})\\
&\ \ \ \ \ \ \ \ \ \ \ \ \ \ \ \ \ \ \ \ +\Big(\big(\p_\alpha \p_x f\big)^{s,m }_\tau(x)\cdot Z^{s,m }_\tau(x)+\big(\p_\alpha\p_x g\big)^{s,m }_\tau(x)\Big)\Big(\big(\p_\mu\alpha\big)^{s,m }_\tau(x,\wt{x})+\big(\p_z\alpha\big)^{s,m }_\tau(x)\cdot\llbracket \p_\mu\gamma\rrbracket^{s,m }_\tau(x,\wt{x})\Big)\\\no 
&\ \ \ \ \ \ \ \ \ \ \ \ \ \ \ \ \ \ \ \ +\big(\p_x f\big)^{s,m }_\tau(x)\cdot \llbracket \p_\mu\gamma\rrbracket^{s,m }_\tau(x,\wt{x})\bigg)\bigg(X^{s,(1-\theta)m+\theta m'}_\tau(\wt{x})-X^{s,m}_\tau(\wt{x})- \theta\int_{\mathbb{R}^{d}}\mathcal{D}_mX^{s,m}_\tau(\wt{x},y)d(m'-m)(y)\bigg)dm(\wt{x})\Bigg|d\tau\\
&+ \int_s^T \Bigg|\int_{\R^{d_x}}\bigg(\Big(\big(\p_\alpha \p_\mu f\big)^{s,m }_\tau(\wh{x},x)\cdot Z^{s,m }_\tau(\wh{x})+\big(\p_\alpha\p_\mu g\big)^{s,m }_\tau(\wh{x},x)\Big)\big(\p_z\alpha\big)^{s,m }_\tau(\wh{x})+\big(\p_\mu f\big)^{s,m }_\tau(\wh{x},x)\bigg)\\
&\ \ \ \ \ \ \ \ \ \ \ \ \ \ \ \ \ \ \ \ \cdot\bigg(\gamma\Big(\tau,X^{s,m}_\tau(\wh{x}),X^{s,m}_\tau\ot\big((1-\theta)m+\theta m'\big)\Big)-\gamma\Big(\tau,X^{s,m}_\tau(\wh{x}),X^{s,m}_\tau\ot m\Big)-\theta\int_{\R^{d_x}}\llbracket \mathcal{D}_m\gamma\rrbracket^{s,m}_\tau(\wh{x},y)d(m'-m)(y)\bigg)dm(\wh{x})\\\no
&\ \ \ \ \ \ \ \ \ +\bigg(\big(\p_x f\big)^{s,m }_\tau(x)+\Big(\big(\p_\alpha \p_x f\big)^{s,m }_\tau(x)\cdot Z^{s,m }_\tau(x)+\big(\p_\alpha\p_x g\big)^{s,m }_\tau(x)\Big)\big(\p_z\alpha\big)^{s,m }_\tau(x)\bigg)\\
&\ \ \ \ \ \ \ \ \ \ \ \ \cdot\bigg(\gamma\Big(\tau,X^{s,m}_\tau(x),X^{s,m}_\tau\ot\big((1-\theta)m+\theta m'\big)\Big)-\gamma\Big(\tau,X^{s,m}_\tau(x),X^{s,m}_\tau\ot m\Big)-\theta\int_{\R^{d_x}}\llbracket \mathcal{D}_m\gamma\rrbracket^{s,m}_\tau(x,y)d(m'-m)(y)\bigg)\Bigg|d\tau+o(|\theta|)\\
\leq &\  3\Lambda_k\int_{\R^{d_x}} \bigg|X^{s,(1-\theta)m+\theta m'}_T(\wt{x})-X^{s,m}_T(\wt{x})- \theta\int_{\mathbb{R}^{d}}\mathcal{D}_mX^{s,m}_T(\wt{x},y)d(m'-m)(y)\bigg|dm(\wt{x})\\
&+2\Lambda_k\bigg|X^{s,(1-\theta)m+\theta m'}_T(x)-X^{s,m}_T(x)- \theta\int_{\mathbb{R}^{d}}\mathcal{D}_mX^{s,m}_T(x,y)d(m'-m)(y)\bigg|\\
&+M_3\int_s^T \int_{\R^{d_x}}\bigg|X^{s,(1-\theta)m+\theta m'}_\tau(\wt{x})-X^{s,m}_\tau(\wt{x})- \theta\int_{\mathbb{R}^{d}}\mathcal{D}_mX^{s,m}_\tau(\wt{x},y)d(m'-m)(y)\bigg|dm(\wt{x})d\tau\\
&+M_4\int_s^T\bigg|X^{s,(1-\theta)m+\theta m'}_\tau(x)-X^{s,m}_\tau(x)- \theta\int_{\mathbb{R}^{d}}\mathcal{D}_mX^{s,m}_\tau(x,y)d(m'-m)(y)\bigg|d\tau\\
&+ \Big((\wb{l}_fL^*_0+\wb{l}_g)L_\alpha+\Lambda_f\Big)\int_s^T \int_{\R^{d_x}}\bigg|\gamma\Big(\tau,X^{s,m}_\tau(\wh{x}),X^{s,m}_\tau\ot\big((1-\theta)m+\theta m'\big)\Big)-\gamma\Big(\tau,X^{s,m}_\tau(\wh{x}),X^{s,m}_\tau\ot m\Big)\\
&\ \ \ \ \ \ \ \ \ \ \ \ \ \ \ \ \ \ \ \ \ \ \ \ \ \ \ \ \ \ \ \ \ \ \ \ \ \ \ \ \ \ \ \ \ \ \ -\theta\int_{\R^{d_x}}\llbracket \mathcal{D}_m\gamma\rrbracket^{s,m}_\tau(\wh{x},y)d(m'-m)(y)\bigg|dm(\wh{x})d\tau\\\no
&+ \Big((\wb{l}_fL^*_0+\wb{l}_g)L_\alpha+\Lambda_f\Big)\int_s^T \bigg|\gamma\Big(\tau,X^{s,m}_\tau(x),X^{s,m}_\tau\ot\big((1-\theta)m+\theta m'\big)\Big)-\gamma\Big(\tau,X^{s,m}_\tau(x),X^{s,m}_\tau\ot m\Big)\\
&\ \ \ \ \ \ \ \ \ \ \ \ \ \ \ \ \ \ \ \ \ \ \ \ \ \ \ \ \ \ \ \ \ \ \ \ \ \ \ \ \ \ \ \ \ \ \ -\theta\int_{\R^{d_x}}\llbracket \mathcal{D}_m\gamma\rrbracket^{s,m}_\tau(x,y)d(m'-m)(y)\bigg|d\tau+o(|\theta|),\\\no
&\text{(by using \eqref{bdd_d1_f}, \eqref{bdd_d2_f}, \eqref{bdd_d2_g_1}, \eqref{bdd_d2_g_2}, \eqref{bdd_d2_k_1}, \eqref{p_xalpha}, \eqref{p_zalpha} and $\vertiii{\gamma}_2\leq L^*_0$)}\\
\leq &\ \Big(\big(3\Lambda_k+M_3(T-s)\big)M_1(T)+\big(5\Lambda_k+(M_3+M_4)(T-s)\big)M_2(T)(T-s)+(\wb{l}_fL^*_0+\wb{l}_g)L_\alpha+\Lambda_f\Big)\\
&\cdot\int_s^T \int_{\R^{d_x}}\bigg|\gamma\Big(\tau,X^{s,m}_\tau(\wh{x}),X^{s,m}_\tau\ot\big((1-\theta)m+\theta m'\big)\Big)-\gamma\Big(\tau,X^{s,m}_\tau(\wh{x}),X^{s,m}_\tau\ot m\Big)-\theta\int_{\R^{d_x}}\llbracket \mathcal{D}_m\gamma\rrbracket^{s,m}_\tau(\wh{x},y)d(m'-m)(y)\bigg|dm(\wh{x})d\tau\\\no
&+ \Big(\big(2\Lambda_k+M_4(T-s)\big)M_1(T)+(\wb{l}_fL^*_0+\wb{l}_g)L_\alpha+\Lambda_f\Big)\\
&\cdot\int_s^T \bigg|\gamma\Big(\tau,X^{s,m}_\tau(x),X^{s,m}_\tau\ot\big((1-\theta)m+\theta m'\big)\Big)-\gamma\Big(\tau,X^{s,m}_\tau(x),X^{s,m}_\tau\ot m\Big)-\theta\int_{\R^{d_x}}\llbracket \mathcal{D}_m\gamma\rrbracket^{s,m}_\tau(x,y)d(m'-m)(y)\bigg|d\tau+o(|\theta|),\\\no
&\text{(by using \eqref{eq_8_14}),}
\end{align*}\normalsize
where $M_1(T)$ was defined in \eqref{M_1}, $M_2(T)$ was defined in \eqref{M_2}, $M_3:=3(\wb{l}_fL^*_0+\Lambda_g)+3(\wb{l}_fL^*_0+\wb{l}_g)L_\alpha(1+L^*_0)+3\Lambda_fL^*_0$, and $M_4:=2(\wb{l}_fL^*_0+\Lambda_g)+(\wb{l}_fL^*_0+\wb{l}_g)L_\alpha(1+L^*_0)+\Lambda_fL^*_0$. 
Thus, by using Gr\"{o}nwall's inequality and \eqref{coro_6_4}, we have
\begin{align*}
&\int_{\R^{d_x}}\bigg|\gamma\Big(s,x,(1-\theta)m+\theta m'\Big)-\gamma\Big(s,x,m\Big)-\theta\int_{\R^{d_x}}\mathcal{D}_m\gamma (s,x,m)(y)d(m'-m)(y) \bigg|dm(x)=o(|\theta|),
\end{align*}\normalsize
and then
\begin{align*}
&\bigg|\gamma\Big(s,x,(1-\theta)m+\theta m'\Big)-\gamma\Big(s,x,m\Big)-\theta\int_{\R^{d_x}}\mathcal{D}_m\gamma (s,x,m)(y)d(m'-m)(y) \bigg|=o(|\theta|).
\end{align*}\normalsize
Therefore, \eqref{linD_gamma} is valid. Moreover, by \eqref{eq_8_14}, we also have $$\bigg|X^{t,(1-\theta)m+\theta m'}_s(x)-X^{t,m}_s(x)- \theta\int_{\mathbb{R}^{d}}\mathcal{D}_mX^{t,m}_s(x,y)d(m'-m)(y)\bigg|=o(|\theta|).$$ Thus \eqref{linD_X} is also valid.

In summary, by Steps 1-3, we have shown that $X^{t,m}_s(x)$, $\gamma(s,x,m)$, and thus $Z^{t,m}_s(x)=\gamma(s,X^{t,m}_s(x),X^{t,m}_s\ot m)$, have linear functional derivatives with respect to $m\in\mc{P}_2(\R^{d_x})$, and $\Big(\frac{\delta}{\delta m}\big(X^{t,m}_s(x)\big)(y),\frac{\delta}{\delta m}\gamma(s,x,m)(y)\Big)=\big(\mathcal{D}_mX^{t,m }_s(x,y),\mathcal{D}_m\gamma (s,x,m)(y)\big)$ up to  a constant. In addition, \eqref{eq_8_3_1} follows from the uniqueness (see Theorem \ref{Thm6_4}) of the solution of the FBODE system \eqref{MPIT}.
\end{proof}

Therefore, we have the following global unique existence of the solution of master equation \eqref{master}:
\begin{theorem}\label{thm_master}
(Global existence of the master equation \eqref{master}). Under Assumption ${\bf(a1)}$-${\bf(a3)}$ and Hypothesis ${\bf(h1)}$-${\bf(h2)}$, the master equation \eqref{master} has a global solution $u(t,x,m)$, which is differentiable in $t\in[0,T]$, $x\in\R^{d_x}$ and is $L$-differentiable in $m\in\mc{P}_2(\R^{d_x})$ with its corresponding derivatives $\p_t u(t,x,m)$, $\p_x u(t,x,m)$ and $\p_m u(t,x,m)(y)$ being continuous in their corresponding arguments. In addition, (i) $\p_t u(t,x,m)$ is differentiable in $x\in\R^{d_x}$; (ii) $\p_x u(t,x,m)$ is differentiable in $t\in[0,T]$, $x\in\R^{d_x}$ and $L$-differentiable in $m\in\mc{P}_2(\R^{d_x})$; (iii) and $\p_m u(t,x,m)(y)$ is differentiable in $x\in\R^{d_x}$.
\end{theorem}
\begin{proof}
Define $u(t,x,m):=\frac{\delta}{\delta m}v(t,m)(x)$ where $v(t,m)$ is the global solution of the Bellman equation \eqref{nomin_1} constructed in Theroem \ref{GlobalSol_Bellman}. Then one can check directly that this $u(t,x,m)$ is a global solution to the master equation \eqref{master}, and it has the regularity deduced from that of $v(t,m)(x)$.
\end{proof}

Moreover, since $\p_x u(t,x,m)=\p_m v(t,m)(x)=Z^{t,m}_t(x)=\gamma(t,x,m)$, and the corresponding pair $\Big(X^{t,m}_s(x),Z^{t,m}_s(x)=\gamma(s,X^{t,m}_s(x),X^{t,m}_s\ot m)\Big)$ solves the FBODE system \eqref{MPIT} whose solution is unique by Theorem \ref{Thm6_4}, the solution of master equation \eqref{master} constructed in Theorem \ref{thm_master} is also unique. More precisely, suppose there exist two different solutions to the master equation \eqref{master}, both of which have the regularity stated in Theorem \ref{thm_master}, let us denote them by $u(t,x,m)$ and $\wb{u}(t,x,m)$, and define $\gamma(t,x,m):=\p_x u(t,x,m)$, $\wb{\gamma}(t,x,m):=\p_x \wb{u}(t,x,m)$. Denote by $X^{t,m }_s(x)$ and by $\wb{X}^{t,m }_s(x)$ the respective unique solutions to 
\small\begin{align*}
X^{t,m }_s(x) = x+\int_t^s f\Big(X^{t,m }_\tau(x),X^{t,m }_\tau\ot m,\alpha\big(X^{t,m }_\tau(x),X^{t,m }_\tau\ot m,\gamma(\tau,X^{t,m }_\tau(x),X^{t,m }_\tau\ot m)\big)\Big)d\tau,\\
\wb{X}^{t,m }_s(x) = x+\int_t^s f\Big(\wb{X}^{t,m }_\tau(x),\wb{X}^{t,m }_\tau\ot m,\alpha\big(\wb{X}^{t,m }_\tau(x),\wb{X}^{t,m }_\tau\ot m,\wb{\gamma}(\tau,\wb{X}^{t,m }_\tau(x),\wb{X}^{t,m }_\tau\ot m)\big)\Big)d\tau.
\end{align*}\normalsize
Also define $Z^{s,m }_\tau(x):=\gamma\big(\tau,X^{s,m }_\tau(x),X^{s,m }_\tau\ot m\big)$ and $\wb{Z}^{s,m }_\tau(x):=\wb{\gamma}\big(\tau,\wb{X}^{s,m }_\tau(x),\wb{X}^{s,m }_\tau\ot m\big)$. Then, by taking the Euclidean derivative $\p_x$ to the master equation \eqref{master} and then applying the method of characteristics, similar to establishing \eqref{eq_9_17_1}, the derived equation, one can find that both $\Big(X^{t,m}_s(x),Z^{t,m}_s(x)\Big)$ and $\Big(\wb{X}^{t,m}_s(x),\wb{Z}^{t,m}_s(x)\Big)$ solve the same FBODE system \eqref{MPIT} and thus, by Theorem \ref{Thm6_4}, $\Big(X^{t,m}_s(x),Z^{t,m}_s(x)\Big)=\Big(\wb{X}^{t,m}_s(x),\wb{Z}^{t,m}_s(x)\Big)$, and hence, $\gamma(t,x,m)=\wb{\gamma}(t,x,m)$. Therefore, by using the master equation \eqref{master} again,
we have $\p_t\big(u(t,x,m)-\wb{u}(t,x,m)\big)=0$ for all $t\in[0,T]$ and $u(T,x,m)-\wb{u}(T,x,m)=0$, and this implies that $u(t,x,m)=\wb{u}(t,x,m)$ for all $t$.

In summary, we have the following uniqueness of the solution to the master equation \eqref{master}.
\begin{theorem}\label{Unique_M} (Uniqueness of the solution to the master equation \eqref{master})
Under Assumption ${\bf(a1)}$-${\bf(a3)}$ and Hypothesis ${\bf(h1)}$-${\bf(h2)}$, the master equation \eqref{master} has at most one solution $u(t,x,m)$ that has the regularity stated in Theorem \ref
{thm_master}.
\end{theorem}

\section{A Non-linear-quadratic Example}\label{sec:nonLQ}
We here provide a non-trivial non-linear-quadratic example in which the drift function $f$, the running cost function $g$ and the terminal cost function $k$ are defined as follows; for simplicity, we just consider the process $x_t$ living in $\R$. Let $\varepsilon_2\in(0,1/4]$, $\varepsilon_3\in(0,1/8]$, $\varepsilon_4\in(0,1/2]$ be fixed constants. Define, for any $x$, $\alpha\in\R$, $\mu\in\mathcal{P}_2(\R)$, 
\begin{align}\label{example_f}
f(x,\mu,\alpha):= &\ x+ \alpha +\int_{\R} yd\mu(y)+\varepsilon_1 x \exp\Big(-x^2-\alpha^2-\Big(\int_\R \phi(y)d\mu(y)\Big)^2\Big),\\\label{example_g}
g(x,\mu,\alpha):=&\ \frac{1}{2}\alpha^2+\frac{1}{2} x^2-\varepsilon_2\bigg(\int_{\R} yd\mu(y)\bigg)^2+\varepsilon_3\alpha \int_{\R} y d\mu(y),\\\label{example_k}
k(x,\mu):=&\ \frac{1}{2}x^2-\varepsilon_4 x \int_{\R} yd\mu(y),
\end{align}
where the constant $\varepsilon_1\in(0,1)$ will be chosen sufficiently small below so that $(\bf{h2})$ will be satisfied, and 
\begin{align}
\phi(y):=\begin{cases}
|y|,\ \text{ for }|y|\geq 1;\\
-\frac{1}{8}y^4+\frac{3}{4}y^2+\frac{3}{8},\ \text{ for }|y|<1.
\end{cases}
\end{align}
\begin{remark}
The example of the drift function $f(x,\mu,\alpha)$ is clearly nonlinear and non-separable in all $x$, $\mu$ and $\alpha$. The running cost $g(x,\mu,\alpha)$ is non-separable in $\alpha$ and $\mu$ and it is not convex in $\mu$ but it is of quadratic-growth in $x$. The terminal cost $k(x,\mu)$ is not convex in $\mu$ yet it is of quadratic-growth in $x$. To the best of our knowledge, this example cannot be covered as a workable one in the theory proposed by the existing literature, such as \cite{cardaliaguet2019master}, \cite{gangbo2022mean} and \cite{gangbo2020global}.
\end{remark}
One can check that $\phi\in C^2(\R)$ satisfies, $\phi(y)\geq |y|$ for all $y\in\R$, 
\begin{align}
\phi'(y)=\begin{cases}
1,\ \text{ for }y\geq 1;\\
-1,\ \text{ for }y\leq 1;\\
-\frac{1}{2}y^3+\frac{3}{2}y,\ \text{ for }|y|<1,
\end{cases}\text{ and  }\ \  
\phi''(y)=\begin{cases}
0,\ \text{ for }|y|\geq 1,\\
-\frac{3}{2}y^2+\frac{3}{2},\ \text{ for }|y|<1.
\end{cases}
\end{align}
In addition, we can check that $\|\mu\|_1\leq \int_\R \phi(y)d\mu(y)$. 
It follows from routine calculations that
\small\begin{align*}
(i)\ &\p_x f(x,\mu,\alpha)=1+\varepsilon_1(1-2 x^2) \exp\Big(-x^2-\alpha^2-\Big(\int_\R \phi(y)d\mu(y)\Big)^2\Big);\\ 
&\p_\mu f(x,\mu,\alpha)(\wt{x})= 1-2\varepsilon_1 \phi'(\wt{x}) x \int_\R \phi(y)d\mu(y)  \exp\Big(-x^2-\alpha^2-\Big(\int_\R \phi(y)d\mu(y)\Big)^2\Big);\\ 
&\p_\alpha f(x,\mu,\alpha)= 1-2\varepsilon_1 \alpha x \exp\Big(-x^2-\alpha^2-\Big(\int_\R \phi(y)d\mu(y)\Big)^2\Big);\\
&\p_x\p_x f(x,\mu,\alpha)=-2x\Big(3-2x^2\Big)\varepsilon_1  \exp\Big(-x^2-\alpha^2-\Big(\int_\R \phi(y)d\mu(y)\Big)^2\Big);\\
&\p_x\p_\mu f(x,\mu,\alpha)(\wt{x})=-2\varepsilon_1 \phi'(\wt{x}) (1-2x^2) \int_\R \phi(y)d\mu(y)  \exp\Big(-x^2-\alpha^2-\Big(\int_\R \phi(y)d\mu(y)\Big)^2\Big)=\p_\mu \p_x f(x,\mu,\alpha)(\wt{x});\\ 
&\p_\mu\p_\mu f(x,\mu,\alpha)(\wt{x},\wh{x})=-2\varepsilon_1 \phi'(\wt{x})\phi'(\wh{x}) x (1-2\Big(\int_\R \phi(y)d\mu(y)\Big)^2)  \exp\Big(-x^2-\alpha^2-\Big(\int_\R \phi(y)d\mu(y)\Big)^2\Big);\\ 
&\p_{\wt{x}}\p_\mu f(x,\mu,\alpha)(\wt{x})=-2\varepsilon_1 \phi''(\wt{x}) x \int_\R \phi(y)d\mu(y)  \exp\Big(-x^2-\alpha^2-\Big(\int_\R \phi(y)d\mu(y)\Big)^2\Big);\\
&\p_x\p_\alpha f(x,\mu,\alpha)=-2\varepsilon_1 \alpha (1-2x^2) \exp\Big(-x^2-\alpha^2-\Big(\int_\R \phi(y)d\mu(y)\Big)^2\Big)=\p_\alpha \p_x f(x,\mu,\alpha);\\
&\p_\alpha\p_\mu f(x,\mu,\alpha)(\wt{x})=4\varepsilon_1 \phi'(\wt{x}) x\alpha \int_\R \phi(y)d\mu(y)  \exp\Big(-x^2-\alpha^2-\Big(\int_\R \phi(y)d\mu(y)\Big)^2\Big)=\p_\mu\p_\alpha f(x,\mu,\alpha)(\wt{x});\\ 
&\p_\alpha\p_\alpha f(x,\mu,\alpha)=-2\varepsilon_1 (1-2\alpha^2) x \exp\Big(-x^2-\alpha^2-\Big(\int_\R \phi(y)d\mu(y)\Big)^2\Big).\\ 
(ii)\ &\p_x g(x,\mu,\alpha)=x,\ \p_\alpha g(x,\mu,\alpha)=\alpha+\varepsilon_3\int_{\R} y d\mu(y);\\
&\p_\mu g(x,\mu,\alpha)(\wt{x})=-2\varepsilon_2\int_{\R} yd\mu(y)+\varepsilon_3\alpha;\\
&G(x,\mu,\alpha)=g(x,\mu,\alpha)+\int_{\R}\frac{\delta}{\delta\mu}g(\wt{x},\mu,\alpha)(x)d\mu(\wt{x})\\
&\ \ \ \ \ \ \ \ \ \ \ \ \ =\frac{1}{2}\alpha^2+\frac{1}{2} x^2-\varepsilon_2\bigg(\int_{\R} yd\mu(y)\bigg)\left(2x+\int_{\R} yd\mu(y)\right)+\varepsilon_3\alpha \left(x+\int_{\R} y d\mu(y)\right);\\
&\p_\alpha\p_\alpha g(x,\mu,\alpha)=\p_x\p_x g(x,\mu,\alpha)=1,\ \p_\alpha\p_\mu g(x,\mu,\alpha)(\wt{x})=\p_\mu\p_\alpha g(x,\mu,\alpha)(\wt{x})=\varepsilon_3;\\ 
&\p_x\p_\alpha g(x,\mu,\alpha)=\p_\alpha\p_x g(x,\mu,\alpha)=\p_{x}\p_\mu g(x,\mu,\alpha)(\wt{x})=\p_{\mu} \p_x g(x,\mu,\alpha)(\wt{x})=\p_{\wt{x}}\p_\mu g(x,\mu,\alpha)(\wt{x})=0;\\
&\p_\mu\p_\mu g(x,\mu,\alpha)(\wt{x},\wh{x})=-2\varepsilon_2;\\
&\p_x\p_x G(x,\mu,\alpha)=1,\ \p_x\p_\mu G(x,\mu,\alpha)(\wt{x})=\p_\mu\p_x G(x,\mu,\alpha)(\wt{x})=-2\varepsilon_2.\\
(iii)\ &\p_x k(x,\mu)=x-\varepsilon_4 \int_{\R} y d\mu(y),\ \p_\mu k(x,\mu)(\wt{x})= -\varepsilon_4 x,\ \p_{xx} k(x,\mu)=1,\ \p_{\mu\mu} k(x,\mu)(\wt{x},\wh{x})=0;\\
&\p_{x}\p_\mu k(x,\mu)(\wt{x})=\p_{\mu} \p_x k(x,\mu)(\wt{x})=-\varepsilon_4,\ \p_{\wt{x}}\p_\mu k(x,\mu)(\wt{x})=0;\\
&K(x,\mu)=k(x,\mu)+\int_{\R}\frac{\delta}{\delta\mu}k(\wt{x},\mu)(x)d\mu(\wt{x})=\left(\frac{1}{2}-\varepsilon_4\right)x^2-\varepsilon_4 x \int_{\R} yd\mu(y);\\
&\p_x\p_x K(x,\mu)=1,\ \p_x\p_\mu K(x,\mu)(\wt{x})=\p_\mu\p_x K(x,\mu)(\wt{x})=-\varepsilon_4.
\end{align*}\normalsize
All of these functions are clearly jointly Lipschitz continuous in their corresponding arguments.  Note also that $0\leq x\exp\big(-x^2\big)\leq \frac{1}{\sqrt{2}}\exp\big(-\frac{1}{2}\big)\leq \frac{1}{2}$ and $0\leq x^2\exp\big(-x^2\big)\leq \exp\big(-1\big)\leq \frac{1}{2}$. We emphisize here that both $G(x,\mu,\alpha)$ and $K(x,\mu)$ do not satisfy Lasry-Lions monotonicity conditions since $\p_x\p_\mu G(x,\mu,\alpha)(\wt{x})=-2\varepsilon_2<0$ and $\p_x\p_\mu K(x,\mu)(\wt{x})=-\varepsilon_4<0$.

Now, we shall check that whether the $f$, $g$ and $k$ defined in \eqref{example_f}-\eqref{example_k} do satisfy Assumption ${\bf(a1)}$-${\bf(a3)}$ and Hypothesis ${\bf(h1)}$-${\bf(h2)}$ in the following:\\
${\bf(a1)}$(i) Since $|\p_\alpha f(x,\mu,\alpha)|^2 \geq (1-\frac{1}{2}\varepsilon_1)^2$, then \eqref{positive_f} is valid with $\lambda_f:=\frac{1}{4}$ for any $0<\varepsilon_1\leq 1$.\\
(ii) Since $|\p_x f(x,\mu,\alpha)| \leq 1+2\varepsilon_1$, $|\p_\mu f(x,\mu,\alpha)(\wt{x})| \leq  1+\frac{1}{2}\varepsilon_1$ and $|\p_\alpha f(x,\mu,\alpha)| \leq 1+\frac{1}{2}\varepsilon_1$, then \eqref{bdd_d1_f} is valid with $\Lambda_f:=3$ for any $0<\varepsilon_1\leq 1$.\\
(iii) Using the decay nature of $\exp\Big(-x^2-\alpha^2-\Big(\int_\R \phi(y)d\mu(y)\Big)^2\Big)$ and the fact that $\|\mu\|_1\leq \int_\R \phi(y)d\mu(y)$, one can show that there exists a generic positive constant $C_1$ such that \eqref{bdd_d2_f} is valid with $\wb{l}_f:=C_1\varepsilon_1$.\\
(iv) The Lipschitz continuity can be verified in the same manner as in ${\bf(a1)}$(iii).\\
${\bf(a2)}$(i) Since $\p_x\p_x g(x,\mu,\alpha)=1$, \eqref{positive_g_alpha} is valid with $\lambda_g:=1$.\\
(ii) First of all, since $\p_x\p_x G(x,\mu,\alpha)=1$, \eqref{positive_g_x} is valid with $\lambda_g=1$; since $\p_x\p_\mu G(x,\mu,\alpha)(\wt{x})=\p_\mu\p_x G(x,\mu,\alpha)(\wt{x})\\=-2\varepsilon_2$, \eqref{positive_g_mu_1} is valid with $l_g:=2\varepsilon_2\leq \frac{1}{2}\lambda_g=\frac{1}{2}$ for any $\varepsilon_2\leq\frac{1}{4}$; in addition, $\int_{\R}G(x,\mu,\alpha)-G(x,\mu',\alpha)d(\mu-\mu')(x)=-2\varepsilon_2\left(\int_{\R} yd(\mu-\mu')(y)\right)^2$, which implies \eqref{positive_g_mu}.\\
(iii) Using the explicit formulae for all second-order derivatives of $g$, one can check that \eqref{bdd_d2_g_1} is valid with $\Lambda_g:=3/2$ for any $\varepsilon_2\leq\frac{1}{4}$ and \eqref{bdd_d2_g_2} is valid with $\wb{l}_g:=\varepsilon_3$.\\
${\bf(a3)}$ (i) Since $\p_x\p_x K(x,\mu)=1$, \eqref{positive_k} are valid with $\lambda_k:=1$.\\
(ii) Since $\p_x\p_\mu K(x,\mu)(\wt{x})=\p_\mu\p_x K(x,\mu)(\wt{x})=-\varepsilon_4$, \eqref{positive_k_mu_1} is valid with $l_k:=\varepsilon_4\leq \frac{1}{2}\lambda_k=\frac{1}{2}$ for any $\varepsilon_4\leq\frac{1}{2}$; in addition, $\int_{\R}K(x,\mu)-K(x,\mu')d(\mu-\mu')(x)=-\varepsilon_4 \left(\int_{\R} yd(\mu-\mu')(y)\right)^2$, which implies \eqref{positive_k_mu}.\\
(iii) Using the explicit formulae for all second-order derivatives of $k$, one can check that \eqref{bdd_d2_k_1} is valid with $\Lambda_k:=3/2$ for any $\varepsilon_4\leq\frac{1}{2}$.\\
${\bf(h1)}$ Evaluating \eqref{example_f} and formulae for derivatives of $g$ and $k$ at $(x,\mu,\alpha,\wt{x})=(0,\delta_0,0,0)$ shows the fulfillment of ${\bf(h1)}$.\\
${\bf(h2)}$ It is worth noting that for the $\wb{l}_g$ and $\lambda_g$ chosen above, we have $\wb{l}_g=\varepsilon_3\leq \frac{1}{8}=\frac{1}{8}\lambda_g$. Furthermore, the constant $L^*_0$ defined in \eqref{L_star_0} 
depends only on $\lambda_f=1/4$, $\Lambda_f=3$, $\lambda_g=1$, $\Lambda_g=3/2$, $\lambda_k=1$ and $\Lambda_k=3/2$ but not on $\wb{l}_f$ or $\varepsilon_1$, so we can always choose an $\varepsilon_1$ small enough, such that $\wb{l}_f=C_1\varepsilon_1\leq \frac{1}{40 \max\{\wb{L}_k,L^*_0\}}\lambda_g=\frac{1}{40 L^*_0}$. Therefore, ${\bf(h2)}$ is satisfied.

\section*{Acknowledgment}
Alain Bensoussan is supported by the National Science Foundation under grants NSF-DMS-1905449
and NSF-DMS-2204795, and grant from the SAR Hong Kong RGC GRF 14301321.
Tak Kwong Wong was partially supported by the HKU Seed Fund for Basic Research under the project code 201702159009, the Start-up Allowance for Croucher Award Recipients, and Hong Kong General Research Fund (GRF) grants with project numbers 17306420, 17302521, and 17315322.
Phillip Yam acknowledges the financial supports from HKSAR-GRF 14301321 with the project title ``General Theory for Infinite Dimensional Stochastic Control: Mean Field and Some Classical Problems'', HKSAR-GRF 14300123 with the project title ``Well-posedness of Some Poisson-driven Mean Field Learning Models and their Applications'', and NSF-DMS-2204795. He also thanks Columbia University for the kind invitation to be a visiting faculty member in the Department of Statistics during his sabbatical leave. Hongwei Yuan thanks the Department of Statistics of The Chinese University of Hong Kong and the Department of Mathematics of University of Macau for the financial support.

\section{Appendix}
\subsection{Proving the Identity \eqref{lifted:estimate_1}}\label{lifted:estimate_1_detail}

By further differentiating the first order condition in the lifted version
\eqref{lifted:First_Order_Con}, 
we obtain, for any $\wt{A}\in L^{2,d_x,d_\alpha}_{m_0}$, $\wt{X},\,\wt{Z}\in L^{2,d_x}_{m_0}$,
\begin{align}\no
(i)\ &\left\langle D_XD_A  F (X,A (X,Z))(\wt{X},\wt{A})+D_AD_A  F (X,A (X,Z))\big(D_XA (X,Z)(\wt{X}),\wt{A}\big), Z\right\rangle_{L^{2,d_x}_{m_0}}  \\\label{D_XFOC}
&\ \ + D_XD_A \wb{G} (X,A (X,Z))(\wt{X},\wt{A})+D_AD_A \wb{G} (X,A (X,Z))\big(D_X A (X,Z)(\wt{X}),\wt{A}\big)=0,\\\no
(ii)\ &\left\langle D_AD_A  F (X,A (X,Z))\big(D_ZA (X,Z)(\wt{Z}),\wt{A}\big), Z\right\rangle_{L^{2,d_x}_{m_0}}+\left\langle D_A  F (X,A (X,Z))(\wt{A}), \wt{Z}\right\rangle_{L^{2,d_x}_{m_0}}  \\\label{D_AFOC}
&\ \ +D_AD_A \wb{G} (X,A (X,Z))\big(D_Z A (X,Z)(\wt{Z}),\wt{A}\big)=0.
\end{align}
We next use Riesz representation theorem to define the inverse mapping of 
\small\begin{align}\label{D_AA_L_App}
\widecheck{A}\in L^{2,d_x,d_\alpha}_{m_0}\mapsto &\wh{A}:=\left\langle D_AD_A F (X,A(X,Z))(\widecheck{A},\cdot) , Z\right\rangle_{L^{2,d_x}_{m_0}}+ D_AD_A \wb{G}(X,A(X,Z))(\widecheck{A},\cdot)\in \mathcal{L}(L^{2,d_x,d_\alpha}_{m_0};\R)\simeq L^{2,d_x,d_\alpha}_{m_0};
\end{align}\normalsize 
under the assumptions that (i) the running cost function is convex in the control (also see \eqref{positive_g_alpha}), that is $D_AD_A \wb{G}(X,A)(\wb{A},\wb{A})\geq \Lambda_g \|\wb{A}\|_{L^{2,d_x,d_\alpha}_{m_0}}^2$ for all $\wb{A}\in L^{2,d_x,d_\alpha}_{m_0}$; and (ii) the second order derivative of the drift function with respect to the control is bounded and decays at the proximity of infinity (also see \eqref{bdd_d2_f}), i.e. $\|D_AD_A F(X,A)(\wb{A},\wb{A})\|_{L^{2,d_x}_{m_0}}\leq l_F\|\wb{A}\|_{L^{2,d_x,d_\alpha}_{m_0}}^2/(1+\|X\|_{L^{2,d_x}_{m_0}})$ for all $\wb{A}\in L^{2,d_x,d_\alpha}_{m_0}$; these two assumptions can then warrant the Hilbert space isomorphism of \eqref{D_AA_L_App} when $(X,Z)$ satisfies a {\it cone condition} $\|Z\|_{L^{2,d_x}_{m_0}}\leq C_0(1+\|X\|_{L^{2,d_x}_{m_0}})$ for some positive constant $C_0$ (also see \eqref{eq_7_44}) and $\Lambda_g-C_0l_F>0$.
By then we can define the inverse mapping
\small
\begin{align}\label{InvMap}
&\bigg(\left\langle D_AD_A F (X,A(X,Z)), Z\right\rangle_{L^{2,d_x}_{m_0}} + D_AD_A \wb{G} (X,A(X,Z))\bigg)^{-1}:\\\no
&\wh{A} \in L^{2,d_x,d_\alpha}_{m_0}\mapsto \widecheck{A}:=\bigg(\left\langle D_AD_A F (X,A(X,Z)), Z\right\rangle_{L^{2,d_x}_{m_0}} + D_AD_A \wb{G} (X,A(X,Z))\bigg)^{-1}(\wh{A},\cdot)\in \mathcal{L}(L^{2,d_x,d_\alpha}_{m_0};\R)\simeq L^{2,d_x,d_\alpha}_{m_0},
\end{align}\normalsize
where the isomorphism $\mathcal{L}(L^{2,d_x,d_\alpha}_{m_0};\R)\simeq L^{2,d_x,d_\alpha}_{m_0}$ actually means, for all $\wb{A}\in L^{2,d_x,d_\alpha}_{m_0}$, \small$$\left\langle \widecheck{A},\wb{A}\right\rangle_{L^{2,d_x,d_\alpha}_{m_0}}=\left(\left\langle D_AD_A F (X,A(X,Z)), Z\right\rangle_{L^{2,d_x}_{m_0}} + D_AD_A \wb{G} (X,A(X,Z))\right)^{-1}(\wh{A},\wb{A}).\ $$\normalsize In other words, \eqref{InvMap} means, for all $\wb{A}\in L^{2,d_x,d_\alpha}_{m_0}$,\small
\begin{align}
\left\langle \wh{A},\wb{A}\right\rangle_{L^{2,d_x,d_\alpha}_{m_0}}=\left\langle D_AD_A F (X,A(X,Z))(\wt{A},\wh{A}) , Z\right\rangle_{L^{2,d_x}_{m_0}} + D_AD_A \wb{G} (X,A(X,Z))(\widecheck{A},\wb{A}).
\end{align}
\normalsize
Therefore, \eqref{D_XFOC} and \eqref{D_AFOC} means that, for any $\wt{A}\in L^{2,d_x,d_\alpha}_{m_0}$, $\wt{X},\,\wt{Z}\in L^{2,d_x}_{m_0}$,
\footnotesize\begin{align}\no
\left\langle D_XA(X,Z)(\wt{X}),\wt{A}\right\rangle_{L^{2,d_x,d_\alpha}_{m_0}}=&\ \bigg(\left\langle  D_AD_A  F (X,A(X,Z)) , Z\right\rangle_{L^{2,d_x}_{m_0}} 
  +D_AD_A \wb{G} (X,A(X,Z))\bigg)^{-1}\\\label{DYA}
&\ \ \bigg(-\left\langle D_XD_A  F (X,A(X,Z))(\wt{X},\wt{A}), Z\right\rangle_{L^{2,d_x}_{m_0}}-D_XD_A \wb{G} (X,A(X,Z))(\wt{X},\wt{A}),\wt{A}\bigg),\\\no
\left\langle D_ZA(X,Z)(\wt{Z}),\wt{A}\right\rangle_{L^{2,d_x,d_\alpha}_{m_0}}
=&\ \bigg(\left\langle  D_AD_A  F (X,A(X,Z)) , Z\right\rangle_{L^{2,d_x}_{m_0}} 
  +D_AD_A \wb{G} (X,A(X,Z))\bigg)^{-1}\\\label{DZA}
&\ \ \bigg(-\left\langle D_{A}  F (X,A(X,Z))(\wt{A}), \wt{Z}\right\rangle_{L^{2,d_x}_{m_0}},\wt{A}\bigg).
\end{align}\normalsize
For any $\wt{Y}\in L^{2,d_x}_{m_0}$,
\scriptsize\begin{align*}
&\bigg\langle D_Y Z^{t,Y}_T(\wt{Y}),D_Y X^{t,Y}_T(\wt{Y})\bigg\rangle_{ L^{2,d_x}_{m_0}}-\bigg\langle D_Y Z^{t,Y}_t(\wt{Y}),D_Y X^{t,Y}_t(\wt{Y})\bigg\rangle_{ L^{2,d_x}_{m_0}}\\
=&\ \int_t^T \frac{d}{ds}\bigg\langle D_Y Z^{t,Y}_s(\wt{Y}),D_Y X^{t,Y}_s(\wt{Y})\bigg\rangle_{ L^{2,d_x}_{m_0}}ds\\
=&\ \int_t^T \bigg(\bigg\langle D_Y Z^{t,Y}_s(\wt{Y}),\frac{d}{ds}D_Y X^{t,Y}_s(\wt{Y})\bigg\rangle_{ L^{2,d_x}_{m_0}}+\bigg\langle \frac{d}{ds}D_Y Z^{t,Y}_s(\wt{Y}),D_Y X^{t,Y}_s(\wt{Y})\bigg\rangle_{ L^{2,d_x}_{m_0}}\bigg)ds\text{ (by using \eqref{lifted:D_YXexistence})}\\
=&\ {\color{orange}\int_t^T} \bigg(\bigg\langle D_Y Z^{t,Y}_s(\wt{Y}),D_X F\left(X^{t,Y}_s,A\left(X^{t,Y}_s,Z^{t,Y}_s\right)\right)(D_Y X^{t,Y}_s(\wt{Y}))\bigg\rangle_{ L^{2,d_x}_{m_0}}\\
&{\color{blue}\ \ \ \ \ \ \ \ -\bigg\langle  D_AD_A F\left(X^{t,Y}_s,A\left(X^{t,Y}_s,Z^{t,Y}_s\right)\right)\Big(D_Z A\left(X^{t,Y}_s,Z^{t,Y}_s\right)(D_Y Z^{t,Y}_s(\wt{Y})),}\\
&{\color{purple}\ \ \ \ \ \ \ \ \ \ \ \ \ \ \ \ \ \ \ \ \ \ \ \ \ \ \ \ \ \ \ \ \ \ \ \ \ \ \ \ \ \ \ \ \ \ \ \ \ \ \ \ \ \ \ \ D_X A\left(X^{t,Y}_s,Z^{t,Y}_s\right)(D_Y X^{t,Y}_s)}{\color{blue}+D_Z A\left(X^{t,Y}_s,Z^{t,Y}_s\right)(D_Y Z^{t,Y}_s(\wt{Y})) \Big),Z^{t,Y}_s\bigg\rangle_{ L^{2,d_x}_{m_0}}}\\
&{\color{blue}\ \ \ \ \ \ \ \ - D_AD_A \wb{G}\left(X^{t,Y}_s,A\left(X^{t,Y}_s,Z^{t,Y}_s\right)\right)\Big(D_Z A\left(X^{t,Y}_s,Z^{t,Y}_s\right)(D_Y Z^{t,Y}_s(\wt{Y})),} \\
&{\color{purple}\ \ \ \ \ \ \ \ \ \ \ \ \ \ \ \ \ \ \ \ \ \ \ \ \ \ \ \ \ \ \ \ \ \ \ \ \ \ \ \ \ \ \ \ \ \ \ \ \ \ \ \ \ \ \ \ D_X A\left(X^{t,Y}_s,Z^{t,Y}_s\right)(D_Y X^{t,Y}_s)}{\color{blue}+D_Z A\left(X^{t,Y}_s,Z^{t,Y}_s\right)(D_Y Z^{t,Y}_s) \Big)}\ \text{ (by using \eqref{D_AFOC})}\\
&\ \ \ \ \ \ \ \ -\left\langle D_XD_X F\left(X^{t,Y}_s,A\left(X^{t,Y}_s,Z^{t,Y}_s\right)(D_Y X^{t,Y}_s(\wt{Y}),D_Y X^{t,Y}_s(\wt{Y}))\right),Z^{t,Y}_s\right\rangle_{L^{2,d_x}_{m_0}}\\
&\ \ \ \ \ \ \ \ -\left\langle D_AD_X F\left(X^{t,Y}_s,A\left(X^{t,Y}_s,Z^{t,Y}_s\right)\right)\left(D_X A\left(X^{t,Y}_s,Z^{t,Y}_s\right)(D_Y X^{t,Y}_s(\wt{Y}))+{\color{purple}D_Z A\left(X^{t,Y}_s,Z^{t,Y}_s\right)(D_Y Z^{t,Y}_s(\wt{Y}))},D_Y X^{t,Y}_s(\wt{Y})\right),Z^{t,Y}_s\right\rangle_{L^{2,d_x}_{m_0}}\\
&\ \ \ \  \ \ \ \ - \left\langle D_X F\left(X^{t,Y}_s,A\left(X^{t,Y}_s,Z^{t,Y}_s\right)\right)(D_Y X^{t,Y}_s(\wt{Y})),D_Y Z^{t,Y}_s(\wt{Y})\right\rangle_{L^{2,d_x}_{m_0}}\\
&\ \ \ \ \ \ \ \ -  D_XD_X \wb{G}\left(X^{t,Y}_s,A\left(X^{t,Y}_s,Z^{t,Y}_s\right)\right)(D_Y X^{t,Y}_s(\wt{Y}),D_Y X^{t,Y}_s(\wt{Y}))\\
&\ \ \ \ \ \ \ \ - D_AD_X \wb{G}\left(X^{t,Y}_s,A\left(X^{t,Y}_s,Z^{t,Y}_s\right)\right)\left(D_X A\left(X^{t,Y}_s,Z^{t,Y}_s\right)(D_Y X^{t,Y}_s(\wt{Y}))+{\color{purple}D_Z A\left(X^{t,Y}_s,Z^{t,Y}_s\right)(D_Y Z^{t,Y}_s(\wt{Y}))},D_Y X^{t,Y}_s(\wt{Y})\right)\bigg){\color{orange}ds}\\
=&\ {\color{orange}\int_t^T} \bigg( -\bigg\langle  D_AD_A F\left(X^{t,Y}_s,A\left(X^{t,Y}_s,Z^{t,Y}_s\right)\right)\Big(D_Z A\left(X^{t,Y}_s,Z^{t,Y}_s\right)(D_Y Z^{t,Y}_s(\wt{Y})),D_Z A\left(X^{t,Y}_s,Z^{t,Y}_s\right)(D_Y Z^{t,Y}_s(\wt{Y})) \Big),Z^{t,Y}_s\bigg\rangle_{ L^{2,d_x}_{m_0}} \\
& \ \ \ \ \ \ \ \ - D_AD_A \wb{G}\left(X^{t,Y}_s,A\left(X^{t,Y}_s,Z^{t,Y}_s\right)\right)\Big(D_Z A\left(X^{t,Y}_s,Z^{t,Y}_s\right)(D_Y Z^{t,Y}_s(\wt{Y})), D_Z A\left(X^{t,Y}_s,Z^{t,Y}_s\right)(D_Y Z^{t,Y}_s) \Big)  \\
&\ \ \ \ \ \ \ \ -\left\langle D_XD_X F\left(X^{t,Y}_s,A\left(X^{t,Y}_s,Z^{t,Y}_s\right)(D_Y X^{t,Y}_s(\wt{Y}),D_Y X^{t,Y}_s(\wt{Y}))\right),Z^{t,Y}_s\right\rangle_{L^{2,d_x}_{m_0}}\\
&\ \ \ \ \ \ \ \ -  D_XD_X \wb{G}\left(X^{t,Y}_s,A\left(X^{t,Y}_s,Z^{t,Y}_s\right)\right)(D_Y X^{t,Y}_s(\wt{Y}),D_Y X^{t,Y}_s(\wt{Y}))\\
&  \ \ \ \ \ \ \ \ -\left\langle D_AD_X F\left(X^{t,Y}_s,A\left(X^{t,Y}_s,Z^{t,Y}_s\right)\right)\left(D_X A\left(X^{t,Y}_s,Z^{t,Y}_s\right)(D_Y X^{t,Y}_s(\wt{Y})),D_Y X^{t,Y}_s(\wt{Y})\right),Z^{t,Y}_s\right\rangle_{L^{2,d_x}_{m_0}} \\
& \ \ \ \ \ \ \ \ - D_AD_X \wb{G}\left(X^{t,Y}_s,A\left(X^{t,Y}_s,Z^{t,Y}_s\right)\right)\left(D_X A\left(X^{t,Y}_s,Z^{t,Y}_s\right)(D_Y X^{t,Y}_s(\wt{Y})) ,D_Y X^{t,Y}_s(\wt{Y})\right) \bigg){\color{orange}ds}\text{ (using \eqref{D_XFOC}, purple terms can be removed)}\\
=&\ {\color{orange}\int_t^T} \bigg({\color{red}\bigg\langle  D_{A} F\left(X^{t,Y}_s,A\left(X^{t,Y}_s,Z^{t,Y}_s\right)\right)\Big(D_Z A\left(X^{t,Y}_s,Z^{t,Y}_s\right)(D_Y Z^{t,Y}_s(\wt{Y})) \Big),D_Y Z^{t,Y}_s(\wt{Y})\bigg\rangle_{ L^{2,d_x}_{m_0}}}\\
&\ \ \ \ \ \ \ \ -\left\langle D_XD_X F\left(X^{t,Y}_s,A\left(X^{t,Y}_s,Z^{t,Y}_s\right)(D_Y X^{t,Y}_s(\wt{Y}),D_Y X^{t,Y}_s(\wt{Y}))\right),Z^{t,Y}_s\right\rangle_{L^{2,d_x}_{m_0}}\\
&\ \ \ \ \ \ \ \ -  D_XD_X \wb{G}\left(X^{t,Y}_s,A\left(X^{t,Y}_s,Z^{t,Y}_s\right)\right)(D_Y X^{t,Y}_s(\wt{Y}),D_Y X^{t,Y}_s(\wt{Y}))\\
&{\color{blue} \ \ \ \ \ \ \ \ +\left\langle D_AD_A F\left(X^{t,Y}_s,A\left(X^{t,Y}_s,Z^{t,Y}_s\right)\right)\left(D_X A\left(X^{t,Y}_s,Z^{t,Y}_s\right)(D_Y X^{t,Y}_s(\wt{Y})),D_X A\left(X^{t,Y}_s,Z^{t,Y}_s\right)(D_Y X^{t,Y}_s(\wt{Y}))\right),Z^{t,Y}_s\right\rangle_{L^{2,d_x}_{m_0}}}\\
&{\color{blue}\ \ \ \ \ \ \ \ + D_AD_A \wb{G}\left(X^{t,Y}_s,A\left(X^{t,Y}_s,Z^{t,Y}_s\right)\right)\left(D_X A\left(X^{t,Y}_s,Z^{t,Y}_s\right)(D_Y X^{t,Y}_s(\wt{Y})) ,D_X A\left(X^{t,Y}_s,Z^{t,Y}_s\right)(D_Y X^{t,Y}_s(\wt{Y}))\right)}\bigg){\color{orange}ds}\\
&\text{ (using \eqref{D_AFOC} for the first red term and using \eqref{D_XFOC} for the last two blue terms in the last expression)},\\
=&\ -\int_t^T {\color{red}\left(D_AD_A L\left(X^{t,Y}_s,A\left(X^{t,Y}_s,Z^{t,Y}_s\right);Z^{t,Y}_s\right)\right)^{-1}\bigg(\left\langle  D_{A} F\left(X^{t,Y}_s,A\left(X^{t,Y}_s,Z^{t,Y}_s\right)\right), D_Y Z^{t,Y}_s(\wt{Y})\right\rangle_{ L^{2,d_x,d_\alpha}_{m_0}},}\\
&\ \ \ \ \ \ \ \ \ \ \ \ \ \ \ \ \ \ \ \ \ \ \ \ \ \ \ \ \ \ \ \ \ \ \ \ \ \ \ \ \ \ \ \ \ \ \ \ \ \ \ \ \ \ \ \ \ \ \ \ \ \ \ \ \ \ \ \ \ \ \ \ \ {\color{red}\left\langle  D_{A} F\left(X^{t,Y}_s,A\left(X^{t,Y}_s,Z^{t,Y}_s\right)\right), D_Y Z^{t,Y}_s(\wt{Y})\right\rangle_{ L^{2,d_x,d_\alpha}_{m_0}}\bigg)}ds\\
&-\int_t^T\bigg(D_XD_X L\left(X^{t,Y}_s,A\left(X^{t,Y}_s,Z^{t,Y}_s\right);Z^{t,Y}_s\right)\big(D_Y X^{t,Y}_s(\wt{Y}),D_Y X^{t,Y}_s(\wt{Y})\big)\\
&\ \ \ \ \ \ \ \ \ \ \ \ -\left(D_AD_A L\left(X^{t,Y}_s,A\left(X^{t,Y}_s,Z^{t,Y}_s\right);Z^{t,Y}_s\right)\right)^{-1}\bigg(\left\langle D_AD_X L\left(X^{t,Y}_s,A\left(X^{t,Y}_s,Z^{t,Y}_s\right);Z^{t,Y}_s\right),D_Y X^{t,Y}_s(\wt{Y}) \right\rangle_{L^{2,d_x,d_\alpha}_{m_0}},\\
&\ \ \ \ \ \ \ \ \ \ \ \ \ \ \ \ \ \ \ \ \ \ \ \ \ \ \ \ \ \ \ \ \ \ \ \ \ \ \ \ \ \ \ \ \ \ \ \ \ \ \ \ \ \ \ \ \ \ \ \ \ \ \ \ \ \ \ \ \ \ \ \ \ \ \ \ \ \ \ \left\langle D_XD_A L\left(X^{t,Y}_s,A\left(X^{t,Y}_s,Z^{t,Y}_s\right);Z^{t,Y}_s\right),D_Y X^{t,Y}_s(\wt{Y}) \right\rangle_{L^{2,d_x,d_\alpha}_{m_0}}\bigg)ds,
\end{align*}\normalsize
where we use \eqref{DYA} and \eqref{DZA} to replace $D_X A(X,Z)$ and $D_Z A(X,Z)$ respectively in the last equality.

\subsection{Proving the Inequality \eqref{lifted:estimate_2}}\label{lifted:estimate_2_detail}
The assumptions that the Schur complement is positive-definite, the terminal cost function $\wb{K}(X)$ is convex in $X$ and $D_{A} F(X,A)^\top D_{A} F(X,A)$ is positive-definite,  implies that there exist positive constants $\lambda_{Schur}$, $\lambda_K$ and $\lambda_F$ such that
\small\begin{align*}
&D_XD_X L\big(D_Y X^{t,Y}_s(\wt{Y}),D_Y X^{t,Y}_s(\wt{Y})\big)-\left(D_AD_A L\right)^{-1}\bigg(\left\langle D_AD_X L,D_Y X^{t,Y}_s(\wt{Y}) \right\rangle_{L^{2,d_x,d_\alpha}_{m_0}},\left\langle D_XD_A L,D_Y X^{t,Y}_s(\wt{Y}) \right\rangle_{L^{2,d_x,d_\alpha}_{m_0}}\\
&\geq \lambda_{Schur}\|D_Y X^{t,Y}_s(\wt{Y})\|_{L^{2,d_x}_{m_0}}^2,\\
&D_XD_X\wb{K}(X)(\wt{X},\wt{X})\geq \lambda_K\|\wt{X}\|_{L^{2,d_x}_{m_0}},\\
&\left\|\left\langle  D_{A} F, D_Y Z^{t,Y}_s(\wt{Y})\right\rangle_{ L^{2,d_x,d_\alpha}_{m_0}}\right\|_{L^{2,d_x,d_\alpha}_{m_0}}\geq \lambda_F\|D_Y Z^{t,Y}_s(\wt{Y})\|_{L^{2,d_x}_{m_0}}^2,
\end{align*}\normalsize
where the arguments $\left(X^{t,Y}_s,A\left(X^{t,Y}_s,Z^{t,Y}_s\right);Z^{t,Y}_s\right)$ of the Lagrangian $L(X,A;Z)$ and the arguments $\left(X^{t,Y}_s,A\left(X^{t,Y}_s,Z^{t,Y}_s\right)\right)$ of the drift function $F(X,A)$ are omitted for simplicity if there is no cause of ambiguity. Thus, by \eqref{lifted:estimate_1} and the initial condition and terminal condition of \eqref{lifted:D_YXexistence}, we have
\small\begin{align}\no
&\lambda_K\|D_Y X^{t,Y}_T(\wt{Y})\|_{L^{2,d_x}_{m_0}}-\bigg\langle D_Y Z^{t,Y}_t(\wt{Y}),\wt{Y}\bigg\rangle_{ L^{2,d_x}_{m_0}}\\\no
\leq &\ \bigg\langle D_XD_X \wb{K} (X^{t,Y}_T)(D_Y X^{t,Y}_T(\wt{Y})),D_Y X^{t,Y}_T(\wt{Y})\bigg\rangle_{ L^{2,d_x}_{m_0}}-\bigg\langle D_Y Z^{t,Y}_t(\wt{Y}),\wt{Y}\bigg\rangle_{ L^{2,d_x}_{m_0}}\\\no
=&\ -\int_t^T \left(D_AD_A L\right)^{-1}\bigg(\left\langle  D_{A} F, D_Y Z^{t,Y}_s(\wt{Y})\right\rangle_{ L^{2,d_x,d_\alpha}_{m_0}},\left\langle  D_{A} F, D_Y Z^{t,Y}_s(\wt{Y})\right\rangle_{ L^{2,d_x,d_\alpha}_{m_0}}\bigg)ds\\\no
&-\int_t^T\bigg(D_XD_X L\big(D_Y X^{t,Y}_s(\wt{Y}),D_Y X^{t,Y}_s(\wt{Y})\big)\\\no
&\ \ \ \ \ \ \ \ \ \ \ \ -\left(D_AD_A L\right)^{-1}\bigg(\left\langle D_AD_X L,D_Y X^{t,Y}_s(\wt{Y}) \right\rangle_{L^{2,d_x,d_\alpha}_{m_0}},\left\langle D_XD_A L,D_Y X^{t,Y}_s(\wt{Y}) \right\rangle_{L^{2,d_x,d_\alpha}_{m_0}}\bigg)ds\\\no
\leq &\ -\frac{\lambda_F}{\Lambda_g+C_0l_F}\int_t^T\|D_Y Z^{t,Y}_s(\wt{Y})\|_{L^{2,d_x}_{m_0}}^2 ds-\lambda_{Schur}\int_t^T\|D_Y X^{t,Y}_s(\wt{Y})\|_{L^{2,d_x}_{m_0}}^2 ds
\end{align}\normalsize
where the constants $\Lambda_g$, $C_0$ and $l_F$ are defined in Section \ref{lifted:estimate_1_detail}; thus \eqref{lifted:estimate_2} is valid with $$c=\min\left\{\frac{\lambda_F}{\Lambda_g+C_0l_F},\lambda_{Schur},\lambda_K\right\}.$$

\subsection{Proof of Propositions in Section \ref{properties}}\label{App:prop}

\begin{proof}[Proof of Proposition \ref{A1}]
First note that,  
for a function $h:\mu\in\mathcal{P}_2(\R^{d_x})\mapsto h(\mu)\in\R^l$, where the dimension $l$ can be any positive integer such as $1$ or $d_x$, having a linear functional derivative $\frac{\delta}{\delta \mu}h(\mu)(\wt{x})$ which is also continuously differentiable in $\wt{x}$ and satisfies the linear growth of $\displaystyle\sup_{\mu\in\mathcal{P}_2(\R^{d_x})}\Big|\p_{\wt{x}} \frac{\delta}{\delta \mu}h(\mu)(\wt{x})\Big|\leq l_1+l_2|\wt{x}|$ for some positive constants $l_1$ and $l_2$, and for any $\mu,\,\mu'\in\mathcal{P}_2(\R^{d_x})$ such that $\pi\in \Pi(\mu,\mu')$, the space of all joint measures such that the respective marginal measures are $\mu$ and $\mu'$, we have
\small\begin{align*}\no
&|h(\mu)-h(\mu')|\\\no
=&\ \left|\int_0^1 \int_{\R^{d_x}} \frac{\delta}{\delta \mu}h(\theta\mu+(1-\theta)\mu')(\wt{x})\,d(\mu-\mu')(\wt{x})d\theta\right|\\\no
=&\ \left|\int_0^1 \left(\int_{\R^{d_x}\times \R^{d_x}} \frac{\delta}{\delta \mu}h(\theta\mu+(1-\theta)\mu')(\wt{x})\,d\pi(\wt{x},\wt{x}')-\int_{\R^{d_x}\times \R^{d_x}} \frac{\delta}{\delta \mu}h(\theta\mu+(1-\theta)\mu')(\wt{x}')\,d\pi(\wt{x},\wt{x}')\right)d\theta \right|\\\no
=&\ \left|\int_0^1 \int_0^1  \int_{\R^{d_x}\times \R^{d_x}} \p_{\wt{x}} \frac{\delta}{\delta \mu}h(\theta\mu+(1-\theta)\mu')(\lambda \wt{x}+(1-\lambda)\wt{x}')\cdot (\wt{x}-\wt{x}')\,d\pi(\wt{x},\wt{x}') \,d\lambda\, d\theta\right|\\
\leq &\ \int_0^1 \int_0^1  \int_{\R^{d_x}\times \R^{d_x}} \Big(l_1+l_2\big|\lambda \wt{x}+(1-\lambda)\wt{x}'\big|\Big)\big|\wt{x}-\wt{x}'\big|\,d\pi(\wt{x},\wt{x}') \,d\lambda\, d\theta\\
\leq &\ l_1 \int_{\R^{d_x}\times \R^{d_x}} \big|\wt{x}-\wt{x}'\big|\,d\pi(\wt{x},\wt{x}')  +l_2 \int_{\R^{d_x}\times \R^{d_x}} \Big(\big|\wt{x}\big|+\big|\wt{x}'\big|\Big)\big|\wt{x}-\wt{x}'\big|\,d\pi(\wt{x},\wt{x}') \\
\leq &\  \ l_1\int_{\R^{d_x}\times \R^{d_x}} \big|\wt{x}-\wt{x}'\big|\,d\pi(\wt{x},\wt{x}') +l_2\left(\|\mu\|_2+\|\mu'\|_2 \right) \left(  \int_{\R^{d_x}\times \R^{d_x}} \left|\wt{x}-\wt{x}'\right|^2\,d\pi(\wt{x},\wt{x}')  \right)^{1/2},
\end{align*} \normalsize
which implies that, by taking infimum over all possible $\pi\in \Pi(\mu,\mu')$, 
\begin{align}\label{eq_9_2}
|h(\mu)-h(\mu')|\leq  l_1W_1(\mu,\mu')+l_2\left( \|\mu\|_2+\|\mu'\|_2 \right) W_2(\mu,\mu').
\end{align}
Particularly, when $l_2=0$, by the same argument as above, we also obtain
\begin{align}\label{eq_9_3}
|h(\mu)-h(\mu')|\leq  l_1  W_1(\mu,\mu').
\end{align}

With $h=f$, under Assumption \eqref{bdd_d1_f}, $l_2=0$ for this $f$, and also using \eqref{eq_9_3}, we have
\begin{align*}
&\left|f(x,\mu,\alpha) - f(x',\mu',\alpha')\right|\\
\leq &\ \left|f(x,\mu,\alpha) - f(x,\mu',\alpha)\right|+\left|f(x,\mu',\alpha) - f(x',\mu',\alpha)\right|+\left|f(x',\mu',\alpha) - f(x',\mu',\alpha')\right|\\
\leq &\  \Lambda_f \big(|x-x'|+|\alpha-\alpha'|+W_1(\mu,\mu')\big);
\end{align*}
and by picking $(x',\mu',\alpha')=(0,\delta_0,0)$, we further have
\begin{align*}
|f(x,\mu,\alpha)|\leq L_f \left(1+|x|+ \|\mu\|_1+|\alpha| \right),
\end{align*}
where $L_f:=\max\{\Lambda_f,f(0,\delta_0,0)\}$.
\end{proof}

\begin{proof}[Proof of Proposition \ref{A2}]
By using \eqref{bdd_d2_g_1}, \eqref{bdd_d2_g_2} and \eqref{eq_9_3}, we also have
\begin{align*}
&\left|\p_\mu g(x,\mu,\alpha)(\wt{x}) - \p_\mu g(x',\mu',\alpha')(\wt{x}')\right|\\
\leq &\ \left|\p_\mu g(x,\mu,\alpha)(\wt{x}) - \p_\mu g(x,\mu',\alpha)(\wt{x})\right|+\left|\p_\mu g(x,\mu',\alpha)(\wt{x}) - \p_\mu g(x',\mu',\alpha)(\wt{x})\right|\\
&+\left|\p_\mu g(x',\mu',\alpha)(\wt{x}) - \p_\mu g(x',\mu',\alpha')(\wt{x})\right|+\left|\p_\mu g(x',\mu',\alpha')(\wt{x}) - \p_\mu g(x',\mu',\alpha')(\wt{x}')\right|\\
\leq &\   \Lambda_g W_1(\mu,\mu')+\wb{l}_g|x-x'|+\wb{l}_g|\alpha-\alpha'|+\Lambda_g|\wt{x}-\wt{x}'|\\
\leq &\   \max\{\Lambda_g,\wb{l}_g\} \left(W_1(\mu,\mu')+|x-x'|+|\alpha-\alpha'|+|\wt{x}-\wt{x}'|\right);
\end{align*}
and now setting $(x',\mu',\alpha',\wt{x}')=(0,\delta_0,0,0)$, we then have
\begin{align}\label{pmugp}
|\p_\mu g(x,\mu,\alpha)(\wt{x})|\leq \max\Big\{\p_\mu g(0,\delta_0,0)(0),\Lambda_g,\wb{l}_g\Big\}\cdot \left(1+|x|+ \|\mu\|_1+|\alpha|  +|\wt{x}|\right).
\end{align}
By the same argument as above, we have,
\begin{align*}
&\left|\p_x g(x,\mu,\alpha) - \p_x g(x',\mu',\alpha')\right|\leq \max\{\Lambda_g,\wb{l}_g\} \left(W_1(\mu,\mu')+|x-x'|+|\alpha-\alpha'|\right),\\
&\left|\p_\alpha g(x,\mu,\alpha) - \p_\alpha g(x',\mu',\alpha')\right|\leq \max\{\Lambda_g,\wb{l}_g\} \left(W_1(\mu,\mu')+|x-x'|+|\alpha-\alpha'|\right);
\end{align*}
\begin{align*}
&|\p_x g(x,\mu,\alpha)|,\,|\p_\alpha g(x,\mu,\alpha)|\leq \max\Big\{\p_x g(0,\delta_0,0),\p_\alpha g(0,\delta_0,0),\Lambda_g,\wb{l}_g\Big\}\cdot \left(1+|x|+ \|\mu\|_1+|\alpha|  \right).
\end{align*}
Therefore, we have \eqref{LGgDerivatives}.
Similarly, by using \eqref{bdd_d2_k_1} and \eqref{eq_9_3}, we can deduce \eqref{LGkDerivatives}.
\end{proof}

\begin{proof}[Proof of Proposition \ref{h3}]
It follows from \eqref{positive_g_alpha} that $\xi^\top\p_\alpha\p_\alpha g(x,\mu,\alpha)\xi\geq \lambda_g|\xi|^2>0$, so, by the Taylor series expansion, we have
\begin{align*}
g(x,\mu,\alpha)=&\ g(x,\mu,0)+\p_\alpha g(x,\mu,0)\alpha+\int_0^1\int_0^1 \lambda\alpha^\top \p_\alpha\p_\alpha g(x,\mu,\theta\lambda\alpha)\alpha d\theta d\lambda\\
\geq &\frac{1}{2}\lambda_g|\alpha|^2-|g(x,\mu,0)|-|\p_\alpha g(x,\mu,0)||\alpha|.
\end{align*}
Recall \eqref{ligf} that $|f(x,\mu,\alpha)|\leq L_f \left(1+|x|+ \|\mu\|_1+|\alpha| \right)$. Therefore, for each fixed $(x,\mu,z)\in\R^{d_x}\times \mc{P}_2(\R^{d_x})\times \R^{d_x}$,
\begin{align*}
&\lim_{|\alpha|\to \infty}\Big(f(x,\mu,\alpha)\cdot z+g(x,\mu,\alpha)\Big)\\
\geq &\lim_{|\alpha|\to \infty}\Big(\frac{1}{2}\lambda_g|\alpha|^2-|g(x,\mu,0)|-|\p_\alpha g(x,\mu,0)||\alpha|-L_f \left(1+|x|+ \|\mu\|_1+|\alpha| \right)|z|\Big)\to \infty.
\end{align*} 
\end{proof}

\begin{proof}[Proof of Proposition \ref{A5}]
By using \eqref{positive_g_alpha}, \eqref{bdd_d2_g_1} and \eqref{p_aa_f}, we have, for each $(x,\mu,z)\in c_{k_0}$ and any $\xi\in\R^{d_\alpha}$, 
\small\begin{align*}
\Big(\Lambda_g+\frac{1}{20}\lambda_g\Big)|\xi|^2\geq\xi^\top\Big(\p_\alpha\p_\alpha f(x,\mu,\alpha)\cdot z+\p_\alpha\p_\alpha g(x,\mu,\alpha)\Big)\xi \geq \frac{19}{20}\lambda_g|\xi|^2,
\end{align*}\normalsize
that is, the function $f(x,\mu,\alpha)\cdot z+g(x,\mu,\alpha)$ is convex in $\alpha$. Since $\displaystyle\lim_{|\alpha|\to \infty}f(x,\mu,\alpha)\cdot z+g(x,\mu,\alpha)\to \infty$, the coercive nature of the function $f(x,\mu,\alpha)\cdot z  + g(x,\mu,\alpha)$ guarantees a unique interior minimizer $\wh{\alpha}(x,\mu,z)\in\R^{d_\alpha}$ for each $(x,\mu,z)\in c_{k_0}$, which satisfies the first order condition $\p_\alpha f(x,\mu,\wh{\alpha})\cdot z  + \p_\alpha g(x,\mu,\wh{\alpha})=0$. By Theorem \ref{GIFT}, the generalized Implicit Function Theorem (catering the presence of the measure argument), to be shown in the rest of this appendix section, we have
\begin{align*}
\partial_x \wh{\alpha}(x,\mu,z)=&\ -\left(\p_\alpha\p_\alpha f(x,\mu,\wh{\alpha})\cdot z+\p_\alpha\p_\alpha g(x,\mu,\wh{\alpha})\right)^{-1} \left(\p_x\p_\alpha f(x,\mu,\wh{\alpha})\cdot z+\p_x\p_\alpha g(x,\mu,\wh{\alpha})\right);\\
\partial_z \wh{\alpha}(x,\mu,z)=&\ -\left(\p_\alpha\p_\alpha f(x,\mu,\wh{\alpha})\cdot z+\p_\alpha\p_\alpha g(x,\mu,\wh{\alpha})\right)^{-1} \p_\alpha f(x,\mu,\wh{\alpha});\\
\p_\mu \wh{\alpha}(x,\mu,z)(\wt{x})=&\ -\left(\p_\alpha\p_\alpha f(x,\mu,\wh{\alpha})\cdot z+\p_\alpha\p_\alpha g(x,\mu,\wh{\alpha})\right)^{-1}\left(\p_\mu\p_\alpha f(x,\mu,\wh{\alpha})(\wt{x})\cdot z+\p_\mu\p_\alpha g(x,\mu,\wh{\alpha})(\wt{x})\right).
\end{align*}
In addition, by using \eqref{bdd_d1_f}, \eqref{bdd_d2_f}, \eqref{bdd_d2_g_2} and \eqref{positive_h}, we have
\begin{align*}
&\Big\|\p_{x}\wh{\alpha}(x,\mu,z)\Big\|_{\mathcal{L}(\R^{d_x};\R^{d_\alpha})}\leq \frac{20\big(\wb{l}_g+\frac{1}{2}k_0\cdot\wb{l}_f\big) }{19\lambda_g},\ \Big\|\p_{\mu}\wh{\alpha}(x,\mu,z)(\wt{x})\Big\|_{\mathcal{L}(\R^{d_x};\R^{d_\alpha})}\leq \frac{20\big(\wb{l}_g+\frac{1}{2}k_0\cdot\wb{l}_f\big) }{19\lambda_g},\\
&\Big\|\p_{z}\wh{\alpha}(x,\mu,z)\Big\|_{\mathcal{L}(\R^{d_x};\R^{d_\alpha})}\leq \frac{20\Lambda_f}{19\lambda_g}.
\end{align*}\normalsize
Furthermore, since, for all $\theta\in[0,1]$, we have $\theta (x,\mu,z)+(1-\theta)(x',\mu',z')\in c_{k_0}$ and $\Big|\frac{d}{d\theta}\wh \alpha\Big(\theta x+(1-\theta)x',\theta\mu+(1-\theta)\mu',\theta z+(1-\theta) z'\Big)\Big|\leq L_\alpha (|x-x'|+W_1(\mu,\mu')+|z-z'|)$, thus, we obtain \eqref{lipalpha_old}; particularly, since, for any $(x,\mu,z)\in c_{k_0}$, any $\mu',z'$ such that $|z'|\leq \frac{1}{2}k_0(1+\|\mu'\|_1)$ and all $\theta\in[0,1]$, we have $\theta (x,\mu,z)+(1-\theta)(0,\mu',z')\in c_{k_0}$ and then we obtain \eqref{linalpha} by taking $x'=0$, $\mu'=\delta_0$ and $z'=0$ in \eqref{lipalpha_old}. 
In addition, for any $x,\,x'\in\R^{d_x}$, $\mu,\,\mu'\in\mc{P}_2(\R^{d_x})$ and all $\theta\in[0,1]$, we have $\Big(\theta x+(1-\theta)x',\theta\mu+(1-\theta)\mu',\Gamma\big(\theta x+(1-\theta)x'+\theta\mu+(1-\theta)\mu'\big)\Big)\in c_{k_0}$ and $\Big|\frac{d}{d\theta}\wh \alpha\Big(\theta x+(1-\theta)x',\theta\mu+(1-\theta)\mu',\Gamma\big(\theta x+(1-\theta)x'+\theta\mu+(1-\theta)\mu'\big)\Big)\Big|\leq L_\alpha(1+L_\Gamma)(|x-x'|+W_1(\mu,\mu'))$; and thus \eqref{lipalpha} is valid.
Moreover, by Assumption ${\bf(a1)}(iv)$ and since the second order derivatives of $g$ are jointly Lipschitz continuous in their corresponding arguments, we can deduce that the first order derivatives of $\wh{\alpha}$ are jointly Lipschitz continuous in their corresponding arguments.
\end{proof}

\begin{proof}[Proof of Proposition \ref{A6}]
By using \eqref{bdd_d2_f} and  \eqref{bdd_d2_g_1}, we have, for all $(x,\mu,z)\in c_{k_0}$ and any $\xi\in\R^{d_x}$, 
\small\begin{align*}
\Big(\Lambda_g+\frac{1}{2}k_0\cdot \wb{l}_f\Big)|\xi|^2\geq\xi^\top\Big(\p_\alpha\p_\alpha f(x,\mu,\alpha)\cdot z+\p_\alpha\p_\alpha g(x,\mu,\alpha)\Big)\xi,
\end{align*}\normalsize
which further implies that
\begin{align*}
&\xi^\top\Big(\sum_{i=1}^{d_\alpha}\p_z \wh{\alpha}_i(x,\mu,z)\otimes \p_{\alpha_i} f(x,\mu,\wh{\alpha})\Big)\xi\\
=&\ -\xi^\top\Big(\p_\alpha\p_\alpha f(x,\mu,\wh{\alpha})\cdot z+\p_\alpha\p_\alpha g(x,\mu,\wh{\alpha})\Big)^{-1}\Big(\sum_{i=1}^{d_\alpha}\p_{\alpha_i} f(x,\mu,\wh{\alpha})\otimes \p_{\alpha_i} f(x,\mu,\wh{\alpha})\Big)\xi\\
\leq&\  -\frac{\lambda_f^2}{\Lambda_g+\frac{1}{2}k_0\cdot \wb{l}_f}|\xi|^2,
\end{align*}\normalsize
by using \eqref{positive_f} in the last inequality.
\end{proof}

\begin{theorem}\label{GIFT}(Implicit function theorem with a measure argument)
Let $\mathcal{D}$ be an open set in $\R^{d_x}\times \mathcal{P}_2(\R^{d_x})\times \R^{d_x}$ and $\wt{q}:(x,\mu,z,\alpha)\in \mathcal{D}\times\R^{d_\alpha}\mapsto \wt{q}(x,\mu,z,\alpha)\in\R^{d_\alpha}$ be jointly continuous in $(x,\mu,z,\alpha)$ and partially differentiable in $\alpha$ for each fixed $(x,\mu,z)\in\mathcal{D}$, such that for some $(x_0,\mu_0,z_0)\in \mathcal{D}$ and an $\alpha_0\in\R^{d_\alpha}$, we have $\wt{q}(x_0,\mu_0,z_0,\alpha_0)=0$. Assume that\\
(i) $\partial_\alpha \wt{q}(x,\mu,z,\alpha)$ is jointly continuous in $(x,\mu,z,\alpha)$ and is also Lipschitz continuous in $\alpha$ uniformly in $(x,\mu,z)\in\mathcal{D}$;\\
(ii) $\partial_\alpha \wt{q}(x_0,\mu_0,z_0,\alpha_0)$ is an invertible $\R^{d_\alpha}\times\R^{d_\alpha}$ matrix.\\ 
Then there is an open neighborhood $\mathcal{D}_0$ of $(x_0,\mu_0,z_0)$ in $\mathcal{D}$ such that there is a unique jointly continuous function $\alpha:(x,\mu,z)\in \mathcal{D}_0\mapsto \alpha(x,\mu,z)\in \R^{d_\alpha}$ satisfying $\alpha(x_0,\mu_0,z_0)=\alpha_0$ and $\wt{q}(x,\mu,z,\alpha(x,\mu,z))=0$ for all $(x,\mu,z)\in \mathcal{D}_0$. Moreover, if $\wt{q}$ is differentiable in $(x,z)$ for each $\mu$ and $\alpha$, it is also $L$-differentiable in $\mu$ for each $(x,z,\alpha)$, and its derivatives are also jointly continuous in $(x,\mu,z,\alpha)$, then $\alpha$ is differentiable in $(x,z)$ for every $\mu$, it is also $L$-differentiable in $\mu$ for every $(x,z)$, and its derivatives are:
\begin{align*}
\partial_x \alpha(x,\mu,z)=&\ -\left(\partial_\alpha \wt{q}(x,\mu ,z, \alpha(x,\mu,z))\right)^{-1} \p_x \wt{q}(x,\mu ,z, \alpha(x,\mu,z));\\
\partial_z \alpha(x,\mu,z)=&\ -\left(\partial_\alpha \wt{q}(x,\mu ,z, \alpha(x,\mu,z))\right)^{-1} \p_z \wt{q}(x,\mu ,z, \alpha(x,\mu,z));\\
\p_\mu \alpha(x,\mu,z)(\wt{x})=&\ -\left(\partial_\alpha \wt{q}(x,\mu ,z, \alpha(x,\mu,z))\right)^{-1}\p_\mu \wt{q}(x,\mu ,z,\alpha(x,\mu,z))(\wt{x});
\end{align*}
moreover, all of these are jointly continuous in $(x,\mu,z)\in \mathcal{D}_0$. 
\end{theorem}

\begin{proof}
Since $\partial_\alpha \wt{q}(x_0,\mu_0,z_0,\alpha_0)$ is an invertible $\R^{d_\alpha}\times\R^{d_\alpha}$ matrix, $\wt{q}(x_0,\mu_0,z_0,\alpha_0)=0$ and both $\wt{q}(x,\mu,z,\alpha_0)$, $\partial_\alpha \wt{q}(x,\mu,z,\alpha_0)$ are continuous in$(x,\mu,z)$, then, for any small $\eps>0$, there exists a neighborhood $\wt{\mathcal{D}}_0(\eps)$ of $(x_0,\mu_0,z_0)$ in $\mathcal{D}$ such that $\partial_\alpha \wt{q}(x,\mu,z,\alpha_0)$ is also an invertible $\R^{d_\alpha}\times\R^{d_\alpha}$ matrix for every $(x,\mu,z)\in\wt{\mathcal{D}}_0(\eps)$ such that $\|\partial_\alpha \wt{q}(x,\mu,z,\alpha_0)^{-1}\|\leq \|\partial_\alpha \wt{q}(x_0,\mu_0,z_0,\alpha_0)^{-1}\|+\eps$, where the matrix norm $\|\cdot\|$ stands for the operator norm induced by the vector norm $|\cdot|$, and $|\wt{q}(x,\mu,z,\alpha_0)|\leq \eps$. For any $(x_1,\mu_1,z_1)\in\wt{\mathcal{D}}_0(\eps)$, we define 
$$\wh{q}(\alpha):=-\p_\alpha \wt{q}(x_1,\mu_1,z_1,\alpha_0)^{-1}\wt{q}(x_1,\mu_1,z_1,\alpha)+\alpha,$$
and then $\p_\alpha \wh{q}(\alpha_0)=0$. Consider the iteration map $I:\alpha^{(n)}\in\R^{d_\alpha}\mapsto \alpha^{(n+1)}=\wh{q}(\alpha^{(n)})\in\R^{d_\alpha}$ with $\alpha^{(0)}=\alpha_0$. Let $C_{\wt{q}}$ be the Lipschitz constant of $\p_\alpha \wt{q}(x_1,\mu_1,z_1,\alpha)$ with respect to $\alpha$. Define $$\eps_0:=\min\left\{1,\dfrac{1}{8C_{\wt{q}}\left(\|\partial_\alpha \wt{q}(x_0,\mu_0,z_0,\alpha_0)^{-1}\|+1\right)^2}\right\},$$ and suppose that 
\begin{align*}
\left|\alpha^{(n)}-\alpha_0 \right|\leq \fr1{4 C_{\wt{q}} \left(\|\partial_\alpha \wt{q}(x_0,\mu_0,z_0,\alpha_0)^{-1}\|+\eps_0\right)},
\end{align*}
which is clearly valid when $n=0$, then, if $\eps\leq \eps_0$, by induction,
\small\begin{align*}
|\alpha^{(n+1)}-\alpha_0|
=&\ \left|\wh{q}(\alpha^{(n)})-\wh{q}(\alpha_0)+\p_\alpha \wt{q}(x_1,\mu_1,z_1,\alpha_0)^{-1}\wt{q}(x_1,\mu_1,z_1,\alpha_0)\right|\\
\leq&\ \left|\left(1-\p_\alpha \wt{q}(x_1,\mu_1,z_1,\alpha_0)^{-1}\int_0^1\p_\alpha \wt{q}(x_1,\mu_1,z_1,\theta \alpha^{(n)}+(1-\theta)\alpha_0)d\theta\right)\cdot(\alpha^{(n)}-\alpha_0)\right|\\
&+\eps_0(\|\partial_\alpha \wt{q}(x_0,\mu_0,z_0,\alpha_0)^{-1}\|+\eps_0)\\
=&\ \Bigg|\bigg(\p_\alpha \wt{q}(x_1,\mu_1,z_1,\alpha_0)^{-1}\int_0^1\p_\alpha \wt{q}(x_1,\mu_1,z_1,\alpha_0)d\theta\\
&-\p_\alpha \wt{q}(x_1,\mu_1,z_1,\alpha_0)^{-1}\int_0^1\p_\alpha \wt{q}(x_1,\mu_1,z_1,\theta \alpha^{(n)}+(1-\theta)\alpha_0)d\theta\bigg)\cdot(\alpha^{(n)}-\alpha_0)\Bigg|\\
&+\eps_0\left(\|\partial_\alpha \wt{q}(x_0,\mu_0,z_0,\alpha_0)^{-1}\|+\eps_0\right)\\
\leq&\ \left(\|\partial_\alpha \wt{q}(x_0,\mu_0,z_0,\alpha_0)^{-1}\|+\eps_0\right)\left(C_{\wt{q}}\left|\alpha^{(n)}-\alpha_0 \right|^2+\eps_0\right)\\
\leq&\  \fr14\left|\alpha^{(n)}-\alpha_0 \right|+ \eps_0\left(\|\partial_\alpha \wt{q}(x_0,\mu_0,z_0,\alpha_0)^{-1}\|+\eps_0\right)\\
\leq&\  \fr1{4 C_{\wt{q}} \left(\|\partial_\alpha \wt{q}(x_0,\mu_0,z_0,\alpha_0)^{-1}\|+\eps_0\right)},
\end{align*}\normalsize
Thus,
\small\begin{align*}
&\left|\wh{q}(\alpha^{(n+1)})-\wh{q}(\alpha^{(n)})\right|\\
=&\ \left|\left(1-\p_\alpha \wt{q}(x_1,\mu_1,z_1,\alpha_0)^{-1}\int_0^1\p_\alpha \wt{q}(x_1,\mu_1,z_1,\theta \alpha^{(n+1)}+(1-\theta)\alpha^{(n)})d\theta\right)\cdot(\alpha^{(n+1)}-\alpha^{(n)})\right|\\
\leq&\ \left|\p_\alpha \wt{q}(x_1,\mu_1,z_1,\alpha_0)^{-1}\int_0^1\left(\p_\alpha \wt{q}(x_1,\mu_1,z_1,\theta \alpha^{(n+1)}+(1-\theta)\alpha^{(n)})-\p_\alpha \wt{q}(x_1,\mu_1,z_1,\alpha_0)\right)d\theta \right|\cdot\left|\alpha^{(n+1)}-\alpha^{(n)}\right|\\
\leq&\ (\|\partial_\alpha \wt{q}(x_0,\mu_0,z_0,\alpha_0)^{-1}\|+\eps_0)\cdot C_{\wt{q}}\left(\int_0^1 \theta |\alpha^{(n+1)}-\alpha_0|+(1-\theta)|\alpha^{(n)}-\alpha_0|d\theta \right)\cdot\left|\alpha^{(n+1)}-\alpha^{(n)}\right|\\
\leq&\  (\|\partial_\alpha \wt{q}(x_0,\mu_0,z_0,\alpha_0)^{-1}\|+\eps_0)\cdot C_{\wt{q}}\cdot\fr1{4 C_{\wt{q}} \left(\|\partial_\alpha \wt{q}(x_0,\mu_0,z_0,\alpha_0)^{-1}\|+\eps_0\right)}\left|\alpha^{(n+1)}-\alpha^{(n)}\right|,\\
=&\ \fr14 \left|\alpha^{(n+1)}-\alpha^{(n)}\right|,
\end{align*}\normalsize
which implies the map $I$ is contractive. Therefore, by Banach Fixed Point Theorem, for any $(x_1,\mu_1,z_1)\\ \in\mathcal{D}_0:=\wt{\mathcal{D}}_0(\eps_0)$, there exists a unique $\alpha_1=\alpha(x_1,\mu_1,z_1)$ such that $\wh{q}(\alpha_1)=\alpha_1$, that is, $\wt{q}(x_1,\mu_1,z_1,\alpha_1)=0$. Moreover, there exists a unique jointly continuous function $\alpha(x,\mu,z)$, and for every $(x,\mu,z)\in\mathcal{D}_0$, $\wt{q}(x_,\mu_,z_,\alpha(x,\mu,z))=0$ since $\wt{q}$ itself is jointly continuous in $(x,\mu,z,\alpha)$.

If $\wt{q}$ is $L$-differentiable in $\mu$, then, by Definition \ref{def_L_diff}, its lifted version $\wt{Q}(x,X,z,\alpha):=\wt{q}(x,\mu=\mathbb{L}_X,z,\alpha)$ is Fr\'echet differentiable in $X\in\mathcal{H}^{d_x}$. Thus, for any random variable $X\in\mathcal{H}^{d_x}$ such that $\mathbb{L}_X=\mu_0$ and any sequence of random variables $Y_n\in\mathcal{H}^{d_x}$ with $\|Y_n\|_{\mathcal{H}^{d_x}}\to 0$ as $n \to \infty$ and $(x_0,\mu_n:=\mathbb{L}_{X+Y_n},z_0)\in\mathcal{D}_0$,
\begin{align}\no 
&\wt{Q}(x_0,X+Y_n,z_0,\alpha(x_0,\mu_n,z_0))-\wt{Q}(x_0,X,z_0,\alpha(x_0,\mu_0,z_0))\\\no
&-D_X\wt{Q}(x_0,X,z_0,\alpha(x_0,\mu_0,z_0))(Y_n)\\\no
&-\int_0^1 \partial_\alpha \wt{q}(x_0,\mu_n,z_0,\theta \alpha(x_0,\mu_n,z_0)+(1-\theta)\alpha(x_0,\mu_0,z_0))d\theta\cdot \left(\alpha(x_0,\mu_n,z_0)- \alpha(x_0,\mu_0,z_0)\right)\\\label{12_11_new}
&=o(\|Y_n\|_{\mathcal{H}^{d_x}}),
\end{align}
where the small-$o$ function $o(\|Y_n\|_{\mathcal{H}^{d_x}})$ means that the term tends to $0$ as $\|Y_n\|_{\mathcal{H}^{d_x}}$ goes to $0$. Then \eqref{12_11_new} implies that
\footnotesize\begin{align*} 
&\alpha(x_0,\mu_n,z_0)- \alpha(x_0,\mu_0,z_0)\\
=&\ -\left(\int_0^1 \partial_\alpha \wt{q}(x_0,\mu_n,z_0,\theta \alpha(x_0,\mu_n,z_0)+(1-\theta)\alpha(x_0,\mu_0,z_0))d\theta\right)^{-1}D_X\wt{Q}(x_0,X,z_0,\alpha(x_0,\mu_0,z_0))(Y_n)-o(\|Y_n\|_{\mathcal{H}^{d}})\\
=&\ -\left(\int_0^1 \partial_\alpha \wt{q}(x_0,\mu_n,z_0,\theta \alpha(x_0,\mu_n,z_0)+(1-\theta)\alpha(x_0,\mu_0,z_0))d\theta\right)^{-1}\int_{\R^{d_x}}\p_\mu\wt{q}(x_0,\mu_0,z_0,\alpha(x_0,\mu_0,z_0))\big(X(\omega)\big)\cdot Y_n(\omega)d\mathbb{P}(\omega)\\
&-o(\|Y_n\|_{\mathcal{H}^{d}}).
\end{align*}\normalsize
Therefore, $\alpha(x_0,\mu,z_0)$ is $L$-differentiable in $\mu$ at $\mu_0$ and
\begin{align}
\p_\mu \alpha(x,\mu,z)(\wt{x})=-\left(\partial_\alpha \wt{q}(x,\mu ,z, \alpha(x,\mu,z))\right)^{-1}\p_\mu \wt{q}(x,\mu ,z,\alpha(x,\mu,z))(\wt{x}),
\end{align}
which is jointly continuous in $(x,\mu,z)$ since $\p_\mu \wt{q}$ is jointly continuous in $(x,\mu,z,\alpha)$ and $\p_\alpha \wt{q}$ is jointly continuous in $(x,\mu,z,\alpha)$ and $\alpha$ is jointly continuous in $(x,\mu,z)$.

If $\wt{q}$ is partially differentiable in $x$, then, for any $x\in\R^{d_x}$ such that $(x,\mu_0,z_0)\in\mathcal{D}_0$,
\begin{align*} 
0=&\ \wt{q}(x,\mu_0,z_0,\alpha(x,\mu_0,z_0))-\wt{q}(x_0,\mu_0,z_0,\alpha(x_0,\mu_0,z_0))\\
=&\ \int_0^1 \p_x \wt{q}(\theta x+(1-\theta)x_0,\mu_0,z_0,\alpha(x_0,\mu_0,z_0))d\theta\cdot(x-x_0)\\
&+\int_0^1 \partial_\alpha \wt{q}(x,\mu_0,z_0,\theta \alpha(x,\mu_0,z_0)+(1-\theta)\alpha(x_0,\mu_0,z_0))d\theta\cdot \left(\alpha(x,\mu_0,z_0)- \alpha(x_0,\mu_0,z_0)\right),
\end{align*}
which implies that
\begin{align*} 
&\alpha(x,\mu_0,z_0)- \alpha(x_0,\mu_0,z_0)\\
=-&\left(\int_0^1 \partial_\alpha \wt{q}(x,\mu_0,z_0,\theta \alpha(x,\mu_0,z_0)+(1-\theta)\alpha(x_0,\mu_0,z_0))d\theta\right)^{-1}\\
&\cdot \int_0^1 \p_x \wt{q}(\theta x+(1-\theta)x_0,\mu_0,z_0,\alpha(x_0,\mu_0,z_0))d\theta\cdot(x-x_0).
\end{align*}
Therefore, by Lebesgue's dominated convergence theorem, $\alpha(x,\mu_0,z_0)$ is differentiable in $x$ at $x_0$ so that
\begin{align}
\partial_x \alpha(x,\mu,z)=&\ -\left(\partial_\alpha \wt{q}(x,\mu ,z, \alpha(x,\mu,z))\right)^{-1} \p_x \wt{q}(x,\mu ,z, \alpha(x,\mu,z)),
\end{align}
which is jointly continuous in $(x,\mu,z)$ since $\p_x \wt{q}$ is jointly continuous in $(x,\mu,z,\alpha)$.
By the same method as above, one can also show that
$\alpha(x_0,\mu_0,z)$ is differentiable in $z$ at $z_0$ and
\begin{align}
\partial_z \alpha(x,\mu,z)=&\ -\left(\partial_\alpha \wt{q}(x,\mu ,z, \alpha(x,\mu,z))\right)^{-1} \p_z \wt{q}(x,\mu ,z, \alpha(x,\mu,z)),
\end{align}
which is jointly continuous in $(x,\mu,z)$ since $\p_z  \wt{q}$ is jointly continuous in $(x,\mu,z,\alpha)$. 
\end{proof}

\begin{corollary}
Under the hypotheses in Theorem \ref{GIFT}, we have the following expression:
\begin{align}\no
&\alpha(x_1,\mu_1,z_1)-\alpha(x_0,\mu_0,z_0)\\\no
=&\ \alpha(x_1,\mu_1,z_1)-\alpha(x_0,\mu_1,z_1)+\alpha(x_0,\mu_1,z_1)-\alpha(x_0,\mu_1,z_0)+\alpha(x_0,\mu_1,z_0)-\alpha(x_0,\mu_0,z_0)\\\no
=&\ \int_0^1 \p_x \alpha(\theta x_1+(1-\theta)x_0, \mu_1, z_1)d\theta\cdot(x_1-x_0)\\\no
&+\int_0^1 \p_z \alpha(x_0,\mu_1,\theta z_1+(1-\theta)z_0)d\theta\cdot(z_1-z_0)\\\label{expansion}
&+\int_0^1 \mathbb{E}\left[ \p_\mu \alpha(x_0,\mathbb{L}_{\theta \xi_1+(1-\theta)\xi_0},z_0)(\theta \xi_1+(1-\theta) \xi_0)\cdot(\xi_1-\xi_0) \right]d\theta,
\end{align}
where $\xi_0$ and $\xi_1$ are independent random variables such that the laws $\mathbb{L}_{\xi_i}=\mu_i,\, i=0,1$.
\end{corollary}
\begin{proof}
The validity of \eqref{expansion} is a consequence of the jointly continuous differentiability, see Section 5.3.4 in \cite{carmona2018probabilistic} for details.
\end{proof}

\subsection{Proof of Theorem \ref{Verification}}\label{proofVerification}
Also see \cite{bensoussan2015master} for the second order case. For an admissible feedback control $\alpha\in\mathbb{A}$, we solve \eqref{SDE1}-\eqref{SDE2} subject to the control $\alpha$, and solve for $(x^{t,\xi,\alpha}_s,x^{t,x,m,\alpha}_s)$, $s\in [t,T]$. Define $X_s^{t,m,\alpha}:x\in\R^{d_x}\mapsto X_s^{t,m,\alpha}(x):=x^{t,x,m,\alpha}_s\in\R^{d_x}$ and $X_s^{t,m,\alpha}\in L^2_m$ since $\sup_{s\in[t,T]}\mathbb{E}[|x^{t,\xi,\alpha}_s|^2]< \infty$ which was described in the paragraph after \eqref{SDE1}-\eqref{SDE2}. 

We consider, by using \eqref{bellman},
\begin{align*}
&\frac{d}{ds}v(s,X_s^{t,m,\alpha}\ot m)\\
=&\ \p_sv(s,X_s^{t,m,\alpha}\ot m) + \int_{\R^{d_x}} \p_\mu v (s,X_s^{t,m,\alpha}\ot m) (y)\cdot f(y,X_s^{t,m,\alpha}\ot m,\alpha(s,y,X_s^{t,m,\alpha}\ot m))d (X_s^{t,m,\alpha}\ot m)(y)\\
=&\  - \int_{\R^{d_x}} \displaystyle\min_{\wt{\alpha}\in \R^{d_\alpha}} \left(f(y, X_s^{t,m,\alpha}\ot m, \wt{\alpha})\cdot \p_\mu v (s,X_s^{t,m,\alpha}\ot m)(y)  + g(y, X_s^{t,m,\alpha}\ot m, \wt{\alpha})\right)  d(X_s^{t,m,\alpha}\ot m)(y) \\
&+ \int_{\R^{d_x}} \p_\mu v (s,X_s^{t,m,\alpha}\ot m) (x)\cdot f(y,X_s^{t,m,\alpha}\ot m,\alpha(s,y,X_s^{t,m,\alpha}\ot m))d (X_s^{t,m,\alpha}\ot m)(y)\\
\geq &-\int_{\R^{d_x}} g(y,X_s^{t,m,\alpha}\ot m,\alpha(s,y,X_s^{t,m,\alpha}\ot m))d (X_s^{t,m,\alpha}\ot m)(y).
\end{align*}
Integrating on both sides and rearranging the terms, it yields
\begin{align*}
v(t,m) &\leq \int_{\R^{d_x}} k(y,X_T^{t,m,\alpha}\ot m)d(X_T^{t,m,\alpha}\ot m)(y) + \int_t^T \int_{\R^{d_x}} g(y,X_s^{t,m,\alpha}\ot m,\alpha(s,y,X_s^{t,m,\alpha}\ot m))d (X_s^{t,m,\alpha}\ot m)(y)\\
&= j(t,m,\alpha),
\end{align*}
therefore, $v(t,m)\leq \min_{\alpha} j(t,m,\alpha)$. In order to show $v(t,m)=\min_{\alpha} j(t,m,\alpha)$, it suffices to find an admissible feedback control $\wb{\alpha}(s,x,\mu)$ such that $v(t,m) = j(t,m,\wb{\alpha})$. Define the feedback optimal control $\wb{\alpha}(s,x,\mu)$ by
\begin{align*}
\wb{\alpha}(s,x,\mu):= \wh{\alpha}\left(x,\mu,\p_\mu v(s,\mu)(x)\right),
\end{align*}
and let $(x^{t,\xi,\wb{\alpha}}_s,x^{t,x,m,\wb{\alpha}}_s)$ be the pair of solutions to 
\begin{align}
x^{t,\xi,\wb{\alpha}}_s =&\  \xi+\int_t^s f\left(x^{t,\xi,\wb{\alpha}}_\tau,\mathbb{L}_{x^{t,\xi,\wb{\alpha}}_\tau},\wh{\alpha}\left(x^{t,\xi,\wb{\alpha}}_\tau,\mathbb{L}_{x^{t,\xi,\wb{\alpha}}_\tau},\p_\mu v(s,\mathbb{L}_{x^{t,\xi,\wb{\alpha}}_\tau})(x^{t,\xi,\wb{\alpha}}_\tau)\right)\right)d\tau,\ s\in[t,T],\\
x^{t,x,m,\wb{\alpha}}_s =&\  x+\int_t^s f\left(x^{t,x,m,\wb{\alpha}}_\tau,\mathbb{L}_{x^{t,\xi,\wb{\alpha}}_\tau},\wh{\alpha}\left(x^{t,x,m,\wb{\alpha}}_\tau,\mathbb{L}_{x^{t,\xi,\wb{\alpha}}_\tau},\p_\mu v(s,\mathbb{L}_{x^{t,\xi,\wb{\alpha}}_\tau})(x^{t,x,m,\wb{\alpha}}_\tau)\right)\right)d\tau,\ s\in[t,T],
\end{align}
and $X_s^{t,m,\wb{\alpha}}:x\in\R^{d_x}\mapsto X_s^{t,m,\wb{\alpha}}(x):=x^{t,x,m,\wb{\alpha}}_s\in\R^{d_x}$ so that $X_s^{t,m,\wb{\alpha}}\in L^2_m$.
Then,
\begin{align*}
&\frac{d}{ds}v(s,X_s^{t,m,\wb{\alpha}}\ot m)\\
=&\ \p_sv(s,X_s^{t,m,\wb{\alpha}}\ot m) \\
&+ \int_{\R^{d_x}} \p_\mu v (s,X_s^{t,m,\wb{\alpha}}\ot m) (y)\cdot f(y,X_s^{t,m,\wb{\alpha}}\ot m,\wh{\alpha}(y,X_s^{t,m,\wb{\alpha}}\ot m,\p_\mu v(s,X_s^{t,m,\wb{\alpha}}\ot m)(y)))d (X_s^{t,m,\wb{\alpha}}\ot m)(y)\\
=&\  - \int_{\R^{d_x}} \displaystyle\min_{ \alpha \in \R^{d_\alpha}} \left(f(y, X_s^{t,m,\wb{\alpha}}\ot m,  \alpha )\cdot \p_\mu v (s,X_s^{t,m,\wb{\alpha}}\ot m)(y)  + g(y, X_s^{t,m,\wb{\alpha}}\ot m,  \alpha )\right)  d(X_s^{t,m,\wb{\alpha}}\ot m)(y) \\
&+ \int_{\R^{d_x}} \p_\mu v (s,X_s^{t,m,\wb{\alpha}}\ot m) (y)\cdot f(y,X_s^{t,m,\wb{\alpha}}\ot m,\wh{\alpha}(y,X_s^{t,m,\wb{\alpha}}\ot m,\p_\mu v(s,X_s^{t,m,\wb{\alpha}}\ot m)(y)))d (X_s^{t,m,\wb{\alpha}}\ot m)(y)\\
=&\  -\int_{\R^{d_x}} g(y,X_s^{t,m,\wb{\alpha}}\ot m,\wh{\alpha}(y,X_s^{t,m,\wb{\alpha}}\ot m,\p_\mu v(s,X_s^{t,m,\wb{\alpha}}\ot m)(y)))d (X_s^{t,m,\wb{\alpha}}\ot m)(y);
\end{align*}
again integrating both sides yields:
\begin{align*}
v(t,m) &= \int_{\R^{d_x}} k(y,X_T^{t,m,\wb{\alpha}}\ot m)d(X_T^{t,m,\wb{\alpha}}\ot m)(y) + \int_t^T \int_{\R^{d_x}} g(y,X_s^{t,m,\wb{\alpha}}\ot m,\wb{\alpha}(s,y,X_s^{t,m,\wb{\alpha}}\ot m))d (X_s^{t,m,\wb{\alpha}}\ot m)(y) \\
&= j(t,m,\wb{\alpha}) .
\end{align*}
Thus $v(t,m)$ is the value function defined by \eqref{valfun} as expected.

\subsection{Proof of Theorem \ref{MPBE}}\label{App:MPBE}
For simplicity, we make use of the following notations, so that these functions omit their corresponding arguments but only keep the parameters $t,\,m,\,x$ and $s$:
\begin{align}\label{notation_1}
&\begin{cases}
\alpha^{t,m}_s(x):= \alpha(X^{t,m}_s(x), X^{t,m}_s\ot m,Z^{t,m}_s(x)),\ g^{t,m}_s(x):=g(X^{t,m}_s(x),X^{t,m}_s\ot m,\alpha^{t,m}_s(x)),\\
f^{t,m}_s(x):=f(X^{t,m}_s(x),X^{t,m}_s\ot m,\alpha^{t,m}_s(x)),\ k^{t,m}_s(x):=k(X^{t,m}_s(x),X^{t,m}_s\ot m).
\end{cases}\ \ \ \  
\end{align} 
Likewise, we adopt the same style of notations for the derivatives of $\alpha$ and $f$, as shown in the following Table \ref{notation_2}: 
\begin{center}
\centering\scriptsize
\renewcommand\arraystretch{2}
\begin{tabular}{|l|l|l|}
\hline
$\big(\p_x\alpha\big)^{t,m}_s(y)=\p_x \alpha(x,\mu,z)$&
$\big(\p_\mu\alpha\big)^{t,m}_s(y,\wt{y})=\p_\mu \alpha(x,\mu,z)(\wt{x})$&
$\big(\p_z\alpha\big)^{t,m}_s(y)=\p_z \alpha(x,\mu,z)$\\ \hline
$\big(\p_x f\big)^{t,m}_s(y)=\p_x f(x,\mu,\alpha)$&
$\big(\p_\mu f\big)^{t,m}_s(y,\wt{y})=\p_\mu f(x,\mu,\alpha)(\wt{x})$&
$\big(\p_\alpha f\big)^{t,m}_s(y)=\p_\alpha f(x,\mu,\alpha)$\\\hline
$\big(\p_x\p_x f\big)^{t,m}_s(y)=\p_x\p_x f(x,\mu,\alpha)$&
$\big(\p_x\p_\mu f\big)^{t,m}_s(y,\wt{y})=\p_x\p_\mu f(x,\mu,\alpha)(\wt{x})$&
$\big(\p_{\wt{x}}\p_\mu f\big)^{t,m}_s(y,\wt{y})=\p_{\wt{x}}\p_\mu f(x,\mu,\alpha)(\wt{x})$\\\hline
$\big(\p_\alpha\p_\alpha f\big)^{t,m}_s(y)=\p_\alpha\p_\alpha f(x,\mu,\alpha)$&
$\big(\p_\alpha\p_\mu f\big)^{t,m}_s(y,\wt{y})=\p_\alpha\p_\mu f(x,\mu,\alpha)(\wt{x})$&
$\big(\p_x\p_\alpha f\big)^{t,m}_s(y)=\p_x\p_\alpha f(x,\mu,\alpha)$\\\hline
\multicolumn{3}{|c|}{$\big(\p_\mu\p_\mu f\big)^{t,m}_s(y,\wt{y},\wh{y})=\p_\mu\p_\mu f(x,\mu,\alpha)(\wt{x},\wh{x})$} \\\hline
\end{tabular}
\captionof{table}{\scriptsize Abbreviations for the derivatives of $\alpha$ and $f$, where $y,\,\wt{y},\,\wh{y}\in\R^{d_x}$, $t\in[0,T]$, $m\in\mathcal{P}_2(\R^{d_x})$, $s\in[t,T]$ and we also evaluated all of them at $(x,\mu,z,\alpha,\wt{x},\wh{x})=(X^{t,m}_s(y), X^{t,m}_s\ot m,Z^{t,m}_s(y),\alpha^{t,m}_s(y),X^{t,m}_s(\wt{y}),X^{t,m}_s(\wh{y}))\in \R^{d_x}\times \mathcal{P}_2(\R^{d_x})\times \R^{d_x}\times\R^{d_\alpha}\times\R^{d_x}\times\R^{d_x}$.}
\label{notation_2}
\end{center}
Using these notations and the chain rule, one can write: 
\footnotesize\begin{equation}
\begin{aligned}
\p_t\big(\alpha^{t,m}_s(y)\big)=&\ \big(\p_x  \alpha\big)^{t,m}_s(y)\cdot\p_t\big(X^{t,m}_s(y)\big)+\displaystyle\int_{\R^{d_x}}\big(\p_\mu \alpha\big)^{t,m}_s(y,\wt{y})\cdot \p_t\big(X^{t,m}_s(\wt{y})\big)dm(\wt{y})+\big(\p_z \alpha\big)^{t,m}_s(y)\cdot\p_t\big(Z^{t,m}_s(y)\big);\\
\p_t\big(f^{t,m}_s(y)\big)=&\ \big(\p_x f\big)^{t,m}_s(y)\cdot\p_t\big(X^{t,m}_s(y)\big)+\displaystyle\int_{\R^{d_x}}\big(\p_\mu f\big)^{t,m}_s(y,\wt{y})\cdot\p_t\big(X^{t,m}_s(\wt{y})\big)dm(\wt{y})+ \big(\p_\alpha f\big)^{t,m}_s(y)\cdot\p_t\big(\alpha^{t,m}_s(y)\big);\\
\p_t\big(g^{t,m}_s(y)\big)=&\ \big(\p_x g\big)^{t,m}_s(y)\cdot\p_t\big(X^{t,m}_s(y)\big)+\displaystyle\int_{\R^{d_x}}\big(\p_\mu g\big)^{t,m}_s(y,\wt{y})\cdot\p_t\big(X^{t,m}_s(\wt{y})\big)dm(\wt{y})+\big(\p_\alpha g\big)^{t,m}_s(y)\cdot\p_t\big(\alpha^{t,m}_s(y)\big);\\
\p_t\big(k^{t,m}_s(y)\big)=&\ \big(\p_x k\big)^{t,m}_s(y)\cdot\p_t\big(X^{t,m}_s(y)\big)+\displaystyle\int_{\R^{d_x}}\big(\p_\mu k\big)^{t,m}_s(y,\wt{y})\cdot\p_t\big(X^{t,m}_s(\wt{y})\big)dm(\wt{y});
\end{aligned}\label{calculus1}
\end{equation}\normalsize
similarly, we also have a set of similar equations for $\p_m \big(\alpha^{t,m}_s(y)\big)(\wt{y})$, $\p_m \big(f^{t,m}_s(y)\big)(\wt{y})$, $\p_m \big(g^{t,m}_s(y)\big)(\wt{y})$ and $\p_m \big(k^{t,m}_s(y)\big)(\wt{y})$, and $\p_x \big(\alpha^{t,m}_s(y)\big)$, $\p_x \big(f^{t,m}_s(y)\big)$, $\p_x \big(g^{t,m}_s(y)\big)$, $\p_x \big(k^{t,m}_s(y)\big)$. For instance, $\p_m \big(\alpha^{t,m}_s(y)\big)(\wt{y})$, $\p_m \big(f^{t,m}_s(y)\big)(\wt{y})$, $\p_x \big(\alpha^{t,m}_s(y)\big)$, $\p_x \big(f^{t,m}_s(y)\big)$ can be written as, for any $y,\,\wt{y}\in\R^{d_x}$, $t\in[0,T]$, $m\in\mathcal{P}_2(\R^{d_x})$ and $s\in[t,T]$,
\footnotesize\begin{equation}
\begin{aligned}
\p_m\big(\alpha^{t,m}_s(y)\big)(\wt{y})=&\ \big(\p_x  \alpha\big)^{t,m}_s(y)\cdot\p_m\big(X^{t,m}_s(y)\big)(\wt{y})+\big(\p_\mu \alpha\big)^{t,m}_s(y,\wt{y})\cdot \p_{\wt{y}}\big(X^{t,m}_s(\wt{y})\big)\\
&+\displaystyle\int_{\R^{d_x}}\big(\p_\mu \alpha\big)^{t,m}_s(y,\wh{y})\cdot \p_m\big(X^{t,m}_s(\wh{y})\big)(\wt{y})dm(\wh{y})+\big(\p_z \alpha\big)^{t,m}_s(y)\cdot\p_m\big(Z^{t,m}_s(y)\big)(\wt{y});\\
\p_m\big(f^{t,m}_s(y)\big)(\wt{y})=&\ \big(\p_x f\big)^{t,m}_s(y)\cdot\p_m\big(X^{t,m}_s(y)\big)(\wt{y})+\big(\p_\mu f\big)^{t,m}_s(y,\wt{y})\cdot \p_{\wt{y}}\big(X^{t,m}_s(\wt{y})\big)\\
&+\displaystyle\int_{\R^{d_x}} \big(\p_\mu f\big)^{t,m}_s(y,\wh{y})\cdot\p_m\big(X^{t,m}_s(\wh{y})\big)(\wt{y}) dm(\wh{y})+ \big(\p_\alpha f\big)^{t,m}_s(y)\cdot \p_m\big(\alpha^{t,m}_s(y)\big)(\wt{y});\\
\p_y\big(\alpha^{t,m}_s(y)\big)=&\ \big(\p_x  \alpha\big)^{t,m}_s(y)\cdot\p_y\big(X^{t,m}_s(y)\big)+\big(\p_z \alpha\big)^{t,m}_s(y)\cdot\p_y\big(Z^{t,m}_s(y)\big);\\
\p_y \big(f^{t,m}_s(y)\big)=&\ \big(\p_x f\big)^{t,m}_s(y)\cdot\p_y \big(X^{t,m}_s(y)\big)+ \big(\p_\alpha f\big)^{t,m}_s(y)\cdot \p_y \big(\alpha^{t,m}_s(y)\big).
\end{aligned}\label{calculus2}
\end{equation}\normalsize
Note that, by using the second last equation of \eqref{MPIT} and Fubini's theorem, we have \small
\begin{align}\no
&\int_{\R^{d_x}}\p_t\big(g^{t,m}_s(x)\big)dm(x)\\\no
=&\ \int_{\R^{d_x}}\bigg(\big(\p_x g\big)^{t,m}_s(x)\cdot\p_t\big(X^{t,m}_s(x)\big)+\int_{\R^{d_x}}\big(\p_\mu g\big)^{t,m}_s(x,\wt{x})\cdot\p_t\big(X^{t,m}_s(\wt{x})\big)dm(\wt{x})+ \p_t\big(\alpha^{t,m}_s(x)\big)\cdot \big(\p_\alpha g\big)^{t,m}_s(x)\bigg)dm(x)\\\no
=&\ \int_{\R^{d_x}}\p_t\big(\alpha^{t,m}_s(x)\big)\cdot \big(\p_\alpha g\big)^{t,m}_s(x) dm(x)\\\label{p_tg}
&-\int_{\R^{d_x}}\bigg(\dfrac{d}{ds}Z^{t,m}_s(x)+\big(\p_x f\big)^{t,m}_s(x)\cdot Z^{t,m}_s(x)+\int_{\R^{d_x}}\big(\p_\mu f\big)^{t,m}_s(\wt{x},x)\cdot Z^{t,m}_s(\wt{x})dm(\wt{x})\bigg)\cdot \p_t\big(X^{t,m}_s(x)\big) dm(x),
\end{align}\normalsize
and
\small\begin{align}\nonumber
&\p_m \left(\int_{\R^{d_x}} g^{t,m}_s(\wt{x})d m(\wt{x})\right)(x)={\color{blue}\int_{\R^{d_x}} \p_m \big(g^{t,m}_s(\wt{x})\big)(x)d m(\wt{x})}+{\color{red}\p_x \big(g^{t,m}_s(x)\big)}\\\nonumber
=&\ {\color{blue}\int_{\R^{d_x}} \bigg(\big(\p_x g\big)^{t,m}_s(\wt{x})\cdot\p_m\big(X^{t,m}_s(\wt{x})\big)(x)+\big(\p_\mu g\big)^{t,m}_s(\wt{x},x)\cdot \p_x\big(X^{t,m}_s(x)\big)}\\\no
& \ \ \ \ \ \ \ \ {\color{blue}+\int_{\R^{d_x}} \big(\p_\mu g\big)^{t,m}_s(\wt{x},\wh{x})\cdot\p_m\big(X^{t,m}_s(\wh{x})\big)(x) dm(\wh{x})+\big(\p_\alpha g\big)^{t,m}_s(\wt{x})\cdot \p_m\big(\alpha^{t,m}_s(\wt{x})\big)(x)\bigg)dm(\wt{x})}\\\no
&{\color{red}+\big(\p_x g\big)^{t,m}_s(x)\cdot\p_x \big(X^{t,m}_s(x)\big)+ \big(\p_\alpha g\big)^{t,m}_s(x)\cdot \p_x \big(\alpha^{t,m}_s(x)\big)}\\\no
=&\ {\color{blue}\int_{\R^{d_x}} \big(\p_x g\big)^{t,m}_s(\wh{x})\cdot\p_m\big(X^{t,m}_s(\wh{x})\big)(x)dm(\wh{x})+\int_{\R^{d_x}} \int_{\R^{d_x}}\big(\p_\mu g\big)^{t,m}_s(\wt{x},\wh{x}) dm(\wt{x})\cdot\p_m\big(X^{t,m}_s(\wh{x})\big)(x)dm(\wh{x})}\\\no
&+{\color{red}\big(\p_x g\big)^{t,m}_s(x)\cdot\p_x \big(X^{t,m}_s(x)\big)}+{\color{blue}\int_{\R^{d_x}}\big(\p_\mu g\big)^{t,m}_s(\wt{x},x)dm(\wt{x})\cdot \p_x\big(X^{t,m}_s(x)\big)}\\\no
&+{\color{red}\big(\p_\alpha g\big)^{t,m}_s(x)\cdot \p_x \big(\alpha^{t,m}_s(x)\big)}+{\color{blue}\int_{\R^{d_x}} \big(\p_\alpha g\big)^{t,m}_s(\wt{x})\cdot \p_m\big(\alpha^{t,m}_s(\wt{x})\big)(x)dm(\wt{x})}\\\no
=&\ -\int_{\R^{d_x}}\bigg(\dfrac{d}{ds}Z^{t,m}_s(\wh{x})+\big(\p_x f\big)^{t,m}_s(\wh{x})\cdot Z^{t,m}_s(\wh{x})+\int_{\R^{d_x}}\big(\p_\mu f\big)^{t,m}_s(\wt{x},\wh{x})\cdot Z^{t,m}_s(\wt{x})dm(\wt{x})\bigg)\cdot \p_m\big(X^{t,m}_s(\wh{x})\big)(x)dm(\wh{x})\\\no
&-\bigg(\dfrac{d}{ds}Z^{t,m}_s(x)+\big(\p_x f\big)^{t,m}_s(x)\cdot Z^{t,m}_s(x)+\int_{\R^{d_x}}\big(\p_\mu f\big)^{t,m}_s(\wt{x},x)\cdot Z^{t,m}_s(\wt{x})dm(\wt{x})\bigg)\cdot \p_x\big(X^{t,m}_s(x)\big)\\\label{p_mg}
&+\big(\p_\alpha g\big)^{t,m}_s(x)\cdot \p_x \big(\alpha^{t,m}_s(x)\big)+\int_{\R^{d_x}} \big(\p_\alpha g\big)^{t,m}_s(\wt{x})\cdot \p_m\big(\alpha^{t,m}_s(\wt{x})\big)(x)dm(\wt{x}),
\end{align}\normalsize
where the last equality holds by using the backward equation of \eqref{MPIT}.
We can write \eqref{MPIT} alternatively as a system of integral equations: 
\small\begin{align}
\begin{cases}       
X^{t,m}_s(x)= x+\displaystyle\int_t^s f^{t,m}_\tau(x)d\tau,\ \text{ for } s\in[t,T]; \\
Z^{t,m}_s(x)=\displaystyle{\color{red}\int_{\R^{d_x}} \big(\p_\mu k\big)^{t,m}_T(\wt{x},x)dm(\wt{x})}+{\color{blue}\big(\p_x k\big)^{t,m}_T(x)}\\
\ \ \ \ \ \ \ \ \ \ \ \ \ \ +\displaystyle\int_s^T\bigg( {\color{orange}\int_{\R^{d_x}}\bigg(\big(\p_\mu f\big)^{t,m}_\tau(\wt{x},x)\cdot  Z^{t,m}_\tau(\wt{x})+ \big(\p_\mu g\big)^{t,m}_\tau(\wt{x},x)\bigg)dm(\wt{x})}\\
\ \ \ \ \ \ \ \ \ \ \ \ \ \ \ \ \ \ \ \ \ \ \ \ +{\color{purple}\big(\p_x f\big)^{t,m}_\tau(x)\cdot Z^{t,m}_\tau(x)+\big(\p_x g\big)^{t,m}_\tau(x)}\bigg)d\tau;
\end{cases}
\end{align}\normalsize
and by taking derivatives on both sides, we can see that $\Big(\p_t\big(X^{t,m}_s(x)\big),\p_t\big(Z^{t,m}_s(x)\big)\Big)$, $\Big(\p_m\big(X^{t,m}_s(x)\big),\\\p_m\big(Z^{t,m}_s(x)\big)\Big)$ and $\Big(\p_x\big(X^{t,m}_s(x)\big),\p_x\big(Z^{t,m}_s(x)\big)\Big)$ satisfy the following systems of \eqref{p_tX}, \eqref{p_mX} and \eqref{p_xX} respectively; the discussion on the local (in time) solvability and global (in time) solvability of these systems are given in Section \ref{sec:local} and Section \ref{sec:global}, respectively.\\
\scriptsize\begin{align}
\label{p_tX}
\left\{ \begin{aligned}
	\frac{d}{ds}\p_t\big(X^{t,m}_s(x)\big) =&\  \p_t\big(f^{t,m}_s(x)\big),\\
	\p_t\big(X^{t,m}_s(x)\big)\Big|_{s=t} =&\  -f^{t,m}_t(x),\\
	-\frac{d}{ds}\p_t\big(Z^{t,m}_s(x)\big)  =&\  {\color{orange}\int_{\R^{d_x}}\p_t\big(X^{t,m}_s(\wt{x})\big)^\top\bigg(\big(\p_x\p_\mu f\big)^{t,m}_s(\wt{x},x)\cdot  Z^{t,m}_s(\wt{x})+ \big(\p_x\p_\mu g\big)^{t,m}_s(\wt{x},x)\bigg)dm(\wt{x})}\\
&+{\color{orange}\p_t\big(X^{t,m}_s(x)\big)^\top\int_{\R^{d_x}}\bigg(\big(\p_{\wt{x}}\p_\mu f\big)^{t,m}_s(\wt{x},x)\cdot  Z^{t,m}_s(\wt{x})+ \big(\p_{\wt{x}}\p_\mu g\big)^{t,m}_s(\wt{x},x)\bigg)dm(\wt{x})}\\
&+{\color{orange}\int_{\R^{d_x}}\int_{\R^{d_x}}\p_t\big(X^{t,m}_s(\wh{x})\big)^\top\bigg(\big(\p_\mu\p_\mu f\big)^{t,m}_s(\wt{x},x,\wh{x})\cdot  Z^{t,m}_s(\wt{x})+ \big(\p_\mu\p_\mu g\big)^{t,m}_s(\wt{x},x,\wh{x})\bigg)dm(\wh{x})dm(\wt{x})}\\
&+{\color{orange}\int_{\R^{d_x}}\p_t\big(\alpha^{t,m}_s(\wt{x})\big)^\top\bigg(\big(\p_\alpha\p_\mu f\big)^{t,m}_s(\wt{x},x)\cdot  Z^{t,m}_s(\wt{x})+ \big(\p_\alpha\p_\mu g\big)^{t,m}_s(\wt{x},x)\bigg)dm(\wt{x})}\\
&+{\color{orange}\int_{\R^{d_x}} \big(\p_\mu f\big)^{t,m}_s(\wt{x},x)\cdot  \p_t\big(Z^{t,m}_s(\wt{x})\big)dm(\wt{x})}\\
&+{\color{purple}\p_t\big(X^{t,m}_s(x)\big)^\top\bigg(\big(\p_x\p_x f\big)^{t,m}_s(x)\cdot Z^{t,m}_s(x)+\big(\p_x\p_x g\big)^{t,m}_s(x)\bigg)}\\
&+{\color{purple}\int_{\R^{d_x}}\p_t\big(X^{t,m}_s(\wt{x})\big)^\top\bigg(\big(\p_\mu\p_x f\big)^{t,m}_s(x,\wt{x})\cdot Z^{t,m}_s(x)+\big(\p_\mu\p_x g\big)^{t,m}_s(x,\wt{x})\bigg)dm(\wt{x})}\\
&+{\color{purple}\p_t\big(\alpha^{t,m}_s(x)\big)^\top\bigg(\big(\p_\alpha\p_x f\big)^{t,m}_s(x)\cdot Z^{t,m}_s(x)+\big(\p_\alpha\p_x g\big)^{t,m}_s(x)\bigg)}+{\color{purple}\big(\p_x f\big)^{t,m}_s(x)\cdot \p_t\big(Z^{t,m}_s(x)\big)},\\
\p_t\big(Z^{t,m}_T(x)\big)=&\  \displaystyle{\color{red}\int_{\R^{d_x}} \big(\p_x\p_\mu k\big)^{t,m}_T(\wt{x},x)\cdot \p_t\big(X^{t,m}_T(\wt{x})\big)dm(\wt{x})+\displaystyle\int_{\R^{d_x}} \big(\p_{\wt{x}}\p_\mu k\big)^{t,m}_T(\wt{x},x)dm(\wt{x})\cdot \p_t\big(X^{t,m}_T(x)\big)}\\
&+\displaystyle{\color{red}\int_{\R^{d_x}} \int_{\R^{d_x}}\big(\p_\mu\p_\mu k\big)^{t,m}_T(\wt{x},x,\wh{x})\cdot \p_t\big(X^{t,m}_T(\wh{x})\big)dm(\wh{x})dm(\wt{x})}+{\color{blue}\big(\p_x\p_x k\big)^{t,m}_T(x)\cdot \p_t\big(X^{t,m}_T(x)\big)}\\
&+\displaystyle{\color{blue}\int_{\R^{d_x}}\big(\p_\mu\p_x k\big)^{t,m}_T(x,\wt{x})\cdot \p_t\big(X^{t,m}_T(\wt{x})\big)dm(\wt{x})};
\end{aligned} \right.
\end{align}\normalsize
\scriptsize\begin{align}
\label{p_mX}
\left\{ \begin{aligned}
	\frac{d}{ds}\p_m\big(X^{t,m}_s(x)\big)(y) =&\  \p_m \big(f^{t,m}_s(x)\big)(y),\\
	\p_m \big(X^{t,m}_t(x)\big)(y) =&\ 0,\\
-\frac{d}{ds}\p_m\big(Z^{t,m}_s(x)\big)(y)=&\  {\color{orange}\int_{\R^{d_x}}\p_m\big(X^{t,m}_s(\wt{x})\big)(y)^\top\bigg(\big(\p_x\p_\mu f\big)^{t,m}_s(\wt{x},x)\cdot  Z^{t,m}_s(\wt{x})+ \big(\p_x\p_\mu g\big)^{t,m}_s(\wt{x},x)\bigg)dm(\wt{x})}\\
&+{\color{orange}\p_m\big(X^{t,m}_s(x)\big)(y)^\top\int_{\R^{d_x}}\bigg(\big(\p_{\wt{x}}\p_\mu f\big)^{t,m}_s(\wt{x},x)\cdot  Z^{t,m}_s(\wt{x})+ \big(\p_{\wt{x}}\p_\mu g\big)^{t,m}_s(\wt{x},x)\bigg)dm(\wt{x})}\\
&+{\color{orange}\int_{\R^{d_x}}\int_{\R^{d_x}}\p_m\big(X^{t,m}_s(\wh{x})\big)(y)^\top\bigg(\big(\p_\mu\p_\mu f\big)^{t,m}_s(\wt{x},x,\wh{x})\cdot  Z^{t,m}_s(\wt{x})+ \big(\p_\mu\p_\mu g\big)^{t,m}_s(\wt{x},x,\wh{x})\bigg)dm(\wh{x})dm(\wt{x})}\\
&+{\color{orange}\int_{\R^{d_x}}\p_y\big(X^{t,m}_s(y)\big)^\top\bigg(\big(\p_\mu\p_\mu f\big)^{t,m}_s(\wt{x},x,y)\cdot  Z^{t,m}_s(\wt{x})+ \big(\p_\mu\p_\mu g\big)^{t,m}_s(\wt{x},x,y)\bigg)dm(\wt{x})}\\
&+{\color{orange}\int_{\R^{d_x}}\p_m\big(\alpha^{t,m}_s(\wt{x})\big)(y)^\top\bigg(\big(\p_\alpha\p_\mu f\big)^{t,m}_s(\wt{x},x)\cdot  Z^{t,m}_s(\wt{x})+ \big(\p_\alpha\p_\mu g\big)^{t,m}_s(\wt{x},x)\bigg)dm(\wt{x})}\\
&+{\color{orange}\int_{\R^{d_x}} \big(\p_\mu f\big)^{t,m}_s(\wt{x},x)\cdot  \p_m\big(Z^{t,m}_s(\wt{x})\big)(y)dm(\wt{x})}\\
&+{\color{purple}\p_m\big(X^{t,m}_s(x)\big)(y)^\top\bigg(\big(\p_x\p_x f\big)^{t,m}_s(x)\cdot Z^{t,m}_s(x)+\big(\p_x\p_x g\big)^{t,m}_s(x)\bigg)}\\
&+{\color{purple}\int_{\R^{d_x}}\p_m\big(X^{t,m}_s(\wt{x})\big)(y)^\top\bigg(\big(\p_\mu\p_x f\big)^{t,m}_s(x,\wt{x})\cdot Z^{t,m}_s(x)+\big(\p_\mu\p_x g\big)^{t,m}_s(x,\wt{x})\bigg)dm(\wt{x})}\\
&+{\color{purple}\p_y\big(X^{t,m}_s(y)\big)^\top\bigg(\big(\p_\mu\p_x f\big)^{t,m}_s(x,y)\cdot Z^{t,m}_s(x)+\big(\p_\mu\p_x g\big)^{t,m}_s(x,y)\bigg)}\\
&+{\color{purple}\p_m\big(\alpha^{t,m}_s(x)\big)(y)^\top\bigg(\big(\p_\alpha\p_x f\big)^{t,m}_s(x)\cdot Z^{t,m}_s(x)+\big(\p_\alpha\p_x g\big)^{t,m}_s(x)\bigg)}+{\color{purple}\big(\p_x f\big)^{t,m}_s(x)\cdot \p_m\big(Z^{t,m}_s(x)\big)(y)}\\
\p_m\big(Z^{t,m}_T(x)\big)(y)=&\  \displaystyle{\color{red}\int_{\R^{d_x}} \big(\p_x\p_\mu k\big)^{t,m}_T(\wt{x},x)\cdot \p_m\big(X^{t,m}_T(\wt{x})\big)(y)dm(\wt{x})+\displaystyle\int_{\R^{d_x}} \big(\p_{\wt{x}}\p_\mu k\big)^{t,m}_T(\wt{x},x)dm(\wt{x})\cdot \p_m\big(X^{t,m}_T(x)\big)(y)}\\
&+\displaystyle{\color{red}\int_{\R^{d_x}} \int_{\R^{d_x}}\big(\p_\mu\p_\mu k\big)^{t,m}_T(\wt{x},x,\wh{x})\cdot \p_m\big(X^{t,m}_T(\wh{x})\big)(y)dm(\wh{x})dm(\wt{x})}\\
&+\displaystyle{\color{red}\int_{\R^{d_x}} \big(\p_\mu\p_\mu k\big)^{t,m}_T(\wt{x},x,y)\cdot \p_y\big(X^{t,m}_T(y)\big)dm(\wt{x})}+{\color{blue}\big(\p_x\p_x k\big)^{t,m}_T(x)\cdot \p_m\big(X^{t,m}_T(x)\big)(y)}\\
&+\displaystyle{\color{blue}\int_{\R^{d_x}}\big(\p_\mu\p_x k\big)^{t,m}_T(x,\wt{x})\cdot \p_m\big(X^{t,m}_T(\wt{x})\big)(y)dm(\wt{x})+\big(\p_\mu\p_x k\big)^{t,m}_T(x,y)\cdot \p_y\big(X^{t,m}_T(y)\big)};
\end{aligned} \right.
\end{align}\normalsize
and
\scriptsize\begin{align}
\label{p_xX}
\left\{ \begin{aligned}
	\frac{d}{ds}\p_x\big(X^{t,m}_s(x)\big) =&\  \p_x \big(f^{t,m}_s(x)\big),\\
	\p_x \big(X^{t,m}_t(x)\big) =&\ \mathcal{I}_{d_x\times d_x},\\
	-\frac{d}{ds}\p_x \big(Z^{t,m}_s(x)\big)  =&\ {\color{orange}\p_x\big(X^{t,m}_s(x)\big)^\top\int_{\R^{d_x}}\bigg(\big(\p_{\wt{x}}\p_\mu f\big)^{t,m}_s(\wt{x},x)\cdot  Z^{t,m}_s(\wt{x})+ \big(\p_{\wt{x}}\p_\mu g\big)^{t,m}_s(\wt{x},x)\bigg)dm(\wt{x})}\\
&+{\color{purple}\big(\p_x f\big)^{t,m}_s(x)\cdot \p_x\big(Z^{t,m}_s(x)\big)+\p_x\big(X^{t,m}_s(x)\big)^\top\bigg(\big(\p_x\p_x f\big)^{t,m}_s(x)\cdot Z^{t,m}_s(x)+\big(\p_x\p_x g\big)^{t,m}_s(x)\bigg)}\\
&+{\color{purple}\p_x\big(\alpha^{t,m}_s(x)\big)^\top\bigg(\big(\p_\alpha\p_x f\big)^{t,m}_s(x)\cdot Z^{t,m}_s(x)+\big(\p_\alpha\p_x g\big)^{t,m}_s(x)\bigg)}\\
\p_x\big(Z^{t,m}_T(x)\big)=&\  \displaystyle{\color{red}\int_{\R^{d_x}} \big(\p_{\wt{x}}\p_\mu k\big)^{t,m}_T(\wt{x},x)dm(\wt{x})\cdot \p_x\big(X^{t,m}_T(x)\big)}+{\color{blue}\big(\p_x\p_x k\big)^{t,m}_T(x)\cdot \p_x\big(X^{t,m}_T(x)\big)};
\end{aligned} \right.
\end{align}\normalsize
where $\mathcal{I}_{d_x\times d_x}$ is the $d_x\times d_x$ identity matrix.
Note that, $v(t,m)$ was defined in \eqref{definev}, and we next show that this $v(t,m)$ satisfies the Bellman equation \eqref{nomin_1}. Since $X^{T,m}_T(x)\equiv x$, by using \eqref{definev} and Fubini's theorem, we have $v(T,m)=\int_{\R^{d_x}} k(x,m)dm(x)$ and 
\begin{align}\nonumber
&\int_{\R^{d_x}} k(x,m)dm(x)-v(t,m)= \int_t^T \p_sv(s,m)ds \\\no
=&\  \int_t^T \p_s\bigg(\int_{\R^{d_x}}k^{s,m}_T(x)dm(x)+\int_s^T\int_{\R^{d_x}}g^{s,m}_\tau(x)dm(x)d\tau\bigg)ds\\\nonumber
=&\  \int_t^T \bigg(\int_{\R^{d_x}}{\color{blue}\p_s\big(k^{s,m}_T(x)\big)dm(x)}+\int_s^T{\color{red}\int_{\R^{d_x}}\p_s\big(g^{s,m}_\tau(x)\big)dm(x)}d\tau-\int_{\R^{d_x}}g^{s,m}_s(x)dm(x)\bigg)ds\\\nonumber
=&\  \int_t^T \bigg(\int_{\R^{d_x}} {\color{blue}\bigg(\int_{\R^{d_x}} \big(\p_\mu k\big)^{s,m}_T(\wt{x},x)dm(\wt{x})+\big(\p_x k\big)^{s,m}_T(x)\bigg)\cdot \p_s\big(X^{s,m}_T(x)\big)}dm(x)-\int_{\R^{d_x}}g^{s,m}_s(x)dm(x)\bigg)ds\\\nonumber
&+ \int_t^T \int_s^T\bigg({\color{red}\int_{\R^{d_x}}\p_s\big(\alpha^{s,m}_\tau(x)\big)\cdot \big(\p_\alpha g\big)^{s,m}_\tau(x) dm(x)}\\\no
&\ \ \ \ \ \ \ \ \ \ \ \ \ \ \ \ {\color{red}-\int_{\R^{d_x}}\bigg(\dfrac{d}{d\tau}Z^{s,m}_\tau(x)+\big(\p_x f\big)^{s,m}_\tau(x)\cdot Z^{s,m}_\tau(x)}\\\label{eq_52}
&\ \ \ \ \ \ \ \ \ \ \ \ \ \ \ \ \ \ \ \ \ \ \ \ \ \ \ \ \ \ {\color{red}+\int_{\R^{d_x}}\big(\p_\mu f\big)^{s,m}_\tau(\wt{x},x)\cdot Z^{s,m}_\tau(\wt{x})dm(\wt{x})\bigg)\cdot \p_s\big(X^{s,m}_\tau(x)\big) dm(x)}\bigg)d\tau ds,
\end{align}
where we use \eqref{p_tg} to obtain the red terms of the last equality.
To simplify \eqref{eq_52}, we note that, by using Fubini's theorem, the last equation in \eqref{MPIT} and \eqref{p_tX}, the following integrand in \eqref{eq_52} can be rewritten as
\small\begin{align}\no
&\int_s^T \int_{\R^{d_x}}\dfrac{d}{d\tau}Z^{s,m}_\tau(x)\cdot \p_s\big(X^{s,m}_\tau(x)\big) dm(x)d\tau=\int_{\R^{d_x}}\int_s^T \dfrac{d}{d\tau}Z^{s,m}_\tau(x)\cdot \p_s\big(X^{s,m}_\tau(x)\big)d\tau dm(x)\\\no
=&\ \int_{\R^{d_x}} {\color{blue}Z^{s,m}_T(x)\cdot \p_s\big(X^{s,m}_T(x)\big)}dm(x)-\int_{\R^{d_x}} Z^{s,m}_s(x)\cdot \p_s\big(X^{s,m}_s(x)\big)dm(x)\\\no
&-\int_{\R^{d_x}}\int_s^T {\color{purple}Z^{s,m}_\tau(x)\cdot \dfrac{d}{d\tau}\p_s\big(X^{s,m}_\tau(x)\big)}d\tau dm(x)\\\no
=&\ \int_{\R^{d_x}} {\color{blue}\bigg(\int_{\R^{d_x}} \big(\p_\mu k\big)^{s,m}_T(\wt{x},x)dm(\wt{x})+\big(\p_x k\big)^{s,m}_T(x)\bigg)\cdot \p_s\big(X^{s,m}_T(x)\big)}dm(x)+\int_{\R^{d_x}} Z^{s,m}_s(x)\cdot f^{s,m}_s(x)dm(x)\\\no
&-\int_{\R^{d_x}}\int_s^T {\color{purple}Z^{s,m}_\tau(x)\cdot \p_s\big(f^{s,m}_\tau(x)\big)}d\tau dm(x)\\\no
=&\ \int_{\R^{d_x}}{\color{blue} \bigg(\int_{\R^{d_x}} \big(\p_\mu k\big)^{s,m}_T(\wt{x},x)dm(\wt{x})+\big(\p_x k\big)^{s,m}_T(x)\bigg)\cdot \p_s\big(X^{s,m}_T(x)\big)}dm(x)+\int_{\R^{d_x}} Z^{s,m}_s(x)\cdot f^{s,m}_s(x)\big)dm(x)\\\no
&-\int_{\R^{d_x}}\int_s^T {\color{purple}Z^{s,m}_\tau(x)\cdot \bigg(\big(\p_x f\big)^{s,m}_\tau(x)\cdot\p_s\big(X^{s,m}_\tau(x)\big)+\displaystyle\int_{\R^{d_x}}\big(\p_\mu f\big)^{s,m}_\tau(x,\wt{x})\cdot\p_s\big(X^{s,m}_\tau(\wt{x})\big)dm(\wt{x})}\\\label{eq_12_28}
&\ \ \ \ \ \ \ \ \ \ \ \ \ \ \ \ \ \ \ \ \ \ \ \ \ \ \ \ \ \ {\color{purple}+ \p_s\big(\alpha^{s,m}_\tau(x)\big)\cdot \big(\p_\alpha f\big)^{s,m}_\tau(x)\bigg)}d\tau dm(x),
\end{align}\normalsize
where we use  the second equation in the chain rule of \eqref{calculus1} to obtain the purple terms in the last equality. 
Substituting \eqref{eq_12_28} into \eqref{eq_52} so as to cancel some similar terms, one has
\begin{align*}
&\int_{\R^{d_x}} k(x,m)dm(x)-v(t,m)\\
=&\ \int_t^T \int_s^T\int_{\R^{d_x}}\bigg(\p_s\big(\alpha^{s,m}_\tau(x)\big)\cdot \Big(\big(\p_\alpha g\big)^{s,m}_\tau(x)+\big(\p_\alpha f\big)^{s,m}_\tau(x)\cdot Z^{s,m}_\tau(x)\Big)\bigg) dm(x)d\tau ds\\
&-\int_t^T \int_{\R^{d_x}}\bigg(Z^{s,m}_s(x)\cdot f^{s,m}_s(x)+g^{s,m}_s(x)\bigg)dm(x)ds\\
=&\ -\int_t^T \int_{\R^{d_x}}\bigg(Z^{s,m}_s(x)\cdot f^{s,m}_s(x)+g^{s,m}_s(x)\bigg)dm(x)ds,
\end{align*}
where we use the first order condition \eqref{usolve} in the last equality,
from which, by differentiation, we deduce that
\small\begin{align}\label{p_t_v}
\begin{cases}
\p_tv(t,m)+\displaystyle\int_{\R^{d_x}}\bigg(Z^{t,m}_t(x)\cdot f(x,m,\alpha(x,m,Z^{t,m}_t(x)))+g(x,m,\alpha(x,m,Z^{t,m}_t(x)))\bigg)dm(x)=0,\\
v(T,m)=\displaystyle\int_{\R^{d_x}} k(x,m)dm(x).
\end{cases}
\end{align}
Now, to show that \eqref{p_t_v} is equivalent to \eqref{nomin_1}, it remains to establish that $\p_m v (t,m)(x)= Z^{t,m}_t(x)$.
To this end, by differentiating \eqref{definev} with respect to $m$ and using \eqref{p_mg}, one has
\footnotesize
\begin{align}\nonumber
&\p_m v (t,m)(x)={\color{blue}\p_x \big(k^{t,m}_T(x)\big)}+  {\color{red}\int_{\R^{d_x}} \p_m \big(k^{t,m}_T(\wt{x})\big)(x)dm(\wt{x})}+{\color{orange}\int_t^T\p_m \left(\int_{\R^{d_x}} g^{t,m}_s(\wt{x})d m(\wt{x})\right)(x)ds}\\\nonumber
=&\ {\color{blue}\big(\p_x k\big)^{t,m}_T(x)\cdot \p_x\big(X^{t,m}_T(x)\big)}+{\color{red}\int_{\R^{d_x}}\big(\p_x k\big)^{t,m}_T(\wt{x})\cdot \p_m \big(X^{t,m}_T(\wt{x})\big)(x) dm(\wt{x})}\\\no
&+  {\color{red}\int_{\R^{d_x}} \int_{\R^{d_x}}\big(\p_\mu k\big)^{t,m}_T(\wt{x},\wh{x})\cdot \p_m \big(X^{t,m}_T(\wh{x})\big)(x)dm(\wh{x})dm(\wt{x})+\int_{\R^{d_x}}\big(\p_\mu k\big)^{t,m}_T(\wt{x},x)dm(\wt{x})\cdot \p_x \big(X^{t,m}_T(x)\big)}\\\nonumber
&+{\color{orange}\int_t^T\bigg(-\int_{\R^{d_x}}\bigg(\dfrac{d}{ds}Z^{t,m}_s(\wh{x})+\big(\p_x f\big)^{t,m}_s(\wh{x})\cdot Z^{t,m}_s(\wh{x})+\int_{\R^{d_x}}\big(\p_\mu f\big)^{t,m}_s(\wt{x},\wh{x})\cdot Z^{t,m}_s(\wt{x})dm(\wt{x})\bigg)\cdot \p_m\big(X^{t,m}_s(\wh{x})\big)(x)dm(\wh{x})}\\\no
&\ \ \ \ \ \ \ \ \ \ \ \ {\color{orange}-\bigg(\dfrac{d}{ds}Z^{t,m}_s(x)+\big(\p_x f\big)^{t,m}_s(x)\cdot Z^{t,m}_s(x)+\int_{\R^{d_x}}\big(\p_\mu f\big)^{t,m}_s(\wt{x},x)\cdot Z^{t,m}_s(\wt{x})dm(\wt{x})\bigg)\cdot \p_x\big(X^{t,m}_s(x)\big)}\\\label{p_mv}
&\ \ \ \ \ \ \ \ \ \ \ \ {\color{orange}+\big(\p_\alpha g\big)^{t,m}_s(x)\cdot \p_x \big(\alpha^{t,m}_s(x)\big)+\int_{\R^{d_x}} \big(\p_\alpha g\big)^{t,m}_s(\wt{x})\cdot \p_m\big(\alpha^{t,m}_s(\wt{x})\big)(x)dm(\wt{x})\bigg) ds.}
\end{align}\normalsize
To simplify \eqref{p_mv}, we note that, by using Fubini's theorem, the last terminal condition in \eqref{MPIT} and \eqref{p_mX}, the following integrand in \eqref{p_mv} can be rewritten as
\scriptsize\begin{align}\no
&\int_t^T\int_{\R^{d_x}}\fr{d}{ds}Z^{t,m}_s(\wh{x})\cdot \p_m\big(X^{t,m}_s(\wh{x})\big)(x)dm(\wh{x})ds\\\no
=&\ \int_{\R^{d_x}}\int_t^T\fr{d}{ds}Z^{t,m}_s(\wh{x})\cdot \p_m\big(X^{t,m}_s(\wh{x})\big)(x)dsdm(\wh{x})\\\no
=&\ \int_{\R^{d_x}} \Bigg({\color{blue}Z^{t,m}_T(\wh{x})\cdot \p_m \big(X^{t,m}_T(\wh{x})\big)(x)}-Z^{t,m}_t(\wh{x})\cdot \p_m \big(X^{t,m}_t(\wh{x})\big)(x)-{\color{red}\int_t^T Z^{t,m}_s(\wh{x})\cdot \fr{d}{ds}\p_m \big(X^{t,m}_s(\wh{x})\big)(x) ds}\Bigg)dm(\wh{x})\\\no
=&\ \int_{\R^{d_x}} \Bigg({\color{blue}\int_{\R^{d_x}}\big(\p_\mu k\big)^{t,m}_T(\wt{x},\wh{x})dm(\wt{x})\cdot \p_m \big(X^{t,m}_T(\wh{x})\big)(x)+\big(\p_x k\big)^{t,m}_T(\wh{x})\cdot \p_m \big(X^{t,m}_T(\wh{x})\big)(x)}\\\no
&\ \ \ \ \ \ \ \ -{\color{red}\int_t^T Z^{t,m}_s(\wh{x})\cdot \p_m \big(f^{t,m}_s(\wh{x})\big)(x) ds }\Bigg)dm(\wh{x})\\\no
=&\ \int_{\R^{d_x}} \Bigg({\color{blue}\int_{\R^{d_x}}\big(\p_\mu k\big)^{t,m}_T(\wt{x},\wh{x})dm(\wt{x})\cdot \p_m \big(X^{t,m}_T(\wh{x})\big)(x)+\big(\p_x k\big)^{t,m}_T(\wh{x})\cdot \p_m \big(X^{t,m}_T(\wh{x})\big)(x)}\\\no
&\ \ \ \ \ \ \ \ -{\color{red}\int_t^T Z^{t,m}_s(\wh{x})\cdot \bigg(\big(\p_x f\big)^{t,m}_s(\wh{x})\cdot\p_m\big(X^{t,m}_s(\wh{x})\big)(x)+\big(\p_\mu f\big)^{t,m}_s(\wh{x},x)\cdot\p_x\big(X^{t,m}_s(x)\big) dm(\wt{x})}\\\label{eq_12_31}
&\ \ \ \ \ \ \ \ \ \ \ \ \ \ \ \ \ \ \ \ \ \ \ \ \ \ \ \ \ \ \ \ {\color{red}+\int_{\R^{d_x}} \big(\p_\mu f\big)^{t,m}_s(\wh{x},\wt{x})\cdot\p_m\big(X^{t,m}_s(\wt{x})\big)(x) dm(\wt{x})+ \big(\p_\alpha f\big)^{t,m}_s(\wh{x})\cdot \p_m\big(\alpha^{t,m}_s(\wh{x})\big)(x)\bigg) ds} \Bigg)dm(\wh{x}),
\end{align}\normalsize
where we use the second equation in the chain rule of \eqref{calculus2} to obtain the red terms in the last equality; and, by using the last terminal condition in \eqref{MPIT} and \eqref{p_xX},
\begin{align}\no
&\int_t^T \fr{d}{ds}Z^{t,m}_s(x)\cdot\p_x  \big(X^{t,m}_s(x)\big) ds\\\no
=&\ {\color{blue}Z^{t,m}_T(x)\cdot\p_x  \big(X^{t,m}_T(x)\big)}-Z^{t,m}_t(x)\cdot\p_x  \big(X^{t,m}_t(x)\big)-{\color{red}\int_t^T Z^{t,m}_s(x)\cdot \fr{d}{ds}\p_x\big(X^{t,m}_s(x)\big) ds}\end{align}
\begin{align}\no
=&\ {\color{blue}\int_{\R^{d_x}}\big(\p_\mu k\big)^{t,m}_T(\wt{x},x)dm(\wt{x})\cdot \p_x\big(X^{t,m}_T(x)\big)+\big(\p_x k\big)^{t,m}_T(x)\cdot \p_x \big(X^{t,m}_T(x)\big)}-Z^{t,m}_t(x)\\\no
&-{\color{red}\int_t^T Z^{t,m}_s(x)\cdot \p_x \big(f^{t,m}_s(x) \big)ds}\\\no
=&\ {\color{blue}\int_{\R^{d_x}}\big(\p_\mu k\big)^{t,m}_T(\wt{x},x)dm(\wt{x})\cdot \p_x\big(X^{t,m}_T(x)\big)+\big(\p_x k\big)^{t,m}_T(x)\cdot \p_x \big(X^{t,m}_T(x)\big)}-Z^{t,m}_t(x)\\\label{eq_12_32}
&-{\color{red}\int_t^T Z^{t,m}_s(x)\cdot \bigg(\big(\p_x f\big)^{t,m}_s(x)\cdot\p_x \big(X^{t,m}_s(x)\big)+ \big(\p_\alpha f\big)^{t,m}_s(x)\cdot \p_x \big(\alpha^{t,m}_s(x)\big)\bigg)ds},
\end{align}\normalsize
where we use the chain rule to obtain the red terms in the last equality. 
Substitute \eqref{eq_12_31} and \eqref{eq_12_32} into \eqref{p_mv} so as to cancel some similar terms, one has
\begin{align}\no
\p_m v (t,m)(x)=&\ Z^{t,m}_t(x)+\int_t^T \bigg(\big(\p_\alpha f\big)^{t,m}_s(x)\cdot Z^{t,m}_s(x)+\big(\p_\alpha g\big)^{t,m}_s(x)\bigg)\cdot \p_x \big(\alpha^{t,m}_s(x)\big)ds\\\no
&+\int_t^T \int_{\R^{d_x}}\bigg(\big(\p_\alpha f\big)^{t,m}_s(\wt{x})\cdot Z^{t,m}_s(\wt{x})+\big(\p_\alpha g\big)^{t,m}_s(\wt{x})\bigg)\cdot \p_m \big(\alpha^{t,m}_s(\wt{x})\big)(x)dm(\wt{x})ds\\\label{p_m_v}
=&\ Z^{t,m}_t(x),
\end{align}
where we use the first order condition \eqref{usolve} in the last equality so that both integrals vanish. Based on all of the derivations shown above, therefore $v$ defined by \eqref{definev}
is a classical solution to
the Bellman equation \eqref{nomin_1}.

\bibliography{Bib}
\bibliographystyle{plain}

\end{document}